\documentclass[11pt]{report} 

\usepackage{amsmath,amssymb,amsthm}
\usepackage{graphics}
\usepackage{enumerate}
\usepackage[dvips]{color}
\usepackage{version}
\usepackage{stmaryrd}
\usepackage[all]{xy}

\usepackage{todonotes}
\usepackage{titlesec}

\titleformat{\chapter}[block]{\Large\bfseries\filcenter}{\thechapter}{1em}{}
\titleformat{\section}[hang]{\normalsize\bfseries\filcenter}{\thesection}{1em}{}
\titleformat{\subsection}[hang]{\normalsize\bfseries\filcenter}{\thesubsection}{1em}{}

\usepackage[margin=1.5in]{geometry}  

\newtheorem{Thm}{Theorem}[chapter] 
\newtheorem{Lem}[Thm]{Lemma}
\newtheorem{Fac}[Thm]{Fact}
\newtheorem{Cor}[Thm]{Corollary}
\newtheorem{Cor.Conj}[Thm]{Corollary of Conjecture}

\newtheorem{Prop}[Thm]{Proposition}
\newtheorem{Prob}[Thm]{Problem}
\newtheorem{Conj}[Thm]{Conjecture}

\newtheorem{Claim}[Thm]{Claim}
\newtheorem{Ass}[Thm]{Assumption}
\newtheorem{THM}{Theorem}

\newtheorem{RTHM}{Theorem}

\newtheorem{MConj}[THM]{Main Conjecture}
\newtheorem{PROP}[THM]{Proposition}
\newtheorem{COR}[THM]{Corollary}
\newtheorem{REM}[THM]{Remark}

\theoremstyle{remark}
\newtheorem{Rem}[Thm]{Remark}
\newtheorem{Ex}[Thm]{Example}

\theoremstyle{definition}
\newtheorem{Def}[Thm]{Definition}

\newtheorem{Step}{Step}
\newtheorem{Stp}{Step}
\newtheorem{Stpp}{Step}  
\newtheorem{Const}[Thm]{Construction}

\newtheorem*{ntt}{Notations}

\newcommand{\Aut}{\mathop{\mathrm{Aut}}\nolimits}

\newcommand{\re}{\mathop{\mathrm{Re}}\nolimits}
\newcommand{\im}{\mathop{\mathrm{Im}}\nolimits}
\newcommand{\Lie}{\mathop{\mathrm{Lie}}\nolimits}

\newcommand{\R}{\ensuremath{\mathbb{R}}}
\newcommand{\C}{\ensuremath{\mathbb{C}}}
\newcommand{\G}{\ensuremath{\mathbb{G}}}
\newcommand{\Z}{\ensuremath{\mathbb{Z}}}
\newcommand{\Q}{\mathbb{Q}}

\begin{document}

\title
{\bf \Large{Collapsing K3 Surfaces, Tropical Geometry \\  
and Moduli Compactifications of Satake, \\ Morgan-Shalen Type}}

\author{Yuji Odaka, Yoshiki Oshima}
\date{}

\maketitle

\newpage

\begin{abstract}
This research monograph mainly aims to provide, or make a step toward, a theory of 
{\it canonical} and explicit compactifications of 
moduli spaces of K-trivial varieties, notably K3 surfaces. 
We require a geometric meaning i.e., to make the boundary of the compactified moduli space
parametrize limits of Ricci-flat K\"ahler metrics, which are often collapsed and 
also coincides with a posteriori defined ``tropicalized version'' of 
original varieties. 
From differential geometric perspective, this work 
provides a 
moduli-theoretic framework for the limiting behavior 
of Ricci-flat K\"ahler metrics. 
We emphasize that the compactification we use are {\it no longer varieties}. More precisely speaking, 
our actual discussion proceeds as follows. 

A general form of our desired compactifications  
can be obtained by applying 
a variant of {\it Morgan-Shalen} type compactifications \cite{MS} 
to 
locally Hermitian symmetric spaces.
Note that the compactifying procedure 
depends on either the theory of valuations or 
 that of dual intersection complexes, 
which often appear in the contexts of birational geometry, 
non-archimedean geometry or tropical geometry. 
Our first observation 
identifies the compactification with 
 {\it the Satake compactification for the adjoint representation}. 
 Note that it is different from the so-called Satake-Baily-Borel compactification, but is still 
a particular example of \cite{Sat2} and hence constructed 
in the context of Lie theory or symmetric spaces. 
 The existence of two descriptions of the same compactification are beneficially used in our work. 

\bigskip

\medskip

We 
apply this compactification to moduli spaces of compact hyperK\"ahler manifolds, which includes 
K3 surfaces but special case of K-trivial varieties. 
Our main conjectures state that the compactification 
 parametrizes the Gromov-Hausdorff limits 
 of the rescaled Ricci-flat (hyperK\"ahler) 
 metrics with fixed diameters.  A benefit of our conjectures is that, once it is confirmed, then for quite general sequences 
 which are not even necessarily ``maximally degenerating", we can determine the Gromov-Hausdorff limits.  
The main body of this monograph aims to give 
(partial) confirmation of the 
 conjectures in various cases. 
We first establish the easier case of 
 abelian varieties, refining the previous work of the first author. 
 Then, we proceed to the cases of K3 surfaces and hyperK\"ahler manifolds. 

 Our results for K3 surfaces provide, for example, a proof 
 of \cite[Conjecture 1]{KS} and \cite[Conjecture 6.2]{GW} for K3 surfaces, 
 which ask the limits of their one parameter maximally degenerating family. In this case, 
 the obtained Gromov-Hausdorff limits are metrized spheres $S^{2}$, appeared and studied in \cite{GW,KS}, 
 which have natural integral affine structures with singularities and are 
 often regarded as tropical analog of K3 surfaces. We also identify such limits along one parameter holomorphic families
 from the monodromy information. We also prove analogous results 
 for (not necessarily algebraic) K\"ahler K3 surfaces.  
 
 \vspace{5mm}
Trying to prove our conjectures for higher dimensional hyperK\"ahler manifolds, this monograph at least 
establishes the non-collapsing part of the 
moduli 
and make certain progress on algebro-geometric preparations for the collapsing part. 
More precise statement of the former progress states that, for a fixed deformation class of polarized irreducible symplectic manifolds, 
the set of all $\mathbb{Q}$-Gorenstein degenerations as polarized  symplectic varieties with ample 
\textit{$\Q$-line bundles} are bounded, and further that 
the corresponding partial compactification of the moduli space becomes an orthogonal locally symmetric variety as in K3 surfaces case. 
Finally, we discuss possible extension of our collapsing picture for general Ricci-flat K\"ahler metrics of $K$-trivial varieties.

In the text, we often review some basic prerequisites to make the documents relatively self-contained. 
\end{abstract}

\newpage
\begin{center}
{\bf \Large{Acknowledgments}}
\end{center}
Since the first author presented the main conjecture \ref{K3.Main.Conjecture} 
at Oxford in 
the Clay conference ``Algebraic geometry --- new and old---" in September 2016 
(resp., Conjecture \ref{K3.Main.Conjecture2} in the 
Mirror symmetry international conference at Kyoto university in December 2016), 
we have talked and discussed  
on the topic at Singapore, Levico Terme, Ann Arbor, 
Shanghai, Xiamen, Moscow, and various cities in Japan. 
We thank the organizers and the hosts for making such enjoyable chances which encouraged us. 

This work is posted on arXiv at arXiv:1810.07685 and its announcement \cite{OO} has been also published. 
Especially until then, our work benefited from the helpful comments 
from our colleagues and friends whom we thank very much; 
Yuki Arano, Kenji Hashimoto, Shouhei Honda, Hiroshi Iritani, Tetsushi Ito, Chen Jiang,  Ryoichi Kobayashi, Teruhisa Koshikawa, 
Radu Laza, Daisuke Matsushita, Giovanni Mongardi, Yosuke Morita, 
Shigeru Mukai, Hiraku Nakajima, 
Yoshinori Namikawa, Yuichi Nohara, Bernd Siebert, Song Sun, Cristiano Spotti, Kazushi Ueda, 
Ken-Ichi Yoshikawa. 
After our announcement paper \cite{OO}, 
we could also have some fruitful discussions with 
Lorenzo Foscolo, Song Sun, David Morrison, Jeff A. Viaclovsky in June 2018, on Conjecture \ref{K3.Main.Conjecture2} and possible 
relations with 
\cite{ChCh}, \cite{Fos}, \cite{HSVZ}, \cite{Morr}. 
Mattias Jonsson, Valentino Tosatti, Keita Goto, and 
Masafumi Hattori 
gave us helpful comments on the manuscript.  
Also the anonymous kind referees gave us various 
helpful comments and suggestions for the revision. We would like to thank them all gratefully. 
\medskip

\begin{center}
{\Large{Financial supports}}
\end{center}
During this research, the first author was partially supported by JSPS Grant-in-Aid for Young Scientists (B) No.\ 26870316, 
Grant-in-Aid (S), No.\ 16H06335 and Grant-in-Aid for Early-Career Scientists 
No.\ 18K13389. The second author was partially supported by JSPS Grant-in-Aid for Young Scientists (B) No.\ 16K17562.

\newpage
\tableofcontents


\newpage
\chapter{Introduction}

\section{Background}
This monograph is a sequel  to \cite{TGC.I, TGC.II} by the first author, 
which compactified the moduli of hyperbolic curves $M_{g}$ and 
of the principally polarized abelian varieties $A_{g}$, by 
attaching moduli of certain ``tropical varieties". 
In this joint work, while we mainly focus on the case of 
moduli of K3 surfaces, 
we introduce new general perspectives using 
the theory of symmetric spaces and the Morgan-Shalen type compactification, 
to reveal more explicit structures of the compactification. 

In general, as partially discussed in the introduction of {\it op.cit.}, 
we expect that for any moduli $\mathcal{M}$ of general polarized 
K\"ahler-Einstein varieties with \textit{non-positive} Ricci curvatures, 
there is a pair of similar (non-variety) compactifications and 
we hope our joint work here will be useful to study their structures, 
extending \cite{TGC.I, TGC.II}. 
The two types of compactifications are, firstly 
the Gromov-Hausdorff 
compactification $\overline{\mathcal{M}}^{\rm GH}$ with respect to rescaled K\"ahler-Einstein metrics 
of fixed diameters, and secondly its refinement which we call 
``tropical geometric compactification'' $\overline{\mathcal{M}}^{\rm T}$ 
whose boundary encodes more structure of the 
collapses rather than just metric structures. 
Here, the superscripts GH stands for Gromov-Hausdorff and 
T stands for tropical. 

Recall that the Gromov-Hausdorff distance is a distance between 
compact metric spaces (cf., e.g., \cite{BBI}), and 
there are notions of Gromov-Hausdorff limit of a sequence of 
compact metric spaces and the 
Gromov-Hausdorff topology on a (sub)set of compact metric spaces 
accordingly. 
Our Gromov-Hausdorff compactification $\overline{\mathcal{M}}^{\rm GH}$ above 
means to attach the set of all Gromov-Hausdorff limits on 
the boundary. The famous Gromov precompactness theorem (cf., e.g., 
\cite{BBI}) 
shows that any set of compact Riemannian manifolds of fixed dimension  with 
the Ricci curvature bounded from both sides and bounded diameters 
is precompact with respect to the Gromov-Hausdorff topology. 
In particular, it ensures the existence of Gromov-Hausdorff 
compactification in the Ricci-flat case. 

The name for $\overline{\mathcal{M}}^{\rm T}$ is inspired by the 
recent studies on the Strominger-Yau-Zaslow conjecture for the mirror symmetry (cf., e.g., \cite{SYZ, GW, KS, Gross}), in which 
the half-dimension collapses of Ricci-flat K\"ahler metrics along one parameter maximal degenerations are often regarded as ``tropicalization'' of the original Calabi-Yau manifolds and are expected to play a crucial role toward the geometric construction of ``mirror'' (cf., e.g., \cite{KS, GS.logI, GS11}). 
However, unfortunately, by now 
we only have case by case definitions of 
 tropical geometric compactifications depending on the particular classes of varieties. 
 Correspondingly we only use the term 
 ``tropical'' briefly only when we explain the background or motivation. 
 In these contexts, ``tropicalization'' of projective varieties (resp., open varieties) vaguely refers to the dual intersection complex  
 of log minimal (dlt) degenerations (resp., boundary divisor of projective dlt compactifications).  
 
In particular, in \cite{TGC.I, TGC.II}, we introduced the 
compactifications for $M_{g}$ and $A_{g}$ cases and studied their 
explicit structures, which we now briefly recall as follows: 

\begin{RTHM}[For $M_{g}$, in {\cite{TGC.I}, \cite{TGC.II}}]\label{TGC.Mg.review}
There is a pair of compact Hausdorff topological spaces 
$\overline{M_{g}}^{\rm GH}$ and $\overline{M_{g}}^{\rm T}$
which both contain 
the moduli space $M_{g}$ for compact hyperbolic curves of genus $g (\ge 2)$
 with the complex analytic topology, as an open dense subset, 
such that the following holds.

\begin{enumerate}
\item 
The boundary $\overline{M_{g}}^{\rm GH}\setminus M_{g}$ of 
$\overline{M_{g}}^{\rm GH}$ 
parametrizes the underlying metric spaces of connected 
metrized graphs $\Gamma$ which satisfy 
$$v_{1}(\Gamma)+b_{1}(\Gamma)\le g,$$
where $v_{1}(\Gamma)$ refers to 
the number of 1-valent vertices, and $b_{1}(\Gamma)$ means 
the first Betti number of $\Gamma$. 
Accordingly, there is a natural finite stratification 
by the homeomorphic type of $\Gamma$ whose strata 
are open simplices 
of dimension at most $3g-4$ 
divided by finite groups. 

\item The boundary $\overline{M_{g}}^{\rm T}\setminus M_{g}$ of 
$\overline{M_{g}}^{\rm T}$ parametrizes the connected 
metrized graphs $\Gamma$ with 
$w\colon V(\Gamma)=\{\text{vertices of }\Gamma\}\to \Z_{\ge 0}$ 
which satisfy 
\begin{itemize}
\item 
$b_{1}(\Gamma)+\sum_{v\in V(\Gamma)} w(v)=g$,  
\item for any $v$ with $w(v)=0$, its degree (valency) ${\rm deg}(v)$ 
is at least $3$, 
\item the diameter of $\Gamma$ is $1$. 
\end{itemize}
These $\Gamma$ (without assuming the last normalization condition) 
is also known as tropical curves in the literatures  
(cf., e.g., \cite{MZ}, 
\cite{BMV}). 
Accordingly, there is again a natural finite stratification 
by the combinatorial type of $(\Gamma,w)$ whose strata 
are open simplices 
of dimension at most $3g-4$ divided by finite groups. 

\item \label{Mg.ii} There exists a continuous map  
$\overline{M_{g}}^{\rm T}\to \overline{M_{g}}^{\rm GH}$ 
whose restriction to $M_{g}$ is the identity map. 
Its moduli-theoretic meaning is to forget $w$ and take 
underlying metric space of $\Gamma$.

 \item \label{Mg.v}
For any algebraic morphism from 
 a punctured curve $f\colon C\setminus \{p\}\to M_{g}$, 
the limit $\lim_{q\to p} f(q)\in \overline{M_{g}}^{\rm T}$ exists and 
the 
corresponding limit metrized graphs satisfy that 
lengths of all edges are the same. 
 \end{enumerate}
\end{RTHM}
Since $\overline{M_{g}}^{\rm GH}$ is nothing but the 
Gromov-Hausdorff compactification of $M_{g}$ (see \cite{TGC.I}), 
combination of the above 
\ref{Mg.ii} and \ref{Mg.v} implies 
the determination of Gromov-Hausdorff limits of 
any algebraic family of hyperbolic proper curves of genus $g (\ge 2)$ 
over $C\setminus \{p\}$ as the metrized graph whose edges have 
all same lengths. We also note that 
in \cite[Theorem 3.7]{TGC.II}, we identified 
$\overline{M_{g}}^{\rm T}$ with a slight variant of 
Morgan-Shalen type compactification of $M_{g}$. 

Our analogous results for the moduli space 
$A_{g}$ for principally polarized abelian varieties of dimension $g$ 
is as follows. 

\begin{RTHM}[For $A_{g}$, in {\cite{TGC.II}}]\label{TGC.Ag.review}
There is a compact Hausdorff topological space 
$\overline{A_{g}}^{\rm T}$ which contains 
the moduli space $A_{g}$ for principally polarized abelian varieties of dimension $g$ with the complex analytic topology, 
as an open dense subset, 
which satisfies the following. 
\begin{enumerate}
\item 
The boundary $\overline{A_{g}}^{\rm T}\setminus A_{g}$ 
parametrizes
 all flat (real) tori $\R^{i}/\Z^{i}$ of diameter $1$,
 where $i$ runs through integers with $1\le i\le g$. 
Hence, the boundary has obvious stratification 
by $i$ accordingly. 
 
\item \label{Ag.ii}
$\overline{A_{g}}^{\rm T}$ 
is exactly the Gromov-Hausdorff compactification of $A_{g}$ 
once we attach rescaled flat K\"ahler metric 
 in the principal polarization, with diameter $1$ to each abelian variety. 
 
\item \label{Ag.iii}
For any algebraic morphism from a punctured curve $f\colon C^{o}=C\setminus \{p\}\to A_{g}$, 
with trivial Raynaud extension, $\lim_{q\to p} f(q)\in \overline{A_{g}}^{\rm T}$ exists
 and such limits, where $(C,p)$ and $f$ run, 
form a dense subset of $\partial \overline{A_{g}}^{\rm T}$,
 which consists of points with rational coordinates. 
(The Raynaud extension triviality assumption is removed in Theorem~\ref{1par.AV} of this paper.) 
 \end{enumerate}
\end{RTHM}
In particular, by the above \ref{Ag.ii} and \ref{Ag.iii}, 
it follows that any algebraic family of $g$-dimensional 
principally polarized abelian varieties has the Gromov-Hausdorff 
limit as flat tori. We postpone more detailed analysis until 
Chapter~\ref{Abel.sec}. 

When we tried to go beyond the above two cases, we faced an essential difficulty which is that, unlike the two cases, 
we do \textit{not} have an explicit description of the K\"ahler-Einstein metrics (yet) in concern. Nevertheless, in this paper, we 
provide some explicit pictures for the 
Gromov-Hausdorff compactifications essentially 
depending on general compactification 
theory of locally symmetric spaces and the Morgan-Shalen type compactification theory of complex varieties.

\section{Outline of this paper}

\cite{OO} is the announcement of this paper. 
Most of the contents in this monograph are stated there, 
while we add some improvements especially in Chapters \ref{Gen.HSD} and \ref{high.dim.HK.sec}. 

\subsection{Compactifying general Hermitian locally symmetric spaces}

In Chapter~\ref{Gen.HSD},
 we prove a general theorem
 on compactifications of Hermitian locally symmetric spaces. 
 Recall that in  \cite[Appendix, especially A.9, A.13]{TGC.II}, generalizing the construction in \cite{MS, BJ16}, we construct generalized 
 Morgan-Shalen compactifications of the coarse moduli spaces 
 of dlt stacky pairs and toroidal stack. 
Furthermore, 
the obtained compactification for the latter 
does not depend on the choice of cone complex as shown in \cite[A.12, 
A.13]{TGC.II}. In the following theorem, 
we apply the compactification procedure 
to locally Hermitian symmetric spaces and identify the outcome with 
one of Satake compactifications \cite{Sat1, Sat2}. 

\begin{THM}[{=Theorem~\ref{MS.Sat}}]\label{MS.Sat.Intro}
For an arithmetic quotient of Hermitian symmetric domain $\Gamma\backslash D$, 
we consider the toroidal compactifications 
in the sense of \cite{AMRT} 
and their associated (generalized) Morgan-Shalen compactifications of $\Gamma\backslash D$ 
introduced in \cite[Appendix]{TGC.II} after \cite{MS, BJ16}. 
Since it does not depend on the choice of cone complex, 
we simply denote it by $\overline{\Gamma\backslash D}^{\rm MSBJ}$. 

On the other hand, we denote the Satake compactification for the 
adjoint representation \cite{Sat1, Sat2} by 
$\overline{\Gamma\backslash D}^{\rm Sat, \tau_{\rm ad}}$. 
Then, 
we have a homeomorphism 
$$
\overline{\Gamma\backslash D}^{\rm MSBJ}\simeq \overline{\Gamma\backslash D}^{\rm Sat, \tau_{\rm ad}}, 
$$
extending the identity of $\Gamma\backslash D$. 
\end{THM}

For the definition of the right hand side, we refer to \S\ref{Gen.HSD.start} and \cite{Sat1, Sat2, BJ}. 
We emphasize that this Satake compactification for the adjoint representation is \textit{not a variety}, hence \textit{different from} the so-called 
Satake-Baily-Borel compactification which is a projective variety. 
We apply the above compactifications to more geometric situations --- 
moduli of (polarized) varieties such as abelian varieties, K3 surfaces and hyperK\"ahler varieties 
and discuss their geometric meanings. 

\subsection{Moduli of abelian varieties}

The first theorem on the moduli of abelian varieties 
below refines our Theorem~\ref{TGC.Ag.review} proved in \cite{TGC.II}. 
Recall that $\overline{A_g}^{\rm T}$ stands for
 the tropical geometric compactification of \textit{op.cit.}, 
 which is the same as the Gromov-Hausdorff compactification in this case. 
$\overline{A_g}^{\rm Sat,\tau_{\rm ad}}$ 
stands for the Satake compactification 
for the adjoint representation $\tau_{\rm ad}$. 

\begin{THM}[{=Theorem~\ref{Ag.TGC.Satake.MS}}]\label{Ag.TGC.Satake.MS.Intro}
There is a homeomorphism 
$$\overline{A_g}^{\rm T} \simeq \overline{A_g}^{\rm Sat,\tau_{\rm ad}},$$
extending the identity map on $A_g$. \end{THM}
\noindent
In particular, the dual intersection complex of 
(orbi-smooth) toroidal compactification of $A_{g}$ is identified with 
the compact moduli space of flat real tori of dimension $i$ ($1\le i\le g$) with diameters $1$. 
On the other hand, 
recall that Abramovich-Caporaso-Payne \cite{ACP} 
showed that the moduli spaces of tropical curves, which also appear as our boundary of $\overline{M_{g}}^{\rm T}$ of Theorem \ref{TGC.Mg.review} (\cite{TGC.I}, \cite{TGC.II}), 
is homeomorphic to the dual intersection complex of 
the boundary of Deligne-Mumford compactification of $M_{g}$ as 
an algebraic stack. 
Since flat real tori are regarded as tropical analog of abelian varieties, 
Theorem~\ref{Ag.TGC.Satake.MS.Intro} gives an abelian varieties 
case analog of the result of \cite{ACP}. 
In Remark \ref{MZ.relation}, we also discuss a connection between these flat tori 
and the notion of {\it tropical principally polarized abelian variety} in \cite{MZ} by Mikhalkin-Zharkov. 

\subsection{Moduli of K3 surfaces}\label{Intro.K3}

Now we move on to more difficult K3 surfaces case. 
This is studied in Chapters \ref{F2d.sec} to \ref{along.boundary.sec}.
We propose a pair of conjectures on the explicit structure 
of the Gromov-Hausdorff compactifications of 
the moduli $\mathcal{F}_{2d}$ of polarized (ADE singular) K3 surface of degree $2d$, 
and the moduli $\mathcal{M}_{\rm K3}$
 of Ricci-flat K\"ahler (ADE singular) K3 surfaces modulo hyperK\"ahler rotation,
 and we partially confirm them. 
An important role is played by an explicit assignment of 
geometric objects (compact metric spaces with some additional structures) 
to each point of the Satake compactification for the 
adjoint representation 
$\overline{\mathcal{F}_{2d}}^{{\rm Sat},\tau_{\rm ad}}$ (resp., 
$\overline{\mathcal{M}_{\rm K3}}^{{\rm Sat},\tau_{\rm ad}}$),
 which we call the {\it geometric realization map}
 and denote by $\Phi_{\rm alg}$ (resp., $\Phi$). 
We leave the details to Chapters \ref{F2d.sec} and \ref{Kahler.section} 
but $\Phi_{\rm alg}|_{\mathcal{F}_{2d}}$ and $\Phi|_{\mathcal{M}_{\rm K3}}$ 
is just to assign corresponding Ricci-flat (ADE singular) K3 surfaces 
by Yau's theorem \cite{Yau}. 

In the case of $\overline{\mathcal{F}_{2d}}^{\rm Sat,\tau_{\rm ad}}$, 
we have two kinds of boundary components: the quotients of $18$-dimensional real 
balls $\mathcal{F}_{2d}(l)$ and $0$-dimensional strata 
$\mathcal{F}_{2d}(p)$. To each point in the ball quotient $\mathcal{F}_{2d}(l)$, 
by $\Phi_{\rm alg}$ 
we assign a certain metrized sphere (underlying a tropical K3 surface, cf.\  Chapter~\ref{F2d.sec}), which appeared in \cite{GW,KS}. 
To each $0$-dimensional 
stratum $\mathcal{F}_{2d}(p)$, we assign the segment of length $1$. The 
geometric realization map $\Phi$ for the case of 
$\mathcal{M}_{\rm K3}$ is more complicated 
where not only metrized spheres or segments but also 
real $3$-dimensional metrized orbifolds appear. 
See Chapter~\ref{Kahler.section} for the details. 

Our conjectures are of the following form, which we partially prove in this paper. 
See Conjectures~\ref{K3.Main.Conjecture} and \ref{K3.Main.Conjecture2}
 and discussion around them for more detailed statements. 
We write ${\it CMet}_{1}$ for the set of isometry classes of compact metric spaces
 with diameter one equipped with the Gromov-Hausdorff topology.

\begin{MConj}\label{Main.Conj}
\begin{enumerate}
\item ($=$ Conjecture \ref{K3.Main.Conjecture}) 
The geometric realization map 
\[\Phi_{\rm alg}\colon \overline{\mathcal{F}_{2d}}^{{\rm Sat},\tau_{\rm ad}}
 \to {\it CMet}_{1}\]
given in Chapter~\ref{F2d.sec} is continuous. 
\item
($=$ Conjecture~\ref{K3.Main.Conjecture2})
The geometric realization map 
\[\Phi\colon \overline{\mathcal{M}_{\rm K3}}^{\rm Sat,\tau_{\rm ad}} 
\to {\it CMet}_{1}\]
 in Chapter~\ref{Kahler.section} is continuous.  
\end{enumerate}
\end{MConj}

Morally speaking, above Conjecture~\ref{Main.Conj} gives geometric meaning to these 
Satake compactifications for the adjoint representation, or the Morgan-Shalen type compactification through Theorem~\ref{MS.Sat.Intro}. 
Once one confirms the above conjecture~\ref{Main.Conj} in future, we would like to regard each compactification 
$\overline{\mathcal{F}_{2d}}^{{\rm Sat},\tau_{\rm ad}}$ and $\overline{\mathcal{M}_{\rm K3}}^{\rm Sat,\tau_{\rm ad}}$, 
with their associated geometric realization maps, as the ``tropical geometric compactifications'' (cf., \cite{TGC.I}, \cite{TGC.II}) 
in the case of K3 surfaces. 
An immediate consequence of the above conjectures are 
a complete classification of Gromov-Hausdorff limits of 
Ricci-flat K3 surfaces with bounded diameters. 

\begin{PROP}[Classification of the Gromov-Hausdorff limits]
\begin{enumerate}
\item \label{classif1}
If Conjecture \ref{K3.Main.Conjecture} is true, 
then the following holds: For any sequence of the form 
$(X_{i},\frac{g_{i}}{{\rm diam}(g_{i})^{2}}) (i=1,2,\cdots)$ 
where each $X_{i}$ is a smooth K3 surface, 
$g_{i}$ is a 
Ricci-flat K\"ahler metric with the primitive integral K\"ahler class of the same 
degree (volume), if its Gromov-Hausdorff limit exists, it is either 
\begin{itemize}
\item a Ricci-flat K3 surface possibly with ADE singularities, 
\item a metrized sphere ($S^{2}$), 
\item or the unit segment, i.e., $[0,1]$. 
\end{itemize}
\item \label{classif2}
If Conjecture \ref{K3.Main.Conjecture2} is true, 
then the following holds: For any sequence of 
Ricci-flat K3 surfaces of fixed diameters, 
if its Gromov-Hausdorff limit exists, it is either 
\begin{itemize}
\item 
a Ricci-flat K3 surface possibly with ADE singularities, 
\item $T^{n}/(\Z/2\Z)$, i.e., a flat 
real $n$-dimensional torus divided by the $(-1)$-multiplication, where $n=1, 2, 3$, 
\item 
a metrized sphere ($S^{2}$), 
\item 
or the unit segment, i.e., $[0,1]$. 
\end{itemize}
\end{enumerate}
\end{PROP}

From the results of \cite{GW, GTZ1, Fos, ChCh, HSVZ}, for instance, 
conversely all above types in \eqref{classif2}
 indeed appear as such Gromov-Hausdorff limits of 
 Ricci-flat K3 surfaces. 

Note that the above conjecturally  
continuous map $\Phi_{\rm alg}$ in Conjecture 
\ref{Main.Conj} (i) combined with Theorem~\ref{MS.Sat.Intro}, 
gives an interpretation of dual intersection complex of 
the boundary of toroidal compactification of $\mathcal{F}_{2d}$ 
morally as moduli of tropicalized version of K3 surface 
(see \S \ref{trop.K3.1} for more details). Therefore, it would again give 
an K3 surfaces case analog of the result of Abramovich-Caporaso-Payne \cite{ACP} for the moduli of curves and its tropical version. 
\medskip

Our main results for complex K3 surfaces can be summarized as follows, which provide 
partial confirmation of Conjecture \ref{Main.Conj}. 

\begin{THM}[{=Theorems~\ref{K3.Main.Conjecture.18.ok}, \ref{K3a.GH.conti}, \ref{K3.Main.Theorem2}}]
\begin{enumerate}\label{K3.Main.Conjecture.18.ok.Intro}
\item (=Theorem~\ref{K3.Main.Conjecture.18.ok})
The restriction of the 
geometric realization map $\Phi_{\rm alg} 
\colon \overline{\mathcal{F}_{2d}}^{{\rm Sat},\tau_{\rm ad}}\to {\it CMet}_{1}$ 
to the complement of the finite points 
$\bigcup_{p}\mathcal{F}_{2d}(p)$ 
is continuous. 
\item (=Theorems \ref{K3a.GH.conti}, \ref{K3.Main.Theorem2}) 
There exists a (real) $36$-dimensional 
boundary component $\mathcal{M}_{\rm K3}(a)$ 
of $\overline{\mathcal{M}_{\rm K3}}^{\rm Sat,\tau_{\rm ad}}$, 
parametrizing tropical K3 surfaces (certain metrized spheres 
cf., Chapters~\ref{F2d.sec} and \ref{Kahler.section} for details) such that the restrictions 
of the geometric realization map $\Phi\colon \overline{\mathcal{M}_{\rm K3}}^{\rm Sat,\tau_{\rm ad}}\to {\it CMet}_{1}$ to 
$\mathcal{M}_{\rm K3}\cup \mathcal{M}_{\rm K3}(a)$ and to 
the closure of $\mathcal{M}_{\rm K3}(a)$ inside 
$\overline{\mathcal{M}_{\rm K3}}^{\rm Sat,\tau_{\rm ad}}$ 
are both continuous. 
\end{enumerate}
\end{THM}

One of the main advantages of our compactification is, 
via Conjecture \ref{Main.Conj}, we can discuss Gromov-Hausdorff 
limits of Ricci-flat K\"ahler metrics along any sequences, 
which are not of special types such as one parameter holomorphic family. 

For a given holomorphic family $(\mathcal{X}^{*},\mathcal{L}^{*})\twoheadrightarrow \Delta^{*}$
 of polarized K3 surfaces, it can be shown that the period map 
$\Delta^{*}\to \mathcal{F}_{2d}$ extends continuously to $\Delta \to \overline{\mathcal{F}_{2d}}^{{\rm Sat},\tau_{\rm ad}}$ (Theorem~\ref{rationality.HSD}, 
Theorem~\ref{monodromy}). Therefore, above Theorem~\ref{K3.Main.Conjecture.18.ok.Intro}(i) 
is enough to prove the conjecture of 
Kontsevich-Soibelman \cite[Conjecture 1]{KS}, Gross-Wilson \cite[Conjecture 6.2]
{GW} and Todorov for the K3 surface case:  
\begin{COR}[K3 surfaces case of 
{\cite[Conjecture 1]{KS}, \cite[Conjecture 6.2]
{GW}}]
Let $(\mathcal{X}^{*},\mathcal{L}^{*})\twoheadrightarrow \Delta^{*}$ be a 
meromorphic punctured family of polarized K3 surfaces of degree $2d$, possibly 
with ADE singularities, and suppose it is of type III in the sense of Kulikov
(``maximal degeneration'') \cite{Kul, PP}. 
For any $t\in \Delta^{*}$, 
put the rescaled Ricci-flat K\"ahler metric (of diameter $1$)
$\frac{d_{\rm KE}(\mathcal{X}_{t})}{{\rm diam}(\mathcal{X}_{t})}$,
where the K\"ahler class of 
$d_{\rm KE}(\mathcal{X}_{t})$ is $c_{1}(\mathcal{L}|_{\mathcal{X}_{t}})$, 
to each fiber $\mathcal{X}_{t}$. Then they
converge in the Gromov-Hausdorff sense to a Monge-Amp\`ere manifold with singularity,
homeomorphic to $S^{2}$ when $t\to 0$. 
\end{COR}
See \S\ref{K3.along.disk} (especially Corollary~\ref{1par.lim.K3.}) for more detailed information. 

Chapter~\ref{GTZ.extend.proof} is devoted to analytic estimates of K\"{a}hler forms which are necessary for the proof of Theorem~\ref{K3.Main.Conjecture.18.ok.Intro}.
Recall that \cite{GW, GTZ1, GTZ2} studied collapsing of 
fixed elliptic K3 surface along a certain specific behavior of Ricci-flat 
K\"ahler metrics, i.e., when the K\"ahler class changes affine-linearly 
to its base, and they made a partial confirmation of the conjecture by
 using hyperK\"ahler rotations. 
However, for the full proof of the conjecture, hyperK\"ahler rotation of their affine
 linear variation of K\"ahler class seems to be not enough
 (see discussions around \cite[\S1, Theorem 6.2]{GW})
 as indeed both the complex structures and the initial K\"ahler class of rotated
 one parameter family can vary in general. 
Our point is that, after their very pioneering important papers, we use our explicit moduli compactification framework in 
\cite[Appendix]{TGC.II}, Chapter~\ref{Gen.HSD} and 
combine with some refinements of their works
in Chapter~\ref{GTZ.extend.proof}, to yield the proof. 
Further, we prove more about the limit $S^{2}$ (see Corollary~\ref{1par.lim.K3.}). 

It would be also very interesting to discuss how we can use or relate our picture to 
the problem of reconstructing one parameter families from the collapsed Monge-Amp\`ere manifolds 
with singularities, with some additional structures as in \cite{KS, GS11}. 
We wish to study this problem in near future. 

We also remark an analog for the non-maximally degenerating case: 

\begin{REM}
If Conjecture \ref{Main.Conj}(i)(=Conjecture \ref{K3.Main.Conjecture}) holds, then the following is true: 
Let $(\mathcal{X}^{*},\mathcal{L}^{*})\twoheadrightarrow \Delta^{*}$ be a 
meromorphic punctured family of polarized K3 surfaces of degree $2d$, possibly  
with ADE singularities, and suppose it is of type II in the sense of Kulikov \cite{Kul, PP}.  
For any $t\in \Delta^{*}$,  
put the rescaled Ricci-flat K\"ahler metric (of diameter $1$) 
$\frac{d_{\rm KE}(\mathcal{X}_{t})}{{\rm diam}(\mathcal{X}_{t})}$, 
where the K\"ahler class of 
$d_{\rm KE}(\mathcal{X}_{t})$ is $c_{1}(\mathcal{L}|_{\mathcal{X}_{t}})$,  
to each fiber $\mathcal{X}_{t}$. Then they 
converge in the Gromov-Hausdorff sense to the segment of length $1$ when $t\to 0$.  
\end{REM}

In Chapter~\ref{along.boundary.sec}, we study tropical K3 surfaces and their moduli space
 in terms of Weierstrass models of elliptic K3 surfaces.
In particular,
 we consider the GIT compactification of the moduli of Jacobian elliptic K3 surfaces
 and compare it with the boundary of Satake compactification
 $\overline{\mathcal{M}_{\rm K3}}^{\rm Sat,\tau_{\rm ad}}$
 given in Chapter~\ref{Kahler.section}.
In \S\ref{ML.limit} we study the asymptotic behavior of metrics on tropical K3 surfaces
 and show that the restriction of $\Phi$ to the closure of $\mathcal{M}_{\rm K3}(a)$
 is continuous.

In Chapter~\ref{high.dim.HK.sec}, we discuss generalizations of our framework  
to higher dimensional Ricci-flat K\"ahler analytic spaces. In particular,  
we give a precise conjecture that the
compact hyperK\"ahler varieties case fits 
to essentially the same picture as the K3 surfaces case (Conjecture~\ref{Main.Conj}), 
since many parts of our partial proof of the K3 surface case 
extend notably because of the existence of the Beauville-Bogomolov-Fujiki form 
on the second singular cohomology and the recent breakthrough 
for the Torelli theorem \cite{Ver, Mark.Torelli, Ver.er}, although several technical gaps are not yet filled 
for substantial progress. 

In Chapter~\ref{high.dim.gen.sec}, we seek a further extension of our framework to general Calabi-Yau manifolds. 

\medskip

Last but not least, to avoid confusion, we would like to give moral or formal remarks on our projects. 
Firstly, the set of non-variety compactifications of 
moduli spaces we discuss in this monograph and the previous series \cite{TGC.I, TGC.II} is \textit{different} 
from what we expect in the framework of K-moduli (cf., e.g., \cite[\S3]{Od12}), at the collapsing locus or the locus in the K-moduli where 
strictly (semi-)log-canonical varieties are parametrized. 
Indeed, even for the simplest $M_g$ case, the K-moduli is nothing but the Deligne-Mumford compactification 
which is a normal projective variety as shown in \cite{Od0, Od2} and different from our \cite{TGC.I}. 

For polarized K3 surfaces case, the algebro-geometrically meaningful projective compactifications of the moduli spaces $\mathcal{F}_{2d}$ 
has been, as a full generality, a classical open problem which had some partial progress such as \cite{Shah, Shah2, Laza2, LO} and \cite{AET} which recently appeared 
after the appearance of our work. Our compactification theory in Chapter~\ref{F2d.sec} 
gives a geometrically meaningful compactification although \textit{not} in the category of varieties.


\newpage

\chapter{Compactifications of 
Hermitian locally symmetric spaces}\label{Gen.HSD}

In this chapter, we introduce some 
compactification theory for general Hermitian locally symmetric spaces, 
notably Theorem~\ref{MS.Sat}. 
We later apply it to various moduli spaces. 

\section{Setting}\label{Gen.HSD.start}

We first set up the stage for this chapter, 
following \cite{AMRT} and \cite{BJ}. 

\begin{ntt}\label{notation.groups}

\begin{enumerate}

\item $\mathbb{G}$ is a reductive algebraic group over $\mathbb{Q}$,
 $G$ is an open subgroup 
 of the Lie group $\mathbb{G}(\mathbb{R})$,
 $K$ is (one of) its maximal compact subgroup.

\item \label{HSD.assumption} 
$D:=G/K$ which we suppose to have a {\it Hermitian} symmetric domain structure. 
We moreover assume that $D$ is irreducible and that $G$ is simple
 in the sense that its identity component (or the Lie algebra) is simple.

\item \label{Gamma}
$\Gamma$ is an arithmetic subgroup of $\mathbb{G}(\mathbb{Q})$, 
i.e., commensurable to $\mathbb{G}(\mathbb{Z})$. We assume $\Gamma$ to 
act holomorphically on $D$. 

\item $N(F)$  is the maximal real parabolic subgroup (``$\mu$-saturated parabolic 
subgroup'' in \cite{BJ}) corresponding to a boundary component $F$ in the 
Harish-Chandra embedding of $D$, following the notation of \cite{AMRT}. 
This $N(F)$ will be also denoted by $Q(P)$ or $Q$ below, following \cite{BJ}. 

\item $W(F)$ is the unipotent radical of $N(F)$ and $U(F)$ is its center.

\end{enumerate}

\end{ntt}

Ichiro Satake \cite{Sat1}, \cite{Sat2} 
constructed compactifications of Riemannian locally symmetric spaces 
associated to certain faithful projective representations, 
then classified them in terms of their highest weights. 

Given a faithful projective complex representation 
 $\tau\colon G\hookrightarrow 
 PSL(\mathbb{C})$ 
 with certain conditions
 (so-called geometric rationality conditions, see Borel-Ji~\cite[{\S}III.3]{BJ} for details),
 the corresponding Satake compactification is defined
 as a compact topological space, and has
 the following stratification (see e.g.\ \cite[p.312]{BJ}):
\begin{align}\label{Satake.stratif}
\overline{\Gamma\backslash D}^{\rm Sat, \tau}
 =\Gamma\backslash D\sqcup \bigsqcup_{P}
(\Gamma\cap Q(P))\backslash M_{P}/(K\cap M_{P}).
\end{align} 
Here, $P$ runs over (representatives of) all the $\Gamma$-conjugacy classes
 of $\mu(\tau)$-connected rational parabolic subgroups, 
 $P=N_{P}A_{P}M_{P}$ denotes the Langlands decomposition, and 
 $Q(P)$ is the $\mu(\tau)$-saturation of $P$ which we also denote by $Q$. 
In particular, $N_P$ is the unipotent radical of $P$,
 $A_P$ is the maximal $\mathbb{R}$-split torus of $P$,
 and $M_P$ is a reductive subgroup. 
We remark that
 in general (also in the case where $\tau=\tau_{\rm ad}$ as below),
 Satake compactifications are stratified by real orbifolds as \eqref{Satake.stratif},
 but do not have a structure of complex variety.

\section{Satake compactification for adjoint representation}\label{Satake.adjoint}
Under our assumption (\ref{HSD.assumption}), the restricted root system of $G$
 associated to a maximal $\mathbb{R}$-split torus in $G$ is of type $BC$ or $C$.
We are particularly interested in the case when
 \textit{the highest weight 
 $\mu(\tau)$ is connected only to the exact opposite of the distinguished root} 
(i.e., ``$\frac{(\gamma_{1}-\gamma_{2})}{2}$'' in the notation of e.g.~\cite[\S\,III.2.8]{AMRT}). 
This is the case if we take $\tau$ to be the adjoint
 representation.
Let $\tau_{\rm ad}$ denote the adjoint representation of $G$
 and we consider
 $\overline{\Gamma\backslash D}^{\rm Sat, \tau_{\rm ad}}$,
 which we call
 {\it the Satake compactification for the adjoint representation}.

To have a more explicit description,
 let $P_{\rm min}$ be a minimal parabolic subgroup of $G$
 and write $\theta$ for the Cartan involution of $G$ such that $G^{\theta}=K$.
Write $P_{0}=M_{0}A_{0}N_{0}$
 for the Langlands decomposition such that
 $M_{0}\subset K$ and $\mathfrak{a}_0\subset \mathfrak{g}^{-\theta}$,
 where $\mathfrak{a}_0:=\Lie(A_{0})$ and $\mathfrak{g}:=\Lie(G)$.
Then we can choose a maximal set of
 strongly orthogonal roots
 $\gamma_1,\dots,\gamma_r\in\mathfrak{a}_0^*$
 in the restricted root system $\Delta(\mathfrak{g},\mathfrak{a}_0)$.
The set of positive roots can be chosen as
\[\{\tfrac{1}{2}(\gamma_i+\gamma_j)\text{ for $i\leq j$};
 \ \tfrac{1}{2}(\gamma_i-\gamma_j)\text{ for $i < j$}\}\]
if $\Delta(\mathfrak{g},\mathfrak{a}_0)$ is of type $C$ and 
\[\{\tfrac{1}{2}(\gamma_i+\gamma_j)\text{ for $i\leq j$};
 \ \tfrac{1}{2}(\gamma_i-\gamma_j)\text{ for $i < j$};
 \ \tfrac{1}{2}\gamma_i \}\]
if $\Delta(\mathfrak{g},\mathfrak{a}_0)$ is of type $BC$.
The corresponding Dynkin diagram is:
\begin{align*}
\begin{xy}
\ar@{-} (0,0) *++!D{\tfrac{\gamma_1-\gamma_2}{2}}
 *{\circ}="A"; (15,0) *++!D{\tfrac{\gamma_2-\gamma_3}{2}} 
 *{\circ}="B"
\ar@{-} "B"; (25,0)  *{}="C"
\ar@{.} "C"; (35,0)  *{}="D"
\ar@{-} "D"; (45,0)  *++!D{\tfrac{\gamma_{r-1}-\gamma_r}{2}} *{\circ}="E"
\ar@{=}  "E"; (60,0) *++!D{(\tfrac{1}{2})\gamma_r} *{\circ}="F"
\end{xy} 
\end{align*}
where the right most root (so-called the \textit{distinguished root})
 is $\gamma_r$ or $\frac{1}{2}\gamma_r$ according to the root system
 is of type $C$ or of type $BC$. Hence, the right most arrow is 
 $\Leftarrow$ for type $C$ and $\Rightarrow$ for type $BC$, although 
 the latter (restricted) root system of type $BC$ is not reduced. 
  
The highest weight $\mu(\tau_{\rm ad})$ of the adjoint representation of $G$ is
 the highest root $\gamma_1$.
Hence it is connected (non-orthogonal) only to $\frac{\gamma_1-\gamma_2}{2}$
 among simple roots.
Since $\tau_{\rm ad}$ is defined over $\Q$,
 the geometric rationality is assured by \cite[Theorem 8]{Sap}.
Hence the Satake compactification
 can be defined as a compact Hausdorff space
 $\overline{\Gamma\backslash D}^{\rm Sat, \tau_{\rm ad}}$.
In the original Satake's construction \cite{Sat1}, when applied to the adjoint representation $\tau_{\rm ad}$, 
 $D=G/K$ is embedded in the real projective space $\mathbb{P}(\mathfrak{g}^{*}\otimes {}^{\theta}\mathfrak{g}^{*})$ 
and partially compactified there. 
Here, ${}^{\theta}\mathfrak{g}^{*}$ 
means the dual vector space $\mathfrak{g}^{*}$ of $\mathfrak{g}$ with the $\theta$-twisted coadjoint $G$-action on $\mathfrak{g}^*$. 
Then in \cite{Sat2}, its $\Gamma$-quotient is considered
 and it is proved that the quotient becomes a compact topological space 
 with respect to the so-called Satake topology. 
Instead of that, we can alternatively obtain the same space by embedding
 $D$ into $\mathbb{P}(\mathfrak{g})$ (see \cite[Lemma~2 and Corollary~4]{Sap} for the proof). 
The embedding $D\hookrightarrow \mathbb{P}(\mathfrak{g})$ can be given as follows. 
By our assumption \eqref{HSD.assumption}, $K$ has a one-dimensional center.
Let $\mathfrak{z}(\mathfrak{k})$ denote its Lie algebra.
Since a vector in $\mathfrak{z}(k)$ is $K$-invariant, we have
 an embedding $\iota\colon D\hookrightarrow \mathbb{P}(\mathfrak{g})$
 by $\iota(g\cdot o)={\rm Ad}(g)[\mathfrak{z}(\mathfrak{k})]$,
 where $o\in G/K$ is a base point.
We note that by e.g.\ \cite[Theorem 2.4]{AMRT},
 a vector in $\mathfrak{z}(\mathfrak{k})$ is
 of the form 
\begin{align}\label{k_center}
Y+\sum_{j=1}^r(X_{\gamma_j}+\theta(X_{\gamma_j})).
\end{align}
Here, $X_{\gamma_j}\in\mathfrak{g}_{\gamma_j}$ are root vectors
 and $Y\in \Lie (M_0)$.

We now see the boundary components
 of $\overline{\Gamma\backslash D}^{\rm Sat, \tau_{\rm ad}}$.
Thanks to the description of $\Q$-roots in \cite[p.467--468]{BB},
 any maximal rational parabolic subgroup of $G$
 is also maximal as a real parabolic subgroup.
Then it turns out that $Q(P)$ in \eqref{Satake.stratif}
 runs over (representatives of)
 all the $\Gamma$-conjugacy classes of maximal rational parabolic subgroups.
Therefore, the boundary components of 
 $\overline{\Gamma\backslash D}^{\rm Sat, \tau_{\rm ad}}$
 correspond bijectively to 
 the $\Gamma$-conjugacy classes of
 maximal rational parabolic subgroups of $G$.
Let $Q(P)$ be a rational parabolic appearing in \eqref{Satake.stratif}.
Replacing $P_0$ by its conjugate if necessary, we may assume $Q(P)\supset P_0$.
Let $Q(P)=M_QA_QN_Q$ be the Langlands decomposition such that
 $M_Q\supset M_0$, $A_Q\subset A_0$, and $N_Q\subset N_0$.
Since $Q(P)$ is maximal, there is only one simple root
 that is not a root for $M_Q$, which we denote by $\alpha$.
If $\alpha=\gamma_r$ or $\frac{\gamma_r}{2}$,
 put $s:=r$.
If not, $s$ is defined by $\alpha=\frac{\gamma_s-\gamma_{s+1}}{2}$.
Let $\{a_i\}\subset A_Q$ be a sequence such that
 $a_i^{\alpha}\to +\infty$.
Then the sequence $a_i\cdot o$ in $D$ is contained in one Siegel set
 and it has a limit in $\mathbb{P}(\mathfrak{g})$ belonging to
 the boundary component corresponding to $Q(P)$.
By using \eqref{k_center}, it is easy to see that
 the limit point in $\mathbb{P}(\mathfrak{g})$ is
 represented by $\sum_{j=1}^s X_{\gamma_j}$.
Therefore, the boundary component
 $M_P/(K\cap M_P)\subset \mathbb{P}(\mathfrak{g})$
 inside $\mathbb{P}(\mathfrak{g})$
 equals the $Q(P)$-orbit through $\sum_{j=1}^s X_{\gamma_j}$.

Note that 
 if a representation $\tau$ is chosen
 in such a way that $\mu(\tau)$ is connected only to the distinguished root,
 the compactification is called the Satake-Baily-Borel compactification
 and known to have a structure of normal complex projective variety
 (\cite{BB}).
The stratification of the Satake-Baily-Borel compactification 
 as in \eqref{Satake.stratif} is written as
\begin{align*}
\overline{\Gamma\backslash D}^{\rm SBB}
 =\Gamma\backslash D\sqcup \bigsqcup_{F}
(\Gamma\cap N(F))\backslash F
\end{align*} 
in the notation of \cite{AMRT}.
Here, $F$ runs over (representatives of)
 all the $\Gamma$-conjugacy classes of rational boundary components
 and then $N(F)$ runs over maximal rational parabolic subgroups.
Therefore, there is a natural bijection between
 the set of boundary components of
 $\overline{\Gamma\backslash D}^{\rm SBB}$
 and that of
 $\overline{\Gamma\backslash D}^{{\rm Sat}, \tau_{\rm ad}}$.
Under this correspondence,
 the closure relations of boundary components of them are opposite.

\section[Morgan-Shalen compactification and monodromy]{Relation with Morgan-Shalen compactification and monodromy}

\subsection{Relation with Morgan-Shalen compactification}
We follow the notation above. 
For an arithmetic quotient of Hermitian symmetric space $\Gamma\backslash D$, 
we consider toroidal compactifications and their associated (generalized) Morgan-Shalen-Boucksom-Jonsson 
compactifications of $\Gamma\backslash D$ 
in the sense of \cite[Appendix]{TGC.II}, which we briefly recall. 

The idea of the pioneering construction of Morgan-Shalen \cite[\S I.3]{MS}, 
where the technique is applied to the character varieties, 
is to compactify a non-compact complex manifold with the boundary 
of ``locally polyhedral'' nature essentially by use of some valuations 
of the meromorphic functions. In particular, 
the obtained compactification is far from being a complex analytic 
space in general. 
Recently, 
Boucksom-Jonsson \cite[\S 2]{BJ16} introduced a variant of \cite[\S I.3]{MS} 
as follows: for a complex manifold $X$ and its open dense subset 
$U$ such that $X\setminus U$ is a simple normal crossing divisor 
with finite number of strata, 
\cite[Definition 2.3]{BJ16} introduces 
a partial compactification of $U$ as $U\sqcup \Delta(X\setminus U)$ 
with a certain topology, where $\Delta(X\setminus U)$ denotes 
the dual intersection complex of
$X\setminus U$. The construction is compatible with that of 
original \cite[\S I.3]{MS}. The 
authors of \cite{BJ16} denoted the obtained 
partial compactification of $U$ by $X^{\rm hyb}$ and 
called it the hybrid space. 
If $X$ is compact, the obtained partial compactification 
$X^{\rm hyb}$ is also compact. 
Later, in \cite[Appendix]{TGC.II}, the construction of \cite{BJ16} 
is further generalized 
to the (coarse moduli spaces of) what {\it loc.cit.}\ called 
{\it dlt stacky pair} (A.9 of {\it loc.cit.})
 and {\it toroidal stack} (A.12 of {\it loc.cit.}) in a compatible manner. The generalization in {\it loc.cit.}\ 
 includes three directions - to allow dlt singularity, 
 toric singularities, and reflect stacky structures. 
 We call the obtained generalization as 
 (generalized) {\it Morgan-Shalen(-Boucksom-Jonsson)} 
 compactification and put MSBJ for the superscript. 
 
 \vspace{2mm}
 Now we come back to our particular setting. 
 For any $\Gamma$-admissible collection of polyhedra $\{\sigma_{\alpha}\}$ (see \cite[Chapter III Definition 5.1]{AMRT}), we consider the 
 so-called toroidal compactification, i.e., 
 the corresponding proper algebraic stack given in 
 \cite[Chapter III, Theorems 5.2 and 7.5]{AMRT} which we denote by 
 $([\Gamma\backslash D]^{\rm AMRT, \{\sigma_{\alpha}\}},
 [\Gamma\backslash D]^{\rm AMRT, \{\sigma_{\alpha}\}}\setminus 
 [\Gamma\backslash D])$. 
 We apply either A.9 or A.12 of \cite[Appendix]{TGC.II} 
 (the result is the same by A.13 of {\it loc.cit}) 
to the pair. 
Recall that it does not depend on the choice of cone complex as shown 
in \cite[A.12, A.13]{TGC.II}. 
Thus we denote it by $\overline{\Gamma\backslash D}^{\rm MSBJ}$. 

The main theorem of this chapter is to compare such $\overline{\Gamma\backslash D}^{\rm MSBJ}$ 
 with $\overline{\Gamma\backslash D}^{\rm Sat, \tau_{\rm ad}}$,
 the Satake compactification for the adjoint representation, 
 which we discussed in the previous section from 
 a priori rather different perspectives. 

Here is the main theorem of this chapter. 

\begin{Thm}\label{MS.Sat}
In the above setting,
 there exists a homeomorphism 
\[\overline{\Gamma\backslash D}^{\rm MSBJ}\simeq 
\overline{\Gamma\backslash D}^
{\rm Sat, \tau_{\rm ad}}\]
extending the identity map on $\Gamma\backslash D$. 
\end{Thm}

In particular, the dual intersection complex of the boundary divisor for 
any toroidal compactification of $\Gamma\backslash D$ (in the sense explained in \cite[Appendix]{TGC.II}) has a finite 
stratification by (explicit) locally symmetric spaces. 
While finishing the manuscript and before submitting this paper, we learned \cite[Proposition 2.1.1]{HZII} 
mentions that statement under some assumptions. Our Theorem~\ref{MS.Sat} elaborates the result at the level of 
compactification of $\Gamma\backslash D$, rather than the homeomorphism of two boundaries, without any assumption. 

\begin{proof}[proof of Theorem~\ref{MS.Sat}]
By our assumption, the highest weight $\mu(\tau_{\rm ad})$
 (with respect to the restrict root system for maximal $\R$-split torus)
 is connected only to the opposite side root of the distinguished root
 in the Dynkin diagram for $G$. 
Let $P$ be an arbitrary $\mu(\tau_{\rm ad})$-connected parabolic subgroup
 and set $Q$ as its $\mu(\tau_{\rm ad})$-saturation.
Since we assumed that $G$ is simple, $Q$ is a real maximal parabolic subgroup.
This can be written as $N(F)$ in the \cite{AMRT} notation,
 when $F$ is the corresponding rational boundary component in the
 Harish-Chandra embedding of $D$. 

In the notation of \cite[p.144]{AMRT}, a Levi component of $Q$
 is decomposed as $G_h(F)\cdot G_{\ell}(F)\cdot M(F)$.
The semisimple part of $P$ equals that of $G_{\ell}(F)$ up to compact factors.
Recall that the cone $C(F)$,
 which is to be divided into rational polyhedral subcones
 for the toroidal (partial) compactifications around $F$,
 is defined in \cite{AMRT} as a certain $G_{\ell}(F)$-orbit in $U(F)$,
 and it is shown in \cite[III.\ Theorem 4.1]{AMRT} that 
\[C(F)\cong G_{\ell}(F)/(K\cap G_{\ell}(F)).\] 
This implies that the boundary component corresponding to $F$, 
to put in the Satake 
compactification for the adjoint representation is 
 $C(F)/\R_{>0}$ 
 (cf.\  e.g.\ \cite[Lemma III 7.7]{BJ}), which is denoted by 
 $S(F)$ from now on. 

Hence we can stratify the compactification as 
\begin{equation}\label{Satake.strat}
\overline{\Gamma\backslash D}^{\rm Sat, \tau_{\rm ad}}
 =\Gamma\backslash D\sqcup \bigsqcup_{F}(\Gamma\cap N(F))\backslash S(F), 
\end{equation}
where $F$ runs through all the equivalence classes of 
the rational boundary strata (in the sense of 
\cite{AMRT}, i.e.\ for the Baily-Borel compactification) with respect to the $\Gamma$-action. 

Now we are ready to consider the structure of the 
Morgan-Shalen type construction in \cite[Appendix A.13 (after A.11, A.12)]{TGC.II},
 which is simply a stacky extension of \cite{MS, BJ16},
 applied to toroidal compactifications
 and compare with the above Satake type compactification. 
Recall that the toroidal compactifications for a given 
 polyhedral decomposition $\Sigma$ naturally 
 have structures as Deligne-Mumford stacks, 
 which are even smooth (as stacks)
 if all the rational polyhedral subcones we take are regular. 
We denote the obtained compactification of $\Gamma\backslash D$
 by $\overline{\Gamma\backslash D}^{\rm MSBJ}$
 after the names of \cite{MS, BJ16}. Note that \cite{TGC.II}
 (or discussion below) shows that
 their isomorphism classes as compactifications are
 independent of $\Sigma$ after all. 

Let us briefly recall the construction of toroidal compactifications
 following \cite[Chapter III]{AMRT}. 
 In {\it op.cit.}, $G$ is assumed to be
 the connected component ${\rm Aut}^{o}(D)$ of ${\rm Aut}(D)$ 
 but our assumptions \eqref{HSD.assumption}, \eqref{Gamma} are enough for the construction. 
 They use a description of the Hermitian symmetric space $D$ as a 
so-called Siegel domain of the third kind 
 (the existence of such presentation is proved in \cite{PS}, \cite{KW}): 
\begin{align}\label{third.kind}
D\cong\{(x,y,z)\in U(F)_{\mathbb{C}}\times \mathbb{C}^{k}\times F
\mid \im(x)\in C(F)+h_{z}(y,y)\},
\end{align}
where $U(F)_{\C}:=U(F)\otimes \C$ and
 $h_{z}$ is a $(U(F)\otimes \R)$-valued real bilinear quadratic form on $\mathbb{C}^{k}$. 
We regard the domain $D$ as an open subset of
 $U(F)_\C \times \mathbb{C}^{k}\times F$,
 which has a projection onto $\mathbb{C}^{k}\times F$. 
The fibers of $D$ inside $U(F)_{\C}$ are some translations of the same
 tube domains $U(F)\times C(F)\subset U(F)_{\mathbb{C}}$. 
As input data, take a $\Gamma$-admissible collection of polyhedra
 $\{\sigma_{\alpha}^{F}\}_{F}$ and we compactify $\Gamma\backslash D$.
Consider the embedding 
\begin{equation*}
D\hookrightarrow D(F):= U(F)_{\C}
 \cdot D(\subset \check{D})\cong U(F)_{\C}
 \times \mathbb{C}^{k}\times F 
\end{equation*}
 and its $U(F)_{\mathbb{Z}} (:=U(F)\cap \Gamma)$-quotient,
 where $\check{D}$ is the compact dual of $D$. 
For each $\Gamma$-admissible collection of polyhedra $\{\sigma_\alpha^F\}$ 
(\cite[III.~Definition~5.1]{AMRT}), 
 we apply toric construction to get a natural partial compactification 
\begin{align}\label{Siegel.domain}
U(F)_{\mathbb{Z}}\backslash D(F)
  &\cong T_{U(F)_{\mathbb{Z}}}\times \mathbb{C}^{k}\times F \\ \nonumber
& \subset T_{U(F)_{\mathbb{Z}}}{\rm emb}\{\sigma_{\alpha}^{F}\}
  \times \mathbb{C}^{k}\times F \\ \nonumber
&=:(U(F)_{\Z}\backslash D(F))_{\{\sigma_{\alpha}^{F}\}}, 
\end{align}
 where $U(F)_{\Z}:=U(F)\cap \Gamma$ is a free abelian group
 of finite rank.
Let
 $\overline{(U(F)_{\mathbb{Z}}\backslash D)}_{\{\sigma_{\alpha}^{F}\}}$ be 
 the interior of the closure of the image of $D$ inside
 $\overline{(U(F)_{\Z}\backslash D(F))}_{\{\sigma_{\alpha}^{F}\}}$. 
Then we divide 
 $\overline{(U(F)_{\mathbb{Z}}\backslash D)}_{\{\sigma_{\alpha}^{F}\}}$ by 
``the rest part of the discrete subgroup''
 $U(F)_{\mathbb{Z}}\backslash (\Gamma\cap N(F))$
 and then naturally glue together to get 
 the desired toroidal compactification
 $\overline{\Gamma\backslash D}_{\{\sigma_{\alpha}^{F}\}}$. 
Then the set of all boundary components (prime divisors) of partial toric compactification
 of $\Gamma\backslash D$ in this $F$-direction has a bijection 
 with rays in $\{\sigma_{\alpha}^{F}\}$. 
From the above construction and our extension, 
we again have the same set-theoretic description 
\begin{eqnarray}\label{MS.strat}
\overline{\Gamma\backslash D}^{\rm MSBJ}
 &\cong &\Gamma\backslash D\sqcup \bigsqcup_{F}
 (\Gamma\cap N(F))\backslash C(F)/\mathbb{R}_{>0} \\ \nonumber
&=& \Gamma \backslash D \sqcup \bigsqcup_{F}
 (\Gamma\cap N(F))\backslash S(F).
\end{eqnarray}

Since the above two compactifications (\ref{Satake.strat}) and (\ref{MS.strat}) have 
the same stratification (cf., also \cite[Proposition 2.1.1]{HZII}), it suffices to show the 
equivalence of the both topologies by some analysis.

Note that the topology on our Satake compactification is (locally) metrizable as it is 
dominated by the Borel-Serre compactification, a manifold with corners up to finite covering, 
and so is the Morgan-Shalen type compactification \cite[Appendix]{TGC.II} from its 
construction. Hence, it is enough to show that 

\begin{Lem}\label{all.seq}
A converging sequence in 
the Satake compactification $\overline{\Gamma\backslash D}^{\rm Sat,\tau_{\rm ad}}$ 
is also a converging sequence in the 
Morgan-Shalen type compactification $\overline{\Gamma\backslash D}^{\rm MSBJ}$.  
\end{Lem}

In the above statement, a natural bijection 
between two compactifications 
\[I\colon \overline{\Gamma\backslash D}^{\rm Sat,\tau_{\rm ad}} \to 
\overline{\Gamma\backslash D}^{\rm MSBJ}\]
which is defined as the identification of two stratifications \eqref{Satake.strat}, \eqref{MS.strat} 
is used. We first reduce the proof of above lemma to a special case by a standard argument: 

\begin{Claim}\label{gen.seq}
If a sequence $\{x_i\}_{i=1,2,\dots}$ in $\Gamma\backslash D$ converges to 
 $x_{\infty}\in \overline{\Gamma\backslash D}^{\rm Sat,\tau_{\rm ad}}$ with
 respect to the Satake topology on
 $\overline{\Gamma\backslash D}^{\rm Sat,\tau_{\rm ad}}$, 
then $\{I(x_i)\}$ converges to $I(x_{\infty})$ in $\overline{\Gamma\backslash D}^{\rm MSBJ}$. 
\end{Claim}

\begin{proof}[proof of ``Claim~\ref{gen.seq} implies Lemma~\ref{all.seq}"]
We put metrics $d_{\rm Sat}$ and $d_{\rm MS}$ on 
$\overline{\Gamma\backslash D}^{\rm Sat,\tau_{\rm ad}}$ and 
$\overline{\Gamma\backslash D}^{\rm MSBJ}$ respectively, 
which are compatible with their topologies. 
Take an arbitrary converging sequence $x_i\to x_\infty$ as $i\to \infty$ 
in $\overline{\Gamma\backslash D}^{\rm Sat,\tau_{\rm ad}}$.
Each $x_i$ (with fixed $i$) 
 can be approximated by a sequence 
 $\{x_i^j\}_{j=1,2,\cdots}$ in $D$ 
 which converge to $x_i$ in the Satake compactification. 
By Claim~\ref{gen.seq}, $I(x_i^j)\to I(x_i)$ as $j\to \infty$. 
For each $i$, take an index $j(i)$ such that 
 $d_{\rm Sat}(x_i^{j(i)}, x_i)<\frac{1}{i}$ and 
 $d_{\rm MS}(I(x_i^{j(i)}), I(x_i))<\frac{1}{i}$.
Then a sequence $\{x_i^{j(i)}\}_i$ converges to $x_\infty$
 and hence $\{I(x_i^{j(i)})\}_i$ converges to $I(x_\infty)$
 by Claim~\ref{gen.seq} again. 
Therefore, $I(x_i)\to I(x_{\infty})$ as $i\to \infty$. 
\end{proof}

It only remains to prove Claim~\ref{gen.seq}. 
Suppose that a sequence $\{x_i\} \in \Gamma\backslash D$
 converges to $x_\infty\in (\Gamma\cap N(F)) \backslash S(F)$. 
Let $Q:=N(F)$ with the Langlands decomposition $Q=W(F)A_Q M_Q$. 
The convergence of $x_i$ to $x_{\infty}$ with respect to the Satake topology 
 means that $x_i$ has a representative $\tilde{x}_i\in D$ with the decomposition
\[\tilde{x}_i = n_ia_il_iz_i\in W(F)\cdot A_{Q}\cdot S(F)\cdot F,\]
such that 
\begin{enumerate}
\item $n_i$ stay bounded, 
\item \label{div}
$a_i^{\alpha}$ diverges to $+\infty$ for 
every $\alpha \in \Delta(\Lie(W(F)), \Lie(A_Q))$, and 
\item $l_i$ converges to $x_{\infty}$ in $(\Gamma\cap N(F))\backslash S(F)$. 
\end{enumerate}

Choose a minimal boundary stratum $F'$ of $F$
 so $(\Gamma\cap N(F')) \backslash F'$ is compact. 
Here, $N(F')$ is the normalizer of $F'$ and we will write $W(F')$ for its unipotent radical.
By replacing $\tilde{x}_i$ and taking a subsequence if necessary we may assume that
\[\tilde{x}_i= w_i v_i z'_i \in W(F')\cdot C(F')\cdot F'\]
 have the following conditions:
\begin{enumerate}
\item $w_i$ stay bounded, 
\item $v_i$ for all $i$ belong to one Siegel set in $C(F')$
 and they diverge to $\infty$,
\item $[v_i]\in S(F')=C(F')/\R_{>0}$ 
 converge to a point
 $\tilde{x}_{\infty}\in S(F)\subset \overline{S(F')}$
 which is sent to $x_{\infty}$ by the quotient map
 $S(F)\twoheadrightarrow (\Gamma\cap N(F))\backslash S(F)$, and
\item $z'_i$ stay bounded.
\end{enumerate}

Fix a $\Gamma$-admissible collection of rational polyhedral
 decomposition 
$\{\sigma_{\alpha}^{F'}\}$, which we assume to be regular for simplicity. Then 
the set of all boundary components (prime divisors) of partial toric compactification
 of $\Gamma\backslash D$ in this $F'$-direction has bijection 
with rays in $\{\sigma_{\alpha}^{F'}\}$. 
Note that the lattice we use
 for the partial toric construction 
is $\Gamma\cap N(F')$ so we denote its dual lattice by $(\Gamma\cap N(F'))^{*}$. 
We may assume all $v_i$ lie in one polyhedral cone $\sigma_{\alpha}^{F'}$
 by \cite[II.\ Corollary~4.3]{AMRT}.
Let 
\[\tilde{x}_i=(u_i, y_i, z'_i) \in U(F')_\C\times \C^k\times F'\]
 be the decomposition with respect to the description \eqref{third.kind} for $F'$.
Recall from \cite[III.\S4.3]{AMRT} that the projection map $D\to \C^k\times F'$
 can be identified with an $N(F')^o$-equivariant map
 $N(F')^o/ (K_{\ell} \cdot K_h \cdot M)
 \to N(F')^o/(U(F')\cdot G_{\ell} \cdot K_h \cdot M)$
 and hence with $W(F')\times C(F') \times F' \to (W(F')/U(F')) \times F'$.
Then our conditions on $\tilde{x}_i=w_iv_iz'_i$ imply that
 $(y_i,z'_i)\in \C^k\times F'$ are bounded.
Also, recall from \cite[A.12, A.13]{TGC.II} that  
 the limit inside the Morgan-Shalen type compactification encodes 
 the limit of ratios of
$$
\bigl\{-\log |e^{2\pi i\langle m, u_i\rangle}|
 =2\pi\im\langle m, u_i\rangle\bigr\}_{m\in {(\sigma_{\alpha}^{F'})}^{\vee}\cap (\Gamma\cap N(F'))^{*}}, 
$$
where $(\sigma_{\alpha}^{F'})^{\vee}
 :={\operatorname{Hom}}_{\text{monoid}}
 (\sigma_{\alpha}^{F'}\cap \Gamma\cap N(F'),\,\R_{\ge 0})$. 
Since $(y_i,z'_i)$ are bounded,
 $\im u_i-v_i=h_{z'_i}(y_i,y_i)$ are also bounded.
Hence the limit of ratios of 
 $\{\im\langle m, u_i\rangle\}$ is equal to
 that of $\{\langle m, v_i\rangle\}$, which corresponds to the point $\tilde{x}_\infty$.
We therefore conclude that the limit of $I(x_i)$ exists and equals $I(x_{\infty})$, 
 completing the proof of Claim~\ref{gen.seq} and that of Theorem~\ref{MS.Sat}. 
\end{proof}


\begin{Rem}[Non-simple Lie group case]
Recall that we assumed $G$ is a \textit{simple} Lie group at the beginning 
of this chapter. 
Even if $G=\mathbb{G}(\mathbb{R})$
 is \textit{not} simple, we can apply
 the Morgan-Shalen-Boucksom-Jonsson compactification
 (\cite{BJ16}, \cite[Appendix]{TGC.II})
 to toroidal compactifications of a locally symmetric space
 $\Gamma\backslash D=\Gamma\backslash G/K$
 for an arithmetic group $\Gamma$. 
From \cite[A.12, A.13 (also cf.\  A.10)]{TGC.II}, this does not 
depend on the choice of the combinatoric data --- the admissible 
cone decomposition. 
We expect this can be still reconstructed in the Satake's 
representation theoretic manner 
\cite{Sat1, Sat2} and hope to treat it more properly in future (cf., comments in \cite[III.3.18]{BJ}). 

\begin{Ex}
Let us think of the Hilbert modular varieties (cf.\  e.g.\ \cite{vG}). 
In this case, 
$\mathbb{G}={\rm Res}_{F/\Q}(SL_{2})$ for a totally real number field $F$. 
Then $G\simeq SL_{2}(\R)^{[F:\Q]}$ as a Lie group 
and $D\simeq \mathbb{H}^{[F:\Q]}$, a product of upper half plane,
 as a symmetric space 
since $F\otimes_{\Q}\R\simeq \R^{[F:\Q]}$. We set $r:=[F: \Q]$. 
(Most typical and classically well-studied from the time of 
Hilbert-Blumenthal, is when $F$ is a (real) quadratic 
field and $\Gamma=SL_{2}(\mathcal{O}_{F})\subset \mathbb{G}(\Q)=SL_{2}(F)$.)

Note that the $\Q$-rank of $\mathbb{G}$ is $1$ and 
the boundary components of our Morgan-Shalen compactification (after \cite[Appendix]{TGC.II}) is 
real $(r-1)$-dimensional while the adjoint representation 
$\mathfrak{g}\simeq \mathfrak{sl}_{2}(\R)^{\oplus r}$ 
is not 
irreducible so that the framework of Satake \cite{Sat2} does not apply directly. 

This is not directly related to our main contents, but note that the Satake-Baily-Borel compactification's 
($0$-dimensional) boundary components  are famously known to be in 
one to one correspondence with the ideal classes of $\mathcal{O}_{F}$. 
Hence the number of the boundary components is the class number 
$\# {\rm Cl}(\mathcal{O}_{F})$. 
\end{Ex}

From the next subsection, again we discuss under the assumption that $G$ is a simple Lie group,
 to avoid complication.
\end{Rem}


\subsection{Extendability of holomorphic morphism and monodromy}\label{extension.monodromy}

The case of moduli of principally polarized abelian varieties, i.e., $A_g$, 
of the following phenomenon 
is partially proved in \cite{TGC.II} by using 
Mumford-Faltings-Chai uniformization description. 

\begin{Thm}[Extension to Satake compacification]\label{rationality.HSD}
Take an arbitrary holomorphic map $f\colon \Delta^*\to \Gamma\backslash D$, where 
$\Delta^*:=\{z\in \mathbb{C}\mid z\neq 0,\, |z|<1\}$. 
Then $f$ also extends to a continuous map $\Delta\to \overline{\Gamma\backslash D}^{\rm Sat,\tau_{\rm ad}}$ 
where $0\in \Delta$ is sent to a ``rational point'', i.e., a point 
in $(C(F)\cap U(F)_{\Q})/\Q_{>0}\subset S(F)$. 
\end{Thm}

\begin{proof}
From the extension theorem (\cite{Bor, Kie, KO}), 
we get a holomorphic extension of $f$ to a holomorphic map from $\Delta$ to the 
Satake-Baily-Borel compactification $\Delta\to \overline{\Gamma\backslash D}^{\rm SBB}$. 
We suppose $0$ is sent to a point in a boundary component $(\Gamma\cap N(F))\backslash F$. 
From the valuative criterion of properness applied to birational proper surjective morphisms 
with the local ring of convergence power series at $0\in \Delta$, it follows that 
for any proper variety $\overline{(\Gamma\backslash D)}$ compactifying $\Gamma\backslash D$, 
$f$ extends to a holomorphic map from $\Delta$ to $\overline{(\Gamma\backslash D)}$. 
In particular, $f$ is automatically meromorphic 
in the sense of e.g., \cite{BJ16}. 
From the construction of toroidal compactification in \cite[Chapter III, \S5, \S6, especially Proposition 6.10]{AMRT},
after a finite base change of $\Delta$, 
we can assume that the morphism $f$  lifts to $\Delta^*\to U(F)_{\Z}\backslash D$, 
which we denote by $\bar{f}$. 
As we reviewed in 
\eqref{Siegel.domain} (or see \cite[Chapter III, \S4]{AMRT}), we obtain an extension 
$$\bar{f}\colon \Delta\to T_{U(F)_{\mathbb{Z}}}{\rm emb}\{\sigma_{\alpha}^{F}\}\times \mathbb{C}^{k}\times F,$$ 
where 
$\Delta^*$ is sent to $T_{U(F)_{\Z}}\times \mathbb{C}^{k}\times F$ 
while we can assume the cone decomposition $\{\sigma_{\alpha}^{F}\}$ 
is regular. We consider its composite 
$$\varphi:=(p_1\circ \bar{f})\colon \Delta\to T_{U(F)_{\mathbb{Z}}}{\rm emb}\{\sigma_{\alpha}^{F}\}$$ 
with the first projection $p_1$. Thus the proof of the desired extendability is reduced to that of the following general theorem, 
which itself is of own interest. 

\begin{Thm}[Extension to Morgan-Shalen compactification]\label{rationality.MSBJ}
Suppose $X$ is a smooth proper variety
 and $Y$ is a simple normal crossing divisor of $X$. Set $U:=X\setminus Y$. 
For a holomorphic morphism $\varphi\colon \Delta\to X$
 whose restriction $\varphi^o:=\varphi|_{\Delta^*}$ maps to $U$, 
$\varphi^o$ extends also to a continuous map
 $$\overline{\varphi^{o}}\colon \Delta\to \overline{U}^{\rm MSBJ}(X)$$
 where $\overline{U}^{\rm MSBJ}(X)$ denotes the Morgan-Shalen-Boucksom-Jonsson compactification
 \cite{BJ16} (cf., also \cite[A.1]{TGC.II}) originally denoted by $\overline{U}^{\rm hyb}(X)$. 
\end{Thm}

\begin{proof}[proof of Theorem~\ref{rationality.MSBJ}]
In fact, the extendability, i.e., the existence of $\overline{\varphi^{o}}$ itself is essentially a special case of the functoriality of the Morgan-Shalen construction, observed in \cite[Theorem A.15]{TGC.II}, 
although we consider the \textit{analytic} open disk $\Delta$. 
Indeed, we can see \textit{op.cit.}\ as a partial generalization
 of Theorem~\ref{rationality.MSBJ}
 to higher dimensional domains. 
Here, we give a direct proof in our situation for the sake of convenience and completeness, which 
includes a specification of the image of the origin. 

Let us  take holomorphic local coordinates $(z_1,\cdots,z_n)$
 on a neighborhood $V$ of $\varphi(0)\in X$
 such that $Y\cap V=\prod_{i=1}^{m}z_i=0$ for some $m\le n$. 
Then, by the coordinate $t$ on $\Delta$, $\varphi$ can be written as
 $\varphi(t)=(f_1(t),f_2(t),\cdots,f_n(t))$ with convergent power series $f_i(t)$. 
Then it is easy to confirm that a continuous map $\overline{\varphi^{o}}$ exists
 such that  $\overline{\varphi^{o}}(0)$ in the dual complex
 $\Delta(Y\cap V)$ (an $(m-1)$-simplex)
 has barycenter coordinate 
\[\Bigl(\frac{{\rm ord}_t f_1}{\sum_i {\rm ord}_t f_i},
 \frac{{\rm ord}_t f_2}{\sum_i {\rm ord}_t f_i}, 
\cdots, \frac{{\rm ord}_t f_m}{\sum_i {\rm ord}_t f_i}\Bigr).\] 
We end the proof of Theorem~\ref{rationality.MSBJ}. 
\end{proof}
Therefore, combined with Theorem~\ref{MS.Sat}, 
the above completes the proof of Theorem~\ref{rationality.HSD}. 
\end{proof}

\bigskip

In the case when the symmetric domain $D$ is a period domain for Hodge structures, 
such as the moduli spaces of marked abelian varieties or marked polarized K3 surfaces,
 above Theorem~\ref{rationality.HSD} can be rephrased as 
the existence of rational Hodge structures as a limit of varying Hodge structures along holomorphic 
degenerations. Now we discuss this more rigorously and systematically from general Hodge theoretic viewpoint. 
(We will also slightly extend Theorem~\ref{rationality.HSD} and 
 Theorem~\ref{rationality.MSBJ} in Proposition~\ref{MS.reconstruction} later.)

Suppose that a variation of polarized $\Z$-Hodge structures
 of weight $n$
 on the punctured disk $\Delta^*$ is given. 
For a reference point in $\Delta^*$, 
 let $H_{\Z}$ be the assigned integral Hodge structure
 with a polarization $Q$.
Let $\G:=\Aut(H_{\Q},Q)$
 be an automorphism group defined over $\Q$ 
 and let $G$ be the connected open subgroup of the Lie group $\G(\R)$.
Then we have a holomorphic period mapping 
 $\Phi\colon \Delta^*\to \Gamma\backslash D$,
 where $D$ is a period domain of Griffiths (\cite{Griffiths}) and $\Gamma$ is an arithmetic subgroup of $\G(\R)$  
 which 
 includes the monodromy at $c\in \Delta\setminus \{0\}$ (we fix $c$) around $0\in\Delta$. 
 
 Note that if we set $K:=K_{\Phi(c)}$ as the isotropy group of the $G$-action on $D$, 
 there is a natural $G$-equivariant homeomorphism $D\simeq G/K$. 
We assume Notations (i), (ii), (iii) at the beginning
 of \S\ref{Gen.HSD.start}. 
In particular, the period domain $D=G/K$
 is a Hermitian symmetric space. 
The monodromy at $c\in \Delta\setminus \{0\}$ around $0\in\Delta$ will be denoted by 
 $\gamma\in \Gamma$. 
If $\gamma=\gamma_s\gamma_u$ 
 is its Jordan decomposition,
 the semisimple part $\gamma_s$ is known to be of finite order 
 (\cite{Lan}, cf., also \cite[\S3]{Griffiths}). 

Let us revisit
 the construction of the Satake compactification
 $\overline{\Gamma\backslash D}^{\rm Sat,\tau_{\rm ad}}$
  (see \S\ref{Satake.adjoint}). 
 For each point $z\in D$,
 we have a Hodge decomposition
 $H_{\C}=\bigoplus_{p+q=n}H^{p,q}_z$.
Define $\iota(z)\in {\rm End}(H_{\C})$ to be
\begin{equation}\label{emb.to.Pg}
\iota(z)|_{H^{p,q}_z}=\sqrt{-1}(p-q).
\end{equation}
On the other hand, we have a $\mathfrak{g}$-action
 on $H_{\C}$ for Lie algebra $\mathfrak{g}=\operatorname{Lie} (G)$. 
 Let us set $K_{z}$ to be the isotropy subgroup of $z$ in $G$, i.e., which preserves the 
 Hodge structure encoded in $z$, and its Lie algebra as $\mathfrak{k}_{z}$. 
Then it is easy to see that $\iota(z)$
 is given by action of an element in $\mathfrak{z}(\mathfrak{k}_{z})$,
 which we also denote by $\iota(z)$. 
 Here, $\mathfrak{z}(\mathfrak{k}_{z})$ 
 is the center of $\mathfrak{k}_{z}$. 
This gives an embedding $\iota\colon D\hookrightarrow \mathfrak{g}$
 and then $\iota\colon D\hookrightarrow \mathbb{P}(\mathfrak{g})$,
 which is compatible with our previous description in \S\ref{Satake.adjoint}. 
 Indeed, this is just a Hodge-theoretic re-description of the (unique) embedding $\iota$ of Hermitian symmetric domain 
 $D$ introduced in \S\ref{Satake.adjoint}. 
The Satake compactification 
 $\overline{\Gamma\backslash D}^{\rm Sat,\tau_{\rm ad}}$
 is defined as a quotient by $\Gamma$ of a partial compactification
 of $D$ in the projective space $\mathbb{P}(\mathfrak{g})$ by \cite[Lemma~2 and Corollary~4]{Sap}, for instance, 
 as we reviewed in \S\ref{Satake.adjoint}. 

Let $N := \log \gamma_u \in \mathfrak{g}$ and assume $N\neq 0$. 
It will turn out (in the proof of Theorem~\ref{monodromy})
 that the point $[N]\in \mathbb{P}(\mathfrak{g})$
 lies in the partial compactification of $D$ and 
 hence gives a point 
 $[N]\in \overline{\Gamma\backslash D}^{\rm Sat,\tau_{\rm ad}}$.

The nilpotent orbit theorem by Schmid~\cite{Schm} relates
 the asymptotic behavior of the period mapping and
 the nilpotent orbit through $N$, which is determined by the monodromy.
This enables us to show directly that 
 the period mapping extends to the Satake compactification
 with the limit point given by $[N]$,
 without relying on our previous Theorem~\ref{MS.Sat},
 while the proof of Theorem~\ref{rationality.HSD} does rely on Theorem~\ref{MS.Sat}. 

\begin{Thm}[Revisiting extension $\&$ Relation with monodromy]\label{monodromy}
In the above setting, 
 the period mapping $\Delta^*\to \Gamma\backslash D$
 extends continuously to
 $\Delta \to \overline{\Gamma\backslash D}^{\rm Sat,\tau_{\rm ad}}$
 which sends $0\in \Delta$ to
 $[N]\in \overline{\Gamma\backslash D}^{\rm Sat,\tau_{\rm ad}}$.
\end{Thm}

\begin{proof}
The proof depends on the nilpotent orbit theorem
 and the $SL_2$-orbit theorem by Schmid~\cite{Schm},
 results on the asymptotics of the period mapping.
By taking a finite covering $z\mapsto z^m$ of $\Delta^*$ if necessary,
 we may and do assume that $\gamma_s$ is the identity element.
We have the universal covering map
 $\mathbb{H}:=\{x+\sqrt{-1}y \in\C \mid y>0\}\to \Delta^*$,
 $z\mapsto e^{2\pi\sqrt{-1}z}$
 and then the map
 $\Phi(e^{2\pi\sqrt{-1}-})\colon
 \mathbb{H} \to \Gamma\backslash D$
 lifts to a holomorphic map
 $\tilde{\Phi}\colon \mathbb{H}\to D$
 such that
 $\tilde{\Phi}(z+1)=\gamma\cdot\tilde{\Phi}(z)$.

We now utilize \cite[(5.26) Theorem]{Schm}, which is a consequence
 of the nilpotent orbit theorem and the $SL_2$-orbit theorem.
This involves
\begin{itemize}
\item
 a homomorphism $\psi: SL(2) \to \G$ of algebraic groups over $\Q$,
\item
 a minimal $\Q$-parabolic subgroup $P$ of $G$
 with Langlands decomposition $P=RTM$,
 where $R$ is the unipotent radical,
 $T$ is a maximal $\Q$-split torus,
 and $M$ is anisotropic over $\Q$.
\end{itemize}
Let $\psi_*\colon \mathfrak{sl}(2,\R)\to \mathfrak{g}$
 denote the differential of $\psi$.
Then 
\[\psi_* \begin{pmatrix}0 & 1 \\ 0 & 0\end{pmatrix} = N,
 \qquad  \psi_* (Y) \in \operatorname{Lie} (T),
 \text{ where } Y=\begin{pmatrix}-1 & 0 \\ 0 & 1\end{pmatrix},\]
and $\langle \alpha, \psi_*(Y)\rangle\leq 0$ for any root
 $\alpha\in \Delta(\operatorname{Lie} (R), \operatorname{Lie} (T))$. 

\cite[(5.26) Theorem]{Schm} states that
 if the base point $o\in D$ is suitably chosen,
 there exist functions $r(x,y)$, $t(x,y)$, $m(x,y)$
 on $\{(x,y)\in \R^2 \mid y>\beta \}$ for some $\beta>0$
 with values in $R$, $T$, $M$, respectively, 
 such that the following are satisfied:
\begin{enumerate}
\item  $\tilde{\Phi}(x+\sqrt{-1}y) = r(x,y)t(x,y)m(x,y)\cdot o$,
\item the limits of $r(x,y)$,
 $\exp(\frac{1}{2}\log y \psi_*(Y))t(x,y)$,
 $m(x,y)$ as $y\to \infty$ exist
 locally uniformly for $x$.
\item $\lim_{y\to \infty}\exp(\frac{1}{2}\log y \psi_*(Y))t(x,y)=e$,
 and $\lim_{y\to \infty}m(x,y)=e$. 
\end{enumerate}
Moreover, it can be seen from the proof of \cite[(5.26) Theorem]{Schm}
 that the limit $\lim_{y\to \infty}r(x,y)\in G$
 fixes the point $[N]\in \mathbb{P}(\mathfrak{g})$.
Furthermore, $\tilde{\Phi}(x+\sqrt{-1}y)$
 for $|x|\leq C$, $y\gg 0$ lie in one Siegel set
 by \cite[(5.29) Corollary]{Schm}. 
Therefore, it is enough to show that 
 $\iota(\tilde{\Phi}(x+\sqrt{-1}y))$ as $y\to \infty$ converges to $[N]$
 in $\mathbb{P}(\mathfrak{g})$ locally uniformly for $x$, 
 where $\iota\colon D\hookrightarrow \mathfrak{g}$ is as in \eqref{emb.to.Pg}. 

Combining above claims, we have
\[
\tilde{\Phi}(x+\sqrt{-1}y) = r(x,y)m(x,y)
\exp(\tfrac{1}{2}\log y \psi_*(Y))t(x,y)
\exp(-\tfrac{1}{2}\log y \psi_*(Y))\cdot o
\]
and  $\lim_{y\to \infty} r(x,y)m(x,y)
 \exp(\frac{1}{2}\log y \psi_*(Y))t(x,y)$
 fixes $[N]$.
Hence it is enough to show that 
$\iota(\exp(-\frac{1}{2}\log y \psi_*(Y))\cdot o)\to [N]$
 as $y\to \infty$.

A Hodge structure on $\mathfrak{sl}(2,\R)$ is given by
\begin{align*}
&\mathfrak{sl}(2,\C)^{-1,1}=\C X_+,\quad 
\mathfrak{sl}(2,\C)^{0,0}=\C Z,\quad
\mathfrak{sl}(2,\C)^{1,-1}=\C X_-,\\
&\text{where}\ \ 
X_+ = \frac{1}{2}\begin{pmatrix}-\sqrt{-1} & 1 \\ 1 & \sqrt{-1}\end{pmatrix},\quad 
Z = \begin{pmatrix} 0 & - \sqrt{-1} \\ \sqrt{-1} & 0\end{pmatrix},\\
&X_- = \frac{1}{2}\begin{pmatrix} \sqrt{-1} & 1 \\ 1 & -\sqrt{-1}\end{pmatrix}.
\end{align*}
Let $H_{\C}=\bigoplus H^{p,q}_o$ be the Hodge decomposition
 assigned to the chosen base point $o$,
 which induces a Hodge structure of $\mathfrak{g}$
 via $\mathfrak{g}\to {\rm End}(H_{\C})$.
Then $\psi_*$ is a mapping of type $(0,0)$.
Hence 
\[
\psi_*(X_+)(H^{p,q}_o) \subset H^{p-1,q+1}_o,\quad
\psi_*(Z)(H^{p,q}_o) \subset H^{p,q}_o,\quad
\psi_*(X_-)(H^{p,q}_o) \subset H^{p+1,q-1}_o.
\]
Recall that $\iota(o)\in{\rm End}(H_{\C})$
 acts by $\sqrt{-1}(p-q)$ on $H^{p,q}_o$, where $\iota$ is the embedding of $D$ into ${\rm End}(H_{\C})$ (or 
 $\mathbb{P}(\mathfrak{g})$) we introduced earlier. 
Therefore, $[\psi_*(X_+), \iota(o)]=2\sqrt{-1}\psi_*(X_+)$.
Since $[\psi_*(X_+),\psi_*(Z)]=\psi_*([X_+,Z])
 =-2\psi_*(X_+)$,
 the operator $\iota(o)+\sqrt{-1}\psi_*(Z)$ commutes with
 $\psi_*(X_+)$.
Similarly, $[\iota(o)+\sqrt{-1}\psi_*(Z),\psi_*(X_-)]=0$
 and hence $\iota(o)+\sqrt{-1}\psi_*(Z)$ commutes with $\psi_*(\mathfrak{sl}(2,\R))$.
Then we calculate
\begin{align*}
&\iota(\exp(-\tfrac{1}{2}\log y \psi_*(Y))\cdot o) \\
&= {\rm Ad}\bigl(\exp(-\tfrac{1}{2}\log y \psi_*(Y))\bigr) \iota(o) \\
&= \iota(o)+\sqrt{-1}\psi_*(Z)
 - {\rm Ad}\bigl(\exp(-\tfrac{1}{2}\log y \psi_*(Y))\bigr)
  (\sqrt{-1}\psi_*(Z)).
\end{align*}
The last term is calculated as 
\begin{align*}
&{\rm Ad}\bigl(\exp(-\tfrac{1}{2}\log y \psi_*(Y))\bigr)
  (\sqrt{-1}\psi_*(Z))\\
&= \psi_*
\Bigl(
\begin{pmatrix} y^{\frac{1}{2}} & 0 \\ 0 & y^{-\frac{1}{2}}\end{pmatrix}
\begin{pmatrix} 0 & 1 \\ -1 & 0 \end{pmatrix}
\begin{pmatrix} y^{\frac{1}{2}} & 0 \\ 0 & y^{-\frac{1}{2}}\end{pmatrix}^{\!\!-1}
\, \Bigr) \\
&= \psi_*
 \begin{pmatrix} 0 & y \\ -y^{-1} & 0 \end{pmatrix}.
\end{align*}
The other terms do not depend on $y$ and hence they are bounded.
We therefore conclude that 
 $\iota(\exp(-\tfrac{1}{2}\log y \psi_*(Y))\cdot o)$
 tends to $$\biggl[\psi_* \begin{pmatrix}0 & 1 \\ 0 & 0\end{pmatrix}\biggr]=[N]$$ 
 as $y\to \infty$ as desired. 
\end{proof}
From the proof above, the statements will remain valid for 
any $\Gamma$ which includes any discrete subgroup of $G$ which contains the monodromy of 
the variation of Hodge structure on $\Delta^{*}$. 

\section{Application to (co)homology}

Here we give a direct application of Theorem~\ref{MS.Sat} to some topological aspect which is again a classical subject, that is, 
the (co)homologies of arithmetic subgroups $\Gamma\subset \mathbb{G}(\Q)$. We do this 
through the well-known isomorphism $H^{i}(\Gamma\backslash D;\Q)\simeq H^{i}(\Gamma,\Q)$
 for a locally symmetric (orbi-)space $\Gamma\backslash D$ (cf., e.g., \cite{Edi}). 

Take an orbi-smooth toroidal compactification of $\Gamma\backslash D=X$ as 
$\overline{X}=X\sqcup \partial\overline{X}$ with orbi-snc divisor $\partial\overline{X}$. 
Then Theorem~\ref{MS.Sat} asserts that its dual intersection complex in the slightly extended sense of 
\cite[A.12, A.13]{TGC.II} has the following structure: 
\begin{align*}
\Delta(\partial\overline{X})=S_{1}\sqcup \cdots \sqcup S_{l}, 
\end{align*}
where $S_{i}$ is the disjoint union of finite locally symmetric spaces $\Gamma_{i,j}\backslash D_{i}$ 
for $1\le j\le m_{i}$ of the same dimensions $d_{i}$. Here, $D_{i}$ is a Riemannian symmetric space for 
a reductive subgroup $\mathbb{G}_{i}$ of certain rational parabolic subgroup of $\mathbb{G}$ and 
$\Gamma_{i,j}$ are arithmetic subgroups of $\mathbb{G}_{i}(\Q)$. We assume $d_{1}>\cdots > d_{l}$ and 
set $S_{0}:=X$, $d:=d_{0}:=\dim_{\C}(X)$. 

On the other hand, due to the theory of the mixed Hodge structure of Deligne \cite{Del71, Del74}, we have the natural weight filtration (increasing filtration) $W$ on $H^{i}(X,\Q)\simeq H^{i}(\Gamma,\Q)$ which satisfies the following. 

\begin{Cor}[of Theorem~\ref{MS.Sat}]\label{application.cohomology}
For a positive integer $k$ with $1\le k\le 2d$, we have 
\begin{align}\label{cohom.isom}
{\rm Gr}^{W}_{2d}H^{2d-k}(\Gamma,\Q)\simeq \tilde{H}_{k-1}(\partial{\overline{\Gamma\backslash D}^{\rm Sat, \tau_{\rm ad}}},\Q). 
\end{align}
Here, the left hand side is the top graded piece $W_{2d}H^{2d-k}(\Gamma,\Q)/W_{2d-1}H^{2d-k}(\Gamma,\Q)$, 
while the right hand side denotes the reduced homology of 
$\partial{\overline{\Gamma\backslash D}^{\rm Sat, \tau_{\rm ad}}}$. 

In particular, for instance, we have 
$$\dim_{\Q}H^{2d-k}(\Gamma,\Q)\ge \dim_{\Q}H_{k-1}(\partial{\overline{\Gamma\backslash D}^{\rm Sat, \tau_{\rm ad}}},\Q),$$ 
for $2\le k\le 2d$. 
\end{Cor}

Note that the boundary $\partial{\overline{\Gamma\backslash D}^{\rm Sat, \tau_{\rm ad}}}$ has a stratification as 
$\sqcup S_{i}$. 
Given Theorem~\ref{MS.Sat}, the above Corollary~\ref{application.cohomology} can be proved by using
 known fact of the Hodge theory. 
Indeed, we can directly apply \cite[Theorem 5.8]{CGP} (see also \cite[Theorem 3.1]{Hac}, \cite[Theorem 4.4]{Pay}) 
to an arbitrary orbi-smooth toroidal compactification of $\Gamma\backslash D$ and obtain the proof of Corollary~\ref{application.cohomology}. 
At least some special cases seem to be known; 
when $\mathbb{G}$ are certain inner forms of $Sp_{4}$ and $G=Sp_{4}(\R)$,
 the isomorphism is shown in \cite[Corollary~4.3]{OdaSchwermer09} 
(cf., also \cite[5.3, 5.5]{OdaSchwermer90}), 
which also uses the analysis of toroidal compactifications.

By further ``decomposing'' the right hand side 
to the contributions of each stratum $S_{i}$, which are all $\Q$-coefficients classifying spaces for the discrete groups $\Gamma_{i,j}$, 
we expect to obtain some relations of (co)homology of $\Gamma$ and those of $\Gamma_{i,j}$. 

In particular, the above isomorphism \eqref{cohom.isom} might be re-interpreted with 
 the theory of Eisenstein cohomologies, but 
 the authors do not know how this can be done. 
See \cite{OdaSchwermer90, OdaSchwermer09, HZII, Harris, Nair} for related works.


\newpage

\chapter{Abelian varieties case}\label{Abel.sec}


\section{Tropical geometric compactification of $A_g$ revisited}

This section studies the tropical geometric 
compactification $\overline{A_g}^{\rm T}$ of the moduli 
variety $A_{g}$ of complex $g$-dimensional principally polarized abelian 
varieties, introduced in \cite{TGC.II}. We refer to Theorem~\ref{TGC.Ag.review}
 to review the statements in {\it op.cit}. 

In algebraic geometry, 
the most classical and popular compactification of $A_g$ by Ichiro Satake \cite{Sat0} 
is the Satake-Baily-Borel compactification 
which is often 
also called simply 
``The'' Satake compactification or the Baily-Borel compactification. 
The history is that \cite{Sat0} constructed the compactification and later \cite{BB} proved it is actually underlying a normal 
projective variety. 

However, what we do here is to 
identify our tropical geometric compactification $\overline{A_g}^{\rm T}$ 
of $A_g$ with \textit{another} Satake's compactification constructed in 
\cite{Sat2}, i.e., that for the adjoint representation introduced in the previous section \S\ref{Satake.adjoint}. 
Recall the usual description of $A_{g}$ through the weight one polarized Hodge structures with
 the corresponding groups $\mathbb{G}=Sp(2g,\Q)$,
 $G=Sp(2g,\R)$, and $\Gamma=Sp(2g,\Z)$ as is well-known. 

\begin{Thm}\label{Ag.TGC.Satake.MS}
There are homeomorphisms between the three compactifications  
\[\overline{A_g}^{\rm T} \simeq \overline{A_g}^{\rm Sat,\tau_{\rm ad}}(\simeq  \overline{A_g}^{\rm MSBJ}),\]
extending the identity map on $A_g$. 
\end{Thm}

\smallskip
\begin{Rem}\label{ACP.Ag}
Recall that Abramovich-Caporaso-Payne \cite{ACP} proved that the moduli spaces of tropical curves, which also appears as our boundary of $\overline{M_{g}}^{\rm T}$ of Theorem \ref{TGC.Mg.review} (\cite{TGC.I}, \cite{TGC.II}), 
is homeomorphic to the dual intersection complex of 
the boundary of Deligne-Mumford compactification of $M_{g}$ as 
an algebraic stack. 

From the definition of the third term $\overline{A_g}^{\rm MSBJ}$,
the above Theorem \ref{Ag.TGC.Satake.MS} morally says that the (compactified) moduli of tropical abelian varieties 
are the dual intersection complex (in a generalized sense) of the toroidal algebro-geometric compactification of the 
moduli of complex abelian varieties. This can be seen as an analog of \cite{ACP} for $M_g$, for abelian varieties case.

We discuss K3 surfaces analog later in Remark \ref{ACP.K3}. 
\end{Rem}

\begin{Rem}\label{MZ.relation}
As we briefly reviewed from \cite{TGC.II} 
in our Introduction as 
Theorem \ref{TGC.Ag.review} from 
\cite{TGC.II}, the above compactification 
$\overline{A_g}^{\rm T}$ 
parametrizes Gromov-Hausdorff limits of 
principally polarized abelian varieties with 
the corresponding K\"ahler-Einstein metrics, with 
the fixed diameters rescale. They are real flat tori of 
dimension $i$ where $1\le i\le g$ so that the 
the boundary of the compactification is stratified with respect to 
$i$ accordingly. 

We clarify here how this real flat torus is related to the 
notion of {\it tropical abelian variety with principal polarization}  
introduced in \cite[\S 5.1]
{MZ} by Mikhalkin-Zharkov. That is a real torus $\R^{i}/\Lambda$ 
of dimension $i$ where $\R^{i}$ is given the natural 
integral affine structure by $\Z^{i}\subset \R^{i}$ and 
the matrix of a basis of $\Lambda$ is symmetric and positive 
definite. The latter condition comes from the assumed existence of 
tropical principal polarization, as a tropical analog of 
the Hodge-Riemann bilinear relation. 

Consider 
a real flat torus $V/\Lambda$ 
where $V$ is a $i$-dimensional real vector space and $\Lambda$ is 
its lattice, 
representing an arbitrary point in 
the boundary of $\overline{A_{g}}^{\rm T}$, we put an 
  integral affine structure simply by the basis of 
$\Lambda$. It is easy to see that its Legendre transform 
is $i$-dimensional principally polarized tropical abelian variety 
in the sense of \cite[\S 5.1]{MZ}. Conversely, 
any $i$-dimensional principally polarized tropical abelian variety 
in the sense of {\it loc.cit.} is obtained in such manner. 
This gives a relation with the theory of \cite[\S 5.1]{MZ}. 
For the idea of Legendre transforms of affine structures, 
it at least traces back to 
\cite[\S 5]{Hit97}, \cite[Proposition 1.2]{Gross}.
\end{Rem}

\begin{proof}

The second canonical homeomorphism exists as we proved in Theorem~\ref{MS.Sat} 
for general locally Hermitian symmetric spaces. 

Let us give an outline of the proof 
 for the first homeomorphism. 
We construct first a natural bijective map 
$$\varphi\colon \overline{A_g}^{\rm T}\to \overline{A_g}^{\rm Sat, \tau_{\rm 
ad}}$$
 extending the identity map of $A_g$.
To prove that $\varphi$ is a homeomorphism, 
 it suffices to show the following two equivalences of convergences: 

\begin{Claim}\label{Cl1.convergence}
The convergence of $[V(i)]\in A_g$ $(i=1,2,\cdots)$ to 
$[T(\infty)]\in \partial \overline{A_g}^{\rm T}$ and that of 
$\varphi([V(i)])\in A_g$ to $\varphi([T(\infty)])$ 
are equivalent. 
\end{Claim}

\begin{Claim}\label{Cl2.convergence}
The convergence of 
$[T(i)]\in \partial \overline{A_g}^{\rm T}$ $(i=1,2,\cdots)$
 to $[T(\infty)]\in \partial \overline{A_g}^{\rm T}$ 
and that of $\varphi([T(i)])\in \partial \overline{A_g}^{\rm Sat,\tau_{\rm ad}}$ to 
 $\varphi([T(\infty)])\in \partial \overline{A_g}^{\rm Sat,\tau_{\rm ad}}$ are 
 equivalent. 
\end{Claim}

Our construction of $\varphi$ is as follows. 
Recall that our tropical geometric compactification $\overline{A_{g}}^{\rm T}$ can be 
set-theoretically written as 
$A_g \sqcup \bigsqcup_{1\le d\le g} MT_{d}$ where $MT_{d}$ stands for 
the moduli of compact flat tori of dimension $d$ of diameter $1$. 
See the parametrized flat tori as $\mathbb{R}^d/\mathbb{Z}^d$ with 
the Gram matrix\footnote{We called it metric matrix in the previous paper \cite{TGC.II}.} 
$G\, {}^{t}G$ where $G$ is a real matrix of $d\times d$
 and ${}^{t}G$ denotes its transpose.
It follows that each $MT_{d}$ is naturally
 homeomorphic to the locally symmetric space  
$$GL(d,\mathbb{Z})\backslash GL(d,\mathbb{R})/(\mathbb{R}_{>0}\cdot O(d)).$$

Since the boundary of the Satake compactification $\overline{A_g}^{\rm Sat,\tau_{\rm 
ad}}$ is stratified as
$\bigsqcup_{1\le d\le g} GL(d,\mathbb{Z})\backslash GL(d,\mathbb{R})/(\mathbb{R}_{>0}\cdot O(d))$ 
by its definition,
 we can define $\varphi$ on the boundary as $\bigsqcup_{d}\psi_{d}^{-1}$, 
where $\psi_d\colon MT_{d}
 \to GL(d,\mathbb{Z})\backslash GL(d,\mathbb{R})/(\mathbb{R}_{>0}\cdot O(d))$ 
are natural homeomorphisms.

From here we are going to prove the above Claims \ref{Cl1.convergence} and \ref{Cl2.convergence}. 
Actually Claim~\ref{Cl2.convergence} follows from
 Claim~\ref{Cl1.convergence} by the same argument as the 
proof of Claim~\ref{gen.seq} 
but for its own interests, we provide proofs of both claims. 

We first prove Claim~\ref{Cl1.convergence}. 
For each principally polarized abelian variety, 
we will take a lift in a Siegel set inside the Siegel upper half space. 
Let us recall some basic facts and set up the notations. 
Set $G=Sp(2g,\mathbb{R})$ and its maximal compact subgroup 
$K\cong U(g)$. 
Recall that we have the Siegel upper half space
$$G/K \cong \mathfrak{H}_g:=\{ X+\sqrt{-1} Y\mid X,Y\in \mathfrak{gl}_g(\mathbb{R}),\, 
Y={}^{t} Y>0,\, X={}^{t} X\}, $$
where $G$ acts by the fractional transformation 
\[
\left(\begin{array}{c c}
     A     &               B    \\ 
   C      &                 D  \\
\end{array}\right)\colon 
\sqrt{-1} I_g \mapsto (\sqrt{-1}A +B)(\sqrt{-1}C+D)^{-1}\]
 and $I_g$ denotes the $g\times g$ identity matrix. 
The Iwasawa decomposition gives us 
$G=NAK$ such that
$$N=\left\{ 
\left(
\begin{array}{c|c}
  \alpha       & \beta    \\ \hline
   O       &  {}^t \alpha^{-1}   
\end{array}
\right) \in G
\right\}, \qquad
A=
\left\{ \tilde{\tau}=
\left(
\begin{array}{c|c}
  \tau       & O    \\ \hline
   O       &  \tau^{-1}  
\end{array}
\right) \in G
\right\},
$$
where $\alpha$ is an upper triangular unipotent $g\times g$ matrix and
 $\tau$ is a diagonal matrix ${\rm diag}(t_1,\cdots,t_g)$.
The product of the two matrices 
\[\left(\begin{array}{c|c}
  \alpha       & \beta    \\ \hline
   O       &  {}^t \alpha^{-1}   
\end{array}\right)
\left(
\begin{array}{c|c}
  \tau       & O    \\ \hline
   O       &  \tau^{-1}  
\end{array}
\right) \]
maps $\sqrt{-1}I_g\in \mathfrak{H}_g$ 
to 

$$(\beta\, {}^{t}\!\alpha)+
\sqrt{-1} (\alpha \, \tau^{2}\, {}^{t}\alpha)
\in \mathfrak{H}_g.$$ 

Also, recall that the Siegel reduction ensures that each point in $\mathfrak{H}_g$ 
can be translated by $Sp(2g,\Z)$-action to a point in the following Siegel set: 
$$\mathfrak{S}_{U,c}:=\{\nu \tilde{\tau}K \mid \nu \in U,\  t_i t_{i+1}^{-1}>c\, (\forall i<g),\  t_g>c\},$$
where $U$ is a relatively compact subset inside $N$ and $c> 0$ is a fixed real number 
(cf.\  \cite[I\S9, III 1.17]{BJ}). 

Now, going back to our sequence of principally polarized abelian varieties 
$V(i)$ $(i=1,2,\cdots)$ in concern, we can lift it to a sequence of points 
$\{(\alpha(i),\beta(i),\tau(i))\}$ 
inside the above $\mathfrak{S}_{U,c}$, and then to a sequence 
$
\begin{pmatrix}
     A(i)     &               B (i)   \\
   C(i)      &                 D(i)  \\
\end{pmatrix}
$
in $G$.

Recall that \cite[Theorem 2.1]{TGC.II} proved that 
the Gromov-Hausdorff limit of a sequence of 
principally polarized abelian varieties $V(i)$ 
is $\mathbb{R}^{g}/\mathbb{Z}^{g}$ 
with the Gram matrix\footnote{possibly degenerate! so that the torus has dimension lower than $g$} 
$\lim_{i\to \infty}[Y(i)] \in \mathbb{P}_{\mathbb{R}}(\mathfrak{gl}_{g}(\mathbb{R}))$, 
where $Y(i)$ is the imaginary part of the 
point $X(i)+\sqrt{-1}Y(i)$ inside the Siegel upper half space $\mathfrak{H}_g$ 
corresponding to $V(i)$ as our previous notation. 
Recall that a simple calculation shows 
$Y(i)=\alpha(i)\cdot \tau(i)^{2}\cdot {}^{t}\alpha(i)$. 

On the other hand, our Satake compactification comes from 
the natural representation 
$Sp(2g,\mathbb{R})\hookrightarrow GL(2g,\mathbb{R})$ 
and hence the limit point inside the Satake compactification corresponds to 
that of 
$\begin{pmatrix}
     A(i)     &               B (i)   \\
   C(i)      &                 D(i)  \\
\end{pmatrix}
\begin{pmatrix}
    {}^{t} A(i)     &      {}^{t} C (i)   \\
   {}^{t} B(i)      &     {}^{t} D(i)  \\
\end{pmatrix}
\in \mathbb{P}_{\mathbb{R}}(\mathfrak{gl}_{2g}(\mathbb{R})). 
$

Since we have $A(i)=\alpha(i)\tau(i)$, $B(i)=\beta(i)\tau(i)^{-1}$,
 $C(i)=0$, $D(i)={}^{t}\alpha(i)^{-1} \tau(i)^{-1}$ 
 in our setting, $\alpha(i)$ and $\beta(i)$ remain in some bounded set.
If we write $\tau(i)={\rm diag}(t_1(i),\dots,t_g(i))$,
 then $t_1(i)\to +\infty$ as $\tau(i)$ are not bounded. 
Therefore, to calculate
\[
\biggl[
\begin{pmatrix}
     A(i)     &               B (i)   \\
   C(i)      &                 D(i)  \\
\end{pmatrix}
\begin{pmatrix}
    {}^{t}A(i)     &      {}^{t}C (i)   \\
   {}^{t}B(i)      &     {}^{t}D(i)  \\
\end{pmatrix}
\biggr]\in \mathbb{P}_{\mathbb{R}}(\mathfrak{gl}_{2g}(\R)),\]
 it is enough to see the upper left part, 
 which is exactly the limit of $[Y(i)]$
 in $\mathbb{P}_{\mathbb{R}}(\mathfrak{gl}_{g}(\R))$. 
We finish the proof of the Claim~\ref{Cl1.convergence}.


The proof of Claim~\ref{Cl2.convergence} is quite similar and somewhat easier as we do not 
need analysis of abelian varieties in \cite{TGC.II}. 
In this case, we compare the Gromov-Hausdorff topology
 with the Satake topology
 of $\bigsqcup_{1\le d\le g}GL(d,\mathbb{Z})\backslash GL(d,\mathbb{R})/
(\mathbb{R}_{>0}\cdot O(g))$. 
The Iwasawa decomposition of $GL(g,\mathbb{R})$ is $N'A'O(g)$ with
$$N'=\Biggl\{ \alpha=
\begin{pmatrix}
     1       &               &\text{\large{$\ast$}}\\
             &  \ddots  &                        \\
     0       &               &        1              \\
             \end{pmatrix}
 \in GL(g,\R) 
\Biggr\},$$
$$
 A'=\Biggl\{  \tau=
\begin{pmatrix}
     t_1       &               &                \\
             &  \ddots  &                      \\
             &               &        t_g          \\
\end{pmatrix}\in GL(g,\R) 
\Biggr\}, 
$$
and the Siegel set $\mathfrak{S}'_{U,c}$ for a fixed bounded set $U\subset N'$ and a positive constant $c$
is 
$$\{  \alpha\tau (\R_{>0}\cdot O(g)) \mid {t_{i}}t_{i+1}^{-1}>c\ (\forall i<g)
\}. 
$$ 
From the above, we see that the convergence
 of arbitrary sequence $(\alpha(i),\tau(i))$ in the above Siegel set 
 with respect to the Satake topology, 
 corresponding to the natural representation
 $GL(g,\mathbb{R})\to GL(g,\mathbb{R})$, 
is that of $[\alpha(i)\cdot \tau(i)^{2}\cdot {}^{t}\alpha(i)]$ in 
$\mathbb{P}_{\mathbb{R}}(\mathfrak{gl}_{g}(\mathbb{R}))$. 

On the other hand, the Gromov-Hausdorff limit of a sequence of real flat tori 
$\mathbb{R}^{g}/\mathbb{Z}^{g}$ with metric matrix $\alpha(i)\tau(i) \cdot {}^{t}\tau(i){}^{t}\alpha(i)$ 
(possibly degenerate, so that the dimension could be less than $g$) 
for $i=1,2,\cdots$ is exactly the torus corresponding to 
the limit of $\alpha(i)\tau(i) \cdot {}^{t}\tau(i){}^{t}\alpha(i)$,
 hence coincides with the limit point 
in the Satake topology discussed above. 
\end{proof}


\section{Monodromy and Gromov-Hausdorff limit}\label{monodromy.AV}

The following relationship between the monodromy
 for a family of polarized abelian varieties and the Gromov-Hausdorff limit, 
 is a direct corollary to the combination of Theorem~\ref{monodromy} and Theorem~\ref{Ag.TGC.Satake.MS}. 

\begin{Thm}[Holomorphic limits and monodromy --- abelian varieties case]
\label{1par.AV}
Suppose
 $(\mathcal{X}^{*},\mathcal{L}^{*})\to \Delta^{*}$
 is a family of $g$-dimensional principally polarized abelian varieties 
 on the punctured disk. 
Let $\gamma\in Sp(2g,\Z)$ denote the monodromy on $H^{1}(\mathcal{X}_{s},\Z)$ for a fixed $s\neq 0$ around $0\in \Delta$ 
with respect to a marking of $H^{1}(\mathcal{X}_{s},\Z)$.  
The Gromov-Hausdorff limit of $\mathcal{X}_{t}$ 
 with the rescaled flat K\"ahler metric of diameter $1$, whose K\"ahler class is in 
 $\R c_{1}(\mathcal{L}_{t})$, 
 as $t\to 0$ exists as a flat real torus 
 and it is determined by $\gamma$. 

More precisely, it follows that the unipotent part $\gamma_u$ of
 the Jordan decomposition of $\gamma$
 is conjugate to the matrix of the form
\[
\left(
\begin{array}{cc}
  I_g       & B    \\ 
   O       &  I_g   
\end{array}
\right), \text{ where } B \text{ is a symmetric real }g\times g \text{ matrix. }
\]
Then the limit metric space is the rescaled flat real torus
 with the Gram matrix $B$, which is possibly degenerate.
The dimension of torus equals $\operatorname{rank} B$. 
\end{Thm}
Soon later in \S\ref{reconst.AV.sub}, we also give an alternative proof for the above Theorem~\ref{1par.AV}.


\section{Revisiting the reconstruction of family 
of abelian varieties}\label{reconst.AV}

We give a general observation on the Morgan-Shalen type compactification,
 in particular aiming to provide a moduli-theoretic interpretation 
 of the recent hard reconstruction theorems of 
 \cite{KS}, \cite{GS11}
 of certain maximal degenerations in the case of 
 abelian varieties. Note that it is nothing but 
 the ``trivial'' case 
 from the viewpoints of \cite{KS, GS11} since their 
 most crucial process of ``quantum correction'' does not exist in 
 this case and hence we do not claim any crucial overlap with 
 their theory. Nevertheless, 
 we hope to pursue this line of thoughts more in future. 

\subsection{Reviewing the extendability of morphisms}

Consider any proper dlt stack $(\mathcal{X},\mathcal{D})$ over $\C$ and 
the corresponding Morgan-Shalen compactification of the coarse moduli space of 
$\mathcal{U}:=\mathcal{X}\setminus \mathcal{D}$ as in \cite[Appendix]{TGC.II}, 
which we denote again by $U\subset \bar{U}^{\rm MSBJ}(\mathcal{X})$. 
The coarse moduli spaces of $\mathcal{X}$ and $\mathcal{D}$ will be 
denoted by $X$ and $D$ respectively. Below, we 
also regard $\mathcal{X}$ as a stack with finite isotropies 
in the category of complex analytic spaces. 

Note that an analogous result to the following Proposition~\ref{MS.reconstruction} for Hermitian locally symmetric spaces 
(resp., general log smooth pairs) 
is obtained as Theorem~\ref{rationality.HSD} 
(resp., as Theorem~\ref{rationality.MSBJ}). Hence, this is simply a 
subtle technical generalization. 

\begin{Prop}\label{MS.reconstruction}
Take any holomorphic 
morphism $\varphi\colon \Delta\to \mathcal{X}$ and its 
descent to the coarse moduli space $\bar{\varphi}\colon \Delta \to X$, 
such that $\bar{\varphi}(0)$ in $X$ 
resides in some $0$-dimensional strata (log-canonical center) of 
$(\mathcal{X},\mathcal{D})$. Then 
$\bar{\varphi}|_{\Delta^{*}}\colon \Delta^{*}\to U$ extends also 
to a continuous map $\Delta\to \bar{U}^{\rm MSBJ}(\mathcal{X})$. 
\end{Prop}

\begin{proof}
Lifting $\varphi\colon \Delta\to \mathcal{X}$ 
to an algebraic \'etale finite cover of $\mathcal{X}$ and replacing by 
small enough open neighborhood of the limit point, i.e., the $0$-dimensional 
lc center, 
we can assume $\mathcal{X}$ is a smooth variety and $D$ is a simple normal 
crossing. Hence, in particular, $\varphi=\bar{\varphi}$.
Let $n:=\dim X$. 
Suppose $D=\bigcup_{i=1}^n D_{i}$ and irreducible components are
$D_{i}=(z_{i}=0)$ for some local holomorphic functions $z_{i}$. 
Then $\varphi(t)$ can be written in terms of the local coordinates 
for $|t|\ll 1$ as 
$$\varphi(t)=(f_{1}(t),\cdots,f_{n}(t)),$$
with holomorphic function $f_{i}(t)$ around $0\in \Delta$. 
Then it is easy to see that $\varphi(t)$ converges as $t\to 0$,
 inside the Morgan-Shalen compactification to a
 point in the $(n-1)$-simplex corresponding to $\bigcap_{i}D_{i}$ and
 its barycenter coordinates is 
\[\Bigl(\frac{{\rm ord}_{t}(f_{1}(t))}
{\sum_{i}{\rm ord}_{t}(f_{i}(t))},\frac{{\rm ord}_{t}(f_{2}(t))}
{\sum_{i}{\rm ord}_{t}(f_{i}(t))},\cdots,
\frac{{\rm ord}_{t}(f_{n}(t))}
{\sum_{i}{\rm ord}_{t}(f_{i}(t))}\Bigr).\] 
\end{proof}

\subsection{On the moduli and degenerations of abelian varieties}
\label{reconst.AV.sub}

Now, we turn back to 
the moduli of abelian varieties and recall that degeneration family 
of abelian varieties is classically 
known to be encoded in some systematic ``degeneration data'', 
following the idea of 
Mumford \cite{Mum72.AV}, Raynaud and Faltings-Chai \cite{FC90}. 
Note there was also a series of works by 
Y.~Namikawa and I.~Nakamura (cf.\  e.g.\ \cite{Nak}, \cite{Nam}) in the 
complex analytic setting after \cite{Mum72.AV}. 

Recently, this was revisited and extended 
by Kontsevich-Soibelman \cite{KS} 
and Gross-Siebert \cite{GS11}
from a different viewpoint as a splendid 
construction of degenerating 
Calabi-Yau families out of affine manifolds with singularities. 
One of their core results is the general execution of 
actual construction of family of Calabi-Yau varieties at least 
at the formal level, from affine manifolds with singularities and 
additional structure. Indeed, the singularities of the 
affine structures are the source of technical difficulties and 
provides depth in their technical accomplishment. 

Our aim in this subsection is rather modest, only 
to provide a moduli-theoretic understanding of their theory in 
the most ``trivial'' case 
i.e., the family of abelian varieties which do not involve 
singularities in the corresponding dual intersection complex, 
which is nothing but the flat tori. 

Consider a punctured holomorphic family of principally polarized 
abelian varieties as 
$(\mathcal{X}^{*},\mathcal{L}^{*})\to \Delta^{*}$ corresponding to 
$\varphi\colon \Delta^{*}\to A_{g}$ which we assume to extend to 
$\Delta\to \overline{A_{g}}^{\rm Vor}$, the Voronoi toroidal compactification. 
We say such a family is {\it meromorphic} in this paper, following \cite{BJ16}. 
Here, 
$\Delta:=\{t\in \C\mid |t|<1\}$ and $\Delta^{*}:=\{t\in \C\mid 0<|t|<1\}$ 
as usual. For our purpose here, we assume it corresponds to a maximal degeneration. 

Recall from \cite{Mum72.AV, FC90} (revisited in \cite[\S 3.1]{TGC.II}) we can write 
\begin{eqnarray*}
\mathcal{X}_t&=&\mathbb{C}^{g} \Bigl/ \left(
\begin{array}{ccc|ccc}
1         &               &              &                &           &                     \\
                  & \ddots &                &             &    {\dfrac{
                  \log(p_{i,j}(t))}{2\pi i}}&    \\ 
                   &            & 1     &               &          &                   
\\
\end{array}
\right)\cdot \mathbb{Z}^{2g}\\
&\simeq&(\C^{*})^{g}/\langle(p_{i,1}(t))_{i},(p_{i,2}(t))_{i},\cdots,(p_{i,g}(t))_{i}\rangle
\end{eqnarray*}
for $t\in \C$ with $|t|\ll 1$, 
where the entries of the symmetric matrix $(p_{i,j}(t))$ value in the 
(convergent) 
meromorphic functions field $\mathbb{C}((t))^{\it mero}\subset \mathbb{C}((t))$. 
(In the notation of \cite{FC90}, $p_{i,j}(t)=\frac{1}{b(y_i,\phi(y_j))(t)}$.) 
In particular, it (interestingly) induces a lifted morphism 
$\Delta^{*}\to (U(F)\cap  Sp(2g,\Z))\backslash \mathfrak{H}_{g}$, 
where $F$ is now a $0$-dimensional cusp, i.e., 
a rational boundary component of $0$-dimension, of 
the Satake-Baily-Borel compactification of the Siegel upper half space 
$\mathfrak{H}_{g}$ and 
$U(F)$ follows Notation~\ref{notation.groups}, i.e., 
the center of the unipotent radical of the corresponding 
parabolic subgroup. 

As shown also in \cite[\S 3.1]{TGC.II}, 
$\mathcal{X}_t$  $(t\neq 0)$ with the flat K\"ahler metrics
 converge to $\mathbb{R}^{r}/\mathbb{Z}^{r}$ with
 the Gram matrix $({\rm ord}_{t}(p_{i,j}(t)))$,
 appropriately rescaled to make the diameter $1$. This Gram matrix is equal to the 
``$B$-part'' of the Faltings-Chai degeneration data \cite{FC90} 
and had been known to be determined by 
the monodromy on $H^{1}(\mathcal{X}_{t},\Z)$. This fact can be at least 
traced back to the 
late work of A.~Grothendieck \cite{SGA7-1}, combined with 
the Deligne's comparison theorem (\cite[3, XIV.2.8]{SGA7-2}, also cf., e.g., \cite[\S 6]{HN}). 
\cite{SGA7-1} introduced ``l'accouplement de monodromie $u$" for 
semiabelian scheme over henselian DVRs, a priori defined 
in terms of the Galois action on the $l$-adic Tate modules (when $l\neq p$, cf.\  
\cite[\S9.1]{SGA7-1} for details in more generality) 
and showed that it gives rise to a symmetric non-degenerate 
pairing on $Y\otimes_{\Z}\Z_{l}$ (cf.\  \cite[Theorem 10.4]{SGA7-1}). 
The discussion here gives an alternative proof of Theorem~\ref{1par.AV}.

We denote the Gromov-Hausdorff limit of $\mathcal{X}_t$
 by $B(\mathcal{X},\mathcal{L})$, that 
is a flat real torus hence can be written as $V/L$ where $L$ is a finitely generated torsion-free abelian group such that 
$L\otimes \R\simeq V$. It is obviously a Monge-Amp\`ere manifold in the sense of \cite{KS}, 
if we define its integral affine structure by the natural real extension of $L\simeq \Z^{g}$ denoted by 
$V\simeq \R^{g}$. The limit metric on $V/L$ is flat. We take the discrete Legendre transform 
(cf.\  \cite{GS.logI}) and denote it by $\check{B}(\mathcal{X},\mathcal{L})$. 
This object encodes the (dual) affine structure plus the flat (Monge-Amp\`ere) metric $g$, same as $B(\mathcal{X},\mathcal{L})$. 

Again, coming back to the setting under the maximal degeneration assumption, 
we observe the following reconstruction of the family $\pi$, 
by using only the moduli space and 
its compactification. We set the valuation field 
$K:=\overline{\C((t))^{\it mero}}$, the algebraic closure of 
the convergent meromorphic function field $\C((t))^{\it mero}$. 

\begin{Prop}\label{AV.GH.recon.prop}
Take a meromorphic family of maximally degenerating principally polarized 
abelian varieties 
$\pi\colon (\mathcal{X},\mathcal{L})\to \Delta$. 
From the above discussion,
 the Gromov-Hausdorff limit of $\mathcal{X}_t$
 as $t\to 0$ exists.
We denote the limit by $B(\mathcal{X},\mathcal{L})$
 and its discrete Legendre transform by
 $\check{B}(\mathcal{X},\mathcal{L})$. 

Then we can enhance the underlying integral affine structure of 
$\check{B}(\mathcal{X},\mathcal{L})$ 
as $K$-affine structure (in the sense of \cite[\S7.1]{KS}) naturally 
via the data of $\pi$. 
Furthermore, such $K$-affine structure recovers $\pi$ 
up to an equivalence relation generated by 
base change (replace $t$ by $t^{a}$ with $a\in \Q_{>0}$). 
\end{Prop}

\begin{proof}
By analyzing moduli space, we previously 
observed (cf., e.g., \cite[\S 3.1]{TGC.II}) 
that the family $\pi$ 
should be written as 
$(\C^{*})^{g}/\langle(p_{i,1}(t))_{i},(p_{i,2}(t))_{i},\cdots,(p_{i,g}(t))_{i}\rangle$, with $p_{i,j}=p_{j,i}\in \C((t))^{\rm mero}$ such that 
$(B_{i,j}:=-{\rm ord}_{t} \, p_{i,j})_{i,j}$ is a positive definite symmetric matrix. 
From the previous description of the toroidal compactification above, 
the limit $\lim_{t\to 0}\varphi(t)$ inside the corresponding Morgan-Shalen 
compactification $\overline{A_{g}}$, the limit is $\R^{g}/\Z^{g}$ with metric 
matrix $cB_{i,j}$ (as we knew from the analysis in \cite{TGC.II} and 
Theorem~\ref{Ag.TGC.Satake.MS}) for some 
$c\in \R_{>0}$. 
Thus the dual affine structure on $\check{B}$ 
of the trivial affine structure with respect 
to the metric is written as translation given by $cB_{i,j}$. 
The desired enhancement to the $K$-affine structure can be simply defined as 
$p_{i,j}(t)$ itself so that our assertion of this theorem can be 
now regarded just as a tautology. 
\end{proof}


\newpage

\chapter{Algebraic K3 surfaces case}\label{F2d.sec}

From now, our focus moves on to the collapsing of 
Ricci-flat \textit{K3 surfaces}. We begin by 
reviewing and setting up the basic background. 

\section{Moduli of polarized K3 surfaces}\label{Mod.pol.K3}
Let us first review the well-known construction of moduli of polarized K3 surfaces, 
and set our notation (cf.\  e.g.\ \cite{Huy}). 
This description uses the Torelli theorem of algebraic K3 surfaces 
(\cite{PS.Sha}) and 
surjectivity of the period maps
 (\cite{Kul}, \cite{PP}).\footnote{Their K\"ahler versions are due to Burns-Rapoport \cite{BR}
 (Torelli theorem) and Todorov \cite{Tod.surj}, \cite{Tod} (surjectivity) respectively. }
Let $U=\mathbb{Z} e_0 \oplus \mathbb{Z} f_0$ be a lattice of rank two
 with a symmetric bilinear form given by $(e_0,e_0)=(f_0,f_0)=0$ and
 $(e_0,f_0)=1$.
Define $\Lambda_{\rm K3}:=U^{\oplus 3}\oplus E_8^{\oplus 2}$ the even unimodular lattice
 of rank $22$. 
Here, $E_8$ denotes the negative definite $E_8$-lattice
 so $\Lambda_{\rm K3}$ has signature $(3,19)$.
Fix a positive integer $d$.
A primitive element $\lambda\in \Lambda_{\rm K3}$ with $(\lambda,\lambda)=2d$ is unique
 up to automorphisms of $\Lambda_{\rm K3}$.
We fix $\lambda=d e_0+f_0$ contained in one of $U$
 and then put $\Lambda_{2d}:=\lambda^{\perp} \subset \Lambda_{\rm K3}$.
The lattice $\Lambda_{2d}$ has signature $(2,19)$ and
 $\Lambda_{2d}\simeq \mathbb{Z}(de_0-f_0)\oplus U^{\oplus 2}\oplus E_8^{\oplus 2}$.

Let $\mathcal{F}_{2d}$ be the moduli space of 
K3 surfaces allowing possible ADE singularities, with primitive polarizations of degree $2d$. 
Recall that a polarization (ample line bundle) $L$ is called primitive if it can not be written as $M^{\otimes m}$ 
with $m>1$. 
Its structure is well-known as follows. Let us set 
$$\Omega(\Lambda_{2d}):=\{[w]\in \mathbb{P}(\Lambda_{2d}\otimes \mathbb{C})\mid
 (w,w)=0,\ (w,\bar{w})>0\},$$ 
which has two connected components. 
 Note there is a natural involution $\iota\colon [w]\mapsto [\bar{w}]$,
 which interchanges the two components. 
We choose one of its connected components and denote
 by $\mathcal{D}_{\Lambda_{2d}}$. 
Also, $\Omega(\Lambda_{2d})$ can be identified with
 the set of positive definite oriented two-dimensional planes in
 $\Lambda_{2d}\otimes \mathbb{R}$
 by assigning $[w]$ to $\mathbb{R} (\re w) \oplus \mathbb{R} (\im w)$.
The choice of one connected component $\mathcal{D}_{\Lambda_{2d}}$
 corresponds to giving orientations on all the positive definite
 planes in $\Lambda_{2d}\otimes \R$.

Let $O(\Lambda_{\rm K3})$ be the automorphism
 group of the lattice $\Lambda$ preserving the bilinear form. 
Let 
$\tilde{O}(\Lambda_{2d})=\{g|_{\Lambda_{2d}} : g\in O(\Lambda_{\rm K3}),\, 
g(\lambda)=\lambda\}$ (cf.\ \cite[1.5.2, 1.6.1]{Nik}). 
The group $\tilde{O}(\Lambda_{2d})$ naturally acts on $\Omega(\Lambda_{2d})$.
Let $\tilde{O}^{+}(\Lambda_{2d})\subset \tilde{O}(\Lambda_{2d})$ 
 be the index two subgroup consisting of the elements 
 which preserve each connected component of $\Omega(\Lambda_{2d})$. 
The following is well-known: 
\begin{Fac}\label{period.alg.K3}
We have an isomorphism 
\begin{align*}
\mathcal{F}_{2d}
\simeq
\tilde{O}(\Lambda_{2d})\backslash \Omega(\Lambda_{2d})
(\simeq 
\tilde{O}^{+}(\Lambda_{2d})\backslash \mathcal{D}_{\Lambda_{2d}}).
\end{align*}
\end{Fac}

We briefly recall how Fact~\ref{period.alg.K3} is proved. The morphism is given by periods of K3 surfaces as follows.
Let $(X,L)$ be a polarized smooth K3 surface of degree $2d$. 
There is an isomorphism $\alpha_X\colon H^2(X,\mathbb{Z})\xrightarrow{\sim} \Lambda_{\rm K3}$
 such that $\alpha_X(c_{1}(L))=\lambda$.
Then the image of $H^{2,0}$ by the morphism 
$\alpha_{X}\otimes \mathbb{C}\colon H^{2}(X,\mathbb{C})
\to \Lambda_{\rm K3}\otimes \mathbb{C}$
 is contained in $\Lambda_{2d}\otimes \mathbb{C}$ and it defines an element
 of $\Omega(\Lambda_{2d})$.
Other choices of $\alpha_X\colon H^2(X,\mathbb{Z})\xrightarrow{\sim} \Lambda_{\rm K3}$
 correspond to $\tilde{O}(\Lambda_{2d})$-translates in $\Omega(\Lambda_{2d})$. 
If $X$ has ADE singularities, we need to 
use markings in the sense of \cite{Mor} as follows. 
We consider the minimal 
resolution $\tilde{X}$ with exceptional $(-2)$-curves $e_{i}$.
If we set\footnote{According to \cite{Mor}, $IH$ stands for 
the Intersection Cohomology.}
 $IH^{2}(X,\Z)$ as the orthogonal complement of $e_{i}$s  in 
$H^{2}(\tilde{X},\Z)$, then any K\"ahler class of $X$
 and the cycle class of any closed subcurve
 which does not pass through the singularity of $X$, are in $IH^{2}(X,\Z)$. 
Then, for 
possibly ADE singular polarized K3 surface $X$, Morrison \cite{Mor} 
defines the marking as an isometric embedding 
 $\alpha_{X}\colon IH^{2}(X,\mathbb{Z})\hookrightarrow \Lambda_{\rm K3}$ 
which can be extended to a whole isometry $H^{2}(\tilde{X},\mathbb{Z})\simeq \Lambda_{\rm K3}$. 
If $X$ is attached with an ample line bundle $L$, we consider such 
$\alpha_{X}$ which satisfies $\alpha_{X}(c_{1}(L))=\lambda$. 
Then we consider  $(\alpha_{X}\otimes \C)([\Omega_{X}])$ as the period, 
 where $\Omega_{X}$ is a non-zero holomorphic $2$-form on $X$. 
This gives a morphism $\mathcal{F}_{2d}
\simeq \tilde{O}(\Lambda_{2d})\backslash \Omega(\Lambda_{2d})$, 
which is an isomorphism thanks to the strong Torelli theorem by \cite{PS.Sha}. 

Let $\mathbb{G}:=O(\Lambda_{2d}\otimes \Q)$ and
 $G=\mathbb{G}(\mathbb{R})=O(2,19)$.
The indefinite orthogonal Lie group $G$ acts on $\Omega(\Lambda_{2d})$
 and we have an isomorphism
\begin{align*}
\Omega(\Lambda_{2d}) \simeq O(2,19)/(SO(2)\times O(19)).
\end{align*}
Let $O^+(2,19)$ be the index two subgroup of $O(2,19)$ which preserves
 each connected component of $\Omega(\Lambda_{2d})$.
Then $O^+(2,19)$ acts holomorphically on the Hermitian symmetric space
 $\mathcal{D}_{\Lambda_{2d}}$ and 
\begin{align*}
\mathcal{D}_{\Lambda_{2d}} \simeq O^+(2,19)/(SO(2)\times O(19)).
\end{align*}
Since $\tilde{O}^{+}(\Lambda_{2d})$ and $\tilde{O}(\Lambda_{2d})$
are arithmetic subgroups of $O(2,19)$,
 $\mathcal{F}_{2d}$ is an arithmetic quotient of
 a Hermitian symmetric space.

Note that by sending $[w]$ to the oriented two-dimensional subspace
 $\mathbb{R} (\re w) \oplus \mathbb{R} (\im w)$, we get another isomorphism
\begin{multline*}
\mathcal{F}_{2d}
\simeq 
\tilde{O}(\Lambda_{2d})\backslash 
\{V\subset \Lambda_{2d}\otimes \mathbb{R}\mid 
 \text{oriented two-dimensional subspaces} \\
\text{with signature $(2,0)$}\}.
\end{multline*}

Also recall that, the subset $\mathcal{F}_{2d}^{o}$ of $\mathcal{F}_{2d}$ parametrizing 
\textit{smooth} K3 surfaces with primitive polarizations of degree $2d$, is Zariski open. More precisely, $\mathcal{F}_{2d}^{o}$ coincides with the 
complement of the union of the Heegner divisors (cf.\  \cite{Huy} etc.) in $\mathcal{F}_{2d}$. 

\section{Description of Satake compactification of $\mathcal{F}_{2d}$}
\label{K3.Sat.sec}

We study the structure of the Satake compactification
 of $\mathcal{F}_{2d}$ for the adjoint representation in our situation. 
We saw that $\mathcal{F}_{2d}$ has a structure of locally symmetric space:
\begin{align*}
\mathcal{F}_{2d}
\simeq 
\tilde{O}^{+}(\Lambda_{2d})\backslash \mathcal{D}_{\Lambda_{2d}}
&\simeq 
\tilde{O}^{+}(\Lambda_{2d})\backslash O^+(2,19)/SO(2)\times O(19)\\
&\simeq 
\tilde{O}^{+}(\Lambda_{2d})\backslash O(2,19)/O(2)\times O(19).
\end{align*}
As we partially explained in the previous section, this fits the general theory in \S\ref{Gen.HSD} by 
 letting $\mathbb{G}=O(\Lambda_{2d}\otimes \Q)$, $G=\mathbb{G}(\R)=O(2,19)$, 
 $K=O(2)\times O(19)$ and $\Gamma=\tilde{O}^{+}(\Lambda_{2d})$.
Hence we can consider its Satake compactification and Morgan-Shalen compactification. 
Let $\overline{\mathcal{F}_{2d}}^{\rm Sat,{\tau_{\rm ad}}}=:
\overline{\mathcal{F}_{2d}}^{{\rm Sat}}$ be the Satake compactification of 
 $\mathcal{F}_{2d}$ 
 corresponding to the adjoint representation of $O(2,19)$. 
One way of seeing its structure is through an identification of the 
 adjoint representation of $O(2,19)$ with $\bigwedge^{2}(\Lambda_{2d}\otimes \R)$. 
From that and the Pl\"ucker embedding,
 we can regard $\overline{\mathcal{F}_{2d}}^{\rm Sat}$ as a quotient of 
partial compactification of $\mathcal{D}_{\Lambda_{2d}}$ in the Grassmannian of 
 real planes ${\rm Gr}_{2}(\Lambda_{2d}\otimes \R)$.
Thus our Satake compactification decomposes as 
\[\overline{\mathcal{F}_{2d}}^{\rm Sat}=
\mathcal{F}_{2d}\sqcup \bigcup_{l} \mathcal{F}_{2d}(l)
\sqcup \bigcup_{p} \mathcal{F}_{2d}(p),\]
where $l$ runs over 
one-dimensional isotropic
 subspaces of $\Lambda_{2d}\otimes \mathbb{Q}$,
 and $p$ runs over 
 two-dimensional isotropic subspaces of
 $\Lambda_{2d}\otimes \mathbb{Q}$. 
 
The boundary component $\mathcal{F}_{2d}(l)$ is given as 
\begin{align}\label{F2d(l)}
\mathcal{F}_{2d}(l)
= \{v\in (l^{\perp}/l) \otimes \mathbb{R} \mid (v,v)>0\}/\sim.
\end{align}
Here $v \sim v'$ if $g\cdot v=c v'$ for some
 $g\in \tilde{O}^{+}(\Lambda_{2d})$ such that $g\cdot l = l$
 and $c\in \mathbb{R}^{\times}$. 
Note that here, $c$ runs over whole $\R^{\times}$ rather than 
 $\R_{>0}$. This is the reason why later in \S \ref{trop.K3.1}, 
 the assigned tropical K3 surfaces to this boundary component $\mathcal{F}_{2d}(l)$ do \textit{not} have canonical orientation. 
If $g\cdot l = l'$ for some $g\in \tilde{O}^{+}(\Lambda_{2d})$, 
 then $\mathcal{F}_{2d}(l)=\mathcal{F}_{2d}(l')$,
 where the identification is given by the action of $g$.
If otherwise, $\mathcal{F}_{2d}(l)\cap\mathcal{F}_{2d}(l')=\emptyset$.
Let $l_{\R}=l\otimes\mathbb{R}$
 and $l_{\R}^{\perp}=l^{\perp}\otimes \mathbb{R}$.
Since $l_{\R}^{\perp}/l_{\R}$ has signature $(1,18)$,
 there is an isomorphism
\begin{align}\label{F2d(l)2}
\{v\in l_{\R}^{\perp}/l_{\R} \mid (v,v)>0\}/{\mathbb{R}}^{\times}
\simeq O(1,18)/(O(1)\times O(18))
\end{align}
and hence $\mathcal{F}_{2d}(l)$
 is an arithmetic quotient of $O(1,18)/(O(1)\times O(18))$.

Let $v\in \mathcal{F}_{2d}(l)$ and take a representative
 in $l^{\perp}_{\R}$,
 which we also denote by the same letter $v$.
The real plane $l_{\R}\oplus \R v$ in $\Lambda_{2d}\otimes \R$
 is positive semidefinite and
 can be described as a limit of positive definite planes.
Recall that we have chosen
 one connected component $\mathcal{D}_{\Lambda_{2d}}$, which corresponds to 
 fixing orientations for all the positive definite real planes
 in $\Lambda_{2d}\otimes \R$.
Therefore, $l_{\R}\oplus \R v$ has also a compatible orientation.
In other words, if we choose a primitive element $e\in l\cap \Lambda_{2d}$, 
 then $v$ with $(v,v)=1$ 
 may be taken canonically modulo $l$.  

The other stratum $\mathcal{F}_{2d}(p)$ is a point
 and $\mathcal{F}_{2d}(p)=\mathcal{F}_{2d}(p')$
 if and only if $g\cdot p=p'$ for some $g\in \tilde{O}^{+}(\Lambda_{2d})$. 

Therefore, if we take representatives of $l$ and $p$
 from each $\tilde{O}^+(\Lambda_{2d})$-orbit,
 we get a finite stratification
\begin{align}\label{K3.Sat.strata}
\overline{\mathcal{F}_{2d}}^{\rm Sat}=
\mathcal{F}_{2d}\sqcup \bigsqcup_{l} \mathcal{F}_{2d}(l)
\sqcup \bigsqcup_{p} \mathcal{F}_{2d}(p).
\end{align}
The number of strata is counted in 
\cite[\S4, \S5]{Sca}. 
For the theory of Siegel sets and the topology on
 $\overline{\mathcal{F}_{2d}}^{\rm Sat}$, 
we refer to \cite{Sat1, Sat2, BJ} for the details. 
We also refer to the end of \S\ref{trop.K3.1} 
for some relation between the parameter spaces discussed in 
\cite[\S3.3, \S6.7.1]{KS}. 

\section{Tropical K3 surfaces}\label{trop.K3.1}

To discuss tropical K3 surfaces,
 we recall the definitions of Hessian metrics and (real) Monge-Amp\`ere equations.
See \cite{Hit97, KS, GS.logI} for details.
On an affine manifold $B$, a {\it Hessian metric} is a Riemannian metric $g$ on $B$
 such that locally for affine coordinate functions $(x_1,\dots,x_n)$,
 there exists a smooth function $K$ such that
 $g_{ij} = \frac{\partial^2 K}{\partial x_i \partial x_j}$.
We say it satisfies {\it Monge-Amp\`ere equation}, 
or simply a {\it Monge-Amp\`ere metric}, if
 $\det (\frac{\partial^2 K}{\partial x_i \partial x_j})$ is a constant. 

In our paper, what we mean by \textit{tropical K3 surface}
 is a topological space $B$ homeomorphic to the sphere $S^{2}$ with 
 two integral affine structures on the complement of certain finite points ${\rm Sing}(B)$ 
 and a metric $g$ which is Monge-Amp\`ere for both of the two affine structures on
 $B\setminus {\rm Sing}(B)$. 
We also assume that 
 the two affine structures are Legendre dual to each other
 with respect to $g$ in the standard sense in this field (cf.\  e.g.\ \cite{Gross}). 
 Such $B\setminus {\rm Sing}(B)$, i.e., 
 with two affine structures and Monge-Amp\`ere metric, 
 is also called Monge-Amp\`ere manifold in \cite{KS}. 
We remark that 
 here we do \textit{not} encode an orientation of $B$, unlike the definition in e.g.\ \cite{HU}. 
In this paper, 
we follow \cite{Gross} for the definition of 
affine structure, that is, an isomorphism class of local affine coordinates atlas. 
Studies of such object as tropical version of 
K3 surfaces are pioneered in well-known papers of Gross-Wilson \cite{GW} and 
Kontsevich-Soibelman \cite{KS}. 

In this section, as a part of geometric realization map $\Phi_{\rm alg}$, 
we shall assign such a tropical K3 surface, 
which we denote by 
$\Phi_{\rm alg}([e,v])$, to each point $[e,v]$ in
 the boundary component $\mathcal{F}_{2d}(l)$ with $l=\Q e$. 
 
Let $l$ be a one-dimensional isotropic
 subspace of $\Lambda_{2d}\otimes \mathbb{Q}$.
Let $e$ be a primitive element of $l\cap\Lambda_{2d}$.
Take a vector $v\in l^{\perp}_{\R}$ such that $(v,v)=1$.
Write $[e,v]$ for the 
corresponding point in $\mathcal{F}_{2d}(l)$. 
Replacing $v$ by $-v$ if necessary, 
 we may assume that the orientation on the plane $l_{\R}\oplus \R v$
 given by $e\wedge v$ agrees the one given in \S\ref{K3.Sat.sec}. 
Then there exists a (not necessarily projective) K3 surface $X$
 and an isomorphism
 $\alpha_{X}\colon H^2(X,\mathbb{Z})\to \Lambda_{\rm K3}$ such that
\begin{enumerate}
\item $\alpha_{X,\mathbb{C}}(H^{2,0})= \mathbb{C}(\frac{1}{\sqrt{2d}}\lambda-\sqrt{-1}v)$,
\item $\alpha_{X}^{-1}(e)\in H^{1,1}$ and it is in the closure of 
K\"ahler cone.
\end{enumerate}
This is well-known to experts but for instance,
 \cite[Chapter 8, Remark 2.13]{Huy} gives the proof. 
Let $L$ be a line bundle on $X$ such that $\alpha_{X}([L])=e$. 
Then by above (ii) and Fact~\ref{bp.free}, 
 we get an elliptic fibration $\pi\colon X\to B(\simeq \mathbb{P}^1)$. 
These are unique up to the change of marking, appears as a composite of reflections along $(-2)$-curves in the 
singular fibers. Its corresponding Jacobian elliptic surface and the Weierstrass model 
are both preserved if we replace $v$ by $v+ce$ for some $c\in \R$. 
Take a holomorphic volume form $\Omega$ on $X$ such that
 $\alpha_X([\re \Omega])=\lambda$.
The map $\pi$ is a Lagrangian fibration with respect to
 the symplectic form $\re \Omega$.
Hence it gives an affine manifold structure
 on $B\setminus{\rm Sing}(B)$,
 where ${\rm Sing}(B)$ denotes the finite set of singular points.
Similarly, the imaginary part $\im \Omega$ gives another affine manifold
 structure on $B\setminus{\rm Sing}(B)$.
If we choose $-e$ instead of $e$, then the elliptic fibration
 corresponding to $[-e,-v]$ is the complex conjugate
 of $\pi\colon X\to B$.

We endow the base space $B$ with metric. 
Let $b\in B\setminus{\rm Sing}(B)$ and $\xi_1,\xi_2\in T_b B$.
Set 
\begin{align}\label{eq:McLean}
g(\xi_1,\xi_2)
:=-\int_{\pi^{-1}(b)}\iota(\tilde{\xi}_1)\re \Omega
 \wedge \iota(\tilde{\xi}_2)\im \Omega, 
\end{align}
where $\tilde{\xi}_1, \tilde{\xi}_2$ are lifts of $\xi_1, \xi_2$,
 respectively.
This is the definition of the McLean metric on the base $B$ (cf.\  \cite{ML}, \cite{Gross}), 
when we regard $f$ as a special Lagrangian fibration after hyperK\"ahler rotation 
(cf.\  \cite{HL}, \cite{SYZ}, \cite{GW}). 
It is also well-known that 
the distance induced by this metric on $B\setminus {\rm Sing}(B)$ 
extends to the whole distance of $B$ naturally as a length space of finite 
diameter. We will include its proof for convenience later in 
Corollary~\ref{McLean.finite.distance} although it is essentially well-known in 
a more general situation (see \cite{Yos10}, \cite[Proposition 2.1]{GTZ2}, \cite[Theorem A]{EMM17},
 \cite[Theorem 3.4]{TZ}). 
$B$ with such distance structure, regarded as a compact metric space, 
\textit{without} orientation, 
will be denoted by $\Phi_{\rm alg}([e,v])$. This is the  
part of our geometric realization map $\Phi_{\rm alg}$ along the boundary 
component $\mathcal{F}_{2d}(l)$. The fact that 
$\Phi_{\rm alg}([e,v])$ does not encode an orientation is 
analogous to the fact that the collapsed $S^{1}$ of flat K\"ahler 
elliptic curves (cf., e.g., \cite{Gross}, \cite{TGC.II}) 
do not have a canonical orientation either. 

Also recall that the symplectic form $\re \Omega$ on $X$
gives an integral affine structure on $B\setminus {\rm Sing}(B)$ 
by integrating it over lifts of paths in $B\setminus {\rm Sing}(B)$. 
Note that replacing $\Omega$ by $e^{i\theta}\Omega$ ($\theta\in \R$) 
does not change $g$ while it changes the affine structure. 
This is a tropical analog of 
the hyperK\"ahler rotation and if $\theta=\frac{\pi}{2}$ 
 it is exactly the Legendre transform of the 
affine structure (cf.\  \cite{Gross}) we mentioned. 
The metric $g$ is Hessian and satisfies the Monge-Amp\`ere equation
 with respect to the affine structure on $B\setminus {\rm Sing}(B)$
 defined by $\re (e^{i\theta}\Omega)$ for any $\theta\in \R$
 (see \cite{Hit97}).

This metric on the base coincides with the 
``special K\"ahler metric" introduced and studied 
in \cite{Str90, DW96, Hit99, Freed99} etc.\ and appears as the metric on $\mathbb{P}^1$ 
in \cite{GTZ2}. 
Set a holomorphic local coordinate $y$ defined at any contractible small open subset $U$ of 
$B\setminus {\rm Sing}(B)$ and local holomorphic section $s$ on $U$. Note 
that whole $B\setminus {\rm Sing}(B)$ is covered by such $U$ so we can work over $U$. 
Via $s$, we can identify $\pi^{-1}(U)$ with its Jacobian fibration hence we can take 
its corresponding Darboux coordinate $y,z$ on $\pi^{-1}(U)$. 
Then $\Omega|_{\pi^{-1}(U)}$ can be assumed to be $dy\wedge dz$ 
(after multiplying a nonvanishing local holomorphic function to $y$ if necessary). 
The elliptic fibers are written as $\mathbb{C}/(\mathbb{Z}+\mathbb{Z}\tau)$ 
where $\tau$ is a local function of $y$. 
Then what we want to show is 
$$
\frac{\sqrt{-1}}{2}((\im \tau)dy\wedge d\bar{y})(\xi_1,\xi_2)
=-\int_{\pi^{-1}(b)}\iota(\tilde{\xi}_1)\re \Omega\wedge 
\iota(J\tilde{\xi}_2)\im \Omega,
$$
for any $b\in U$ and for any $\xi_1, \xi_2 \in T_bU$. 
Here, $\tilde{\xi}_1$ and $\tilde{\xi}_2$ denote the lifts of $\xi_1$ and $\xi_2$, respectively as before. 
The above equality can be verified by a straightforward computation;
 if we write 
 $\xi_i=a_i\frac{\partial}{\partial \re y}
 +b_i\frac{\partial}{\partial \im y}$
 for $i=1,2$ then the both hand sides are equal to 
$(\im \tau) (a_1b_2-a_2b_1)$. 

It is also known that the Ricci form of the K\"{a}hler metric
 $g$ defined as \eqref{eq:McLean}
 coincides with the Weil-Petersson metric (\cite[Proposition 4.1]{Tos}).

\begin{Rem}\label{metric.class}
Recall the notion of the \textit{class of metric} (resp., the 
\textit{radiance obstruction}) 
 of Monge-Amp\`ere manifolds $B$ with singularities 
introduced in \cite{KS.earlier} and discussed in \cite{GS.logI} in more details (resp., introduced by \cite{GoldH}). 
We denote them respectively by 
$m(B)\in H^1(B, i_*\tilde{\mathcal{T}}^{\vee}\otimes \R)$, 
$r(B)\in H^1(B,i_*\mathcal{T})$.
Here, $\mathcal{T}$ is the affine structure
 as a $\mathbb{Z}^{{\rm dim}(B)}$-local system on $B$
 in the tangent bundle $T(B\setminus {\rm Sing}(B))$;
 $-^{\vee}$ denotes the  
dual local system of $-$;
 and $\tilde{\mathcal{T}}^{\vee}$ is the 
local system of integral affine functions,
 and $i\colon (B\setminus {\rm Sing}(B))\hookrightarrow B$. 
In particular, by seeing the slope of affine functions, we naturally have a morphism of local systems 
$f\colon \tilde{\mathcal{T}}^{\vee}\twoheadrightarrow \mathcal{T}^{\vee}$ which induces 
$f_*\colon H^1(B, i_*\tilde{\mathcal{T}}^{\vee}\otimes \R)
\to H^1(B, i_*\mathcal{T}^{\vee}\otimes \R)$. 
We can extract the ``linear" part of metric class as $f_* m(B)\in 
H^{1}(B,i_{*}\mathcal{T}^{\vee}\otimes \R)$. 

Now, we consider a tropical K3 surface $B_{[v]}$ which is obtained in 
\S\ref{trop.K3.1} for a point in a boundary component 
$[e,v]\in \partial \overline{\mathcal{F}_{2d}}^{\rm Sat}$. 
Then the class of metric recovers
 $\bar{v}\in l_{\R}^{\perp}/l_{\R}$, 
 namely, $f_*m(B_{[v]})$ is sent to $\bar{v}$
 under the natural $\R$-linear map 
$$H^1(B_{[v]}, i_*\mathcal{T}^{\vee}\otimes \R)\hookrightarrow l_{\R}^{\perp}/l_{\R}$$ 
coming from the Leray spectral sequence applied to the elliptic fibration 
$\pi\colon X\twoheadrightarrow B_{[v]}$. 
Here, we used the $\mathbb{Z}$-simpleness 
$R^1\pi_*\Z_X\simeq i_*\mathcal{T}^{\vee}$ which follows from 
\cite[Lemma 1.3.13]{FMg} because there are no multiple fibers, 
 and we used $H^0(B_{[v]}, i_*\mathcal{T}^{\vee})=H^2(B_{[v]}, i_*\mathcal{T}^{\vee})=0$. 
If all the fibers of $\pi$ are irreducible, the above injective map is an isomorphism.
The above discussion is a tropical analog of the theory of 
the period and the K\"ahler class for complex K3 surfaces.

A germ of our general idea of \S4 is related to 
\cite[\S3.3, \S6.7.1]{KS}. In particular,  our $\mathcal{F}_{2d}(l)$ 
can be regarded as a leaf of 
foliation $\mathcal{F}$ of $\mathcal{P}$ in \cite[\S3.3]{KS}. 
Here, our $\lambda$ is regarded as a vector in
 the set $\{v\in \Lambda_{2,18}\otimes \R\mid (v,v)>0\}/{\rm Aut}(\Lambda_{2,18})$
 in \cite[\S6.7.1]{KS} which determines the foliation $\mathcal{F}$.
In particular, the monodromy representation
 $\pi_1(B_{[v]}\setminus {\rm Sing}(B_{[v]}))\to SL(2,\Z)\ltimes \R^2$
 does not depend on $v\in \mathcal{F}_{2d}(l)$ at least for the generic case
 where the corresponding elliptic fibration $\pi_{[v]}\colon X_{[v]}\to B_{[v]}$
 has $24$ $I_1$-type degenerations.

Also, our previous results in \cite{TGC.II} and \S\ref{Ag.TGC.Satake.MS} for abelian 
varieties case 
can be re-interpreted by such ``tropical periods'' (but with ``weight'' one). 
\end{Rem}

\section{Gromov-Hausdorff collapse of K3 surfaces}\label{Alg.K3.statements.sec}

In this section, 
we assign a metric space to each point in
 $\overline{\mathcal{F}_{2d}}^{\rm Sat}$. 
For a polarized K3 surface $(X,L)\in \mathcal{F}_{2d}$, 
 we endow $X$ with a Ricci-flat metric with K\"ahler class $c_1(L)$,
 which uniquely exists by Yau's theorem~\cite{Yau}.
For $[e,v]\in \mathcal{F}_{2d}(l)$ we defined a metric space $B_{[v]}$
 in \S\ref{trop.K3.1}. 
For a point in $\mathcal{F}_{2d}(p)$ we assign a (one-dimensional) segment.
Let us normalize these metric spaces so that their diameters are $1$.
We thus obtain a map
 $\Phi_{\rm alg}\colon\overline{\mathcal{F}_{2d}}^{\rm Sat}
 \to {\it CMet}_{1}$, where ${\it CMet}_{1}$ denotes
 the set of isometry classes of compact metric spaces with diameter one
 equipped with the Gromov-Hausdorff topology.
We call $\Phi_{\rm alg}$ the geometric realization map.

\begin{Conj}\label{K3.Main.Conjecture}
The geometric realization map 
\[\Phi_{\rm alg}\colon \overline{\mathcal{F}_{2d}}^{\rm Sat}
 \to {\it CMet}_{1}\]
defined above is continuous.
\end{Conj}

This is one of the main conjectures in this monograph (also see Conjecture \ref{K3.Main.Conjecture2}
 and \S\ref{high.dim.HK.sec}, \S\ref{high.dim.gen.sec}). 
\begin{Rem}
There exists an action of Galois group ${\rm Gal}(\mathbb{C}/\mathbb{R})$ on 
 $\mathcal{F}_{2d}$, which sends 
 K3 surfaces to their complex conjugates.
This $\mathbb{Z}_2$-action extends to $\overline{\mathcal{F}_{2d}}^{\rm Sat}$
 and the map $\Phi_{\rm alg}$ factors through the quotient by this action.
\end{Rem}

\begin{Rem}\label{remark.HNU}
We remark that 
\textit{Kenji Hashimoto, Yuichi Nohara, Kazushi Ueda} \cite{HNU} 
studied a certain $2$-dimensional 
analytic subspace of the moduli $\mathcal{F}_{2d}$, i.e., the moduli 
of $(E_{8}^{\oplus 2}\oplus U)$-polarizable K3 surfaces 
and its Gromov-Hausdorff limits.
Their study matches well with the above conjecture. 
Moreover, it directly follows from a result of Hashimoto and Ueda \cite{HU}
 that a generic elliptic K3 surface with a holomorphic section
 is determined up to complex conjugation
 by the metric on $B\setminus {\rm Sing}(B)$ defined in \eqref{eq:McLean}.
As a result, 
$\Phi_{\rm alg}|_{\mathcal{F}_{2d}(l)}\colon \mathcal{F}_{2d}(l)\to\Phi_{\rm alg}(\mathcal{F}_{2d}(l))$
 is ``generically" one to one (resp.\ two to one)
 if the map $j\colon \mathcal{F}_{2d}\to j(\mathcal{F}_{2d})$
 defined at the end of \S\ref{STK.MK3} is generically one to one (resp.\ two to one).\footnote{This corrects our explanation after Theorem 4.4 in the announcement paper~\cite{OO}, 
 where we wrote the map is ``generically" one to one.}
See also Remark~\ref{remark.HNU2}.
We appreciate their gentle discussion shared with us. 
\end{Rem}

\begin{Rem}\label{ACP.K3}
This conjecture is analogous to Theorem~\ref{Ag.TGC.Satake.MS} for the moduli of principally polarized abelian varieties $A_g$. 
Assuming that Conjecture~\ref{K3.Main.Conjecture} holds, as in Remark \ref{ACP.Ag}, 
it gives an analog of the result of Abramovich-Caporaso-Payne \cite{ACP} for K3 surfaces case. 
\end{Rem}

Towards the settlement of Conjecture~\ref{K3.Main.Conjecture}, 
let us first start from several kinds of easy partial confirmation. 
Recall that $\mathcal{F}_{2d}^o(\subset \mathcal{F}_{2d})$
 be the open subset 
consisting of 
isomorphism classes of {\it smooth} polarized K3 surfaces of degreed $2d$. 
Then one can first easily confirm: 

\begin{Prop}\label{easy.confirmation}
The restriction of $\Phi_{\rm alg}$ to $\mathcal{F}_{2d}^o$ is continuous.
\end{Prop}

\begin{proof}
It is a straightforward consequence of the implicit function theorem applied to 
the complex Monge-Amp\`ere equation determining the 
Ricci-flat condition. 
\end{proof}

Another easy partial confirmation of Conjecture \ref{K3.Main.Conjecture} 
is Theorem~\ref{Kummer}, which is essentially reduced to 
the case of abelian varieties, 
i.e., \cite{TGC.II} plus our Theorem~\ref{Ag.TGC.Satake.MS}. 
Recall that from a 
principally polarized abelian surface $(A,M)$ we can 
divide $(A,M^{\otimes 2})$ via $(-1)$-multiplication (involution), say $\iota$,  
to get $(K(A),L)$ of degree $4$ with $16$ nodes. 
The moduli space $\mathcal{M}_{\rm Km,alg}$ of such $(K(A),L)$ forms a $3$-dimensional 
subvariety inside $\mathcal{F}_{4}$ as the intersection of $16$ 
Heegner divisors. 
\begin{Thm}\label{Kummer}
The restriction of $\Phi_{\rm alg}$ to the closure 
$\overline{\mathcal{M}_{\rm Km,alg}}$
of the $3$-dimensional variety $\mathcal{M}_{\rm Km,alg}$,
 the moduli space of Kummer surfaces, is continuous.
\end{Thm}

\begin{proof}
Let us recall that $\mathcal{M}_{\rm Km,alg}$ is known to be naturally isomorphic to the 
Siegel modular variety $A_{2}$. Any 
Kummer surface $K(A)$ with its Ricci flat K\"ahler metric is simply the 
$\Z/2\Z$-quotient of a flatly metrized principally polarized 
abelian surface. 
From general construction of the Satake compactification,
 there exists a continuous map
 $\overline{A_2}^{\rm Sat,\tau_{\rm ad}}\to\overline{\mathcal{F}_{4}}^{\rm Sat}$,
 which extends the natural inclusion 
 $A_2\simeq \mathcal{M}_{\rm Km,alg}\hookrightarrow \mathcal{F}_{4}$.
Moreover,
 we can identify $\overline{A_2}^{\rm Sat,\tau_{\rm ad}}$ with $\overline{A_2}^{\rm T}$
 by Theorem~\ref{Ag.TGC.Satake.MS}, 
 and with our construction in \S\ref{trop.K3.1}, 
its boundary parametrizes flat $2$-dimensional 
tori divided by $\iota$, which are doubled parallelograms 
(having $4$ singularities at the vertices), and flat $1$-dimensional 
tori, namely $S^{1}$, divided by $\iota$, hence the segment. 
\end{proof}

\begin{Rem}
Group-theoretically, the isomorphism
 $\mathcal{M}_{\rm Km,alg}\simeq A_{2}$
 comes from the fact that
 $Sp(4,\R)$ is a double covering of $SO_0(2,3)$. 
\end{Rem}

By using the study of McLean metrics in \S\ref{along.boundary.sec}
 we obtain the continuity along the boundary:

\begin{Prop}\label{F2d.boundary.conti}
The restriction of $\Phi_{\rm alg}$ to
 the boundary 
 $\partial \overline{\mathcal{F}_{2d}}^{\rm Sat}$
 is continuous. 
\end{Prop}

\begin{proof}
The proposition follows from Theorem~\ref{K3a.GH.conti} and Remark~\ref{MK3.to.F2d}.
Note that Theorem~\ref{K3a.GH.conti} is a K\"{a}hler K3 version of
 Proposition~\ref{F2d.boundary.conti}
 and it is a consequence of Theorem~\ref{Mwbar.GH.conti},
 which in turn follows from the analysis of McLean metrics in \S\ref{along.boundary.sec}.
\end{proof}

We now move on to more substantial confirmation
 of Conjecture~\ref{K3.Main.Conjecture}. 
Since our discussions below involve various kinds of complications, 
lying between several areas, 
we divide into several statements and try to describe them step by step. 
The first result Theorem~\ref{GW} 
is obtained as an application of \cite{GW}, and 
confirms the above conjecture along a certain (real) $20$-dimensional direction
 (among the whole $38$-dimension) passing through 
generic point of $18$-dimensional boundary strata. 
Later, we extend Theorem~\ref{GW} to more general and stronger 
Theorem~\ref{K3.Main.Conjecture.18.ok} 
by replacing the use of \cite{GW} by some technical 
refinements of \cite{Tos.lim, Tos, GTZ1,GTZ2}. 
Nevertheless, the main idea of our approach should be already clear 
in the proof of Theorem~\ref{GW} below. 

\begin{Thm}[A weaker version of Theorem~\ref{K3.Main.Conjecture.18.ok}]\label{GW}
Let $l$ be a one-dimensional isotropic subspace in $\Lambda_{2d}\otimes \Q$, 
 let $e\in \Lambda_{2d}$ be its primitive generator, 
 and let $[e,v]\in \mathcal{F}_{2d}(l)$ be a point in the corresponding
 boundary component with generic $v$
 (this genericity assumption will be made explicit as Assumption~\ref{I1} during the proof). 
Let $\{x_i\}_{i=1,2,\dots}$ be a sequence in $\Lambda_{2d}\otimes \C$ such that 
\begin{enumerate}
\item $(\re{x_i},\re{x_i})=(\im{x_i},\im{x_i})=1$ and $(\re{x_i},\im{x_i})=0$.
\item $\im{x_i} \pmod{l_{\R}} = v$.
\item $(\re{x_i},e)=\epsilon_i>0$ and $\epsilon_i\to 0\ (i\to \infty)$.
\end{enumerate}
Then the sequence $[x_i]\in \mathcal{F}_{2d}$ converges to
 $[e,v]$ in $\overline{\mathcal{F}_{2d}}^{\rm Sat}$. 
Moreover, the Gromov-Hausdorff limit of
 $\Phi_{\rm alg}([x_i])$ as $i\to \infty$
 is $\Phi_{\rm alg}([e,v])$.
\end{Thm}

\begin{Rem}
The above theorem~\ref{GW} is not enough for proving
 the continuity of $\Phi_{\rm alg}$ on a neighborhood of
 the boundary component $\mathcal{F}_{2d}(l)$,
 i.e.\ for proving Theorem~\ref{K3.Main.Conjecture.18.ok}. 
Indeed, two particular conditions are imposed in Theorem~\ref{GW}
 on the sequences $[x_i]\in \mathcal{F}_{2d}$ 
 among those converge to $[e,v]$.
One is that $\im{x_i} \pmod{l_{\R}}$ is constant (condition (ii))
 and the other is that $v$ is generic (Assumption~\ref{I1}).
To prove general Theorem~\ref{K3.Main.Conjecture.18.ok}, 
 we need to remove these two conditions, 
 which will be done later in this section and the next chapter.
\end{Rem}

\begin{proof}[proof of Theorem~\ref{GW}]
Take an isotropic element $f\in\Lambda_{2d}\otimes \Q$ such that $(e,f)=1$. 
Let us define $\Lambda'_{2d}:=\{x\in \Lambda_{2d}\mid (x,e)=(x,f)=0\}$.
Let $v'\in \Lambda'_{2d}\otimes \R$ such that $v' \pmod {l_{\R}} = v$ and $(v',v')=1$.
Our condition (ii) implies that we may write
 $x_i=w_i+\sqrt{-1}(v'+c_i e)$
 with $w_i\in\Lambda_{2d}\otimes \R$ and $c_i\in \R$.
By conditions (iii), we may write $w_i=N_i e + \epsilon_i f + w'_i$,
 where $N_i\in \R$ and $w'_i\in\Lambda'_{2d} \otimes \R$.

We set $G:=O(\Lambda_{2d}\otimes \R)$, 
the isometry group which is isomorphic to $O(2,19)$, 
and define $Q$ to be the stabilizer of $l_{\R}$,
 which is a real maximal parabolic subgroup of $G$. 
In order to prove the first statement, we will replace $x_i$
 by their certain $(\tilde{O}^{+}(\Lambda_{2d})\cap Q)$-translates, i.e., 
 change the markings of corresponding K3 surfaces. 
For a vector $\xi\in l^{\perp}\cap \Lambda_{2d}$ one can construct an element
 (called elementary transformation or the Eichler transvection \cite[\S3]{Eich}, \cite{Sca})
 $\phi_{e,\xi}\in \tilde{O}^{+}(\Lambda_{2d})\cap Q$ defined as
\begin{equation}\label{Eichler}
\phi_{e,\xi}(x):=x+(x,\xi)e-\frac{1}{2}(\xi,\xi)(x,e)e-(x,e)\xi
\quad (x\in \Lambda).
\end{equation}
It is easy to see that \eqref{Eichler} stabilizes
 $l_{\R}$ and $l_{\R}^{\perp}$, respectively. 
If we write $w_i=N_i e + \epsilon_i f + w'_i$ as above,
 $\phi_{e,\xi}(w_i)
 = \epsilon_i f + w'_i - \epsilon_i\xi \pmod {l_{\R}}$.
Take a bounded set $\mathcal{U}\subset \Lambda'_{2d}\otimes \R$ such that 
 $\Lambda'_{2d}+\mathcal{U} = \Lambda'_{2d}\otimes \R$,
 or equivalently,
 $\epsilon_i\Lambda'_{2d}+\epsilon_i \mathcal{U}
 = \Lambda'_{2d}\otimes \R$.
Then for each $i$ we can replace $x_i$ by the above transformation
 and get $w'_i \in \epsilon_i \mathcal{U}$.
This replacement does not change the conditions and the claims of the theorem,
 so we assume $w'_i\in \epsilon_i \mathcal{U}$ in the following.
Under this assumption,
 $(\re x_i, \re x_i)=1$ implies $N_i \epsilon_i\to \frac{1}{2}$
 and also $(\re x_i, \im x_i)=0$ implies $\{c_i\}_{i}\subset \R$ is bounded.
 Let us summarize as follows.  
\begin{Claim}[Reduction]\label{Reduction}
 After applying an appropriate element of $\tilde{O}^{+}(\Lambda_{2d})\cap Q$ to 
 $x_{i}\in \Lambda_{2d}\otimes \C$ 
 for each $i$ (the group element depends on $i$), we can assume that 
 $\{x_{i}\}_{i}$ satisfies the following two conditions: 
 \begin{enumerate}
 \item $w'_i \in \epsilon_i \mathcal{U}\,(\subset 
 \Lambda'_{2d}\otimes \R)$,
 \item $\{c_i\}_{i}\subset \R$ is bounded. 
 \end{enumerate}
\end{Claim}
This is essentially known as ``Siegel reduction". 
As discussed above, the second condition automatically follows from 
the first, under our assumptions in Theorem~\ref{GW}.  
For $t\in \R_{>0}$, let $a_t \in G$ be the group element given by
\[
e\mapsto te,\quad f\mapsto t^{-1}f,\quad
 x\mapsto x\ (x\in \Lambda'_{2d}\otimes \R).
\]
Define $A$ to be the subgroup of $G$ consisting of all $a_t$.
Then $A$ is a connected component of a split torus in 
the center of Levi component of $Q$, and 
 we have the Langlands decomposition $Q=NAM$.
We note that $M$ acts by $\pm 1$ on $\R e +\R f$ and stabilizes
 $\Lambda'_{2d}\otimes \R$.
Also, $Q$ stabilizes the flag
 $0\subset l_{\R} \subset l_{\R}^{\perp} \subset \Lambda_{2d}\otimes \R$
 and $N$ acts as identities on its successive quotients.
Let $x_o := \frac{1}{2}e + f + \sqrt{-1}v' \in \Lambda_{2d}\otimes \C$ and 
 choose a base point of the Hermitian symmetric domain
 $o\in \mathcal{D}_{\Lambda_{2d}}$ as $o=[x_o]$. 
The stabilizer of $o$ is a maximal compact subgroup $K$ of $G$.
We can write $[x_i]=n_ia_im_i\cdot o$ for $n_i\in N$, $a_i\in A$, $m_i\in M$.
This means $n_ia_im_i\cdot x_o = c x_i$ for some $c\in \C^{\times}$.
Since $(x_o,e), (x_i,e) \in \R^{\times}$ and $n_ia_im_i$ stabilizes $l_{\R}=\R e$,
 we must have $c\in \R^{\times}$.
Moreover, since $(\re x_o, \re x_o)=(\re x_i, \re x_i)=1$,
 it follows that $c=\pm 1$.
Then we see that 
\begin{equation}\label{A.Siegel}
a_i=a_{\epsilon_i^{-1}},
\end{equation}
 and $m_i v' = \pm v'$. 
The latter implies 
\begin{equation}\label{M.Siegel}
\text{ all $\{m_i\}_i$  lie in a bounded set. }
\end{equation}
To see the behavior of $n_i$, let us observe that the map 
\[N\to l_{\R}^{\perp} / l_{\R},
 \qquad n\mapsto n \cdot f - f \pmod {l_{\R}}\]
 is an isomorphism, where ``$n \cdot $ '' means the group action. 
By the assumption $w'_i\in \epsilon_i \mathcal{U}$, 
 we have $n_i\cdot f - f \in \mathcal{U}$.
Hence 
\begin{equation}\label{N.Siegel}
\text{ all $\{n_i\}_i$  lie in a bounded set. }
\end{equation}
From (\ref{A.Siegel}), (\ref{M.Siegel}), (\ref{N.Siegel}), 
we conclude that $[x_i]$ lie in the same Siegel set. 
Recall that the symmetric space $\mathcal{D}_{\Lambda_{2d}}$
 can be embedded into the Grassmannian
 ${\rm Gr}_2(\Lambda_{2d}\otimes \R)$,
 which moreover has the Pl\"ucker embedding into 
 $\mathbb{P}(\bigwedge^2 (\Lambda_{2d}\otimes \R))$.
The adjoint representation of $G\simeq O(2,19)$ is
 identified with $\bigwedge^2 (\Lambda_{2d}\otimes \R)$.
Following the original definition of the Satake topology
 (\cite{Sat1, Sat2}), 
 we can take the limit of $[x_i]$ in the Grassmannian.
Since $\R(\re x_i) \to \R e$
 and  $\im x_i  = v' \pmod {l_{\R}}$,
 the sequence of planes $\mathbb{R}(\re x_i)\oplus \mathbb{R}(\im x_i)$
 converges to $\mathbb{R}e\oplus \mathbb{R}v'$.
This means that
 $[x_i]\in \mathcal{F}_{2d}$ converges to
 $[e,v]\in\overline{\mathcal{F}_{2d}}^{\rm Sat}$
 in the Satake compactification.  

For the second statement of Theorem~\ref{GW} on the Gromov-Hausdorff convergence, 
let us recall the following facts about 
elliptic K3 surfaces. 

\begin{Fac}\label{bp.free}
Let $X$ be a possibly ADE singular K3 surface
 and let $e\in IH^{2}(X,\Z)$ be a primitive isotropic element. 
If $e$ belongs to the closure of the K\"ahler cone, then
 there exists an elliptic fibration $X\to \mathbb{P}^1$
 with the fiber class $e$.
\end{Fac}

For the proof of Fact~\ref{bp.free}, we refer to the recent textbook 
\cite[Chapter II, 3.10]{Huy} for instance. 
In \textit{op.cit.}, projectivity and smoothness 
of $X$ are assumed. Nevertheless, for smooth $X$, its proof extends to above 
analytic statement verbatim, once we replace the use of ample class by K\"ahler class. Indeed, 
what is needed is the existence of uniform positive lower bound of the degree of $(-2)$-curves along the fixed K\"ahler class, 
which simply follows from the fact that $(-2)$-curves cannot accumulate in the closure of the positive cone. 
Furthermore, the case for general ADE singular K3 surface 
is directly reduced to that of the minimal resolution. 

\begin{Fac}\label{Jac.equal}
Let $X_i\ (i=1,2)$ be a K3 surface
 with a holomorphic volume form $\Omega_i$. 
Let $e_i\in H^2(X_i,\Z)$ be a primitive isotropic element
 and write $\pi_i\colon X_i\to \mathbb{P}^1$ for an elliptic fibration
 with fiber class $e_i$.
If there exists an isomorphism (isometry) $\iota\colon H^2(X_1,\Z)\to H^2(X_2,\Z)$ such that 
 $\iota(e_1)=e_2$ and $\iota([\Omega_1])\in \C[\Omega_2]+\C e_2$,
 then the relative Jacobians of $f_i$ are isomorphic (over $\mathbb{P}^{1}$) 
 with each other.
\end{Fac}

This immediately extends to the ADE case,
 by reducing it to Fact~\ref{Jac.equal} by taking the minimal resolutions. 

\begin{Fac}\label{Jac.equal.ADE}
Let $X_i\ (i=1,2)$ be a K3 surface possibly with ADE singularities
 and let $\Omega_i$ be its holomorphic volume form. 
Let $e_i\in IH^2(X_i,\Z)$ be a primitive isotropic element
 and write $\pi_i\colon X_i\to \mathbb{P}^1$ for an elliptic fibration
 with fiber class $e_i$. 
Suppose there exists an isomorphism (isometry)
 $\iota\colon IH^2(X_1,\Z)\to IH^2(X_2,\Z)$ that is extendable to
 an isometry $\tilde{\iota}\colon H^2(\tilde{X_1},\Z)\to H^2(\tilde{X_2},\Z)$,
 where $\tilde{X_i}$ is the minimal resolution of $X_i$,  
such that 
 $\iota(e_1)=e_2$ and $\iota([\Omega_1])\in \C[\Omega_2]+\C e_2$. 
 Then, the relative Jacobians of $f_i$ are isomorphic (over $\mathbb{P}^{1}$) 
 with each other. 
\end{Fac}

We put a proof of the above Fact \ref{Jac.equal} for readers' convenience: 
\begin{proof}[proof of Fact \ref{Jac.equal}]
Let us denote the relative Jacobian of $X_{i}$ by $J_{i}$.
By a $C^{\infty}$-section of $X_{i}\to \mathbb{P}^1$,
 the underlying $C^{\infty}$-manifolds of $X_{i}$ and $J_{i}$
 are identified with each other, giving an isomorphism
 $H^{2}(X_{i},\Z)\simeq H^{2}(J_{i},\Z)$. 
Let $f_{i}$ denote the Poincar\'e dual (cohomology class)
 of the zero section of $J_{i}$. 
Then, we consider an elementary transformation 
$\phi_{e_{1},\iota^{-1}(f_{2})-f_{1}}$
 on $H^{2}(X_{1},\Z)$ as (\ref{Eichler})
which satisfies 
$\phi_{e_{1},\iota^{-1}(f_{2})-f_{1}} (\iota^{-1}(f_{2})) = f_{1}$. 
Therefore, changing $\iota$ by the above elementary transformation,
 we can assume without loss of generality that
 $\iota(e_1)=e_2$ and $\iota(f_1)=f_2$.
Therefore, $J_{1}$ and $J_{2}$
 share K\"ahler classes of the form $e_{1}+\epsilon f_{1}$
 with $0<\epsilon \ll 1$. 
Furthermore, by intersecting with $f_{i}$, 
we can see that automatically we have
 $\iota([\Omega_{1}])\in \C[\Omega_{2}]$. 
Hence, by applying the strong Torelli theorem, 
we obtain an isomorphism between $J_{1}$ and $J_{2}$, 
preserving the fibration structures.
We confirmed Fact~\ref{Jac.equal}.
\end{proof}

Now, we come back to Theorem~\ref{GW} 
by turning to the proof of its second (main) statement, i.e., 
the Gromov-Hausdorff convergence. 
We prove it as follows by reducing it to a result of Gross-Wilson~\cite{GW}. 
For our sequence $x_i$, take corresponding K3 surfaces $X_i$ 
with the polarization and the marking. 
\begin{Const}[HyperK\"ahler rotation]\label{HK.rotation.}
By applying the hyperK\"ahler rotation, 
 we obtain K3 surfaces $X_i^{\vee}$
 with a holomorphic volume form $\Omega_i^{\vee}$
 and a marking $\alpha_i \colon H^2(X_i^{\vee},\Z)\to \Lambda_{\rm K3}$
 such that 
\begin{itemize}
\item $\alpha_i([\Omega_i^{\vee}])=\frac{1}{\sqrt{2d}}\lambda-\sqrt{-1}(v'+c_ie)$, and
\item $\alpha_i^{-1}(w_i)$ belongs to the K\"ahler cone of $X_i^{\vee}$.
\end{itemize}
\end{Const}

In order to apply Gross-Wilson~\cite[Theorem 6.4]{GW}, 
 we will show the following claim:
 
\begin{Claim}\label{e.nef}
There is a positive constant $C(v')$ which depends only on 
 $v'$ and $\mathcal{U}$, such that for any $i$ with 
$0<\epsilon_{i}<C(v')$, 
$\alpha_i^{-1}(e)$
 belongs to the closure of the K\"ahler cone of $X_{i}^{\vee}$. 
 \end{Claim}
 \begin{proof}[proof of Claim~\ref{e.nef}]
Recall that the K\"{a}hler cone of a K3 surface $X$ is 
\[K(X):=\{x\in C^+(X) \mid (x,\delta) >0 \ (\forall\delta\in \Delta(X)^+)\},\]
where $C^+(X)\subset H^{1,1}(X,\R)$ denotes the positive cone and
\[\Delta(X)^+ = \{\delta\in H^{1,1}(X,\Z)\mid
 (\delta,\delta)=-2 \text{ and $\delta$ is effective}\}.\]
Since $(w_i,e)>0$, $\alpha_i^{-1}(e)\in \overline{C^+(X_i^{\vee})}$. Otherwise, 
$-\alpha_i^{-1}(e)\in \overline{C^+(X_i)}$ which contradicts by the Cauchy-Schwarz inequality. 
To prove $\alpha_i^{-1}(e) \in \overline{K(X_i^{\vee})}$,
 it is enough to show
 that if $\epsilon_i$ is sufficiently small,
 there does not exist $\delta_i\in \Lambda_{2d}$ such that
 $(\delta_i, \delta_i)=-2$, $(\delta_i,v'+c_ie)=0$, 
 $(\delta_i,w_i)>0$, and $(\delta_i,e)<0$.
Let us write $\delta_i=s_ie-t_if+\delta_i'
\ (s_i,t_i\in \R,\ \delta'_i\in\Lambda'_{2d}\otimes \R)$.
The above four conditions become
\begin{align*}
&(\delta_i, \delta_i)=-2s_it_i+(\delta'_i,\delta'_i)=-2,\qquad
(\delta_i,v'+c_ie)=-c_it_i+(\delta'_i,v')=0,\\
&(\delta_i, w_i)=s_i\epsilon_i-t_iN_i+(\delta'_i,w'_i)>0,\qquad
(\delta_i,e)=-t_i \in \Z_{<0}.
\end{align*}
On $\Lambda'_{2d}\otimes \R$ we define a symmetric form $(\cdot, \cdot)_E$ by 
\[(x,y)_{E,v'}:=-(x,y-2(y,v')v').\]
Since $\Lambda'_{2d}\otimes \R$ has signature $(1,18)$, this is positive definite. 
Also, it is easy to see that $(x,y)^2\leq (x,x)_{E,v'}(y,y)_{E,v'}$ 
by the Cauchy-Schwarz inequality. 
Since $-c_it_i+(\delta'_i,v')=0$ and $c_i$ is bounded, there exists
 a constant $C_1$
 such that $$|(\delta'_i,v')|\leq C_1t_i.$$ Then we have 
\begin{align*}
(\delta'_i,\delta'_i)_{E,v'}=-(\delta'_i,\delta'_i-2(\delta'_i,v')v')
&=-(\delta'_i,\delta'_i)+2(\delta'_i,v')^2 \\
&\leq (-2s_it_i+2)+2C_1^2t_i^2.
\end{align*}
Since $(\delta'_i,\delta'_i)_{E,v'}\geq 0$, we must have
 \begin{equation}\label{s.t.ineq}
 s_i\leq \frac{1}{t_i}+ C_1^2 t_i\leq 1 + C_1^2 t_i.
 \end{equation}
Recall we assumed $\epsilon_i^{-1}w'_i$ is bounded.
Set a constant 
$$C_2:=\max \{ (w,w)_{E,v'}\mid
 w\in \mathcal{U}\subset \Lambda'_{2d}\otimes \R \}>0$$ 
so that $(w'_i,w'_i)_{E,v'} \leq 
\epsilon_i^2 C_2$.
Then we have 
\begin{align*}
|(\delta'_i,w'_i)|^2
 \leq (w'_i,w'_i)_{E,v'} (\delta'_i,\delta'_i)_{E,v'}
\leq 2 C_2 \epsilon_i^2 (C_1^2t_i^2-s_it_i+1).
\end{align*}
Since $\epsilon_iN_i\to \frac{1}{2}$ as $\epsilon_i \to 0$,
we may assume $N_i\geq \frac{1}{3} \epsilon_i^{-1}$. 
Hence
\begin{align}\label{w.e}
(\delta_i, w_i)=s_i\epsilon_i-t_iN_i+(\delta'_i,w'_i)
\leq -\frac{t_i}{3\epsilon_i}
 + \epsilon_i\Bigl(s_i+\sqrt{2C_2(C_1^2t_i^2-s_it_i+1)}\Bigr).
\end{align}
We claim that
 $$s_i+\sqrt{2C_2(C_1^2t_i^2-s_it_i+1)} \leq C_3t_i$$
 for some constant $C_3>0$, which depends only on $C_1$ and $C_2$. 
Indeed, we calculate
\begin{align*}
s_i+\sqrt{2C_2(C_1^2t_i^2-s_it_i+1)}
&=s_i+\sqrt{-2C_2s_it_i+2C_2(C_1^2t_i^2+1)} \\
&\leq s_i+\sqrt{|2C_2s_it_i|}+\sqrt{2C_2(C_1^2t_i^2+1)}\\
&\leq s_i+|s_i|+\frac{C_2t_i}{2}+\sqrt{2C_2(C_1^2t_i^2+1)}\\
&\leq \max\{0, 2s_i\}+\frac{C_2t_i}{2}+\sqrt{2C_2(C_1^2t_i^2+1)}.
\end{align*}
The claim follows from \eqref{s.t.ineq}. 
If $\epsilon_i$ is sufficiently small, more precisely if 
$\epsilon_i < (3C_3)^{-\frac{1}{2}}, $
then the above value \eqref{w.e}
cannot be positive, 
 which contradicts our assumption $(\delta_i, w_i)>0$.
Note that our constants $C_{1}, C_{2}, C_{3}$ only depends on 
 $v'$
 and a fundamental domain $\mathcal{U}$ in $\Lambda'_{2d}\otimes \R$ 
 and hence we complete the proof of Claim~\ref{e.nef}. 
\end{proof}

We return to the proof of Theorem~\ref{GW}. 
By Fact~\ref{bp.free} and Claim~\ref{e.nef}, we obtain an elliptic fibration
 $\pi_i\colon X_i^{\vee}\twoheadrightarrow \mathbb{P}^1$ with the fiber class $e$ 
for sufficiently large $i$. 
Moreover, by Fact~\ref{Jac.equal.ADE}, relative Jacobians of
 $\pi_i\colon X_i^{\vee}\twoheadrightarrow \mathbb{P}^1$
 are all isomorphic to each other.
We now assume the genericity condition on $v$ 
(as we wrote in the statements of Theorem~\ref{GW}). 
Suppose that: 
\begin{Ass}\label{I1}
 the obtained elliptic fibration $\pi_i$ only has (exactly $24$) $I_{1}$-type degenerations 
 in the Kodaira's notation, i.e., all singular fibers are irreducible nodal rational curves. 
 \end{Ass}
Now we can apply \cite[Theorem 6.4]{GW}. 
It proves that the Gromov-Hausdorff limit of $(X_{i}^{\vee},\omega_{i}^{\vee})$
 as $i\to \infty$,
 where $\omega_{i}^{\vee}$ is the unique Ricci-flat K\"ahler form
 whose K\"ahler class is $w_{i}$, 
 is given by 
 the base $B(\simeq \mathbb{P}^{1})$ 
 with the (singular) K\"ahler metric which \cite{GW} denote by 
  ``$W_{0}^{-1}dy\otimes d\bar{y}$''. 
It is, by \cite{GTZ1,GTZ2}, also a special K\"ahler metric
 in the sense of (\cite{Str90, DW96, Hit99, Freed99} etc.) 
 and also coincides with 
 the McLean metric (\cite[Remark 6.5]{GW})
 due to straightforward calculations as we wrote in \S\ref{trop.K3.1}.\footnote{This is also called in yet another name e.g. in \cite{Lof}. }
Anyhow, the second statement of Theorem~\ref{GW}
 directly follows from \cite[Theorem 6.4]{GW}.
\end{proof}

We generalize and strengthen Theorem~\ref{GW} as Theorem 
\ref{K3.Main.Conjecture.18.ok} below.  
The proof is also a technical extension of the above arguments 
proving Theorem~\ref{GW}, which will be done as replacing the direct 
use of \cite{GW} by certain refinements of \cite{Tos.lim,Tos,GTZ1,GTZ2}
 in the K3 surface case. 
Later in \S\ref{K3.along.disk},
 we show that this also implies 
 Kontsevich-Soibelman~\cite[Conjecture 1]{KS}
 for polarized algebraic K3 surfaces case in the full generality.

\begin{Thm}
\label{K3.Main.Conjecture.18.ok}
Our geometric realization map $\Phi_{\rm alg}$ is continuous except
 for the finite points $\bigcup_{p}\mathcal{F}_{2d}(p)$, i.e., 
\begin{align*}
\Phi_{\rm alg}\colon \overline{\mathcal{F}_{2d}}^{\rm Sat}
 \setminus\bigcup_{p}\mathcal{F}_{2d}(p)
 =\mathcal{F}_{2d}\cup \bigcup_{l}\mathcal{F}_{2d}(l) \to 
{\it CMet}_{1}
\end{align*}
 is continuous. 
\end{Thm}

Recall from  \S\ref{K3.Sat.sec} that 
$p$ (resp.\ $l$) denotes the finite equivalence classes of rational 
isotropic planes (resp.\ isotropic lines) inside $\Lambda_{2d}\otimes \Q$ and 
$\mathcal{F}_{2d}(p)$ is one point for each $p$ (resp.\ 
$\mathcal{F}_{2d}(l)$ is the corresponding $18$-dimensional boundary component). 

\begin{proof}[proof of Theorem~\ref{K3.Main.Conjecture.18.ok}]
First we prove that $\Phi_{\rm alg}$ is continuous on $\mathcal{F}_{2d}$. 

The continuity on $\mathcal{F}_{2d}^{o}$, the locus of 
\textit{smooth} polarized K3 surfaces, was proved in
 Proposition~\ref{easy.confirmation}. 

On the other hand, the following claim
 follows from (or can be proved similarly to) Proposition~\ref{MK3.conti}. 
\begin{Claim}[Around the discriminant locus]\label{GH.conti.ADE}
$\Phi_{\rm alg}$ is continuous 
 on a neighborhood of $\mathcal{F}_{2d}\setminus \mathcal{F}_{2d}^{o}$,
 the locus of ADE singular polarized K3 surfaces. 
\end{Claim}
Here, we give an alternative proof which uses
 a general result from \cite{DS}. 
First we take $[(X_{\infty},L_{\infty})]\in 
\mathcal{F}_{2d}\setminus \mathcal{F}_{2d}^{o}$ and suppose 
a sequence $[(X_{i},L_{i})]\in \mathcal{F}_{2d}^{o}$ converges to 
$[(X_{\infty},L_{\infty})]$ in the analytic topology of  
$\mathcal{F}_{2d}$. If we write $\omega_{i}$ for
 the Ricci-flat K\"ahler metric of $X_{i}$ in the class $c_{1}(L_{i})$, 
then by \cite[Theorem 3.1]{Tos.lim} the diameter 
${\rm diam}(X_{i},\omega_{i})$ is uniformly bounded above
 (see \S\ref{Tos.lim.uni}). 
Then combined with the Bishop-Gromov inequality, which infers the 
non-collapsingness condition \cite[(1.2)]{DS}, we can 
apply \cite[Theorem 1.2]{DS} to the sequence $(X_{i},L_{i},\omega_{i})$ 
so that its Gromov-Hausdorff limit is still 
a weak Ricc-flat log terminal projective variety. 
This has to be an ADE singular K3 surface in our case. 
If we take a closer look at the proof in \cite{DS}, 
 it furthermore proves that this ADE K3 surface is a limit inside the Hilbert scheme 
of $X_{i}$ embedded by $L_{i}^{\otimes m}$ for $m\gg 0$.
Hence it is the limit of $[(X_{i},L_{i})]\in \mathcal{F}_{2d}$
 with respect to the analytic topology of $\mathcal{F}_{2d}$,
 i.e., $[(X_{\infty},L_{\infty},\omega_{\infty})]$. 
For a general sequence $[(X_{i},L_{i})]\in \mathcal{F}_{2d}$ converging to 
$[(X_{\infty},L_{\infty})]$ with respect to the complex analytic topology of 
$\mathcal{F}_{2d}$ with possibly singular $X_{i}$,
 we employ the usual diagonalizing argument (as in the proof of
 ``Claim~\ref{gen.seq} implies Lemma~\ref{all.seq}" in \S\ref{Gen.HSD})
 to reduce it to the case of smooth $X_{i}$.
Therefore, we confirmed the continuity of $\Phi_{\rm alg}$
 on the whole $\mathcal{F}_{2d}$. 

Now we move on to the most analytically complicated part of this paper,
 i.e., estimates of Ricci-flat metrics for the collapsing case. 
We first make our setup as follows. 

Define the period domain to be
\[\Omega(\Lambda_{\rm K3}):=
 \{[w]\in \mathbb{P}(\Lambda_{\rm K3}\otimes \mathbb{C})\mid
 (w,w)=0,\ (w,\bar{w})>0\}\] 
 and define 
\[\Delta(\Lambda_{\rm K3})
 :=\{\delta\in \Lambda_{\rm K3}\mid (\delta,\delta)=-2\}.\]
For $[w]\in \Omega(\Lambda_{\rm K3})$, let 
\[V([w]):=\{\kappa\in \Lambda_{\rm K3}\otimes \R
\mid \kappa\perp [w] \text{ and } (\kappa,\kappa)=1\},
\]
which has two connected components.
We choose one of the connected components of $V([w])$
 which is continuous in $[w]$ and denote it by $K_{[w]}$.
The choice of $K_{[w]}$ is assumed to be compatible with
 the choice of one connected component
 $\mathcal{D}_{\Lambda_{2d}}$ in $\Omega(\Lambda_{2d})$
 in the following sense:
if $[w]\in \mathcal{D}_{\Lambda_{2d}}$ and  
 $[u]:=[\frac{1}{\sqrt{2d}}\lambda-\sqrt{-1}\im w]$, then $[\re w]\in K_{[u]}$.
Define the weakly polarized domain 
 $K\Omega$ as (see \cite{KT})
\[K\Omega:=\{[[w],\kappa]\in \Omega(\Lambda_{\rm K3})
 \times (\Lambda_{\rm K3}\otimes \R) \mid \kappa\in K_{[w]}\}.\]
This can be regarded as 
 the union of the K\"ahler cones 
 for possibly ADE singular marked K3 surfaces.
Its open subset 
\[K\Omega^{o}:=\{[[w],\kappa]\in K\Omega \mid 
([w], \delta) \neq 0
 \text{ or } (\kappa,\delta)\neq 0
 \text{ for all } \delta\in \Delta(\Lambda_{\rm K3})\}\]
 represents that of smooth marked K3 surfaces. 
We write ${\rm pr}\colon K\Omega\twoheadrightarrow \Omega(\Lambda_{\rm K3})$
 for the natural projection so we have $K_{p}={\rm pr}^{-1}(p)$.
Let $K_{p}^{o}:={\rm pr}^{-1}(p)\cap K\Omega^{o}$. 
The decomposition of $K_{p}^o$ into its connected components
 is denoted by $K_{p}^{o}=\bigsqcup_{i} K_{p}^{o,i}$. 
Each connected component $K_{p}^{o,i}$
 is the image of the K\"ahler cone of a smooth K3 surface by a marking. 
More generally, whole $K_{p}$ for a fixed $p=[w]$ can be
 stratified by walls as follows. 
For $\kappa\in K_{p}$, define 
\begin{align*}
&I_+ := \{\delta\in \Delta(\Lambda_{\rm K3})
 \cap [w]^{\perp} \mid (\delta,\kappa) >0\},\\
&I_0 := \{\delta\in \Delta(\Lambda_{\rm K3})
 \cap [w]^{\perp} \mid (\delta,\kappa) =0\},\\
&I_- := \{\delta\in \Delta(\Lambda_{\rm K3})
 \cap [w]^{\perp} \mid (\delta,\kappa) <0\}.
\end{align*}
This gives the decomposition
 $$\Delta(\Lambda_{\rm K3})\cap [w]^{\perp}=I_+\sqcup I_0\sqcup I_-.$$
We write $I(\kappa):=(I_+,I_0,I_-)$.
Then $K_p$ can be stratified by $I(\kappa)$, namely, 
for each decomposition $I$ of 
$\Delta(\Lambda_{\rm K3})\cap [w]^{\perp}$ into three disjoint 
subsets, we define a stratum by
 $K_{p}^{I}:=\{\kappa\in K_{p}\mid I=I(\kappa)\}$, 
 which is the image of the K\"ahler cone of a possibly ADE singular K3 surface 
 by a marking in the sense of \cite{Mor}. 
We discuss this with some different viewpoint again
 in \S\ref{Kahler.section} but see \cite{Mor, KT} for the details. 

Fix an isotropic vector $e\in \Lambda_{\rm K3}$ and 
define 
\[\Omega^{e}(\Lambda_{\rm K3})
 :=\{[w]
 \in \Omega(\Lambda_{\rm K3}) 
\mid 
[w] \perp e\}.\]
Moreover, we set 
\[K\Omega^{e\ge 0}
:=\bigsqcup_{p\in \Omega^e(\Lambda_{\rm K3})}K_{p}^{e\ge 0},\]
where 
\[K_{p}^{e\ge 0}:=\bigsqcup_{\overline{K_{p}^{I}}\ni e}K_{p}^{I}.\] 
By Fact \ref{bp.free} and the discussion before that, 
the condition $\overline{K_{p}^{I}}\ni e$ 
means that 
$s\in K\Omega^{e\ge 0}$ 
gives an elliptic fibration structure to the corresponding
 possibly ADE singular K3 surface 
$$\pi_{s}\colon (X_{s},\omega_{X_{s}})\twoheadrightarrow 
B_{s}(\simeq \mathbb{P}^{1}),$$ 
with a marking 
in the sense of \cite{Mor}, 
whose fiber class is $e$.


To continue our proof of Theorem~\ref{K3.Main.Conjecture.18.ok}, 
 we introduce the following necessary analytic input. 
We leave its proof to \S\ref{GTZ.extend.proof}
 as the discussion needed there is of very different nature and is lengthy.

\begin{Thm}[cf.\ \cite{Tos}, \cite{GTZ1}, \cite{GTZ2}, \cite{TZ}]\label{GTZ.extend}
Let $S$ be an open subset of $K\Omega\cap {\rm pr}^{-1}(\Omega^e(\Lambda_{\rm K3}))$
 such that $S\subset K\Omega^{e\ge 0}$ and the closure
 $\overline{S}$ in $K\Omega$ is a compact subset of $K\Omega^{e\ge 0}$. 
Let $\{\pi_{s}\colon 
(X_{s},\omega_{X_{s}})\twoheadrightarrow B_{s}\simeq \mathbb{P}^{1}\}_{s\in S}$
 be a collection of
 elliptic (K\"ahler, not necessarily algebraic)
 K3 surfaces, possibly with ADE singularities on $X_{s}$, 
 with markings (in the sense of \cite{Mor}), 
 whose $\pi_s$-fiber class is $e$ and
 the period and the K\"ahler class of $\omega_{X_{s}}$
 is given by $s$ via the marking. 

Consider the Ricci-flat K\"ahler (orbi-)metrics $\tilde{\omega}_{s,t}\in 
[\pi_{s}^{*}\omega_{B_{s}}+t\omega_{X_{s}}]$, where $\omega_{B_{s}}$ is a 
 smooth family of K\"ahler metrics on the base $B_{s}$
 and $t>0$. 
Let ${\rm disc}(\pi_{s})$ denote the discriminant locus of $\pi_s$. 
Define a metric on $B_{s}\setminus {\rm disc}(\pi_{s})$ as in \eqref{eq:McLean}
 and write $\omega_{s,\rm{ML}}$ for the corresponding K\"ahler form.
Write $(B_{s}, \omega_{s,{\rm ML}})$ for the metric completion of
 $(B_{s}\setminus {\rm disc}(\pi_{s}), \omega_{s,{\rm ML}})$.
Then $(X_{s},\tilde{\omega}_{s,t})$ converge to $(B_{s},\omega_{s,{\rm ML}})$
 in the Gromov-Hausdorff sense when $t\to 0$. 
Moreover, this convergence is uniform with respect to $s\in S$. 
\end{Thm}

In this section, we prove Theorem~\ref{K3.Main.Conjecture.18.ok} assuming Theorem~\ref{GTZ.extend}. 
As the reference metric $\omega_{X_{s}}$, we take the Ricci-flat K\"ahler metrics. 
Take an arbitrary sequence in $\mathcal{F}_{2d}$
 which converges
 to a point $[e,v]\in \mathcal{F}_{2d}(l)$ for an isotropic line $l=\Q e$. 
(This is not necessarily of the form in Theorem~\ref{GW}.) 
We take corresponding K3 surfaces $X_i$
 with polarizations of degree $2d$
 and want to identify the Gromov-Hausdorff limit
 with a (underlying metric space of) specific tropical K3 surface. 

Let us employ the Siegel reduction as follows. 
As in Theorem~\ref{GW} (and its proof),
 we take an isotropic vector $f\in \Lambda_{2d}\otimes \Q$ with $(e,f)=1$
 and define the lattice
 $\Lambda'_{2d}=\langle e, f \rangle ^{\perp}\subset\Lambda_{2d}$
 in the same way.  
Recall the Langlands decomposition of the stabilizer $Q$ of $l_{\R}=\R e$ 
from the proof of Claim~\ref{Reduction}. 
By the definition of Satake topology,
 we can take representatives $[x_i]\in \mathcal{D}_{\Lambda_{2d}}$
 of the given sequence in $\mathcal{F}_{2d}$
 which are in the same Siegel set. 
For a bounded open subset
 $\mathcal{U}$, $\mathcal{V}$ of $\Lambda'_{2d}\otimes \R$
 and $t>0$, define the set 
$\mathfrak{S}_{\mathcal{U},t,\mathcal{V}}
(\subset \mathcal{D}_{\Lambda_{2d}})$ 
(cf.\ \S\ref{K3.Sat.sec}, \cite{Sat1, Sat2, BJ}) 
 as the set consisting of $[x]\in \mathcal{D}_{\Lambda_{2d}}$
 for $x\in \Lambda_{2d}\otimes \C$ satisfying the conditions
\begin{align}\label{Siegel.condition}\nonumber
&(\re x, \re x)=(\im x, \im x)=1,\quad (\re x, \im x)=0, \\ 
&\re x\in N e + \epsilon f +\epsilon \, \mathcal{U} \quad
 (N>0,\ 0<\epsilon<t),\\ \nonumber
&\im x = c e + v' \quad  (c \in\R,\ v'\in \mathcal{V}).
\end{align}
The value $c$ is automatically bounded on
 $\mathfrak{S}_{\mathcal{U},t,\mathcal{V}}$
 by the condition $(\re x, \im x)=0$.
Then we may assume that our sequence
 $[x_i]\in \mathcal{D}_{\Lambda_{2d}}$ lie in 
 this set for some bounded sets $\mathcal{U}$ and $\mathcal{V}$
 and $t>0$ which do not depend on $i$,
 namely, we have
\begin{align*}
&(\re {x_i}, \re {x_i})=(\im {x_i}, \im {x_i})=1,
 \quad (\re {x_i}, \im {x_i})=0, \\
&\re {x_i}\in N_i e + \epsilon_i f +\epsilon_i \, \mathcal{U} \quad
 (N_i,\epsilon_i>0),\\
&\im {x_i}= c_i e + v'_i \quad  (c_i \in\R,\ v'_i\in \mathcal{V}).
\end{align*}
The convergence $[x_i]\to [e,v]$ with respect to the Satake topology
 implies that $\epsilon_i\to 0$ and $v'_i \pmod {l_\R}\to v$
 as $i\to \infty$.
By applying a hyperK\"ahler rotation to
 the polarized K3 surfaces $X_i$  as in the proof of Theorem~\ref{GW}, 
 we obtain marked K3 surfaces $X_i^{\vee}$
 with holomorphic volume form $\Omega_i^{\vee}$
 and K\"ahler form $\omega_i^{\vee}$ such that 
 $\alpha_i([\Omega_i^{\vee}]) = \frac{1}{\sqrt{2d}}\lambda-\sqrt{-1}\im x_i$
 and  $\alpha_i([\omega_i^{\vee}]) = \re x_i$.

For $x\in \Lambda_{2d}\otimes \C$
 satisfying \eqref{Siegel.condition} for small enough $t$, 
 we define $X^{\vee}$ as above.
Combining Fact~\ref{bp.free} and a refinement of Claim~\ref{e.nef} as 
Claim~\ref{e.nef.general}
 we obtain an elliptic fibration structure on $X^{\vee}$:
\[\pi\colon X^{\vee}\twoheadrightarrow B\simeq \mathbb{P}^{1},\] 
with the fiber class $[\pi^{-1}(\text{point})]=\alpha^{-1}(e)$.
The obtained fibration can be also regarded
 as a \textit{special Lagrangian (torus) fibration} on $X$ which is 
``close to" large complex structure limit. 
In fact, in view of Claim~\ref{e.nef}, 
we constructed such fibration structures on K3 surfaces parametrized 
 by an explicit neighborhood of 
 the boundary $\mathcal{F}_{2d}(l)$. 
Recall that in general,
 it is a challenging problem to construct a special Lagrangian fibration
 on general Calabi-Yau varieties around large complex structure limits. 

Also recall that on $\mathcal{F}_{2d}$, more precisely on a finite cover
 or in a stacky sense, we have a holomorphic proper family of
 (possibly ADE) polarized K3 surfaces $(X,L)$ with degree $2d$. 
By Fact~\ref{bp.free} and Claim~\ref{e.nef.general}, 
 for each $\mathcal{U}$ and $\mathcal{V}$, 
 the existing family $\mathcal{X}=\{(X,L,\alpha)\}$
 of marked ADE singular polarized K3 surfaces on 
 $\mathfrak{S}_{\mathcal{U},t,\mathcal{V}}$
 for small enough $t>0$, has simultaneous elliptic fibration structure
 $X\twoheadrightarrow B(\simeq \mathbb{P}^{1})$
 (with non-canonical isomorphism $\simeq$)
 such that the fiber classes are $\alpha^{-1}(e)$. 

Now we come back to our sequence $\{X_i\}_i$
 and its hyperK\"ahler rotation $X_i^{\vee}$. 
We apply Claim~\ref{e.nef.general}
 to obtain $\pi_i\colon X_i^{\vee}\twoheadrightarrow B_i\simeq \mathbb{P}^1$.
Let $\mathcal{Z}$ be a bounded
 open neighborhood of $\frac{\lambda}{\sqrt{2d}}$
 in the set $\{z\in l_{\R}^{\perp} \mid (z,z)=1\}$.
For $\epsilon_0>0$, define $S_{\epsilon_0}\subset K\Omega$ to be 
\begin{multline*}
S_{\epsilon_0}=\{[[z-\sqrt{-1}\im x], \re x]\in K\Omega
\mid \text{$x\in \Lambda_{2d}\otimes \C$ satisfies
 \eqref{Siegel.condition}
 for $\epsilon_0<\epsilon<2\epsilon_0$}\\
\text{ and $z \in \mathcal{Z}$}\}.
\end{multline*}
Then for sufficiently small $\epsilon_0$,
 Claim~\ref{e.nef.general} implies that
 $S_{\epsilon_0}\subset K\Omega^{e\geq 0}$
 and that $S_{\epsilon_0}$ satisfies the assumption 
 on $S$ in Theorem~\ref{GTZ.extend}.
Moreover, for sufficiently large $i$,
 there exists a Ricci-flat K\"ahler form $\omega_{X_i^{\vee}}$
 on $X_i^{\vee}$ and $t_i>0$ such that 
 $(X_i^{\vee}, \alpha_i, \omega_{X_i^{\vee}})$ is parametrized
 in $S_{\epsilon_0}$ and that 
 the cohomology class $[\omega_i^{\vee}]$
 equals $[t_i\omega_{X_i^{\vee}}] + e$ up to $\R^{\times}$-multiplication.
Furthermore, $t_i\to 0$ as $i\to \infty$.
We can therefore apply Theorem~\ref{GTZ.extend} to $S=S_{\epsilon_0}$
 and $\omega_i^{\vee}$ appears 
 as ``$\tilde{\omega}_{s,t_i}$" up to rescaling
 in the notation of Theorem~\ref{GTZ.extend}.
This shows that for any sequence $[x_i]\in \mathcal{F}_{2d}$
 which converges to $[e,v]\in \mathcal{F}_{2d}(l)$,
 the metric spaces $\Phi_{\rm alg}([x_i])$ converge to
 $\Phi_{\rm alg}([e,v])$ with respect to the Gromov-Hausdorff topology.

By an argument as in the proof of
 ``Claim~\ref{gen.seq} implies Lemma~\ref{all.seq}" in \S\ref{Gen.HSD}, 
 we obtain Theorem~\ref{K3.Main.Conjecture.18.ok}. 
(See also the related extension,
 the proof of Theorem~\ref{K3.Main.Theorem2}.) 
\end{proof}

\begin{Rem}[Tropical Heegner divisor]\label{Trop.ADE.F2d}
Consider the Heegner divisor 
$\mathcal{H}_{\delta}\subset \mathcal{D}_{\Lambda_{2d}}$ 
for each $(-2)$-element
 $\delta\in\Delta(\Lambda_{\rm K3})\cap \Lambda_{2d}$. 
Recall that the union of its image 
${\rm Im}(\bigcup_{\delta} \mathcal{H}_{\delta})\subset \mathcal{F}_{2d}$ 
exactly parametrizes ADE singular (non-smooth) polarized K3 surfaces 
(cf.\ e.g.\ \cite{Huy}). 

Similarly, inside each $18$-dimensional boundary component 
$\mathcal{F}_{2d}(l)$ (see \S\ref{K3.Sat.sec})
we can define a closed subset $\mathcal{F}_{2d}^{\rm ADE}(l)$ of 
real codimension $1$ (i.e., dimension $17$) which parametrizes 
those appearing as the base of ADE singular (\textit{non-}smooth) 
elliptic K3 surfaces. It is defined as 
$$
\mathcal{F}_{2d}^{\rm ADE}(l):=
\overline{{\rm Im}(\bigcup_{\delta} \mathcal{H}_{\delta})}\cap 
\mathcal{F}_{2d}(l). 
$$
Here, $\{\mathcal{H}_{\delta}\}$ is the set of the Heegner divisors 
parametrizing ADE polarized K3 surfaces 
($\delta$ runs over the set of the 
$\tilde{O}^{+}(\Lambda_{2d})$-equivalence classes of 
$(-2)$-elements in $\Lambda_{2d}$), which is known to be finite. 
Indeed, for instance, by generic smoothness of polarized K3 surfaces parametrized by $\mathcal{F}_{2d}$ with respect to 
the Zariski topology implies the desired finiteness. 

Now we confirm a tropical analog of such finiteness. 
For each point $[e,v]\in\mathcal{F}_{2d}(l)$, 
 and for a subset $U(\subset \mathcal{D}_{\Lambda_{2d}})$ of the form \eqref{Siegel.condition}
 which maps to the intersection of a neighborhood of $[e,v]$ with $\mathcal{F}_{2d}$, we show the following: 
\begin{Claim}\label{delta.finite} 
the $(\tilde{O}^{+}(\Lambda_{2d})\cap
 {\rm stab}(l_{\R}))$-\textit{equivalence classes of} $\delta$
 with $\mathcal{H}_{\delta}\cap U\neq \emptyset$ form a finite set,
 where ${\rm stab}(l_{\R})$ stands for the stabilizer subgroup for $l_{\R}$ 
 in $O(\Lambda_{2d}\otimes \R)$. 
\end{Claim}

\begin{proof}
Suppose that $x\in \mathcal{H}_{\delta}\cap U$.
As in the proof of Claim~\ref{e.nef},
 we decompose as $\delta=se-tf+\delta'\,(\delta'\in \Lambda'_{2d})$.
In that proof (see also Claim~\ref{e.nef.general}),
 we assumed $(\delta, \re x)>0$ and $(\delta,e)<0$ and
 got a contradiction if $(\re x, e)$ is small enough.
But, it can be similarly seen that
 $(\delta, \re x)\geq 0$ and $(\delta,e)<0$ yield a contradiction.
Hence in our situation, if $(\delta, \re x)=0$, then $(\delta,e)\geq 0$.
By replacing $\delta$ by $-\delta$, we also have $(\delta,e)\leq 0$
 and therefore $(\delta,e)=0$.

Let $\mathcal{V}\subset \Lambda'\otimes \R$ be a bounded set
 appearing in \eqref{Siegel.condition} and 
 set for $y\in \Lambda'_{2d}\otimes \R$ 
\[
(y,y)_{E,\mathcal{V}}:=\inf \{(y,y)_{E,v'}\mid v'\in \mathcal{V}\} 
\]
(cf.\ the proof of Claim~\ref{e.nef}), 
which may not be a bilinear form. However, from the
 positive definiteness of each
 $(,)_{E,v'}$ and the boundedness of $\mathcal{V}$,
 we see that
 $\{y\in \Lambda'_{2d}\otimes \R\mid (y,y)_{E,\mathcal{V}}<C\}$
 for a fixed constant $C$ is bounded in $\Lambda'_{2d}\otimes \R$. 
Then $(\delta',v')=0$ implies 
 $(\delta',\delta')_{E,v'}=2$, where
 $v'$ is the $\Lambda'_{2d}$-component of $\im x$,
 showing $\delta'$ is contained in a bounded set. 
Hence there are only finitely many such $\delta'$.
By the Eichler transvections, we may change $s$ so that $s$ becomes bounded.
Therefore it follows that the equivalence classes of $\delta$ is finite.
\end{proof}

Due to Claim~\ref{delta.finite},
 we get that $\mathcal{F}_{2d}^{\rm ADE}(l)$ is
 a $17$-dimensional subset of $\mathcal{F}_{2d}(l)$,
 which can be seen as tropical analog of Heegner divisors. 
\end{Rem}

\section{Holomorphic limits and monodromy}\label{K3.along.disk}

In this section, as in \cite{Gross, KS, TGC.II} 
 and our \S\ref{monodromy.AV}, \S\ref{reconst.AV}, 
we consider a one parameter holomorphic family of polarized K3 surfaces 
and consider the Gromov-Hausdorff limit of the rescaled Ricci-flat K\"ahler metrics on the fibers along the degeneration. 
As discussed in \cite{TGC.II}, even the existence of the limit 
 (independent of the choice of convergent sequence to $0$), 
is not a priori clear. Nevertheless, our previous discussion implies the following: 

\begin{Cor}\label{1par.lim.K3.}
Let  $\pi^{*}\colon (\mathcal{X}^{*},\mathcal{L}^{*})\twoheadrightarrow \Delta^{*}$ be a 
punctured holomorphic family of polarized K3 surface of degree $2d$, possibly 
with ADE singularities, and suppose it is of type III in the Kulikov's sense 
(``maximal degeneration'') \cite{Kul, PP}. 
That is, if we write as $\gamma\in O(\Lambda_{2d}) \subset O(2,19)$ the monodromy on 
$H^{2}(\mathcal{X}_{s},\Z)$ for a fixed $s\neq 0$ 
 around $0\in \Delta$, 
with respect to a marking of $H^{2}(\mathcal{X}_{s},\Z)$, then its nilpotent index is $3$. 
We write the Jordan decomposition of the monodromy as $\gamma=\gamma_{s}\gamma_{u}$, where $\gamma_{s}$ is the semisimple part 
while $\gamma_{u}$ is the unipotent part. Set $N:=\log\gamma_{u}$ (as in \S\ref{extension.monodromy}). 
We denote the holomorphic map from the base $\Delta^{*}\to \mathcal{F}_{2d}$ corresponding to $\pi^{*}$ by $\varphi^{o}$. 

\begin{enumerate}

\item \label{1par.K3.mon}
$\varphi^o$ extends to a continuous map $\overline{\varphi^o}\colon \Delta\to \overline{\mathcal{F}_{2d}}^{\rm Sat}$ 
whose limit $\overline{\varphi^o}(0)$ is in 
the $18$-dimensional strata $\bigsqcup_{l}\mathcal{F}_{2d}(l)$ (cf.\ \S\ref{K3.Sat.sec}) and exactly coincides with 
the image of the monodromy point $[N]\in \mathbb{P}(\mathfrak{g})$. Recall that now 
we have $G=O(\Lambda_{2d}\otimes \R)$ so that $\mathfrak{g}\simeq \mathfrak{o}(2,19)$. 

\item \label{Rel.FriS} (Relation with \cite{FriS})
Note further that 
$$\mathfrak{g}\simeq \wedge^{2}(\Lambda_{2d}\otimes \R)
\subset {\rm End}(\Lambda_{2d}\otimes \R).$$ 
The last inclusion is induced by the bilinear form on $\Lambda_{2d}(\otimes \R)$. 
Through this identification, $N$ is identified (up to constant) with $e\wedge \delta$ for an isotropic element $e\in \Lambda_{2d}$ 
where $\delta$ is defined as \cite[p.7, before (1.1) Lemma]{FriS} and $\Z e$ is the $0$-th piece $W_0$ of the weight monodromy filtration 
(cf., e.g., \cite[beginning of \S1]{FriS}). 

\item ({\cite[Conjecture 1]{KS}, \cite[Conjecture 6.4]{GW}, \cite[Conjecture 5.4]{Gross}} for K3 surfaces) \label{KS.GW.conj} 
For $t\to 0$, the K3 surfaces with rescaled Ricci-flat K\"ahler metrics 
$\bigl(\mathcal{X}_{t},\frac{d_{\rm KE}(\mathcal{X}_{t})}{{\rm diam}(\mathcal{X}_{t})}\bigr)$ converge (collapse) to a 
tropical polarized K3 surface, i.e., a metrized $S^{2}$, in the Gromov-Hausdorff sense. 
\end{enumerate}
For the last statement \ref{KS.GW.conj}, we can further specify the limit metrized $S^{2}$ as 
$\Phi_{\rm alg}(\bar{\varphi}(0))$ explicitly defined in \S\ref{trop.K3.1}. 
\end{Cor}

\begin{proof}
The first assertion \ref{1par.K3.mon} of the theorem 
simply follows from Theorems~\ref{rationality.HSD}, \ref{monodromy} 
for $\mathcal{F}_{2d}$ (cf.\ also Proposition~\ref{MS.reconstruction}). 
The limit $\overline{\varphi^{o}}(0)$ is in $\bigsqcup_{l}\mathcal{F}_{2d}(l)$
 (cf.\ \S\ref{K3.Sat.sec}) because of the maximal degeneration condition of $\pi^{*}$.
This can be seen from that the extension of $\varphi^{o}$ to
 the Satake-Baily-Borel compactification sends $0$ to 
 a point in $0$-dimensional boundary components
 (i.e.\ $F$ in the statement of Theorem~\ref{rationality.HSD} becomes
 a $0$-dimensional component in our situation).
Alternatively, the assumption that $N$ has nilpotent index $3$
 and our description of boundary component in $\mathbb{P}(\mathfrak{g})$
 after \eqref{k_center} 
 also show $\overline{\varphi^{o}}(0)\in \bigsqcup_{l}\mathcal{F}_{2d}(l)$.

We prove the second assertion \ref{Rel.FriS}, depending on \cite{FriS}. 
Let us consider the lift of $\varphi^o$ as 
$$\widetilde{\varphi^o}\colon \mathbb{H}\to D\subset {\rm Gr}_2(\Lambda_{2d}\otimes \R)\subset \mathbb{P}(\mathfrak{g}).$$
Here, the last inclusion is induced by the Pl\"ucker embedding. 
From Theorem~\ref{monodromy} and its proof, combined with \cite[(1.1) Lemma]{FriS}, we get that for a fixed $x\in \R$ we have 
$$\lim_{y\to +\infty}\widetilde{\varphi^o}(x+\sqrt{-1}y)=[N]=[e'\wedge \delta],$$ 
where $e'$ is the generator\footnote{originally denoted by $\gamma$ in \cite[\S1]{FriS} but we change the notation to avoid confusion with our 
monodromy $\gamma$.} of the last piece of the weight monodromy filtration $W_0$. 
On the other hand, our Theorem~\ref{K3.Main.Conjecture.18.ok} and its proof gives that 
$\lim_{y\to +\infty}\widetilde{\varphi^o}(x+\sqrt{-1}y)$ in 
${\rm Gr}_2(\Lambda_{2d}\otimes \R)$ is a positive semi-definite $2$-plane with exactly one isotropic direction. 
In the meantime, \cite{FriS} showed $\delta^2>0$ where $e' \perp \delta$. Hence, $e'^2=0$ must hold (so that we denote $e'$ simply by $e$ 
to follow the setup of \S\ref{trop.K3.1}). Therefore, \ref{Rel.FriS} follows. 

The assertion \ref{KS.GW.conj} and the last statement follow from Theorem~\ref{K3.Main.Conjecture.18.ok}, 
combined with the first assertion \ref{1par.K3.mon}. 
\end{proof}

Recall the analog for abelian varieties case holds as Theorem~\ref{1par.AV}. 
We also expect the compact hyperK\"ahler varieties version 
should follow from the same line of proof, 
once we settle some technical difficulties (see \S\ref{high.dim.HK.sec}). 

\medskip

Moreover, recall from Remark~\ref{metric.class} that the limit data $N$ or $\delta$ can be regarded as the 
tropical analog of period (linearized version of the metric class) for the limit tropical K3 surface. The following 
remark is in a similar direction. 

\begin{Rem}
Recall from Remark~\ref{metric.class} that the data parametrized at the ($18$-dimensional) boundary components 
are interpreted as their metric classes of the limit Monge-Amp\`ere manifolds with singularities, 
or equivalently to their radiance obstructions of their Legendre duals. 
When we presented some outline of our works, 
the authors also learned from Kazushi Ueda that his student 
\textit{Yuto Yamamoto} had had some then-ongoing interesting work \cite{Yam} 
which seems to be related to our works, 
where he constructs a sphere with an integral affine structure 
from the tropicalization of any given smooth anticanonical hypersurface  
in a toric Fano 3-fold, and computes its radiance obstruction. 
There is a related result for more general Calabi-Yau weighted hypersurfaces as \cite{Iritani, AGIS}. 
We thank them for the kind explanations. 
\end{Rem}

\begin{Rem}
This is a note added in our revision. 
Five months after our work appeared on arXiv, an interesting and closely related paper  \cite{AET} on an explicit 
algebro-geometric compactification of $\mathcal{F}_{2} (d=1)$ but with tropical geometric background appeared. 
In particular, what the authors call the monodromy theorem, 
\cite[Theorem 8.38]{AET}, is very closely related to our work above. 
Indeed, starting from Type III (maximal) degeneration of polarized K3 surfaces (data (i) of {\it op.cit.}), 
above Corollary~\ref{1par.lim.K3.} combined with our construction of the Lagrangian fibration $\Phi_{\rm alg}([e,v])$ in \S\ref{trop.K3.1} (giving data (ii) of {\it op.cit.})~satisfy their statements, and we see a connection between them indirectly through hyperK\"ahler rotation (Construction~\ref{HK.rotation.}) in \S\ref{Alg.K3.statements.sec}. 
We thank Valery Alexeev and Philip Engel for the discussions. 
\end{Rem}


\chapter{Uniform adiabatic limits of the metrized K3 surfaces}\label{GTZ.extend.proof}

\section{Overview of our analysis}

We now focus on fairly analytic contents, 
i.e., the proof of Theorem~\ref{GTZ.extend} which requires 
some refinements of known  
a priori estimates results for the solutions of the 
complex Monge-Amp\`ere equations. 

Suppose we are in the setting of Theorem~\ref{GTZ.extend}.
First of all, if $s\in S$ is fixed and $X_s$ is smooth, then the Gromov-Hausdorff convergence
 $(X_s, \tilde{\omega}_{s, t}) \to (B_s, g_s)$ as $t\to0$ follows from 
\cite[Theorem~1.2]{GTZ1} and \cite[Theorem~1.1]{GTZ2}. 
Even if $X_{s}$ is ADE singular, if $s\in S$ is fixed then 
we only need to slightly modify the proof of \cite[Theorem~1.2]{GTZ1} and \cite[Theorem~1.1]{GTZ2} in the orbi-setting. 

The essential technical difficulties we face for the full proof of Theorem~
\ref{GTZ.extend}  are twofolds --- 
the necessity of the uniformity of convergence with respect to $s$, 
 and the fact that we allow (degenerations to) ADE singular $X_{s}$. 
 For these reasons, our theorem \ref{GTZ.extend} 
 does \textit{not} directly follow from 
 \cite{Tos, GTZ1, GTZ2, TZ} and its proofs. 

Nevertheless, we of course build up our discussion heavily on 
the arguments in \cite{Tos, GTZ1, GTZ2, TZ} 
 and make various estimates uniform with respect to $s$ one by one. 
Before going to the details, here we explain some necessary modifications we made 
by showing examples. 

Firstly, we use two kinds of (families of) reference metrics: 
one is Ricci-flat K\"ahler metric 
 and the other is constructed in \S\ref{Grauert.singular.metric}, 
 in order to get various uniform estimates
 while allowing \textit{degeneration to ADE singular K3 surfaces}. 
 The main point of this new reference metric construction is 
 to get a uniform bound of \textit{holomorphic bisectional curvatures}
 in Claim \ref{bisec.bded}. 
 Indeed, lack of such uniform upper bound of bisectional curvatures of 
 Ricci-flat K\"ahler metrics (while degenerating to ADE singularities 
 cf.\ \cite{BKN, And,Kro89a, Kob90} etc.) would prevent us from 
 directly applying the standard $C^{2}$-estimate method with 
 the Chern-Lu inequality in order to follow arguments in \cite{Tos}. 

Secondly, possibly due to lack of the 
 best language of talking about a family of metrized complex analytic spaces 
 over whole $K\Omega$, 
 various other estimates in \cite{Tos} become nontrivial when $s$ varies 
 (e.g., in \S\ref{H.upper}, the Kronheimer family of 
 ALE metrics are used \cite{Kro89a, Kro89b} 
 to get a continuous control around degeneration 
 to ADE singularities). 
 
Many parts of our arguments below use the assumption $\dim X_s=2$. 
For example, the fiberwise $L^{\infty}$-estimate in \S\ref{(3.9)} 
uses the fact that $\dim \pi_{s}^{-1}(y)=1$ which makes the 
Monge-Amp\`ere equation on the fiber linear.
Also, we prove the Gromov-Hausdorff convergence by 
estimating the distorsions (cf.\ \cite{BBI}) 
of $\pi_{s}\colon X_{s}\twoheadrightarrow B_{s}$
 and their $C^{\infty}$-section maps as in \cite[\S6]{GW}, 
while \cite{GTZ1,GTZ2,TZ} analyze the map ``$\phi$'' 
between the base and a priori 
Gromov-Hausdorff limits by using the Fukaya-Cheeger-Colding limit measure 
\cite{Fuk}, \cite{ChCo}. 
That is, we replace the arguments for Gromov-Hausdorff convergence of \cite[\S5]{GTZ1}
 by some elementary arguments based on the Bishop-Gromov
 inequality, following \cite[\S6]{GW}. This discussion also 
 uses that $\dim(B_{s})=1$. 
Also some other arguments necessary for our purpose 
 do not directly follow from \cite{Tos, GTZ1, GTZ2}, 
 so we also complete such parts. 

Rather than repeating the whole arguments of \cite{Tos, GTZ1}, 
 we only explain the points where we need some changes. 
Below, ``$s$-uniform'' means the uniformity (independence) of constant for $s$,
 which is our key word. 

\section{Setting}\label{Setting}

Let $\{\pi_{s}\colon 
(X_s,\omega_{X_s})\to B_s\}_{s\in S}$ be a collection of elliptic K3 surfaces, 
possibly with ADE singularities, as in Theorem~\ref{GTZ.extend}. The reference metric $\omega_{X_s}$ is the Ricci-flat K\"ahler metric in the 
K\"ahler class specified by $s$. 
We set $$S_{\rm ADE}:=\{s\in S\mid X_s \text{ has ADE singularities}\},$$ 
which is equal to $S\cap (K\Omega\setminus K\Omega^{o})$ (cf.\ \cite{KT}). 

Since $S$ is relatively compact, it is enough to prove the theorem 
 for a small neighborhood of each $s_0\in S$.
Fix $s_0 \in S$ and we want to make a family of $X_s$ for $s$ near $s_0$.
For a (possibly) ADE singular K3 surface $X_s$, let $X_s^o$ be the smooth part of $X_s$
 so that $X_s\setminus X_s^o$ is the set of finite singular points.
If $s\not\in S_{\rm ADE}$, then $X_s^o=X_s$.
The collection $\{X_s^o\}_{s\in S}$ forms a real analytic family on $S$ with
 a family of the restricted Ricci-flat K\"{a}hler form $\omega_{X_s}|_{X_{s}^{o}}$  (see \cite{KT}).
For each $s\in S$ the K\"{a}hler form $\omega_{X_s}|_{X_s^o}$ extends to
 $X_s$ as an orbifold K\"{a}hler form $\omega_{X_s}$. 

We will also require the following holomorphic family of K3 surfaces.
Let $S'={\rm pr}(S)$ be the image of $S$ by the natural projection
 ${\rm pr}:K\Omega \to \Omega(\Lambda_{\rm K3})$.
Then $S'$ is an open subset of $\Omega^e(\Lambda_{\rm K3})$.
Let $s_0=[p_0,\kappa_0]$ so that ${\rm pr}(s_0)=p_0 \in S'$.
Recall from the discussion above Theorem~\ref{GTZ.extend} 
 that the K\"{a}hler cone $K_{p_0}$ has a stratification $K_{p_0}^I$.
If we choose an open stratum $K_{p_0}^I$ such that $\overline{K_{p_0}^I}\ni \kappa$,
 then (replacing $S$ and $S'$ by smaller subsets if necessary)
 we obtain a holomorphic family $\tilde{\mathcal{X}}\to S'$, where
 $\tilde{\mathcal{X}}=\bigcup_{p\in S'}X_p$ and 
 $X_p$ are marked smooth K3 surfaces with
 period $[\Omega_{X_p}]=p$. 
For our $s_{0}=[p_{0},\kappa_{0}]\in K\Omega$, 
in $X_{{\rm pr}(s_{0})}$, we consider all $(-2)$-curves corresponding 
to the root hyperplanes passing through $\kappa_{0}$. We can extend this contraction 
to simultaneous contraction to $X_{p}$ for all $p\in S'$ as a 
contraction of $\tilde{\mathcal{X}}$ by \cite{Riem}. 
In this way, we obtain a deformation family $\mathcal{X}\to S'$
 of the (possibly) ADE singular K3 surface $X_{s_0}$.
Our assumption $S\subset K\Omega^{e\geq 0}$ implies that
 this family has a structure of elliptic fibrations $\mathcal{X}\to \mathbb{P}^1$, by Fact \ref{e.nef}. 
Here, we shrink $S$ again if necessary and trivialize $B_{s}$ to be $\mathbb{P}^{1}$. 

For a positive real number $t$, let $\omega_{s,t}:=\pi_s^*\omega_{B_s}+t\omega_{X_s}$.
We set up notation of the complex Monge-Amp\`ere equations in our concern: 
\[\omega_{s,t}+\sqrt{-1} \partial\bar{\partial}\varphi_{s,t}=\tilde{\omega}_{s,t}\]
where $\tilde{\omega}_{s,t}$ is a Ricci-flat K\"ahler form and $\sup_{X_{s}}\varphi_{s,t}=0$.
We may take the same K\"ahler metric $\omega_{B_s}$ for all $s\in S$
 as a reference metric on $B_s$
 under the trivialization $B_s\simeq \mathbb{P}^1$. 
 
The discriminant locus of $\pi_{s}$ is denoted by ${\rm disc}(\pi_{s})$ 
 and $X_{s}\setminus \pi_{s}^{-1}({\rm disc}(\pi_{s}))$, i.e., the union of 
 smooth proper fibers, is denoted by $X_{s}^{\rm sm}$. 
The superscript ``sm'' stands for the relative smoothness. 
 
Also recall from Theorem \ref{GTZ.extend} 
 that the McLean metric on the base $B_{s}$ is 
 denoted by $g_{s,{\rm ML}}$ and the corresponding K\"ahler form 
 on $B_{s}\setminus {\rm disc}(\pi_s)$ is denoted by $\omega_{s,\rm{ML}}$. 


\section{A priori estimates} 

\subsection{$s$-uniform diameters upper bound of $\tilde{\omega}_{s,t}$} \label{Tos.lim.uni}

It is straightforward from the proof therein, that the diameter bound obtained in 
\cite[Theorem 3.1]{Tos.lim} can be taken uniformly with respect to 
bounded variations of ``$(X,\omega_{0})$'' (in the notation of \cite{Tos.lim}). 
We apply this to get uniform upper bound of the diameters of
 $\tilde{\omega}_{s,t}$ (and thus those of their Gromov-Hausdorff limits). 
We refer to Lemma~\ref{diam.Kahler} for a similar discussion in more details.


\subsection{Upper bound for the $H$-type function}\label{H.upper}

We will find an $s$-uniform upper bound of the following function on $X_s$:
\begin{equation*}
H_{s}:=\dfrac{\omega_{X_s}\wedge \pi_s^{*}\omega_{B_s}}{\omega_{X_s}^2}
= {\rm tr}_{\omega_{X_s}} \pi_s^{*}\omega_{B_s}
\end{equation*}
This part is something trivial and not discussed in 
\cite{Tos}, but becomes nontrivial and new in our 
situation where $s$ varies. 
First, it is easy to see that the above function $H_{s}$ 
forms a continuous function $H$ on 
 the family $\bigcup_{s\in S} X_s^o$. 
Therefore, it is enough to prove that $H$ is bounded on a neighborhood
 of each singular point of $X_{s_0}$ for $s_0=[{\rm pr}(s_0),\kappa_0]\in S$.
This will be carried out by using Kronheimer's family of ALE spaces.

Let $s_0=[{\rm pr}(s_0), \kappa_0]\in S$ and set 
\[\Delta_{s_0}:=\{\delta\in\Lambda_{\rm K3}\mid (\delta,\delta)=-2,\ 
 \delta\perp \langle {\rm pr}(s_0), \kappa_0 \rangle\}.\]  
Then $\Delta_{s_0}$ is finite and forms a root system of type $A$, $D$ or $E$. 
Suppose that $S$ is a sufficiently small neighborhood of $s_0$
 so that if $s=[{\rm pr}(s), \kappa]\in S$
 and $\delta\in \Lambda_{\rm K3}$, $(\delta,\delta)=-2$, and
 $\delta\perp \langle {\rm pr}(s), \kappa \rangle$,
 then $\delta\in \Delta_{s_0}$.
Recall that we have a family $\mathcal{X} \to S'$ of ADE K3 surfaces
 and its simultaneous desingularization $\tilde{\mathcal{X}} \to S'$
 corresponding to an open stratum $K_{p_0}^I$ such that
 $\overline{K_{p_0}^I}\ni \kappa_0$.
We choose such a stratum $K_{p_0}^I$, which corresponds to
 the choice of positive roots $\Delta_{s_0}^+$ in $\Delta_{s_0}$.
Define an open subset $S_+$ of $S$ as
\[S_+:=\{s=[p,\kappa]\in S \mid (\kappa,\delta)>0
 \text{ for } \delta \in \Delta_{s_0}^+\},\]
 corresponding to one chamber of ${\rm pr}^{-1}({\rm pr}(s_{0}))\cap 
 K\Omega^{o}$. 
Since there are only finitely many choices for $\Delta_{s_0}^+$,
 it is enough to show that $H$ is bounded on $\bigcup_{s\in S_+} X_s^o$.

For simplicity we consider the case where $X_{s_0}$ has only one singular point $x_0$.
The general case can be treated in a similar way.
Suppose that the germ at $x_0$ is
 $\mathbb{C}^{2}/\Gamma$ with
 $\Gamma \subset SL(2,\mathbb{C})$.
Consider the corresponding Kronheimer family of
 ALE K\"ahler orbifolds (\cite{Kro89a}, \cite{Kro89b}). 
Slightly changing the notation of \textit{op.cit.}, 
 we denote the family by
 $\mathcal{Y}=\bigcup_{\zeta\in Z\otimes \mathbb{R}^{3}}
 Y_{\zeta}\to Z\otimes \mathbb{R}^{3}$, 
 where $Z$ is the center of the Lie algebra
 for the Hamiltonian action in \cite{Kro89a}. 
Also $Z$ can be identified with the underlying vector space for the root
 system $\Delta_{s_0}$.
For each $\zeta=(\zeta_{1},\zeta_{2},\zeta_{3})\in Z\otimes \mathbb{R}^{3}$,
 $Y_{\zeta}$ is an ALE space with K\"{a}hler form $\omega_{Y_{\zeta}}$ and
 for $\zeta=(0,0,0)$, $Y_{0}\simeq \mathbb{C}^2/\Gamma$.
Since $x_0\in X_{s_0}$ and $0\in Y_{0}$ has the same singularity, 
 we obtain a local isomorphism between $x_0\in X_{s_0}$ and $0\in Y_{0}$.
The K\"{a}hler form on $Y_{0}\simeq \mathbb{C}^2/\Gamma$ is a flat metric
 on the covering $\mathbb{C}^2$ so we may write
 $\omega_{Y_{0}}
 = \sqrt{-1}(dz_1\wedge d\bar{z}_1 + dz_2\wedge d\bar{z}_2)$
 with a flat coordinate $z_1,z_2$ on $\mathbb{C}^2$.
Since $\omega_{X_{s_0}}$ is an orbifold K\"{a}hler form, 
 it can be locally written
 as $\sqrt{-1}f(z_1,z_2)(dz_1\wedge d\bar{z}_1 + dz_2\wedge d\bar{z}_2)$ 
 for a real analytic function $f(z_1,z_2)$.
We may assume $f(0,0)=1$ by replacing
 the local isomorphism between $X_{s_0}$ and $Y_{0}$.

The subset $\bigcup_{\zeta_{2},\zeta_{3}\in Z}
 Y_{(0,\zeta_{2},\zeta_{3})}$ of $\mathcal{Y}$ forms a holomorphic flat family 
$\mathcal{Y}_0\to Z\otimes \mathbb{C}$. 
Actually, $\mathcal{Y}_0$ is an affine variety mapping algebraically 
to $Z\otimes \mathbb{C}$ by the Kemp-Ness theorem and 
the general theory of the geometric invariant 
theory \cite{Mum65}. 
If $\zeta_{1}\in Z$ is regular dominant with respect to $\Delta_{s_0}^+$, then
 $\mathcal{Y}_{\zeta_1}:=\bigcup_{\zeta_{2},\zeta_{3}\in Z} Y_{(\zeta_{1},\zeta_{2},\zeta_{3})}
 \to \mathcal{Y}_{0}$ is a simultaneous desingularization.
We note that $\mathcal{Y}_{\zeta_1}$ for various regular dominant $\zeta_1$
 are canonically isomorphic to each other.
Moreover, this family
 $\mathcal{Y}_{\zeta_1}\to \mathcal{Y}_{0}\to Z\otimes \mathbb{C}$
 has the semi-universality for such deformations by \cite[Corollary 5.3]{Hui}.
Therefore, we get a map $q'\colon S'\to Z\otimes \mathbb{C}$ such that
 the family $\mathcal{X}$ is isomorphic to 
 the pull-back $S' \times_{Z\otimes \mathbb{C}} \mathcal{Y}_{0}$
 on a neighborhood of $x$ which extends the local isomorphism
 between $X_{s_0}$ and $Y_{0}$.
Also we have a local isomorphism between its resolution $\tilde{\mathcal{X}}$
 and the pull-back $S' \times_{Z\otimes \mathbb{C}} \mathcal{Y}_{\zeta_1}$.
This last isomorphism is defined on a neighborhood of $V$
 if $V$ denotes the fiber of $x_0$ for the map
 $\rho:\tilde{\mathcal{X}}\to \mathcal{X}$,
 which is a union of $(-2)$-curves.

Define a smooth map $q:S_+\to Z\otimes \mathbb{R}^3$ by
 $q([p,\kappa])=(\zeta_1,\zeta_2,\zeta_3)$ with
 $\zeta_2+\sqrt{-1}\zeta_3=q'(p)$ and $(\zeta_1,\delta)=(\kappa,\delta)$
 for $\delta\in \Delta_{s_0}$.
Then we have a local isomorphism between
 $S_+ \times_{S'} \tilde{\mathcal{X}}$
 and $S_+ \times_{Z\otimes \mathbb{R}^3} \mathcal{Y}$,
 which is defined in a neighborhood of $V$.
Now we want to glue two K\"{a}hler forms
 $\omega_{X_s}$ and $\omega_{Y_{q(s)}}$
 to get another K\"{a}hler form $\omega_{s,{\rm gl}}$ on $X_s$.
We first consider the gluing of $\omega_{X_{s_0}}$ and $\omega_{Y_{0}}$.
On a neighborhood of $x_0$, $X_{s_0}$ can be identified with that of 
 $Y_0 \simeq \mathbb{C}^2/\Gamma$.
Recall that the two forms can be written on $\mathbb{C}^2$ as
\begin{align*}
&\omega_{X_{s_0}}=\sqrt{-1}f(z_1,z_2)(dz_1\wedge d\bar{z}_1 + dz_2\wedge d\bar{z}_2),\\
&\omega_{Y_0}=\sqrt{-1}(dz_1\wedge d\bar{z}_1 + dz_2\wedge d\bar{z}_2).
\end{align*}
A potential function of $\omega_{Y_0}$ can be chosen as $r^2$,
 where $r:=(|z_1|^2+|z_2|^2)^{\frac{1}{2}}$. 
Since we have assumed $f(0,0)=1$,
 a potential function $u$ of $\omega_{X_{s_0}}$
 can be chosen as $r^2+O(r^3)$. 
 Take a bump function $\beta(t)$ on $\mathbb{R}_{\geq 0}$ which is a smooth function and 
 takes value $1$ for $|t|<c$ and $0$ on $|t|>2c$.
If $c$ is small enough, then the function
 $u_{\rm gl}=r^2 \beta(c^{-1}r) + u (1-\beta(c^{-1}r))$ becomes pluri-subharmonic.
Therefore, we define a K\"{a}hler form $\omega_{s_0,{\rm gl}}$ by
\begin{align*}
\omega_{s_0,{\rm gl}}=\begin{cases}
\omega_{X_{s_0}} &\text{ on } X_{s_0}\setminus\{|r|\leq 2c\},\\
\sqrt{-1}\partial\bar{\partial} u_{\rm gl} &\text{ on } \{c<|r|<2c\},\\
\omega_{Y_0} &\text{ on } \{|r|\leq c\}.\\
\end{cases}
\end{align*}
Let $U:=\{c<|r|<2c\}$. Then $\omega_{s_0,{\rm gl}}$ is Ricci-flat outside $U$.
We will now glue K\"{a}hler forms on $X_s$ for $s\in S_+$.
Let $\mathcal{U}'=\bigcup_{p\in S'} U_p \simeq U\times S'$
 be a neighborhood of $U$ in $\mathcal{X}$
 and let $\mathcal{U} = S_+ \times_{S'} \rho^{-1}(\mathcal{U}')$
 be its pull-back so $\mathcal{U}\subset S_+\times_{S'} \tilde{\mathcal{X}}$.
Then we may write $\mathcal{U}= \bigcup_{s\in S_+} U_s$.
Since the two K\"{a}hler forms  $\omega_{X_s}$ and $\omega_{Y_{q(s)}}$
 on $U_s$ vary smoothly when $s$ moves $s_0$ to $S_+$, 
 the potential functions for them on $U_s$ also vary smoothly.
Therefore, shrinking $S$ (and then $S_+$) if necessary, we may glue 
 two forms $\omega_{X_s}$ and $\omega_{Y_{q(s)}}$ on $U_s$ by a bump function
 as above to get a K\"{a}hler form $\omega_{s,{\rm gl}}$ on $X_s$
 for $s\in S_+$.
We note that $\omega_{s,{\rm gl}}$ are Ricci-flat outside $U_s$ so
 the Ricci curvature of $\omega_{s,{\rm gl}}$ is uniformly bounded.

Since we have defined $q$ in such a way that $\kappa$ corresponds to $\zeta_1$,
 we can see that $\omega_{X_s}$ and $\omega_{s,{\rm gl}}$ are cohomologous.
Hence there exists a function $\psi_s$ such that
 $\omega_{X_s}=\omega_{s,{\rm gl}}+\sqrt{-1}\partial \bar{\partial}\psi_s$
 and $\sup_{X_s} \psi_s=0$.
We claim that 
\begin{Claim}\label{C0.RFK.gl}
$\|\psi_s\|_{L^\infty}$ is uniformly bounded for $s\in S_+$. 
\end{Claim} 
This a priori $C^0$-estimate for the Monge-Amp\`ere equation 
follows from a standard Moser iteration method as in \cite{Yau}. 
See \cite[\S3.4]{Sze} or our \S\ref{unif.L.infinity} for a similar discussion with more details. 
Indeed, the input we need here is that both 
 the diameter and the Ricci curvature 
 of $(X_s,\omega_{s,\rm gl})$ are uniformly bounded
 and also the total volume of $(X_s,\omega_{s,\rm gl})$ has a uniform positive lower bound. 
This implies that the Sobolev constant and the Poincar\'e constant
 are also uniformly bounded. 
 
Once we have Claim~\ref{C0.RFK.gl},
 the bound of ${\rm tr}_{\omega_{X_s}}\pi_s^*\omega_{B_s}$
 is reduced to the bound of ${\rm tr}_{\omega_{s,{\rm gl}}}\pi_s^*\omega_{B_s}$. 
Indeed, the method of the proof of \cite[Lemma 3.1]{Tos} works.
By the Chern-Lu formula, we have
\[\Delta_{\omega_{X_{s}}}
 \log {\rm tr}_{\omega_{X_{s}}}\pi_{s}^{*}\omega_{B_{s}}
 \geq -A{\rm tr}_{\omega_{X_{s}}}\pi_{s}^{*}\omega_{B_{s}}\]
for a constant $A$ which does not depend on $s$.
If we have a bound
 ${\rm tr}_{\omega_{s,{\rm gl}}}\pi_s^*\omega_{B_s}\leq C$, then 
\[
\Delta_{\omega_{X_{s}}} \psi_s
 = 2 - {\rm tr}_{\omega_{X_{s}}} \omega_{s,{\rm gl}}
 \leq 2 - C^{-1} {\rm tr}_{\omega_{X_{s}}} \pi_s^*\omega_{B_s}
\]
and we get
\[
\Delta_{\omega_{X_s}}
(\log{\rm tr}_{\omega_{X_{s}}}\pi_{s}^{*}\omega_{B_{s}}-C(A+1)\psi_{s})
\geq {\rm tr}_{\omega_{X_{s}}}\pi_{s}^{*}\omega_{B_{s}}-2C(A+1).
\]
By applying the maximum principle to the function  
$\log{\rm tr}_{\omega_{X_{s}}}\pi_{s}^{*}\omega_{B_{s}}-C(A+1)\psi_{s}$
 and by Claim~\ref{C0.RFK.gl},
 we obtain a bound of ${\rm tr}_{\omega_{X_s}}\pi_s^*\omega_{B_s}$.

Moreover, since $\omega_{s,{\rm gl}}=\omega_{Y_{q(s)}}$ holds near $S_+ \times_{S'} V$, 
 it is enough to prove that 
 ${\rm tr}_{\omega_{Y_{q(s)}}}\pi_s^*\omega_{B_s}$ is bounded above
 near $S_+ \times_{S'} V$. 

The elliptic fibrations $\pi_s:X_s \to B_s$ for $s\in S_+$ form
 $S_+ \times_{S'} \tilde{\mathcal{X}} \to \mathbb{P}^1$,
 where $B_s$ are trivialized to be $\mathbb{P}^1$.
This map factors as $S_+ \times_{S'} \tilde{\mathcal{X}}\to 
 S_+ \times_{S'} \mathcal{X} \to \mathbb{P}^1$,
 where the first map is the pull-back of $\tilde{\mathcal{X}}\to \mathcal{X}$.
Then via the local isomorphism, this map is identified with
 $S_+ \times_{Z\otimes \mathbb{R}^3} \mathcal{Y}\to 
 S_+ \times_{Z\otimes \mathbb{C}} \mathcal{Y}_0 \to \mathbb{P}^1$.
We want to estimate the function ${\rm tr}_{\omega_{Y_{q(s)}}} \pi_s^{*}\omega_{B_s}$.
The value of this function at $x\in Y_{q(s)}$
 is given as $\sup_{\xi} \frac{\|(\pi_s)_*\xi\|}{\|\xi\|}$,
 where $\xi$ runs over the nonzero tangent vectors of $Y_{q(s)}$ at $x$
 and the lengths are defined with respect to $\omega_{Y_{q(s)}}$ and $\omega_{B_s}$.
Recall that $\mathcal{Y}$ is defined as a quotient of the vector space ${\bf M}$ by a compact group
 and we can take a horizontal lift $\tilde{\xi}$ of $\xi$ as a tangent vector in ${\bf M}$. 
The length of $\tilde{\xi}$ with respect to the hermitian metric equals $\|\xi\|$.
Since $\varpi: S_+ \times_{Z\otimes \mathbb{R}^3} {\bf M} \to \mathbb{P}^1$ is a smooth map,
 the function $\sup_{\eta} \frac{\|(\varpi)_*\eta\|}{\|\eta\|}$
 where $\eta$ runs over the nonzero tangent vectors of ${\bf M}$ at a point is continuous.
Therefore, we conclude that the function
 $\sup_{\xi} \frac{\|(\pi_s)_*\xi\|}{\|\xi\|}$ is also bounded 
 near $S_+\times_{S'} V$, which proves the desired uniform upper bound of $H_s$.


\subsection{Another reference metric} \label{Grauert.singular.metric}

We now construct useful reference metrics $\omega_{X_s, {\rm new}}$ to 
sometimes replace the 
Ricci-flat K\"ahler reference metric $\omega_{X_s}$
 in the same K\"ahler classes. 
In the following discussions, we will use both $\omega_{X_{s}}$ and 
$\omega_{X_{s},{\rm new}}$. 

Recall that we have the natural map ${\rm pr}\colon K\Omega
 \twoheadrightarrow \Omega(\Lambda_{\rm K3}).$ 
To start our construction,
 let us take $s_0=[{\rm pr}(s_0), \kappa_0]\in S$ and define $\Delta_{s_0}$
 as in the previous subsection. 

We will construct a family of K\"ahler forms
 $\omega_{X_{s},{\rm new}}$ on $X_s$ for $s$ near $s_0$.
Take $\kappa_i\in \Lambda_{\rm K3}\otimes \R\ (i=1,\dots,20)$
 which satisfy the following conditions:
\begin{itemize}
\item $\kappa_i$ are orthogonal to ${\rm pr}(s_0)$, 
\item $[{\rm pr}(s_0),\kappa_i]\in S$, 
\item $(\kappa_i,\delta) \neq 0$ for any $1\leq i\leq 20$ and $\delta\in \Delta_{s_0}$, 
\item $\kappa_0=\sum_{i=1}^{20}a_i\kappa_i$
 for some positive real numbers $a_i$.
\end{itemize}
This is possible because the subspace of $\Lambda_{\rm K3}\otimes \R$
 orthogonal to ${\rm pr}(s_0)$ has real dimension $20$. 
We then take a small neighborhood $S'$ of ${\rm pr}(s_0)$
 in $\Omega(\Lambda_{\rm K3})$
 and take local sections $p\mapsto [p,\psi_i(p)]$ of
 $K\Omega^{e\ge 0}\twoheadrightarrow \Omega^{e}(\Lambda_{\rm K3})$ on $S'$
 such that $\psi_i({\rm pr}(s_0))=\kappa_i$ for $1\leq i\leq 20$.

For each $i$, we have a holomorphic flat family of smooth K3
 surfaces $\{(X_{p,i},\omega_{p,i})\mid p\in S'\}$ 
 corresponding to $\{[p,\psi_i(p)]\}_{p\in S'}$, 
 associated with the differentiable family $\{\omega_{p,i}\}$ 
of their Ricci-flat-K\"ahler metrics. 
If $S'$ is sufficiently small,
 then $[p,\psi_i(p)]\in K\Omega^o$
 and hence $\{(X_{p,i},\omega_{p,i})\mid p\in S'\}$
 form a smooth proper holomorphic family,
 which we denote by $\pi_{i}\colon \tilde{\mathcal{X}}_{i}\twoheadrightarrow S'$. 

We set 
\[S'_{\rm ADE}:=
 \{p\in S'\mid p \perp \delta \text{ for some } \delta\in \Delta_{s_0}\}\subset S',\]
which is a union of Heegner divisors. 
For $p\in S'\setminus S'_{\rm ADE}$,
  all $X_{p,i}$ are smooth and
 canonically isomorphic with each other.
They are also isomorphic to $X_s$ if ${\rm pr}(s)=p$.
Moreover, the families $\tilde{\mathcal{X}}_{i}$ are all
 isomorphic on $S'\setminus S'_{\rm ADE}$. 
Across $p\in S'_{\rm ADE}$, the isomorphism does not necessarily extend
 due to the effect of flops,
 while abstract biholomorphic classes of $X_{p,i}$ are the same with different markings. 
For the details of its proof, we refer to \cite{BR}, 
\cite[\S4]{Fjk} (also \cite{MM}) for instance. 

We define $X_{p}^{o}:=X_{p}$ if $p\not\in S'_{\rm ADE}$ and
 define $X_{p}^{o}:=X_{s}^{\rm sm}$ if ${\rm pr}(s)=p \in S'_{\rm ADE}$,
 which is well-defined as it does not depend on the choice of $s$. 
Then $\{X_{p}^{o}\}$ naturally form a (non-proper) holomorphic family over $S'$. 
In both cases, $X_{p}^o$ can be naturally regarded as an open
 subset of $X_{p}$ and also of $X_s$ if ${\rm pr}(s)=p$. 

We replace $S$ by a smaller set if necessary so that
 any $s=[p, \kappa]\in S$ can be written as 
 $\kappa=\sum_{i=1}^{20} x_i \psi_i(p)$ with $c<x_i<C$,
 where $c$ and $C$ are positive constants.
Then we define a new reference metric on $X_p^o$ by
\[\omega_{X_s,{\rm new}}:=\sum_{i=1}^{20} x_i \omega_{p,i}\]
 and the smooth family of  non-compact K\"ahler manifolds
$\{(X_{p}^{o}, \omega_{X_s,{\rm new}})\}_{s=[p,\kappa]}$
on $S$.
Later in \S\ref{Thm2.2.}, we make use of the following key property. 
\begin{Claim}\label{bisec.bded}
The holomorphic bisectional curvatures of $\omega_{{X_s},{\rm new}}$ 
are $s$-uniformly bounded above. 
\end{Claim}
\begin{proof}
The 
holomorphic bisectional curvatures of $\omega_{p,i}$ for all $i$ 
are $s$-uniformly bounded above. 
Then apply \cite[\S4]{GK}
 to the diagonal embedding $X_{p}^o \to \prod_{i=1}^{20}X_{p,i}$.
Here, $\prod_{i=1}^{20}X_{p,i}$ is equipped with the product metric
 with the $i$-th factor $x_i\omega_{p,i}$.  
This gives the desired uniform bound. 
\end{proof}


\subsection{Lower bound for the $H$-type function} \label{2.4}

We show the $s$-uniform version of \cite[(2.4)]{Tos}
 for our new reference metric $\omega_{X_{s},{\rm new}}$
 defined in \S\ref{Grauert.singular.metric},
 namely, we will get a lower estimate of the function
\[H_{s,\rm new}:=
 \frac{\pi_s^*\omega_{B_s}\wedge
 \omega_{X_{s},{\rm new}}}{\omega_{X_{s},{\rm new}}^2}\]
 on $X_{{\rm pr}(s)}^o$
 by some power of a certain analytic function $\sigma$.
This will be done by the following steps and 
 will be used in \S\ref{Thm2.2.}
 for our extension of \cite[Theorem 2.2]{Tos}. 

\begin{Step}[$H$-function]
Let $\pi_i\colon\tilde{\mathcal{X}}_{i}\twoheadrightarrow S'$
 be as in \S\ref{Grauert.singular.metric}.
We consider the fiber product of them over $S'$ 
\[\tilde{\mathcal{X}}_{1}\times_{S'}\tilde{\mathcal{X}}_{2}
\times_{S'}\cdots\times_{S'}\tilde{\mathcal{X}}_{20}.\]
If we write $\mathcal{X}^0$ for a family
 $\bigcup_p X_p^o \to S'$, then
 there is a natural inclusion
 $\mathcal{X}^0 \subset \tilde{\mathcal{X}}_{i}$ for every $i$.
Inside the above fiber product, we consider the diagonal
 $\Delta_{\mathcal{X}^0}(\tilde{\mathcal{X}}_{i})$ as the image of 
$\mathcal{X}^0\to \tilde{\mathcal{X}}_{1}\times_{S'}\tilde{\mathcal{X}}_{2}
\times_{S'}\cdots\times_{S'}\tilde{\mathcal{X}}_{20}$ 
sending $x$ to $(x,x,\cdots,x)$. Then we take its closure in 
$\tilde{\mathcal{X}}_{1}\times_{S'}\tilde{\mathcal{X}}_{2}
\times_{S'}\cdots\times_{S'}\tilde{\mathcal{X}}_{20}$ 
and denote it by $\mathcal{Y}$. 
Further, we take the resolution of singularities of $\mathcal{Y}$ as 
 $\tilde{\pi}\colon \tilde{\mathcal{Y}}\to S'$. 
Write ${\rm pr}_{i}\colon \tilde{\mathcal{Y}}
 \twoheadrightarrow \mathcal{X}_{i}$
 for the natural $i$-th projection. 
This contains $\mathcal{X}^o$ as an open dense subset.

Take a real analytic volume form $\nu_{S'}$ on $S'$.
Recall that $\omega_{X_{s},{\rm new}}$ is a family of smooth K\"ahler
 forms on the smooth family
 $\tilde{\pi}\colon\mathcal{X}^o\to S'$,
 which depends on the numbers $x_1,\dots,x_{20}$.
Then $\pi_s^*\omega_{B_s}\wedge\omega_{X_{s},{\rm new}}
 \wedge \tilde{\pi}^*\nu_{S'}$
 and $\omega_{X_{s},{\rm new}}^2 \wedge \tilde{\pi}^*\nu_{S'}$
 define smooth volume forms on $\mathcal{X}^o$.
Hence we may write
\[H_{s,\rm new}=
 \frac{\pi_s^*\omega_{B_s}\wedge
 \omega_{X_{s},{\rm new}}\wedge \tilde{\pi}^*\nu_{S'}}
 {\omega_{X_{s},{\rm new}}^2 \wedge \tilde{\pi}^*\nu_{S'}},\]
which depends on $s\in S$, hence on $x_i$s. 

Take a finite open covering $\{U_j\}_j$ of $\tilde{\mathcal{Y}}$
 and take analytic K\"ahler forms $\omega_{U_j}$ on $U_j$ such that
 ${\rm pr}_{i}^*\omega_{p,i}|_{U_j\cap X_p^{o}}
 \leq \omega_{U_j}|_{U_j\cap X_p^{o}}$
 for $1\leq i\leq 20$ and $p\in \tilde{\pi}(U_j)$.
Since $\omega_{{X_s},{\rm new}}=\sum_{i=1}^{20} x_i\omega_{{\rm pr}(s),i}$,
 we have
 $\omega_{{X_s},{\rm new}} \leq \sum_{i=1}^{20} x_i \omega_{U_j}
 \leq 20C \omega_{U_j}$ on $U_j\cap X_p^{o}$.
Therefore, 
\[\omega_{X_{s},{\rm new}}^2 \wedge \tilde{\pi}^*\nu_{S'}
 \leq (20C)^2\omega_{U_j}^2 \wedge \tilde{\pi}^*\nu_{S'}\]
on $U_j\cap \mathcal{X}^{o}$.
On the other hand, we have
\begin{align*}
\pi_s^*\omega_{B_s}\wedge
 \omega_{X_{s},{\rm new}}\wedge \tilde{\pi}^*\nu_{S'}
&= \sum_{i=1}^{20}
 \pi_s^*\omega_{B_s}\wedge
 x_i {\rm pr}_i^* \omega_{{\rm pr}(s),i}
 \wedge \tilde{\pi}^*\nu_{S'}\\
&\geq c \pi_s^*\omega_{B_s}\wedge
 {\rm pr}_1^* \omega_{{\rm pr}(s),1}
 \wedge \tilde{\pi}^*\nu_{S'}.
\end{align*}
Hence 
\begin{align*}
H_{s,\rm new}\geq 
 \frac{c\pi_s^*\omega_{B_s}\wedge
 {\rm pr}_1^* \omega_{{\rm pr}(s),1} \wedge \tilde{\pi}^*\nu_{S'}}
 {(20C)^2\omega_{U_j}^2 \wedge \tilde{\pi}^*\nu_{S'}} =:H_j
\end{align*}
on $U_j\cap \mathcal{X}^{o}$.
We note that the function $H_j$ is a real analytic function
 defined on $U_j$ and is positive on $U_j\cap \mathcal{X}^{o}$, 
 whose  definition obviously does \textit{not} use $x_i$s.
This will be used in Step \ref{Step3.H.sigma}. 
\end{Step}

\begin{Step}[$\sigma$-function] 
On the other hand, for the function $\sigma$ of \cite[(2.2)]{Tos}, we consider as follows. 
First, we take a family of relative Jacobian K3 surfaces 
$\bigsqcup_{s}\overline{{\rm Jac}}(X_{s}/B_{s})\twoheadrightarrow 
\bigsqcup_{s} B_{s}$ 
and take its Weierstrass models. 
Note that the discriminant locus (image of critical locus) 
does not change while passing to Jacobian fibration and further to its 
Weierstrass model, so that discriminant locus on the base has 
no ambiguity. It is 
determined by the corresponding discriminant (holomorphic!) 
section $\Delta(:=g_{2}^{3}-27g_{3}^{2})$ 
of $H^0(\bigsqcup_{s}B_s, \bigsqcup_{s}
L^{\otimes 12})$ for the Weierstrass model 
(cf.\ \cite[\S1.4.1]{FMg}) as it vanishes exactly at the discriminant locus. In particular the discriminant locus is a divisor of $\bigsqcup_s B_s$. 
We fix a hermitian real analytic metric $h$ on $L$
 and define a function $\sigma$ to be $|\Delta|_{h}$, 
 which forms a real analytic family of real analytic functions on $B_{s}$. 
\end{Step}

\begin{Step}[{$H$-function vs.\ $\sigma$-function \cite[(2.4)]{Tos}}]\label{Step3.H.sigma}
Now we use {\L}ojasiewicz's theorem \cite{Loja59} to compare the 
$H$-type functions and $\sigma$-type function constructed in the 
previous steps. 
Recall that the function $H_j$ on $U_j$ is real analytic 
 and the zero set of $H_j$ is contained in
 $U_j\setminus \mathcal{X}^o$.

Take a positive symmetric tensor (Riemannian metric) $g_{B}$ 
 on the base $\bigcup_{s}B_{s}$ and
 write  
 $\tilde{\pi}^*g_B$ for its pull-back by
 $\tilde{\pi}\colon\tilde{\mathcal{Y}} \to \bigcup_s B_s$.
Also, take a metric $g_{\tilde{\mathcal{Y}}}$ on $\tilde{\mathcal{Y}}$.
By applying {\L}ojasiewicz's theorem (\cite{Loja59}),
 it follows that
\begin{align*}
C_1{\rm dist}
 (x,[H_j=0]; \tilde{\pi}^*g_{B} + g_{\tilde{\mathcal{Y}}})^{\lambda_1}\le H_j(x), 
\end{align*}
where ${\rm dist}(x,[H_j=0]; \tilde{\pi}^*g_{B} + g_{\tilde{\mathcal{Y}}})$ denotes
 the distance between $x$ and $[H_{j}=0]$ 
with respect to the metric  $\tilde{\pi}^{*}g_{B}+g_{\tilde{\mathcal{Y}}}$, 
 for $x\in U_j$ 
 with certain positive uniform constants $C_1$ and $\lambda_1$. 

On the other hand, since $\tilde{\pi}^{-1}([\Delta=0])\supset [H_j=0]$, 
 we have some positive constants $C_2$ and $\lambda_2$ such that 
\begin{align*}
C_2\sigma(\tilde{\pi}(x))^{\lambda_2}
&\leq {\rm dist}(\tilde{\pi}(x),[\Delta=0];g_{B})\\
&\leq {\rm dist}(x,[H_j=0];\tilde{\pi}^{*}g_{B})\\ 
&\leq {\rm dist}(x,[H_j=0];\tilde{\pi}^{*}g_{B}+g_{\tilde{\mathcal{Y}}}). 
\end{align*}
Combining above two inequalities, we get that with uniform positive 
constants $C_3$ and $\lambda_3$ such that 
\[C_3\sigma(y)^{\lambda_3}\le \inf_{X_{s,y}} H_{j},\]
where $y\in B_s\setminus {\rm disc}(\pi_s)$ and $X_{s,y}:=\pi_s^{-1}(y)$.
Therefore, combining with the estimate in Step 1, we get 
the desired uniform version of \cite[(2.4)]{Tos}:
\begin{align}\label{H.lower.bd}
C_3\sigma(y)^{\lambda_3}\le \inf_{X_{s,y}} H_{s,\rm new}.
\end{align}
\end{Step}

\subsection{Uniform $L^{\infty}$-bound on the total space} \label{unif.L.infinity}

Next we need to 
make the $L^{\infty}$-bound 
of the K\"ahler-Einstein potential $\varphi_{s,t}$ quoted at 
\cite[Theorem 2.1]{Tos} as theorems of Demailly-Pali~\cite{DP}
 and Eyssidieux-Guedj-Zeriahi~\cite{EGZ}. The original statement is uniform 
 only for $t$ and \textit{not} for $s$, but we need an $s$-uniform version. 
That is what we prove here. 

Our proof heavily depends on the method of Yau~\cite{Yau} and 
later developments, known as Moser iteration. 
We consulted \cite{Naka99} for learning it, which is reflected in the following arguments. 
(See also \cite[Lemma~5.3]{SoTi} for similar arguments in the context of K\"ahler-Ricci flow.)

%
In this subsection only, 
 we normalize $\varphi_{s,t}$ by adding a positive constant to
 satisfy 
\begin{align*}
\int_{X_s} \varphi_{s,t} \omega_{X_s}^2 = 0,
\end{align*}
which is different from
 the previous normalization $\sup_{X_s} \varphi_{s,t}=0$.
We claim that the inequalities 
\begin{equation}\label{vol.comp}
-tC_{1} \omega_{X_s}^{2} \leq 
 \tilde{\omega}_{s,t}^{2}-\omega_{s,t}^{2}\le tC_{2} \omega_{X_s}^{2}
\end{equation}
hold for positive constants $C_1$ and $C_2$.
Indeed, since $\tilde{\omega}_{s,t}$ is Ricci-flat, 
 $\tilde{\omega}_{s,t}^{2}=c_{s,t}\cdot t\omega_{X_s}^{2}$ 
 where a constant $c_{s,t}$ is uniformly bounded by positive 
 constants from above and below, which directly
 proves the right hand side of \eqref{vol.comp}.
The left hand side follows from
 $$\omega_{s,t}^{2}=2t\pi_s^*\omega_{B_s}\wedge \omega_{X_s}+ t^2 \omega_{X_s}^2$$
 and the uniform upper bound of $H_s$
 already proved in \S\ref{H.upper}.

For any $\alpha\geq 0$, \eqref{vol.comp} implies that 
\begin{align*}
t C \int_{X_s} |\varphi_{s,t}|^{\alpha+1} \omega_{X_s}^2
\geq 
-\int_{X_s} \varphi_{s,t} |\varphi_{s,t}|^{\alpha}
 (\tilde{\omega}_{s,t}^2-\omega_{s,t}^2),
\end{align*} 
where $C:=\max\{C_1,C_2\}$.
On the other hands, 
\begin{align*}
&-\int_{X_s} \varphi_{s,t} |\varphi_{s,t}|^{\alpha}
 (\tilde{\omega}_{s,t}^2-\omega_{s,t}^2) \\
&= -\int_{X_s} \varphi_{s,t} |\varphi_{s,t}|^{\alpha}
 \sqrt{-1}\partial\bar{\partial}\varphi_{s,t}
 \wedge (\tilde{\omega}_{s,t}+\omega_{s,t}) \\
&= -\int_{X_s} (\varphi_{s,t} |\varphi_{s,t}|^{\alpha}
 \sqrt{-1}\partial\bar{\partial}\varphi_{s,t}
 -\sqrt{-1}d(\varphi_{s,t}|\varphi_{s,t}|^{\alpha}\bar{\partial}\varphi_{s,t}))
 \wedge (\tilde{\omega}_{s,t}+\omega_{s,t})\\
&= \sqrt{-1}\int_{X_s} (\alpha+1) |\varphi_{s,t}|^{\alpha}
 \partial\varphi_{s,t} \wedge \bar{\partial} \varphi_{s,t}
 \wedge (\tilde{\omega}_{s,t}+\omega_{s,t}) \\
&= \int_{X_s} (\alpha+1) |\varphi_{s,t}|^{\alpha}
 |\partial\varphi_{s,t}|_{\omega_{X_s}}^{2} \omega_{X_s}
  \wedge (\tilde{\omega}_{s,t}+\omega_{s,t}) \\ 
&\geq 
t\int_{X_s} (\alpha+1) |\varphi_{s,t}|^{\alpha}
 |\partial\varphi_{s,t}|_{\omega_{X_s}}^2
 \omega_{X_s}^2.
\end{align*}
Therefore,
\begin{align}
\label{ineq:1}
C \int_{X_s} |\varphi_{s,t}|^{\alpha+1} \omega_{X_s}^2
&\geq \int_{X_s} (\alpha+1) |\varphi_{s,t}|^{\alpha} |\partial\varphi_{s,t}|^2
 \omega_{X_s}^2 \\ \nonumber
&\geq \frac{\alpha+1}{(\frac{\alpha}{2}+1)^2}
 \int_{X_s}  |\partial (\varphi_{s,t} |\varphi_{s,t}|^{\frac{\alpha}{2}})|^2
 \omega_{X_s}^2.
\end{align}

We abbreviate $\int_{X_s} \cdot\, \omega_{X_s}^2$
 to $\int_{X_s} \cdot$ for the remaining argument.
We will apply the Sobolev inequality 
\[ A_{1} \| \cdot \|_{W^{1,2}} \geq \| \cdot \|_{L^4}\]
 to $\varphi_{s,t} |\varphi_{s,t}|^{\frac{\alpha}{2}}$,
 where the Sobolev constant $A_{1}$ is known to be 
 determined only by the diameter and a lower bound of the Ricci curvatures.
Since $\omega_{X_s}$ is Ricci flat and the diameter of $(X_s,\omega_{X_s})$
 is uniformly bounded by \S\ref{Tos.lim.uni},
 we get
\begin{align*}
A_{1}^{2} 
\Bigl(\int_{X_s} |\partial (\varphi_{s,t} |\varphi_{s,t}|^{\frac{\alpha}{2}})|^2 
+ \int_{X_s} |\varphi_{s,t}|^{\alpha+2}\Bigr)
\geq 
\Bigl(\int_{X_s}  |\varphi_{s,t}|^{2(\alpha+2)}\Bigr)^{\frac{1}{2}}
\end{align*}
for a uniform constant $A_1$.
Combining with \eqref{ineq:1}, 
\begin{align*}
\Bigl(\int_{X_s}  |\varphi_{s,t}|^{2(\alpha+2)}\Bigr)^{\frac{1}{2}}
&\leq 
A_{1}^{2} 
\Bigl(\int_{X_s} |\partial (\varphi_{s,t} |\varphi_{s,t}|^{\frac{\alpha}{2}})|^2 
+ \int_{X_s} |\varphi_{s,t}|^{\alpha+2}\Bigr)\\
&\leq 
A_{2} \Bigl(
(\alpha+2)\int_{X_s}  |\varphi_{s,t}|^{\alpha+1}
 + \int_{X_s} |\varphi_{s,t}|^{\alpha+2}
 \Bigr)\\
&\leq 
A_{3} \Bigl(
(\alpha+2) \Bigl(\int_{X_s}  |\varphi_{s,t}|^{\alpha+2}
\Bigr)^{\frac{\alpha+1}{\alpha+2}}
 + \int_{X_s} |\varphi_t|^{\alpha+2}
 \Bigr)\\
&\leq 
A_{4} (\alpha+2) \max\Bigl\{1,\int_{X_s} |\varphi_{s,t}|^{\alpha+2}\Bigr\}.
\end{align*}
For the second inequality above, we used (\ref{ineq:1}) and 
this is one point of our arguments slightly different from \cite{Naka99}. 
Hence we get 
\begin{align}
\label{MI.1step}
\max\{1,\| \varphi_{s,t} \|_{L^{2(\alpha+2)}}\}
\leq (A_{4}(\alpha+2))^{\frac{1}{\alpha+2}} 
\max\{1,\| \varphi_{s,t} \|_{L^{\alpha+2}}\}.
\end{align}
Starting from $\alpha=0$ and utilizing this inequality (\ref{MI.1step}) iteratively 
for $\alpha=2^{k}-2 (k\in \mathbb{Z}_{>0})$, it holds that 
\begin{align*}
\max\{1,\| \varphi_{s,t} \|_{L^{2^{N+1}}}\}
\leq
\Bigl(\prod_{k=1}^N (A_{4}2^k)^{2^{-k}}\Bigr) 
\max\{1,\| \varphi_{s,t} \|_{L^2}\}.
\end{align*}
Since
\begin{align*}
\Bigl(\prod_{k=1}^{\infty} (A_{4}2^k)^{2^{-k}}\Bigr) \leq A_{5}
\end{align*}
for some constant $A_{5}$, we get
\begin{align*}
\sup |\varphi_{s,t}| \leq A_{5}\max\{1,\| \varphi_{s,t} \|_{L^2}\}.
\end{align*}

It suffices to prove $\| \varphi_{s,t} \|_{L^2}$
 is bounded above by a constant.
Since we normalized as $\int_{X_s} \varphi_{s,t}=0$.
 the Poincar\'{e} inequality \cite{LY} gives
\begin{align*}
\| \varphi_{s,t} \|_{L^2}
 \leq A_{6} \| \partial\varphi_{s,t} \|_{L^2}
\end{align*}
for some constant $A_6$.
Then by \eqref{ineq:1} with $\alpha=0$, 
\begin{align*}
\| \partial\varphi_{s,t} \|_{L^2}^2 \leq C \|\varphi_{s,t} \|_{L^1}.
\end{align*}
Hence 
\begin{align}\label{ineq:2}
\| \varphi_{s,t} \|_{L^2}^2 \leq A_{7} \|\varphi_{s,t} \|_{L^1}.
\end{align}
On the other hand, 
\begin{align}\label{ineq:3}
{\rm vol}(X_s,\omega_{X_s})^{\frac{1}{2}}
\| \varphi_{s,t} \|_{L^2} \geq \| \varphi_{s,t} \|_{L^1}
\end{align}
by H\"{o}lder's inequality.
Therefore, inequalities \eqref{ineq:2} and \eqref{ineq:3}
 show the $L^{2}$-bound 
\[\| \varphi_{s,t} \|_{L^2}\leq A_{8},\] 
 which proves the desired uniform bound of $\| \varphi_{s,t} \|_{L^{\infty}}$. 
 

\subsection{Base metric vs.\ the collapsing Ricci-flat metrics} \label{3.1}

\cite[Lemma 3.1]{Tos} gives a bound of ${\rm tr}_{\tilde{\omega}_{s,t}}
(\pi_{s}^{*}\omega_{B_{s}})$. 
This can be naturally extended to $s$-uniform version, 
including ADE singular $X_{s}$s verbatim. 
It is because we can use the previous
 estimate in \S\ref{unif.L.infinity} and also that
 the constant $A$ in the proof of \cite[Lemma 3.1]{Tos}
 can be taken as an explicit
 universal constant, which is the supremum of
 the Riemannian curvature of reference metrics $\omega_{B_{s}}$ on the base. 
We thus obtain 
\begin{align}\label{eq:lem3.1}
{\rm tr}_{\tilde{\omega}_{s,t}}
(\pi_{s}^{*}\omega_{B_{s}}) \leq C
\end{align}
for $C>0$ independent of $s$ and $t$.

\subsection{Fiberwise estimate} \label{(3.7).newref}

Now we prove an $s$-uniform version of \cite[(3.7)]{Tos}
 for the new K\"{a}hler form
 $\omega_{X_s,{\rm new}}=\sum_i x_i\omega_{{\rm pr}(s),i}$.

Let $p:={\rm pr}(s)$.
Since $\omega_{p,i}$ are Ricci-flat K\"ahler, 
 we have $\omega_{p,i}^2=c_{i,j}(p)\omega_{p,j}^2$ for any $i$ and $j$ 
 with continuous functions $c_{i,j}(p)$ on $S'$,
 which are bounded by positive constants from both sides. 
Let 
\[H'_{s,{\rm new}}:=
\frac{\pi_{s}^*\omega_{B_s}\wedge \omega_{X_s,{\rm new}}}{\omega_{p,1}^2}
 =\sum_{i=1}^{20}
 x_i\cdot c_{i,1}(p)
\frac{\pi_{s}^*\omega_{B_s}\wedge \omega_{p,i}}{\omega_{p,i}^2}.\]
The estimate \cite[(2.4)]{Tos} for each $\omega_{p,i}$ gives 
\begin{align}\label{2.4.omega.i}
\sigma(y)^{\lambda}\leq C \inf_{X_{s,y}} H'_{s,\rm new}. 
\end{align}
Write $\tilde{\omega}_{s,y}:=\tilde{\omega}_{s,t}|_{X_{s,y}}$
 and $\omega_{s,y,{\rm new}}:=\omega_{X_s,{\rm new}}|_{X_{s,y}}$.
Then 
\begin{align}\nonumber
\frac{\tilde{\omega}_{s,y}}{\omega_{s,y,{\rm new}}}
 &= \frac{\tilde{\omega}_{s,t}\wedge \pi_{s}^{*}\omega_{B_{s}}}{\omega_{X_s,{\rm new}}
 \wedge \pi_{s}^{*}\omega_{B_{s}}}\\ \label{3.7.new.tochuu}
 &= \frac{\tilde{\omega}_{s,t}\wedge \pi_{s}^{*}\omega_{B_{s}}}{\tilde{\omega}_{s,t}^2}
 \cdot  \frac{\tilde{\omega}_{s,t}^2}{H'_{s,{\rm new}} \omega_{p,1}^2}\\ \nonumber
 &= \operatorname{tr}_{\tilde{\omega}_{s,t}}(\pi_{s}^*\omega_{B_{s}}) 
 \cdot \frac{t A_t}{H'_{s,\rm new}},
\end{align}
where the constant $A_{t}$ is defined as
 $t A_{t}\omega_{p,1}^{2}=\tilde{\omega}_{s,t}^{2}$, 
 which plays a similar role to ``$a_{t}$'' in \cite[p.431]{Tos}.  
Then it is easy to see that
 $A_{t}$ is bounded by positive constants from both sides.
By \eqref{eq:lem3.1} and \eqref{2.4.omega.i}, 
 we can see that the last term of (\ref{3.7.new.tochuu}) 
 is $s$-uniformly bounded above by $Ct\sigma^{-\lambda}$,
 with ($s$-uniform) positive constants $C$ and $\lambda$. 
This is the desired $s$-uniform estimate of \cite[(3.7)]{Tos} 
for $\omega_{X_{s},{\rm new}}$, namely, 
\begin{align}\label{eq:(3.7)}
\frac{\tilde{\omega}_{s,y}}{\omega_{s,y,{\rm new}}}
\leq Ct\sigma(y)^{-\lambda}.
\end{align}

\subsection{Fiberwise $L^{\infty}$-estimate I --- For Ricci-flat reference metrics} \label{(3.9)}

Next we find an $L^{\infty}$-estimate of the \textit{fiberwise} potentials in some weak form. 
That is, we make the estimate \cite[(3.9)]{Tos}
 uniform with respect to $s$ as far as it runs over a compact subset
 \textit{away from the ADE locus} (cf.\  \S\ref{Setting}), 
 where we take $\omega_{X_{s}}$ to be Ricci-flat K\"ahler. 
In \cite{Tos}, the proof depends again on Yau's $L^{\infty}$-estimate 
method by using the Sobolev type inequality 
\cite[Lemma 3.2]{Tos}, which in turn implicitly uses smoothness of the total space. 
Since the fiber of $\pi_s$ is one-dimensional in our setting, 
 the Monge-Amp\`ere equation on the fiber becomes Laplace's equation. 
Hence we can simply calculate the potential function
 in terms of the Green function. 

In this subsection, we assume that $S$ is a relatively compact subset
 in $K\Omega\setminus K\Omega^o$ so that the closure of
 $S$ in $K\Omega$ is contained in $K\Omega^o$.
We will prove that 
\begin{align}\label{better.3.9}
\sup_{X_{s,y}} |\varphi_{s,t}-{\underline{\varphi}}{}_{s,t}|
&\le tC |\sigma(y)|^{-\lambda}
\end{align}
for some constants $C$ and $\lambda$,
 where we defined $X_{s,y}:=\pi_s^{-1}(y)$,
 $\omega_{s,y}:=\omega_{X_s}|_{X_{s,y}}$ and 
\[{\underline{\varphi}}{}_{s,t}
=\Bigl(\int_{X_{s,y}} \omega_{s,y}\Bigr)^{-1}\int_{X_{s,y}} \varphi_{s,t} \omega_{s,y}.
\]

For $y\in B_s\setminus {\rm disc}(\pi_s)$, let us write
 $X_{s,y}=\mathbb{C}_z/(\mathbb{Z}+\tau_{s,y} \mathbb{Z})$
 with $\tau_{s,y}\in \C$ in the standard fundamental domain, i.e., 
 $$|\tau_{s,y}|\geq 1 \text{ and } \left|\re \tau_{s,y}\right|\leq \frac{1}{2}.$$ 
Let us write $$\omega_{s,y}:=\omega_{X_s}|_{X_{s,y}}=\sqrt{-1}h_{s,y}\,dz\wedge d\bar{z}$$ 
for the coordinate $z$. 
Let $\Omega_s$ be a non-vanishing holomorphic volume form on $X_s$
 (which varies continuously in $s$). 
Since we assumed $\omega_{X_s}$ to be Ricci-flat K\"{a}hler, we have 
 $$C_{1}  \Omega_s\wedge \overline{\Omega}_s=\omega_{X_s}^2$$
 for some constant $C_{1}$. 
Recall from \S\ref{H.upper} that we defined $H_s=\frac{\omega_{X_s}\wedge \pi_s^{*}\omega_{B_s}}{\omega_{X_s}^2}$. 
This will form a continuous non-negative function on $\bigsqcup_s X_s$ which vanishes exactly at 
the $\pi_s$-critical points. 
Then $\omega_{X_s}\wedge \pi_s^{*}\omega_{B_s}=C_1H_s \Omega_s\wedge\overline{\Omega}_s$ implies that 
$h_{s,y}$ is proportional to $H_s$ along each fiber $\pi_s^{-1}(y)$. The constant of proportionality depends on $s$ and $y$. 

Due to the continuity of (the family of) $H_s$, we may suppose that $$\sup H_s < C_{2}$$ 
for some positive constant $C_{2}$. 
Also  recall that $$C_{3} \inf_{X_{s,y}} H_s \geq |\sigma(y)|^{\lambda_{1}},$$ 
where $\lambda_{1}$ can be also taken as an $s$-uniform constant (cf.\ \S\ref{2.4}). 
Therefore, 
\begin{equation}\label{sup.inf.ratio}
\frac{\sup_{X_{s,y}} h_{s,y}}{\inf_{X_{s,y}} h_{s,y}} 
=\frac{\sup_{X_{s,y}} H_s}{\inf_{X_{s,y}} H_s} \leq C_{2}C_{3} |\sigma(y)|^{-\lambda_{1}}. 
\end{equation}
On the other hand, the volume of the fiber $\pi_s^{-1}(y)$ does not depend on $y$
 and depends continuously on $s$, so it is bounded by a constant $C_4$.
Then
\begin{align*}
C_{4}\geq \int_{X_{s,y}}\omega_{s,y} = \int_{X_{s,y}}\sqrt{-1}h_{s,y}dz\wedge d\bar{z}
 \geq (\inf_{X_{s,y}} h_{s,y}) \int_{X_{s,y}}\sqrt{-1} dz\wedge d\bar{z}.
\end{align*}
Hence we obtain 
\begin{align*}
 \inf_{X_{s,y}} h_{s,y}  \leq \frac{C_{4}}{{\rm Im}(\tau_{s,y})}
\end{align*}
So, combined with our (\ref{sup.inf.ratio}) above we have now established 
\begin{equation*}
 \sup_{X_{s,y}} h_{s,y}  \leq  C_{2}C_{3}C_{4}
 {|\sigma(y)|^{-\lambda_{1}}{\rm Im} (\tau_{s,y})}^{-1}. 
\end{equation*}
On the other hand, 
since ${\rm Im}(\tau_{s,y})$ is lower bounded by a positive constant, 
we obtain 
$$\sup_{X_{s,y}}h_{s,y}\le C_{5}|\sigma(y)|^{-\lambda_{1}}.$$

On the other hand, \cite[(3.7)]{Tos} gives 
$$\frac{t^{-1}\tilde{\omega}_{s,y}}{\omega_{s,y}}\le C_{6}|\sigma(y)|^{-\lambda_2}.$$
Hence, if we write 
$$t^{-1}\tilde{\omega}_{s,y}=\sqrt{-1} \tilde{h}_{s,y}dz\wedge d\bar{z},$$ then we get 
\begin{Claim}\label{sup.h.upp.bd}
\begin{equation*}
\sup_{X_{s,y}}|\tilde{h}_{s,y}-h_{s,y}|\le C_{7}|\sigma(y)|^{-\lambda_3}.
\end{equation*}
\end{Claim}

Now we use the Green function of Laplacian to give a uniform upper bound of 
\begin{align*}
\psi:=
\psi_{s,t}:=\dfrac{(\varphi_{s,t}-\bar{\varphi}_{s,t})}{t}. 
\end{align*}
Here, $\bar{\varphi}_{s,t}$ is defined in a way similar
 to $\underline{\varphi}{}_{s,t}$, but the volume form is replaced 
 by $d(\re z)\wedge d(\im z)$: 
\begin{align*}
\bar{\varphi}_{s,t}:=
\frac{\int_{X_{s,y}}\varphi_{s,t}\, d(\re z)\wedge 
 d(\im z)}{\int_{X_{s,y}} d(\re z)\wedge d(\im z)} 
=\frac{\int_{X_{s,y}}\varphi_{s,t}d(\re z)\wedge 
d(\im z)}{\im \tau_{s,y}}.
\end{align*}
If we write $\operatorname{osc} \varphi
 := \left|\sup \varphi - \inf \varphi\right|$,  then
\[
\frac{1}{2}\operatorname*{osc}_{X_{s,y}} \varphi_{s,y} \leq 
\sup_{X_{s,y}}|\varphi_{s,t}-
\underline{\varphi}{}_{s,t}|
\le
\operatorname*{osc}_{X_{s,y}} \varphi_{s,y}
\]
and a similar inequalities hold when replacing
 $\underline{\varphi}{}_{s,t}$ by $\bar{\varphi}_{s,t}$.
Hence it is enough to prove \eqref{better.3.9}
 where $\underline{\varphi}{}_{s,t}$ replaced by $\bar{\varphi}_{s,t}$. 

Note that 
 $\Delta\psi =  \Phi $ on $X_{s,y}$, 
 where  $\Delta$ is the standard Laplacian
 with respect to the coordinate $z$, and 
 $\sqrt{-1}\Phi dz\wedge d\bar{z}:=t^{-1}\tilde{\omega}_{s,y}-\omega_{s,y}$,
 i.e., $\Phi:=\tilde{h}_{s,y}-h_{s,y}.$ 
Therefore, 
\begin{align*}
\psi(z) &= - \int_{X_{s,y}} \Phi(w)G_{s,y}(z,w) d(\re w)\wedge d(\im w) \\ 
&=- \int_{X_{s,y}} (\tilde{h}_{s,y}(w)-h_{s,y}(w))G_{s,y}(z,w) d(\re w)\wedge d(\im w),  
\end{align*}
where $G_{s,y}(z,w)$ is the Green function on the flat torus $X_{s,y}$. 
From now on, we do some analysis on general complex torus $\C/(\Z+\Z\tau)$
 with any $\tau\in \mathbb{H}$, 
and later apply the obtained estimation to the case $\tau=\tau_{s,y}$. 
It is known (cf.\  e.g.\ \cite[\S7]{LiWa}) that:
\begin{Fac}[Green function formula]
Let $G(z,w)$ denotes the Green function for the Laplacian on the (general) complex torus $\C_z/(\Z+\Z\tau)$ with $\tau\in \mathbb{H}$. 
\begin{enumerate}
\item We have
\begin{align*}
G(z,w)&=G(z-w)\\
&=-\frac{1}{2\pi}\log \Bigl|\dfrac{\vartheta_1(z-w)}{\eta(\tau)}\Bigr|
 + \frac{(\im (z-w))^2}{2\im \tau} + {\rm const.},  
\end{align*}
where $\vartheta_{1}$ is one of the classical Jacobi theta functions, 
and $\eta$ is the Dedekind eta function. 
\item
There is also a formula
\begin{align*}
&\dfrac{\vartheta_1(z)}{\eta(\tau)}
= e^{\pi iz}(1-e^{-2\pi iz})
 \prod_{m=1}^{\infty}
 (1-e^{2m\pi i\tau}e^{2\pi i z})
 (1-e^{2m\pi i\tau}e^{-2\pi i z}).
\end{align*}
\end{enumerate}
\end{Fac} 

Suppose $\im \tau > \frac{1}{2}$, $\left|\re z\right| \leq \frac{1}{2}$,
 and $\left|\im z\right|\leq \frac{\im \tau}{2}$.
Then 
\begin{align*}
\sum_{m=1}^{\infty} \bigl| \log |1-e^{2m\pi i \tau}e^{2\pi i z}|\bigr|
\ \text{ and }\ 
\sum_{m=1}^{\infty} \bigl| \log |1-e^{2m\pi i \tau}e^{-2\pi i z}|\bigr|
\end{align*}
are bounded.
Hence 
\begin{align*}
G_1(z)&:=
 G(z)+{\rm const.}\, (\text{which depends on $\tau$ but not on $z$}) \\
&= -\frac{1}{2\pi}\log |e^{\pi i z}-e^{-\pi i z}|
 + \frac{(\im z)^2}{2\im\tau}
 + \text{a bounded function}.
\end{align*}
Then it is easy to see that 
\begin{align*}
|G_1(z)|\leq -\frac{1}{2\pi}\log |z| + C_8 \im \tau
\end{align*}
for some constant $C_8$.
This implies 
\begin{align*}
&\int_{\C/(\Z+\Z\tau)} |G_1(z)| \,d(\re z)\wedge d(\im z) \\
&\leq \int_{\substack{\left|\re z\right| \leq \frac{1}{2} \\
  \left|\im z\right| \leq \frac{\im\tau}{2}}}
 \Bigl(-\frac{1}{2\pi}\log |z| + C_8 \im \tau \Bigr)  \,d(\re z)\wedge d(\im z)\\
&\leq C_9 (\im \tau)^2.
\end{align*}
If we suppose $\Delta \psi=\Phi$ on $\C/(\Z+\tau \Z)$, then
\begin{align*}
\psi(z) 
&= - \int_{\C/(\Z+\tau \Z)} \Phi(w)G(z,w) \,d(\re w)\wedge d(\im w) \\
&= - \int_{\C/(\Z+\tau \Z)} \Phi(w)G_1(z-w) \,d(\re w)\wedge d(\im w)
\end{align*}
and we have
\begin{align*}
|\psi(z)| &\leq \sup |\Phi| \cdot
 \int_{\C/(\Z+\tau \Z)} |G_1(z)| \,d(\re z)\wedge d(\im z) \\
&\leq \sup |\Phi|  \cdot C_9 (\im \tau)^2.
\end{align*}
We thus proved that 
\begin{Claim}\label{laplace.conclude}
if $\Delta \psi=\Phi$ on $\C/(\Z +\Z \tau)$
 and if $\tau$ is in the standard fundamental domain, then 
 $$|\psi(z)|\le C (\im \tau)^2 \sup |\Phi|,$$
\end{Claim}
\noindent
where the constant $C$ does not depend on $\tau$. 

Now, we come back to our particular setting on $X_{s,y}$ and apply
 Claim~\ref{laplace.conclude} to 
$\psi=\frac{1}{t}(\varphi_{s,t}-\bar{\varphi}_{s,t})$
 on $X_{s,y}=\C_z/(\Z+\Z \tau_{s,y})$. 
Combining with Claim~\ref{sup.h.upp.bd}, 
we get that 
\begin{align*}
\sup_{X_{s,y}} |\varphi_{s,t}-\bar{\varphi}_{s,t}|
&\le t C_{10} (\im\tau_{s,y})^{2}|\sigma(y)|^{-\lambda_3}.
\end{align*}
Then the known fact (cf.\  e.g.\ Lemma~\ref{ML.estimate}) 
\begin{equation*}
\im \tau_{s,y}\leq C\log(|\sigma(y)|^{-1})
\end{equation*} 
around the discriminant locus gives \eqref{better.3.9}.

This is a better bound than \cite[(3.9)]{Tos} for this special 
case, i.e., elliptic (Ricci-flat K\"ahler metrized) K3 surfaces. 
For our particular purpose of proving Theorem \ref{GTZ.extend}, 
we can also avoid the above improvement but for its own interest and records, 
we keep the above estimates here.

\subsection{Fiberwise $L^{\infty}$-estimate II --- For new reference metric}
 \label{(3.9)new.ref.metric}

We now see the $s$-uniform version \cite[(3.9)]{Tos} 
for the newly constructed reference metric 
$\omega_{X_s,{\rm new}}$ (in \S\ref{Grauert.singular.metric}) 
rather than the Ricci-flat metric $\omega_{X_s}$. 
Define the function $\varphi_{s,t,{\rm new}}$ on $X_s$ 
 to be a solution to the Monge-Amp\`ere equation
\[
t\omega_{X_s,{\rm new}}+\pi_{s}^*\omega_{B_{s}}
+\sqrt{-1}\partial\bar{\partial}\varphi_{s,t,{\rm new}}
=\tilde{\omega}_{s,t}\]
 such that $\sup_{X_s} \varphi_{s,t,{\rm new}} = 0$
 and define 
\[{\underline{\varphi}}{}_{s,t,{\rm new}}
=\Bigl(\int_{X_{s,y}} \omega_{s,y}\Bigr)^{-1}\int_{X_{s,y}}
 \varphi_{s,t,{\rm new}} \omega_{s,y}
\]
as in \S\ref{(3.9)}.
We will show 
\begin{align}\label{better.3.9.new}
\sup_{X_{s,y}} |\varphi_{s,t,{\rm new}}-{\underline{\varphi}}{}_{s,t,{\rm new}}|
&\le tC |\sigma(y)|^{-\lambda}
\end{align}
without the assumption that $S$ is away from the ADE locus
 in \S\ref{(3.9)}.

Define the functions $h_{s,y,{\rm new}}$ and $h_{p,i,y}$ on the smooth elliptic 
fiber $X_{s,y}=\pi_{s}^{-1}(y)$ by 
\[
\omega_{{X_s},{\rm new}}|_{X_{s,y}}= \sqrt{-1} h_{s,y,{\rm new}} dz\wedge d\bar{z}
\ \text{ and }\ 
\omega_{p,i}|_{X_{s,y}}= \sqrt{-1} h_{p,i,y} dz\wedge d\bar{z}.
\]
Then  
\begin{align*}
\Delta(\varphi_{s,t,{\rm new}}|_{X_{s,y}})
=t(\tilde{h}_{s,y}-h_{s,y,{\rm new}}).
\end{align*}
We obtained an estimate
\begin{align*}
\sup_{X_{s,y}} h_{p,i,y} \leq C_1|\sigma(y)|^{-\lambda_1}
\end{align*}
in \S\ref{(3.9)}.
Hence
\begin{align*}
\sup_{X_{s,y}} h_{s,y,{\rm new}}
= \sup_{X_{s,y}} \sum_{i=1}^{20} x_i h_{p,i,y} \leq C_2|\sigma(y)|^{-\lambda_2}.
\end{align*}
Since 
\begin{align*}
\frac{\tilde{h}_{s,y}}{h_{s,y,{\rm new}}}\leq  C_3 |\sigma(y)|^{-\lambda_3}
\end{align*}
 by \eqref{eq:(3.7)},
 we get
\begin{align*}
\sup_{X_{s,y}} |\tilde{h}_{s,y}-h_{s,y,{\rm new}}|\leq C_4 |\sigma(y)|^{-\lambda_4}
\end{align*}
as in Claim~\ref{sup.h.upp.bd}.
Therefore, \eqref{better.3.9.new} follows from  Claim~\ref{laplace.conclude}.

\section{A priori estimates II} 

\subsection{$C^2$-estimate} \label{Thm2.2.}

Next, we consider a certain $C^2$-estimate like \cite[Theorem 2.2]{Tos}.
We prove 
\begin{Thm}\label{2.9.}
For any compact subset $V\subset \bigcup_{s}X_{s}^{\rm sm}$, 
there is a positive constant $C_{V}$ such that for any $s\in S$, 
the inequality 
\begin{equation}\label{metric.comparison}
\frac{t}{C_{V}}\omega_{X_{s}}\le \tilde{\omega}_{s,t}\le C_{V}\omega_{X_{s}} 
\end{equation}
holds on $V\cap X_{s}$.
\end{Thm}

Recall from \S\ref{Setting} that
 $X_{s}^{\rm sm}$ denotes the $\pi_{s}$-smooth locus of $X_{s}$. 
For a fixed $s$ (with smooth $X_{s}$),
 the above theorem is weaker than \cite[Theorem 2.2]{Tos}. 
In the proof of \cite[Theorem 2.2]{Tos}, a $t$-uniform upper bound of 
 a certain function defined at \cite[bottom line of p.435]{Tos} is given, 
 which implies the right hand side inequality of \cite[Theorem 2.2]{Tos}. 
For this one needs an inequality used
 in the Yau's $C^{2}$-estimate method \cite[\S2]{Yau} and
 thus a uniform lower bound of holomorphic bisectional
 curvatures of $\omega_{X_s}$ is required. 
This part is not trivial if we vary $s$
 and some $X_{s}$ have ADE singularities,
 which is the reason why we need more discussion here. 

\begin{proof}[proof of Theorem \ref{2.9.}]
We first prove the left hand side inequality of \eqref{metric.comparison}
 for the newly constructed
 reference metrics $\omega_{X_{s},{\rm new}}$
 in \S\ref{Grauert.singular.metric}: 
\begin{Lem}\label{2.9.new}
There exist positive constants $A,B,C$ such that
 the inequality 
\begin{equation*}
\frac{t}{C e^{Ae^{B\sigma^{-\lambda}}}}\omega_{X_{s}, {\rm new}}\le \tilde{\omega}_{s,t}
\end{equation*}
holds on $X_{s}^{\rm sm}$ for any $s\in S$.
\end{Lem}

\begin{proof}[proof of Lemma \ref{2.9.new}]
We prove the lemma 
benefiting from that the
 holomorphic bisectional curvatures of $\omega_{X_{s}, {\rm new}}$
 are $s$-uniformly upper bounded
 as we saw in Claim~\ref{bisec.bded}. 

Indeed, the proof of the left hand side inequality of \cite[(2.9)]{Tos}
 works almost verbatim (see arguments from the last paragraph of \cite[p.438]{Tos}).
As in there, we define the function 
\[K_1
:=e^{-B\sigma^{-\lambda}}\Bigl(
\log (t\cdot \operatorname{tr}_{\tilde{\omega}_{s,t}} \omega_{X_s,{\rm new}})
 -\frac{A}{t}(\varphi_{s,t,{\rm new}}-\underline{\varphi}{}_{s,t,{\rm new}})
\Bigr)
\]
on $X_s^{\rm sm}$ for large constants $A$, $B$.
The function 
$\operatorname{tr}_{\tilde{\omega}_{s,t}} \omega_{X_s,{\rm new}}$
 is continuous on $X_s^{\rm sm}$ and
 bounded by $C\sigma^{-\lambda}$ near
 singular fibers $X_s\setminus X_s^{\rm sm}$.
Hence $K_1$ takes its maximum on $X_s^{\rm sm}$ for each $s$.

To get the estimate \cite[(3.25)]{Tos},
 we need Laplacian estimates as in \cite{Yau}
 (cf., e.g.~\cite[p.12, Lemma 4-2.]{Li}):
\[
\Delta_{\tilde{\omega}_{s,t}} \log
\operatorname{tr}_{\tilde{\omega}_{s,t}} \omega_{X_s,{\rm new}}
\geq -C\operatorname{tr}_{\tilde{\omega}_{s,t}} \omega_{X_s,{\rm new}} -C.
\]
For this estimate we use that the holomorphic bisectional curvatures of
  $\omega_{X_s,{\rm new}}$ is bounded.

In addition, \eqref{H.lower.bd} and \eqref{better.3.9.new} are also used 
 when we show estimates like \cite[(3.12) and (3.18)]{Tos}.
\end{proof}

Lemma~\ref{2.9.new} 
 implies the left hand side inequality of Theorem~\ref{2.9.} for 
 the original Ricci-flat K\"ahler reference $\omega_{X_{s}}$
 because the ratio of $\omega_{X_s}$ to $\omega_{X_s,{\rm new}}$ is bounded
 on the compact set $V$.

In our two-dimensional situation, 
 the right hand side inequality of Theorem~\ref{2.9.}
 is easily deduced from the left hand side inequality.
It is because we can write $a_{s,t} \omega_{X_s}^2 = \tilde{\omega}_{s,t}^2$
 for a constant $a_{s,t}$ and 
 $t^{-1}a_{s,t}$ is uniformly bounded.

Therefore, we complete the proof of Theorem~\ref{2.9.}.
\end{proof}


\subsection{Fiberwise $C^2$-estimate} \label{2.10}

We also replace the fiberwise version of the
 estimate \cite[(2.10)]{Tos} by its $s$-uniform but rough bounds version as follows. 
(We do not discuss \cite[(2.11)]{Tos} as it is 
not necessary for our particular purpose.) 

\begin{Thm}\label{fiber.C2}
For any compact subset $W\subset \bigcup_{s}(B_{s}\setminus {\rm disc}(\pi_{s}))$, 
there is a positive constant $C_{W}$ so that for any $y\in B_{s}\cap W$, 
we have 
\begin{equation*}
\frac{t}{C_{W}}\omega_{X_{s}}|_{\pi_{s}^{-1}(y)}\le 
\tilde{\omega}_{s,t}|_{\pi_{s}^{-1}(y)}\le t C_{W}\omega_{X_{s}}|_{\pi_{s}^{-1}(y)}. 
\end{equation*}
\end{Thm}

The left hand side inequality is a direct consequence of Theorem~\ref{2.9.}.
The right hand side follows from \eqref{eq:(3.7)} and that
 the ratio of $\omega_{X_s,{\rm new}}$ to $\omega_{X_s}$ is bounded
 on the compact set $\bigcup_{s} \pi_s^{-1}(B_s\cap W)$.

\subsection{Convergence of potential functions} \label{Tos.Thm4.1}

Here, we prove that the potential functions $\varphi_{s,t}$
 converges to $\psi$ as $t\to 0$, where
 $\psi$ is a potential function for the McLean metric.
Our statement Theorem~\ref{potential.convergence}
 is similar to \cite[Theorem 4.1]{Tos}, 
 but the arguments below is a little different from \cite{Tos}. 

Let $\omega_{s,\rm{ML}}$ denote the McLean metric for $\pi_s\colon X_s\to B_{s}$,
which is given as 
\begin{align}\label{adiab(4.3)}
\omega_{s,\rm{ML}} :=
 \frac{\int_{X_s}\pi_s^*\omega_{B_{s}} \wedge
   \omega_{X_s}}{\int_{X_{s,y}}\omega_{X_s}\cdot \int_{X_s}\omega_{X_s}^2}
 (\pi_s)_* \omega_{X_s}^2
\end{align}
(see \cite[(4.3)]{Tos}). 
Let $\psi\in L^1(B_s)\cap L^{\infty}(B_s)$
 be a potential function for $\omega_{s,\rm{ML}}$, namely,
 $\omega_{B_s}+\sqrt{-1}\partial\bar{\partial}\psi = \omega_{s,\rm{ML}}$.
Recall $\tilde{\omega}_{s,t} = \pi_s^*\omega_{B_{s}}+t\omega_{X_s}
 +\sqrt{-1}\partial\bar{\partial}\varphi_{s,t}$ is a Ricci flat K\"ahler metric. 
Thus 
$\tilde{\omega}_{s,t}^2 = a_{s,t}\omega_{X_s}^2$ with 
\begin{align*}
a_{s,t}= 2t \frac{\int_{X_s}\pi_s^*\omega_{B_s}\wedge
 \omega_{X_s}}{\int_{X_s}\omega_{X_s}^2} + O(t^2).
\end{align*}

Let $\eta\in L^1(B_s)$ and denote its pull-back to $X_s$ also by $\eta$.
We compute
\begin{align*}
\int_{X_s} \eta \tilde{\omega}_{s,t}^2
=\int_{X_s} \eta (\pi_s^*\omega_{B_s}+t\omega_{X_s}
 +\sqrt{-1}\partial\bar{\partial}\varphi_{s,t})^2
\end{align*}
when $t$ tends to zero.
We have
\begin{align*}
&(\pi_s^*\omega_{B_{s}}+t\omega_{X_s}
 +\sqrt{-1}\partial\bar{\partial}\varphi_{s,t})^2 \\
&=((\pi_s^*\omega_{B_{s}}+\sqrt{-1}\partial\bar{\partial}\underline{\varphi}{}_{s,t})
 +t\omega_{X_s}+\sqrt{-1} \partial\bar{\partial}
 (\varphi_{s,t}-\underline{\varphi}{}_{s,t}))^2 \\
&=2(\pi_s^*\omega_{B_{s}}+\sqrt{-1}\partial\bar{\partial}\underline{\varphi}{}_{s,t}) \wedge (t\omega_{X_s}+\sqrt{-1}\partial\bar{\partial}
 (\varphi_{s,t}-\underline{\varphi}{}_{s,t}))\\ 
&+ (t\omega_{X_s}+\sqrt{-1}\partial\bar{\partial}
 (\varphi_{s,t}-\underline{\varphi}{}_{s,t}))^2
\end{align*}
and then
\begin{align*}
&\int_{X_s}\eta
 (t\omega_{X_s}+\sqrt{-1}\partial\bar{\partial}
 (\varphi_{s,t}-\underline{\varphi}{}_{s,t}))^2\\
&= \int_{X_s}\eta t^2\omega_{X_s}^2
 + \int_{X_s}\eta (2t\omega_{X_s}
 + \sqrt{-1}\partial\bar{\partial}(\varphi_{s,t}-\underline{\varphi}{}_{s,t}))
 \wedge \sqrt{-1}\partial\bar{\partial}
 (\varphi_{s,t}-\underline{\varphi}{}_{s,t})\\
&= \int_{X_s}\eta t^2\omega_{X_s}^2
 + \sqrt{-1}\int_{X_s} (\varphi_{s,t}-\underline{\varphi}{}_{s,t})
 \bar{\partial}\partial\eta\wedge
 (2t\omega_{X_s}
 + \sqrt{-1}\partial\bar{\partial}(\varphi_{s,t}-\underline{\varphi}{}_{s,t})).
\end{align*}
Here, $\bar{\partial}\partial\eta$ is considered as a current in general.
Moreover,
\begin{align*}
&\int_{X_s} (\varphi_{s,t}-\underline{\varphi}{}_{s,t})
 \bar{\partial}\partial\eta\wedge t\omega_{X_s}
= \int_{B_s}
 (\pi_s)_*((\varphi_{s,t}-\underline{\varphi}{}_{s,t})t\omega_{X_s})
 \bar{\partial}\partial\eta=0,\\
&\int_{X_s} \eta (\pi_s^*\omega_{B_s}
 +\sqrt{-1}\partial\bar{\partial}\underline{\varphi}{}_{s,t})
\wedge 
 \sqrt{-1}\partial\bar{\partial}
 (\varphi_{s,t}-\underline{\varphi}{}_{s,t}) \\
&=\int_{X_s}
\sqrt{-1}(\varphi_{s,t}-\underline{\varphi}{}_{s,t})
 \bar{\partial}\partial\eta
 \wedge (\pi_s^*\omega_{B_{s}}
 +\sqrt{-1}\partial\bar{\partial}\underline{\varphi}{}_{s,t})=0,
\end{align*}
and 
\begin{align*}
\int_{X_s} \eta (\pi_s^* \omega_{B_{s}}
 + \sqrt{-1}\partial\bar{\partial}\underline{\varphi}{}_{s,t})
 \wedge t\omega_{X_s} 
=t \int_{X_{s,y}}\omega_{X_s}
 \cdot \int_{B_{s}} \eta (\omega_{B_{s}}
 + \sqrt{-1}\partial\bar{\partial}\underline{\varphi}{}_{s,t}).
\end{align*}
We thus obtain
\begin{align*}
\int_{X_s} \eta \tilde{\omega}_{s,t}^2 
&= 2t \int_{X_{s,y}}\omega_{X_s}
 \cdot \int_{B_{s}} \eta (\omega_{B_{s}}
 + \sqrt{-1}\partial\bar{\partial}\underline{\varphi}{}_{s,t}) \\
& + t^2 \int_{X_s} \eta \omega_{X_s}^2 
 - \int_{X_s}
 (\varphi_{s,t}-\underline{\varphi}{}_{s,t})
 \bar{\partial}\partial\eta
 \wedge \partial\bar{\partial}(\varphi_{s,t}-\underline{\varphi}{}_{s,t}).
\end{align*}
On the other hand,
\begin{align*}
\int_{X_s} \eta \tilde{\omega}_{s,t}^2 
&=\int_{X_s}\eta a_{s,t}\omega_{X_s}^2 \\ 
&= \Bigl(2t\frac{\int_{X_s}\pi_s^*\omega_{B_{s}}\wedge
 \omega_{X_s}}{\int_{X_s}\omega_{X_s}^2} + O(t^2) \Bigr)
 \int_{X_s}\eta \omega_{X_s}^2 \\
& = 2t \int_{X_{s,y}}\omega_{X_s} \int_{B_{s}} \eta \omega_{s,\rm{ML}}
 + O(t^2)  \int_{X_s}\eta \omega_{X_s}^2
\end{align*}
by \eqref{adiab(4.3)}.
Since $\int_{X_{s,y}}\omega_{X_s}$ is bounded from below and above,
\begin{align}\label{estimate.thm4.1}
\nonumber
&\int_{B_{s}} \eta
 \sqrt{-1}\partial\bar{\partial}(\underline{\varphi}{}_{s,t}-\psi)\\ 
&=\int_{B_{s}} \eta (\omega_{B_s}
 + \sqrt{-1}\partial\bar{\partial}\underline{\varphi}{}_{s,t})
 -\int_{B_{s}} \eta \omega_{s,\rm{ML}} \\ \nonumber
&= O(t) \int_{X_s} \eta \omega_{X_s}^2
 + O(t^{-1}) \int_{X_s}
 (\varphi_{s,t}-\underline{\varphi}{}_{s,t})
 \bar{\partial}\partial \sqrt{-1}\eta \wedge
 \sqrt{-1}\partial\bar{\partial}(\varphi_{s,t}-\underline{\varphi}{}_{s,t}).
\end{align}
Let $K$ be a compact subset of $\bigcup_{s\in S} B_{s}\setminus {\rm disc} (\pi_s)$
 and let $K_s:=K\cap B_s$.
By Theorem~\ref{fiber.C2}
\begin{align*}
-Ct\omega_{X_s}|_{\pi_s^{-1}(y)}
\leq \sqrt{-1}\partial\bar{\partial}(\varphi_{s,t}-\underline{\varphi}{}_{s,t})
 |_{\pi_s^{-1}(y)}
\leq Ct\omega_{X_s}|_{\pi_s^{-1}(y)}
\end{align*}
for $y\in K_s$. 
Then by Claim~\ref{laplace.conclude}, 
 we have  
\begin{align}\label{(3.9).RFK}
|\varphi_{s,t}-\underline{\varphi}{}_{s,t}| \leq C t
\end{align}
 on $K_s$.
Here, $C$ is a constant depending on $K$ but not on $s$.
Hence if we fix $\eta$ to be a smooth function on $B_{s}$
 and if $\eta$ is constant near ${\rm disc} (\pi_s)$, then
\begin{align}\label{estimate.thm4.1.eta}
\int_{B_{s}} \eta
 \sqrt{-1}\partial\bar{\partial}(\underline{\varphi}{}_{s,t}-\psi) 
=\int_{B_{s}} \eta (\omega_{B_s}
 + \sqrt{-1}\partial\bar{\partial}\underline{\varphi}{}_{s,t})
 -\int_{B_{s}} \eta \omega_{s,\rm{ML}}
= O(t)
\end{align}
by \eqref{estimate.thm4.1}.

We next show that for any $\epsilon>0$, there exists a small open neighborhood
 $V$ of ${\rm disc}(\pi_s)$ and $t_0>0$ such that 
\begin{align}\label{estimate.thm4.1.near.sing.}
\Bigl|\int_V f \partial\bar{\partial}\underline{\varphi}{}_{s,t} \Bigr|
< \epsilon \sup_V |f| 
\end{align}
for any bounded function $f$ on $V$ and $t<t_0$.
To show this, fix a smooth nonnegative-valued function $\eta$ on the base 
$B_{s}$
 which takes constant value 1 on a neighborhood $V$ of ${\rm disc} (\pi_s)$
 and $\int_{B_{s}} \eta \omega_{s,\rm{ML}}<\frac{\epsilon}{4}$.
Then by \eqref{estimate.thm4.1.eta}, 
\begin{align*}
\int_{B_{s}} \eta (\omega_{B_{s}}
 + \sqrt{-1}\partial\bar{\partial}\underline{\varphi}{}_{s,t})
 < \frac{\epsilon}{2}
\end{align*}
for small $t$.
We have
 $(\pi_s)_* (\sqrt{-1}\partial\bar{\partial}{\varphi}_{s,t}\wedge\omega_{X_s})
 =\sqrt{-1} (\int_{X_{s,y}}\omega_{X_s}) \cdot
 \partial\bar{\partial}\underline{\varphi}{}_{s,t}$
and $c_s:=\int_{X_{s,y}}\omega_{X_s}$ is bounded from below and above
 by positive constants.
Hence
\begin{align*}
\int_{X_s} \eta \tilde{\omega}_{s,t}\wedge \omega_{X_s}
= c_s
 \int_{B_{s}} \eta (\omega_{B_s}
 + \sqrt{-1}\partial\bar{\partial}\underline{\varphi}{}_{s,t})
 + \int_{X_s} \eta t \omega_{X_s}^2 
< \frac{3}{4}\epsilon c_s
\end{align*}
for small $t$.
Since $\tilde{\omega}_{s,t}\wedge \omega_{X_s}$ is positive, 
\begin{align*}
\Bigl|\int_{\pi_s^{-1}(V)} f \tilde{\omega}_{s,t}\wedge \omega_{X_s}
 \Bigr|
&\leq  \sup_V |f| \cdot
 \int_{\pi_s^{-1}(V)} \tilde{\omega}_{s,t}\wedge \omega_{X_s} \\
&\leq \sup_V |f| \cdot
 \int_{X_s} \eta \tilde{\omega}_{s,t}\wedge \omega_{X_s}
< \frac{3}{4}\epsilon c_s \sup_V |f|.
\end{align*}
By shrinking $V$ if necessary, we obtain 
\begin{align*}
& \Bigl|\int_V f \partial\bar{\partial}\underline{\varphi}{}_{s,t} \Bigr|
 =  c_s^{-1}\Bigl|\int_{\pi_s^{-1}(V)} f \partial\bar{\partial} \varphi_{s,t}
 \wedge \omega_{X_s} \Bigr|  \\
& = c_s^{-1} \Bigl|\int_{\pi_s^{-1}(V)} f
 (\tilde{\omega}_{s,t}-(t\omega_{X_s}+\pi_s^*\omega_{B_s}))
 \wedge \omega_{X_s} \Bigr|  \\
&\leq c_s^{-1}
 \Bigl(\Bigl|\int_{\pi_s^{-1}(V)} f \pi_s^*\omega_{B_s} \wedge \omega_{X_s} \Bigr|
 +\Bigl|\int_{\pi_s^{-1}(V)} f t \omega_{X_s}^2 \Bigr|
 +\Bigl|\int_{\pi_s^{-1}(V)} f \tilde{\omega}_{s,t}\wedge \omega_{X_s}
  \Bigr|\Bigr)\\
&< \epsilon \sup_V |f|,
\end{align*}
showing \eqref{estimate.thm4.1.near.sing.}.

We also note that a similar estimate
\begin{align}\label{estimate.thm4.1.near.sing.2}
\Bigl|\int_V f \partial\bar{\partial}\psi \Bigr|
< \epsilon \sup_V |f| 
\end{align}
can be also proved (more easily) by using the definitions of
 $\psi$ and the McLean metric.

Let us normalize $\varphi_{s,t}$ and $\psi$ as
$$\int_{B_{s}} \underline{\varphi}{}_{s,t} \omega_{B_{s}}
 = \int_{B_{s}} \psi \omega_{B_{s}}= 0.$$ 
\begin{Claim}\label{base.level}
We have a uniform convergence $\underline{\varphi}{}_{s,t}\to \psi$ 
 on any given compact set $K\subset \bigcup_s B_{s}\setminus {\rm disc}(\pi_s)$.
\end{Claim}
\begin{proof}[proof of Claim \ref{base.level}] 
For $y\in B_{s}$, let $F_y$ be the function on $B_{s}$ which solves
 $$\sqrt{-1}\bar{\partial} \partial F_y = \delta_y -c\omega_{B_{s}},$$ 
 where $c:=\int_{B_{s}}\omega_{B_{s}}$, namely, 
\[\int_{B_{s}} F_y \sqrt{-1}\partial\bar{\partial} f   
 = f(y) - c \int_{B_{s}} f \omega_{B_{s}}.\]
Then $F_y \in L^1(B_{s})$ and is smooth on $B_{s}\setminus y$. 
We will prove the convergence $\underline{\varphi}{}_{s,t}\to \psi$
 by estimating 
\begin{align*}
(\underline{\varphi}{}_{s,t}-\psi) (y)
=\int_{B_{s}} F_y \sqrt{-1}\partial\bar{\partial}(\underline{\varphi}{}_{s,t}-\psi)
\end{align*}
for $y\in K$.
Let $V\subset B_{s}$ be as in \eqref{estimate.thm4.1.near.sing.}.
We may assume $K\cap \overline{V}=\emptyset$.
Take an open neighborhood $V'$ of ${\rm disc}(\pi_s)$
 such that $\overline{V'}\subset V$.
Take a smooth function $0\leq \phi\leq 1$ on $B_{s}$ such that 
 $\phi=1$ on $V'$ and $\operatorname{supp} \phi \subset V$.
By \eqref{estimate.thm4.1.near.sing.} and
 \eqref{estimate.thm4.1.near.sing.2}, 
\begin{align*}
\int_{B_{s}} \phi F_y
 \sqrt{-1}\partial\bar{\partial}(\underline{\varphi}{}_{s,t}-\psi)
< 2\epsilon \sup_V |F_y|
\leq 2\epsilon \sup_{y\in K} \sup_V |F_y|.
\end{align*}
By \eqref{estimate.thm4.1},
\begin{align*}
&\int_{B_{s}} (1-\phi)F_y
 \sqrt{-1}\partial\bar{\partial}(\underline{\varphi}{}_{s,t}-\psi)\\
&=O(t) \int_{X_s} (1-\phi)F_y \omega_{X_s}^2 \\ 
&+ O(t^{-1}) \int_{X_s}
 (\varphi_{s,t}-\underline{\varphi}{}_{s,t})
 \pi_s^*( \sqrt{-1}\bar{\partial}\partial (1-\phi)F_y)
 \wedge \partial\bar{\partial}(\varphi_{s,t}-\underline{\varphi}{}_{s,t}).
\end{align*}
The $2$-form
\begin{align*}
\alpha_y := \sqrt{-1}\bar{\partial}\partial (1-\phi)F_y-\delta_y
\end{align*} 
is bounded and $\operatorname{supp} \alpha_y \subset B_{s}\setminus V'$.
Then 
\begin{align*}
&\int_{X_s} (\varphi_{s,t}-\underline{\varphi}{}_{s,t})
 \pi_s^*( \sqrt{-1}\bar{\partial}\partial (1-\phi)F_y)
 \wedge \partial\bar{\partial}(\varphi_{s,t}-\underline{\varphi}{}_{s,t})\\
&= \int_{\pi_s^{-1}(y)} (\varphi_{s,t}-\underline{\varphi}{}_{s,t})
  \partial\bar{\partial}(\varphi_{s,t}-\underline{\varphi}{}_{s,t})
 + \int_{X_s} (\varphi_{s,t}-\underline{\varphi}{}_{s,t})
 \alpha_y
 \wedge \partial\bar{\partial}(\varphi_{s,t}-\underline{\varphi}{}_{s,t}).
\end{align*}
By Theorem~\ref{fiber.C2} and \eqref{(3.9).RFK}, 
 the last integrals are $O(t^2)$.
This proves the uniform convergence
 $\underline{\varphi}{}_{s,t}\to \psi$ on $K$. 
We end the proof of Claim \ref{base.level}. 
\end{proof}

Combining Claim \ref{base.level} and \eqref{(3.9).RFK}, we obtain
 a convergence of potential functions
 (recall that $X_{s}^{\rm sm}=X_{s}\setminus \pi_s^{-1}({\rm disc}(\pi_s))$
 is the union of smooth fibers of $\pi_s$):
\begin{Thm}\label{potential.convergence}
We have a uniform convergence ${\varphi}_{s,t}\to \psi$ 
 on any given compact set $K\subset \bigsqcup_s X_{s}^{\rm sm}$. 
\end{Thm}

\subsection{$C^2$-estimate II} \label{4.1}

The $s$-uniform extension of \cite[Lemma 4.1]{GTZ1} is the following:
for any given compact set
 $K\subset \bigsqcup_s X_{s}^{\rm sm}$,
\begin{align}\label{GTZ4.1}
C^{-1}(\pi_s^*\omega_{B_s}+t\omega_{X_s})
\leq \tilde{\omega}_{s,t} \leq C(\pi_s^*\omega_{B_s}+t\omega_{X_s})
\end{align}
holds, where $C$ is independent of $s,t$.
This can be proved as in \cite{GTZ1}.
The left hand side inequality
 follows from our \eqref{eq:lem3.1} and Theorem~\ref{2.9.}.
For the right hand side inequality, we use 
\[\frac{\tilde{\omega}_{s,t}^2}{(\pi_s^*\omega_{B_s}+t\omega_{X_s})^2}
\leq \frac{C_1 t{\omega}_{X_s}^2}{\pi_s^*\omega_{B_s}\wedge t\omega_{X_s}}
\leq \frac{C_2 t{\omega}_{X_s,{\rm new}}^2}{\pi_s^*\omega_{B_s}\wedge
  t\omega_{X_s,{\rm new}}}
= \frac{C_2}{H_{s,{\rm new}}} \leq C_3\sigma^{-\lambda}, 
\]
which follows from \eqref{H.lower.bd} and 
 that the ratio of $\omega_{X_s,{\rm new}}$
 to $\omega_{X_s}$ is bounded on $K$.

\subsection{$C^k$-estimate} \label{4.3}

Now we 
make \cite[Proposition 4.3, Lemma 4.5]{GTZ1} 
uniform with respect to $s$. Then 
they imply the $s$-uniform version of \cite[(4.18)]{GTZ1} 
which is the purpose of this subsection. 
That is, we want to prove that 
\begin{Prop}\label{4.18}
For any given compact subset
 $K\subset \bigsqcup_s X_{s}^{\rm sm}$ and $k\in \mathbb{Z}_{\geq 2}$, 
there exists a constant $C$ such that 
$$\|\varphi_{s,t}\|_{C^{k}(K\cap X_{s})}\le C$$ 
holds for any $s\in S$. 
\end{Prop}

The above statement with $k=3$ is enough for our
 particular purpose of proving the $s$-uniform
 Gromov-Hausdorff convergence of $X_{s}$ to $B_{s}$. 

\begin{proof}
We follow the arguments of \cite[\S3, \S4]{GTZ1}. 
Consider the real analytic family 
$$\bigsqcup_{s}X_{s}^{\rm sm}\to \bigsqcup_s (B_{s}\setminus {\rm disc}(\pi_{s})).$$ 
We fix a real analytic family of local holomorphic coordinates 
$$
\{((B_{s}\setminus {\rm disc}(\pi_{s}))\supset V_{s},\ y_{s}\colon V_{s}\to \C)\},
$$ 
and set $\mathcal{V}:=\bigsqcup_{s}V_{s}$. 
We then take the uniformization of $\pi_{s}^{-1}(V_{s})$ as 
$$\bigsqcup p_{s}\colon (\bigsqcup V_{s})\times \mathbb{C} \twoheadrightarrow 
\bigsqcup \pi_{s}^{-1}(V_{s}),$$
which is 
real analytic and holomorphic on $V_s\times\C$ for each fixed $s$. 
This is possible for small enough $\mathcal{V}$. 
Then 
\cite[\S3]{GTZ1} constructs a real analytic family of 
semi-flat forms $\sqrt{-1} \partial\bar{\partial}\eta_{s}$ on 
$\bigsqcup_{s} (\pi_{s}\circ p_{s})^{-1} V_{s}$, which descend to 
$\bigsqcup_{s} \pi_{s}^{-1}V_{s}$ as $\omega_{{\rm SF},s}$ 
in an explicit manner. 
We note that in our case, as the smooth fibers of $\pi_{s}$ are all 
elliptic curves, hence in particular all projective. Thus we do not need 
the projectivity assumption of $X_{s}$ for the arguments in \cite{GTZ1, GTZ2} 
(cf.\  also  \cite[``Reviewer's comment"]{H} and 
\cite{HT}) and our $s$-uniform extension. 

From the arguments of \cite[Proposition 3.1]{GTZ1}, 
we get a real analytic family of holomorphic sections of $\pi_{s}$ as 
$\sigma_{s}\colon V_{s}\to (\pi_{s}\circ p_{s})^{-1}(V_{s})$ satisfying 
$$T_{\sigma_{s}}^{*}\omega_{\rm SF,s}-\omega_{X_{s},{\rm new}}=\sqrt{-1} 
\partial\bar{\partial}\xi|_{\pi_s^{-1}(V_s)}$$ 
with real analytic $\xi\colon \bigsqcup_{s} \pi_{s}^{-1}(V_{s})\to \R$.
Here, $T_{\sigma_{s}}$ denotes the translation by the section $\sigma_{s}$
 with respect to the image of zero section by $p_{s}$. 
Note that the construction of 
$\sigma_{s}$ is concrete which comes from the decomposition of ``$\zeta^{0,1}$'' 
discussed in \cite[after (3.6)]{GTZ1}. 
Hence the statements of \cite[Lemma 4.2, Proposition 4.3]{GTZ1} 
can be made $s$-uniform verbatim. Now we prove them after \textit{op.cit}. 

By \eqref{GTZ4.1}, we obtain the $s$-uniform version of \cite[Lemma 4.2]{GTZ1}.
Then we want to make 
 \cite[Proposition 4.3]{GTZ1} $s$-uniform for compact set
 $K=\bigsqcup_{s} K_{s}\subset \bigsqcup_{s}X_{s}^{\rm sm}$. 
We write $$p_{s}^{*}(\pi_{s}^{*}\omega_{B_{s}}+\omega_{{\rm SF},s})=
\sqrt{-1} \partial\bar{\partial}
F_{s}$$ with a real analytic function $F_{s}$. 
 If we take a suitable finite covering of $K$ as 
 $$\{\bigsqcup_{s} K_{s}^{(i)}\mid i=1, \cdots,m\}$$ and passing to 
 $K_{s}^{(i)}$, then the Evans-Krylov estimates (cf.\  \cite[Chapter 17]{GT}, 
 \cite{Blocki}, \cite{Siu}, etc.) for the Monge-Amp\`{e}re equation  
 $$\lambda_{t}^{*}p_{s}^{*}T_{-\sigma_{s}}^{*}\tilde{\omega}_{s,t}
 =p_{s}^{*}\pi_{s}^{*}\omega_{B_{s}}+\omega_{{\rm SF},s}+
 \sqrt{-1} \partial\bar{\partial}u_{s,t}
 $$
 is applied to ensure 
 $$
 \|F_{s}+u_{s,t}\|_{C^{2,\alpha}(K_{s}^{(i)})}\le C
 $$
 for $s$-uniform $C$ and fixed $\alpha\in (0,1)$. 
 Hence we have 
 $$
 \|u_{s,t}\|_{C^{2,\alpha}(K_{s})}\le C 
 $$
for all $0<t<1$. 

Then, as \cite[proof of Proposition 4.3]{GTZ1} discussed, 
we do the usual ``bootstrapping argument'' to prove higher order estimates 
\begin{align}\label{Ck.esti}
\|u_{s,t}\|_{C^{k,\alpha}(K_{s})}\le C 
\end{align}
for $k\in \mathbb{Z}_{\geq 2}$ and all $0<t<1$. 

More details are as follows. 
Suppose (\ref{Ck.esti}) holds for fixed $k$ (we start 
with $k=2$). Apply Schauder estimates (\cite[Problem 6.1]{GT}, 
\cite[Theorem 1.37]{Naka99}, etc.) to
 the linearized elliptic partial differential equation satisfied by the 
partial derivatives 
$\partial_{z}u_{s,t}$ of $u_{s,t}$ with respect to the local 
holomorphic coordinates $\{z\}$. Then the 
coefficients of the equation is of $C^{k-1,\alpha}$-class with bounded norms by 
our assumption of the induction. 
Then the Schauder estimates implies \eqref{Ck.esti} for $k+1$. 
Hence we can inductively prove \eqref{Ck.esti} for all $k$. 
In particular, we obtain the desired estimate Proposition~\ref{4.18}. 
\end{proof}

\section{Gromov-Hausdorff adiabatic limits}
\subsection{Smooth convergence of metrics} \label{Ck.loc}

The previous $s$-uniform version of \cite[(4.18)]{Tos}, i.e., 
Proposition \ref{4.18} for $k=3$ combined with
 Theorem~\ref{potential.convergence}, the uniform convergence
 of potential functions, imply the following claim 
 by the Gagliardo-Nirenberg-Sobolev inequality. 
\begin{Claim}\label{Ck.loc.claim}
The potential functions $\varphi_{s,t}$ of $\tilde{\omega}_{s,t}$ 
 converge to that of $\pi_{s}^{*}\omega_{s,\text{\rm ML}}$ in 
the $C^{2}_{\rm loc}(X_{s}^{\rm sm})$-sense.
Therefore, 
$$\tilde{\omega}_{s,t}\to \pi_{s}^{*}\omega_{s,\text{\rm ML}}\quad (t\to +0)$$ 
in the $C^{0}_{\rm loc}(X_{s}^{\rm sm})$-sense. 
\end{Claim}

\subsection{Diameter bounds of neighborhoods of singular fibers}  \label{Sing.fiber.nbhd}

Next we give a uniform upper bound of diameters of small 
\textit{neighborhoods of} singular fibers in $X_{s}$. 

In \cite[\S5]{GTZ1}, 
a local isometry $\varphi_{s}\colon 
(B_{s}\setminus {\rm disc}(\pi_{s}))\to L_{s}$ is constructed, 
where $L_{s}$ is an arbitrary 
Gromov-Hausdorff limit of $(X_{s},\tilde{\omega}_{s,t_{i}})$ 
with a sequence converging to $0$, i.e., $t_{i}\to 0$. 
There, ${\rm Im}(\varphi_{s})$ 
is proved to be dense inside the 
Gromov-Hausdorff limit $L_{s}$. 
Here, we extract an essence out of their arguments in the following 
elementary way, i.e., by simply using the Bishop-Gromov inequality on the 
total space. Our main claim in this subsection is the following. 
A point is that this is uniform for small enough $t$ (and all $s$). 

\begin{Claim}\label{sing.fiber.nbhd.claim}
For any $\epsilon>0$, 
there exist a positive constant $t_{0}>0$ and a 
small enough open neighborhood of $\bigsqcup_{s}{\rm disc}(\pi_{s})
\subset\bigsqcup_{s} B_{s}$ which will be denoted by $\mathcal{U}=\bigsqcup_{s}U_{s}$ 
with $U_{s}\subset B_{s}$ such that the following is satisfied. If we decompose 
$U_{s}$ into connected components as 
$U_{s}=\bigcup_{i}U_{s}^{(i)}$, then we have 
\begin{align*}
\sum_i {\rm diam}(\pi_{s}^{-1}(U_{s}^{(i)}),\tilde{\omega}_{s,t})<\epsilon
\end{align*}
for any $0<t<t_{0}$ and any $s$. 
\end{Claim}
Note that the claim does not give a bound of diameters
 of the singular fibers themselves
 but those of their small neighborhoods and their boundaries.

\begin{proof}[proof of Claim \ref{sing.fiber.nbhd.claim}]
Let us take a small neighborhood of
 the discriminant locus $\mathcal{U}=\bigcup_{s}U_{s}$
 as in the proof of Proposition~\ref{ML.estimate.disc}. 
Then $U_{s}$ is a union of $24$ small disks
 containing ${\rm disc}(\pi_s)$ and hence the number of
 connected components of $U_s$ is less than or equal to $24$.
If we write $U_s=\bigcup_i U_s^{(i)}$ for the decomposition
 into connected components,
 then Proposition~\ref{ML.estimate.disc} shows that
\begin{align}\label{ML.boundary}
\text{ ${\rm diam}(\partial \overline{U}_s^{(i)},\, \omega_{s,{\rm ML}})$
 is arbitrarily small}
\end{align}
for sufficiently small $\mathcal{U}$.
Moreover, by the proof of Proposition~\ref{prop:MW.conti}, we also see that
\begin{align}\label{ML.disk}
\text{ ${\rm diam}(U_s^{(i)},\, \omega_{s,{\rm ML}})$
 is arbitrarily small}
\end{align}
for sufficiently small $\mathcal{U}$.
\eqref{ML.disk} will be used in \S\ref{GH.conv}.

For any $x\in \pi_{s}^{-1}(U_{s})$ and $r>0$, we have 
\begin{align}\label{Bishop-Gromov}
\dfrac{{\rm vol}(B(x,r))}{t}\ge C\dfrac{r^{4}}{{\rm diam}(X_{s},\tilde{\omega}_{s,t})^{4}}
\end{align}
from the Bishop-Gromov inequality. 
Here, $B(x,r)$ denotes the $r$-ball with center $x$ with respect
 to the metric $\tilde{\omega}_{s,t}$.
Indeed, note $${\rm vol}(X_{s},\tilde{\omega}_{s,t})=(c_{s}+o(1))t$$ 
with a positive constant $c_{s}$ continuous with respect to $s$. 
On the other hand, for $U_{s}\subset B_{s}$, 
$$\dfrac{{\rm vol}(\pi_{s}^{-1}(U_{s}),
 \tilde{\omega}_{s,t})}{{\rm vol}(X_{s},\tilde{\omega}_{s,t})}$$
do not depend on $t$ once we fix $U_{s}$, 
because of the Ricci-flatness of $\tilde{\omega}_{s,t}$. 
Moreover, this value becomes arbitrarily small if we take small $\mathcal{U}$ 
since each singular fibers are one-dimensional and hence with volume zero. 
Then by \eqref{Bishop-Gromov}, for any given $r>0$, 
 if we take small enough $\mathcal{U}=\bigcup_{s}U_{s}$,
 it holds that 
\begin{equation}\label{intersection.boundary}
B(x,r)\cap \pi_{s}^{-1}(\partial \overline{U}_{s}^{(i)})\neq \emptyset
\end{equation}
for any $s$, $t$ and $x\in \pi_{s}^{-1}(U_{s}^{(i)})$.
On the other hand, from Claim~\ref{Ck.loc.claim}, we obtain that 
$${\rm diam}(\pi_{s}^{-1}(\partial\overline{U}_{s}^{(i)}),\tilde{\omega}_{s,t})
\to {\rm diam}(\partial \overline{U}_{s}^{(i)},\omega_{s,\rm{ML}})
\quad (t\to +0),$$
uniformly with respect to $s$. Hence, there exists 
$t_{0}>0$ such that for any $0<t<t_{0}$ and any $s$, 
\begin{equation*}
{\rm diam}(\pi_{s}^{-1}(\partial\overline{U}_{s}^{(i)}),\tilde{\omega}_{s,t})
<2{\rm diam}(\partial\overline{U}_{s}^{(i)},\omega_{s,\rm{ML}})=:c(U_s^{(i)}). 
\end{equation*}
Then we use the triangle inequality to
 arbitrary two points in $\pi_{s}^{-1}(\partial\overline{U}_{s}^{(i)})$
 with points in (\ref{intersection.boundary}), we obtain 
\[{\rm diam}(\pi_{s}^{-1}(U_{s}^{(i)}),
 \tilde{\omega}_{s,t})<c(U_s^{(i)})+2r.\]
Therefore, taking sufficiently small $\mathcal{U}$ and $r>0$,
 we get the desired estimate by \eqref{ML.boundary}.
\end{proof}

Concerning the behavior of McLean metrics near discriminant points
 such as \eqref{ML.boundary} and \eqref{ML.disk}, 
 see \cite{Yos10}, \cite[Proposition 2.1]{GTZ2}, \cite[Theorem A]{EMM17},
 \cite[Theorem 3.4]{TZ} for results in more general settings. 

To have good estimates of the diameter of singular fiber itself
 is not necessary for our particular purpose, 
 but see \cite{GW, Y.Li} for that direction. 
We will only use the above weaker estimate at the end of 
the proof of Theorem~\ref{GTZ.extend}, which is enough  
 for finding the Gromov-Hausdorff limits.

\subsection{Gromov-Hausdorff convergence}\label{GH.conv} 

With all the above $s$-uniform estimates in our hands, 
 we now prove the $s$-uniform 
 Gromov-Hausdorff convergence Theorem \ref{GTZ.extend} 
 by an argument similar to \cite[\S6]{GW}. 
For that, it is enough to show that 
for any fixed $\delta>0$, there exists small enough $t_{0}$
 such that for any $s$ and 
$(0<)t<t_{0}$, 
\begin{align}
\label{dist.proj}
&{\rm dist}(\pi_{s}):=\sup_{x_1,x_2\in X_s} \{
|d(x_{1},x_{2};\tilde{\omega}_{s,t})
 - d(\pi_{s}(x_{1}),\pi_{s}(x_{2});\omega_{s,{\rm ML}})|\}
<\delta, \\
\label{dist.sec}
&{\rm dist}(\sigma_{s}):=\sup_{y_{1},y_{2}\in B_{s}} \{
|d(y_{1},y_{2};\omega_{s,{\rm ML}})
- d(\sigma_{s}(y_{1}),\sigma_{s}(y_{2});\tilde{\omega}_{s,t}) |
\}<\delta.
\end{align}
Here, $\sigma_s$ is a continuous section for $\pi_s$ and 
 ``dist'' above stands for the distortion function (\cite{BBI}). 

We will consider small enough neighborhood
 $\mathcal{U}=\bigcup_s U_s$ of $\pi_{s}$-critical locus 
 (discriminant locus)
 as in (the proof of) Claim~\ref{sing.fiber.nbhd.claim}. 
Write $U_s=\bigcup_{i}U_s^{(i)}$ for the decomposition
 into connected components.
Outside $\pi_{s}^{-1}(U_s)$, by Claim~\ref{Ck.loc.claim}, 
 we have the $s$-uniform convergence 
\begin{align}\label{sm.loc}
\tilde{\omega}_{s,t}\to \pi_{s}^{*}\omega_{s,\rm{ML}}
\end{align} 
in the $C^0_{\rm loc}(X_s^{\rm sm})$-sense. 

Fix $\delta>0$ and let $\delta':=\frac{\delta}{6}$.
Let us take $\mathcal{U}$
 and $t_0$ such that 
\begin{align}\label{base.bd}
\sum_i{\rm diam}(U_{s}^{(i)},\omega_{s,\rm{ML}})\le \delta',
\quad
\sum_i{\rm diam}(\pi^{-1}(U_{s}^{(i)}),\tilde{\omega}_{s,t})\le \delta'
\end{align}
for any $s$ and $0<t<t_0$.  
This is possible by Claim~\ref{sing.fiber.nbhd.claim} and
 by \eqref{ML.disk}.

For $x_{1}, x_{2}\in X_{s}$,
 we take a curve $\gamma$ connecting $x_{1}$ and $x_{2}$
 such that
\[{\rm length}(\gamma) < d(x_1,x_2;\tilde{\omega}_{s,t}) + \delta'.\]
Then from \eqref{base.bd}, 
 we can find another curve $\gamma'$ by modifying $\gamma$ such that 
$I_{i}:=\{t\in [0,1]\mid \gamma'(t)\in \pi_{s}^{-1}(U_{s}^{(i)})\}$ 
are connected and
\[{\rm length}(\gamma')\leq {\rm length}(\gamma)+\delta'\]
following the arguments of \cite[\S6]{GW}. 
Decompose $\gamma'=\gamma_1'\cup\gamma'_2$ such that
 $\gamma'_1\subset \pi_s^{-1}(B_s\setminus U_s)$
 and $\gamma'_2\subset \pi_s^{-1}(U_s)$.
The curve $\pi_s(\gamma')$ in $B_s$ connects
 $\pi_s(x_1)$ and $\pi_s(x_2)$.
By \eqref{base.bd}, we can modify $\pi_s(\gamma'_2)$
 and get a curve $\bar{\gamma}=\pi_s(\gamma'_1)\cup(\bar{\gamma}\cap U_s)$
 connecting $\pi_s(x_1)$ and $\pi_s(x_2)$ which satisfies
\[{\rm length}(\bar{\gamma}\cap U_s)\leq \delta'.\]
Also, \eqref{sm.loc} implies, for small enough $t_0$
 (independent of $\gamma$ and $s$),
\[{\rm length}(\pi_s(\gamma'_1)) 
\leq {\rm length}(\gamma'_1) + \delta'
\]
for $t<t_0$.
Combining above inequalities, we get
\begin{align*}
d(\pi_s(x_1),\pi_s(x_2))
&\leq {\rm length}(\bar{\gamma})
={\rm length}(\pi_s(\gamma'_1))+{\rm length}(\bar{\gamma}\cap U_s)\\
&\leq {\rm length}(\gamma'_1) + 2\delta'
\leq {\rm length}(\gamma') + 2\delta'
< d(x_1,x_2)+\delta.
\end{align*}

To show the other inequality
 $d(x_1,x_2)\leq d(\pi_s(x_1),\pi_s(x_2))+\delta$,
 take continuous sections $\sigma_s:B_s\to X_s$.
We start from a curve $\bar{\gamma}$
 connecting $\pi_s(x_1)$ and $\pi_s(x_2)$ such that
 ${\rm length}(\bar{\gamma})\leq d(\pi_s(x_1),\pi_s(x_2))+\delta'$.
As above, we can modify $\bar{\gamma}$
 to get another curve $\bar{\gamma}'$ such that
$I_{i}:=\{t\in [0,1]\mid \bar{\gamma}'(t)\in U_{s}^{(i)}\}$ 
are connected and
 ${\rm length}(\bar{\gamma}')\leq {\rm length}(\bar{\gamma})+\delta'$.
Let $\bar{\gamma}'_1:= \bar{\gamma}'\cap (B_s\setminus U_s)$
 and $\bar{\gamma}'_2:= \bar{\gamma}'\cap U_s$.
We then consider the curve $\sigma_s(\bar{\gamma}')$ in $X_s$,
 which connects $\sigma_s(\pi_s(x_1))$ and $\sigma_s(\pi_s(x_2))$.
By modifying $\sigma_s(\bar{\gamma}'_2)$, we get another curve $\gamma$
 which also connects $\sigma_s(\pi_s(x_1))$ and $\sigma_s(\pi_s(x_2))$,
 and ${\rm length}(\gamma\cap \pi_s^{-1}(U_s))\leq \delta'$.
Then \eqref{sm.loc} implies 
\[
{\rm length}(\gamma\cap \pi_s^{-1}(B_s\setminus U_s))
={\rm length}(\sigma_s(\bar{\gamma}'_1))
 \leq {\rm length}(\bar{\gamma}_1)+\delta'
\]
for small $t$.
We therefore have
\begin{align*}
d\bigl(\sigma_s(\pi_s(x_1)),\sigma_s(\pi_s(x_2))\bigr)
\leq d(\pi_s(x_1),\pi_s(x_2))+4\delta'.
\end{align*}
It remains to estimate $d(x_1,\sigma_s(\pi_s(x_1)))$.
By \eqref{sm.loc}, the diameter of the fiber $\pi_s^{-1}(y)$
 for $y\in B_s\setminus U_s$ uniformly goes to zero
 as $t\to 0$.
Hence $d(x_1,\sigma_s(\pi_s(x_1)))< \delta'$ for small $t$
 if $x_1\notin \pi_s^{-1}(U_s)$.
If $x_1\in \pi_s^{-1}(U_s)$, then 
 $d(x_1,\sigma_s(\pi_s(x_1)))< \delta'$
 by \eqref{base.bd}.
The same argument gives $d(x_2,\sigma_s(\pi_s(x_2)))< \delta'$.
We thus proved $d(x_1,x_2)< d(\pi_s(x_1),\pi_s(x_2))+\delta$
 and hence \eqref{dist.proj}. 

The inequality \eqref{dist.sec} directly follows from \eqref{dist.proj}. 

The inequalities \eqref{dist.proj} and \eqref{dist.sec}
 prove the $s$-uniform Gromov-Hausdorff convergence
 $(X_s,\tilde{\omega}_{s,t})\to (B_s,\omega_{s,{\rm ML}})$
 and thus we complete the proof of Theorem~\ref{GTZ.extend}.


\chapter{General K\"ahler K3 surfaces case}\label{Kahler.section}
In this chapter, we extend our framework of \S\ref{F2d.sec} on 
the moduli of \textit{algebraic} polarized K3 surfaces to that of 
all K\"ahler (\textit{not} necessarily algebraic) K3 surfaces.

\section{Satake compactification of moduli of K\"ahler K3 surfaces}
\label{STK.MK3}

The whole moduli space 
of all the Ricci-flat-K\"ahler marked K3 surfaces up to rescaling, 
possibly with ADE singularities is often denoted by $K\Omega$, 
which is a real $59$-dimensional manifold. This is known to 
have a structure of the homogeneous space 
$$SO_{0}(3,19)/(SO(2)\times SO(19)),$$
which can be regarded as one of the connected components
 of $O(3,19)/SO(2)\times O(19)$ 
 (cf.\  \cite{Tod}, \cite{Looi}, then \cite[\S4 Theorem 5]{KT}, 
after \cite{Mor}) through the refined period map (cf., 
\cite{Tod, Looi, KT, Kro89a, Kro89b}, also \S\ref{HK.Kahler}). 
By forgetting the markings (on the second integral cohomology), we get the moduli space 
$$O^+(\Lambda_{\rm K3})\backslash SO_{0}(3,19)/(SO(2)\times SO(19))$$
of K3 surfaces with possibly ADE singularities,
 where $O^+(\Lambda_{\rm K3})$ denotes the
 index two subgroup of $O(\Lambda_{\rm K3})$ preserving
 each connected component of $O(3,19)/SO(2)\times O(19)$.
Let us introduce the equivalence relation $\sim$ on $K\Omega$
 which identifies those connected by the hyperK\"ahler rotation
 and consider the quotient by this relation
 $R\Omega:=K\Omega/\!\!\sim$.
It\footnote{``$R$" of $R\Omega$ comes from the term ``R"iemannian metric"
 as ``$K$" of $K\Omega$ should come from the term ``K"\"ahler metric}
 has a structure 
 $SO_{0}(3,19)/(SO(3)\times SO(19))\simeq O(3,19)/(O(3)\times O(19))$.
Note that $R\Omega$ can be identified with the set of all positive definite
 $3$-dimensional subspaces of $\Lambda_{\rm K3}\otimes \mathbb{R}$.
By forgetting the markings again, the $57$-dimensional topological space 
\begin{align}\label{MK3.describe}
\mathcal{M}_{\rm K3}&:=O(\Lambda_{\rm K3})
 \backslash {\rm Gr}_{3}^{+,{\rm or}}(\Lambda_{\rm K3}\otimes \R)\\ \nonumber
&\simeq 
O^+(\Lambda_{\rm K3})
 \backslash SO_{0}(3,19)/(SO(3)\times SO(19)) \\ \nonumber
 &\simeq O^+(\Lambda_{\rm K3}) 
 \backslash O(3,19)/(O(3)\times O(19))
\end{align}
can be thought of as a space which parametrizes equivalence classes
 of K\"ahler K3 orbisurfaces with respect to the hyperK\"ahler rotations. 
Here, ${\rm Gr}_{3}^{+,{\rm or}}(\Lambda_{\rm K3}\otimes \R)$ denotes 
the set of all positive definite oriented $3$-dimensional real vector spaces in 
$\Lambda_{\rm K3}\otimes \R$, which consists of two connected components. 
The change of orientations switches two components.
 Note that taking such equivalence class of a Ricci-flat K\"ahler K3 surface 
 with possibly ADE singularity, is slightly different from simply taking its underlying Riemannian orbifolds (i.e., regarding as metric spaces), since the holonomy group can be strictly smaller than $Sp(1)$, 
 e.g., $\pm 1$-quotients of flat $2$-dimensional complex tori
 which appear later as the locus $\mathcal{M}_{\rm Km}$ (\S\ref{Kummer.confirm}).  
Since $O(\Lambda_{\rm K3})\setminus O^+(\Lambda_{\rm K3})$ contains
 an element $-\mathrm{id}_{\Lambda_{\rm K3}}$, 
 which acts trivially, we have 
\[\mathcal{M}_{\rm K3}
 \simeq O(\Lambda_{\rm K3})\backslash O(3,19)/(O(3)\times O(19)).\]
The moduli space $\mathcal{M}_{\rm K3}$ is an arithmetic quotient of a Riemannian symmetric space, 
although not Hermitian. 
 Hence it is still possible to compare its Satake compactifications
 with the Gromov-Hausdorff compactification as in the case of $\mathcal{F}_{2d}$ 
 (\S\ref{F2d.sec}).  
Here we take the Satake compactification of adjoint type again,
 which we denote by $\overline{\mathcal{M}_{\rm K3}}^{{\rm Sat},\tau_{\rm ad}}$
 or simply by $\overline{\mathcal{M}_{\rm K3}}^{{\rm Sat}}$.
By the definition of the Satake compactification,
 $\overline{\mathcal{M}_{\rm K3}}^{\rm Sat}$ is stratified by
 a finite number of locally symmetric spaces.

To study the stratification of $\overline{\mathcal{M}_{\rm K3}}^{\rm Sat}$,
 set $\mathbb{G}=O(\Lambda_{\rm K3}\otimes \mathbb{Q})$
 and $G:=\mathbb{G}(\mathbb{R})\simeq O(3,19)$.
The $\mathbb{R}$-rank and $\mathbb{Q}$-rank of $\mathbb{G}$
 are both equal to $3$
 and the $\mathbb{R}$-root system for $G$ is $B_3$.
We label the simple roots as 
\begin{align*}
\begin{xy}
\ar@{-} (0,0) *+!D{\alpha_1} *{\circ}="A"; (10,0) *+!D{\alpha_2} 
 *{\circ}="B"
\ar@{=>} "B"; (20,0)*+!D{\alpha_3}  *{\circ}="C"
\end{xy} 
\end{align*}
The highest weight $\mu$ of the adjoint representation is orthogonal
 to $\alpha_1$ and $\alpha_3$, but not orthogonal to $\alpha_2$.
By the construction of Satake compactification
 (see \cite{Sat2} and \cite{BJ}),  
\[\overline{\mathcal{M}_{\rm K3}}^{\rm Sat}
 = {\mathcal{M}}_{\rm K3} \sqcup \bigsqcup_{f} {\mathcal{M}}_{\rm K3}(f),\]
where $f$ runs over the $O(\Lambda_{\rm K3})$-conjugacy classes
 of $\mu$-saturated rational parabolic subgroups of $\mathbb{G}$.
There are four $G$-conjugacy classes of $\mu$-saturated parabolic subgroups
 of $G$ which are listed as
\begin{align*}
\begin{xy}
\ar@{-} (0,0)  *{\circ}; (10,0)  *{\bullet}="A"
\ar@{=>} "A"; (20,0)  *{\bullet}
\ar@{} "A"; (10,-5)  *{(a)}
\ar@{-} (30,0)  *{\bullet}; (40,0)  *{\bullet}="B"
\ar@{=>} "B"; (50,0)  *{\circ}
\ar@{} "B"; (40,-5)  *{(b)}
\ar@{-} (60,0)  *{\circ}; (70,0)  *{\bullet}="C"
\ar@{=>} "C"; (80,0)  *{\circ}
\ar@{} "C"; (70,-5)  *{(c)}
\ar@{-} (90,0)  *{\bullet}; (100,0)  *{\circ}="D"
\ar@{=>} "D"; (110,0)  *{\bullet}
\ar@{} "D"; (100,-5)  *{(d)}
\end{xy} 
\end{align*}
Here, black nodes are the roots for the Levi components
 of the corresponding parabolic subalgebras.
For example, parabolic subalgebras of type $(a)$, the leftmost one, 
 have Levi component $\mathfrak{so}(2,18)\oplus \mathbb{R}$.

Let us study the $O(\Lambda_{\rm K3})$-conjugacy classes
 of rational parabolic subgroups of those types.
First, the rational parabolic subgroups of type $(a)$ are stabilizers
 of $1$-dimensional isotropic subspaces of
 $\Lambda_{\rm K3}\otimes \mathbb{Q}$.
Hence such parabolic subgroups correspond bijectively to
 isotropic lines.
Since primitive isotropic vectors in $\Lambda_{\rm K3}$
 are unique up to the $O(\Lambda_{\rm K3})$-action, 
 all the rational parabolic subgroups of type $(a)$
 are $O(\Lambda_{\rm K3})$-conjugate.
Similarly, the rational parabolic subgroups of type $(d)$ are stabilizers
 of $2$-dimensional isotropic subspaces of
 $\Lambda_{\rm K3}\otimes \mathbb{Q}$ and they
 are $O(\Lambda_{\rm K3})$-conjugate as well.

Next, the rational parabolic subgroups of type $(b)$ are stabilizers 
 of $3$-dimensional isotropic subspaces of 
 $\Lambda_{\rm K3}\otimes \mathbb{Q}$. 
If $V\subset \Lambda_{\rm K3}\otimes \mathbb{Q}$ is such a $3$-dimensional 
 subspace, let $V_{\mathbb{Z}}=V\cap \Lambda_{\rm K3}$. 
Then the quotient $V_{\mathbb{Z}}^{\perp}/V_{\mathbb{Z}}$ 
 is an even unimodular negative definite lattice of rank $16$. 
Due to a classification result of Witt (cf., \cite[V.1.4, V.2.3.1]{Serre}), 
the even unimodular positive definite lattices of rank $16$ 
 are known to be isomorphic to either 
 $E_8^{\oplus 2}$ or another lattice written as $\Gamma_{16}$ 
 in \textit{loc.cit}. The latter is sometimes also written in other literatures as $D_{16}^{+}$ as an overlattice of $D_{16}$.
Note that the discriminant lattice $A(D_{16})$ of $D_{16}$
 is isomorphic to $(\mathbb{Z}/2\Z)^2$ and possesses a nontrivial 
 isotropic subgroup, which corresponds to $D_{16}^{+}$. 
Let $L$ be one of such unimodular rank $16$ lattice
 and $U$ denotes the unimodular indefinite lattice of rank $2$. 
Then there exists an isomorphism from $U^{\oplus 3}\oplus L$ 
 to $\Lambda_{\rm K3}$ by the uniqueness of 
 the even unimodular lattice of signature $(3,19)$.
If $V_{\Z}$ is an isotropic sublattice of $U^{\oplus 3}$
 of rank $3$, then we have $V_{\Z}^{\perp}/V_{\Z}\simeq L$. 
Conversely, it is easy to see that 
 every isotropic sublattice
 $V_{\Z}\subset \Lambda_{\rm K3}$ of rank $3$ arises in this way.
Therefore, there are exactly two $O(\Lambda_{\rm K3})$-conjugacy classes
 of this type. 
Let us refer to the one corresponding to $\Gamma_{16}$ as type $(b_1)$
 and the other one as type $(b_2)$. 
 
Finally, the rational parabolic subgroups of type $(c)$, are stabilizers
 of flags $(V_1\subset V_3 \subset \Lambda_{\rm K3}\otimes \mathbb{Q})$,
 where $V_i$ are isotropic $i$-dimensional subspaces. We denote 
 $V_{i,\Z}:=V_{i}\cap \Lambda_{\rm K3}$. 
The $O(\Lambda_{\rm K3})$-conjugacy classes of such flags
 correspond bijectively to the isomorphic classes of
 $(V_{3,\Z})^{\perp}/(V_{3,\Z})$.
Hence by the previous classification of the conjugacy class of $V_{3}$, 
we have two $O(\Lambda_{\rm K3})$-conjugacy classes 
 $(c_1)$ and $(c_2)$ similarly to the type $(b)$.

We therefore have
\begin{align*}
\overline{\mathcal{M}_{\rm K3}}^{\rm Sat}
 = {\mathcal{M}}_{\rm K3} \sqcup {\mathcal{M}}_{\rm K3}(a)
 & \sqcup {\mathcal{M}}_{\rm K3}(b_1)  \sqcup {\mathcal{M}}_{\rm K3}(b_2) \\
 & \sqcup {\mathcal{M}}_{\rm K3}(c_1)  \sqcup {\mathcal{M}}_{\rm K3}(c_2)
  \sqcup {\mathcal{M}}_{\rm K3}(d).
\end{align*}
These strata are locally symmetric spaces and 
 their closure relation is given by the 
inclusion relation of the corresponding 
$\mu$-connected real parabolic subgroups of $\mathbb{G}$ as follows (cf., \cite{Sat1, BJ}): 
\[
\xymatrix{
  & \mathcal{M}_{\rm K3} \ar@{-}[dl]
 \ar@{-}[d] \ar@{-}[dr] &  \\
\mathcal{M}_{\rm K3}(b_1) \ar@{-}[d]   & 
\mathcal{M}_{\rm K3}(a) \ar@{-}[dl] \ar@{-}[dr]
   & \mathcal{M}_{\rm K3}(b_2)\ar@{-}[d]  \\
 \mathcal{M}_{\rm K3}(c_1) \ar@{-}[dr] &  
 & \mathcal{M}_{\rm K3}(c_2)\ar@{-}[dl]  \\
 & \mathcal{M}_{\rm K3}(d) &  }
\]
\begin{itemize}
\item 
${\mathcal{M}}_{\rm K3}(a)$ 
 is $36$-dimensional and 
 an arithmetic quotient of the symmetric space $O(2,18)/(O(2)\times O(18))$, 
 \item 
${\mathcal{M}}_{\rm K3}(b_1)$ and ${\mathcal{M}}_{\rm K3}(b_2)$
 are $5$-dimensional and 
 arithmetic quotients of the symmetric space $SL(3,\mathbb{R})/SO(3)$, 
 \item 
${\mathcal{M}}_{\rm K3}(c_1)$ and ${\mathcal{M}}_{\rm K3}(c_2)$
 are $2$-dimensional and 
 arithmetic quotients of the symmetric space $SL(2,\mathbb{R})/SO(2)$, and 
 \item 
${\mathcal{M}}_{\rm K3}(d)$ is a point.
\end{itemize}

We end this section by viewing a relationship between
 $\overline{\mathcal{M}_{\rm K3}}^{\rm Sat}$
 and $\overline{\mathcal{F}_{2d}}^{\rm Sat}$.
There is a natural map $j\colon \mathcal{F}_{2d}\to \mathcal{M}_{\rm K3}$
 sending a polarized K3 surface $(X,L)\in \mathcal{F}_{2d}$
 to the Ricci-flat K\"ahler K3 surface $(X,\omega_X)$ with K\"ahler class $c_1(L)$.
The map $j$ sends $(X,L)$ and its complex conjugate to the same point
 and hence $j\colon \mathcal{F}_{2d} \to j(\mathcal{F}_{2d})$ is generically two-to-one.
In terms of locally symmetric spaces,
\[j\colon
 \tilde{O}^{+}(\Lambda_{2d})\backslash O(2,19)/O(2)\times O(19)
 \to O(\Lambda_{\rm K3})\backslash O(3,19)/O(3)\times O(19)\]
 is induced by the natural inclusion
 $O(2,19)\to O(3,19)$.
Therefore, $j$ extends to a continuous map between Satake compactifications 
 $j\colon \overline{\mathcal{F}_{2d}}^{\rm Sat}\to
 \overline{\mathcal{M}_{\rm K3}}^{\rm Sat}$.
By comparing strata of each compactifications,
 it is easy to see that
 $j(\mathcal{F}_{2d}(l))\subset \mathcal{M}_{\rm K3}(a)$ and 
 $j(\mathcal{F}_{2d}(p))=\mathcal{M}_{\rm K3}(d)$.
Note that the map 
 $j\colon \mathcal{F}_{2d}(l)\to j(\mathcal{F}_{2d}(l))$
 is generically one-to-one or two-to-one
 depending on $d$ and $l$.
This is according to the action of
 $\tilde{O}(\Lambda_{2d})/\tilde{O}^+(\Lambda_{2d})$
 on $\mathcal{F}_{2d}(l)$ is trivial or not.

\section{Geometric meaning of the boundary}\label{Geom.Meaning}
Let us define a map $\Phi$ from $\overline{\mathcal{M}_{\rm K3}}^{\rm Sat}$
 to the set of isometry classes of compact metric spaces with diameter $1$, which we call
 the \textit{geometric realization map}. This gives a geometric meaning 
 to the Satake compactification discussed in the previous section. 
 

 \subsubsection*{Case 1}

For a point in ${\mathcal{M}}_{\rm K3}$, we have a real $4$-dimensional
 Ricci-flat Riemannian orbifold underlying the corresponding 
 K3 orbifold. 
By rescaling, we get a compact metric space with diameter $1$.


\subsubsection*{Case 2}

Let $l$ be a $1$-dimensional isotropic subspace of
 $\Lambda_{\rm K3}\otimes \mathbb{Q}$,
 which is unique up to $O(\Lambda_{\rm K3})$.
Then there is a natural isomorphism
\[
{\mathcal{M}}_{\rm K3}(a)\simeq 
 (O(\Lambda_{\rm K3})\cap {\rm stab}(l))
 \backslash {\rm Gr}_{2}^{+}(l_{\R}^{\perp}/l_{\R}).\]
Here, ${\rm Gr}_{2}^{+}(l_{\R}^{\perp}/l_{\R})$
 denotes the set of all
 positive definite $2$-dimensional subspaces of
 $l_{\R}^{\perp}/l_{\R}$, 
 and ${\rm stab}(l)$ denotes the stabilizer of $l$
 so $O(\Lambda_{\rm K3})\cap {\rm stab}(l)$
 acts on $l_{\R}^{\perp}/l_{\R}$.
Note that if we take a unimodular indefinite lattice
 $U\subset \Lambda_{\rm K3}$ which contains $l$,
 then we have
\[
{\mathcal{M}}_{\rm K3}(a)\simeq 
 O(U^{\perp}) \backslash {\rm Gr}_{2}^{+}(U^{\perp}\otimes \R).
\]

A point in ${\mathcal{M}}_{\rm K3}(a)$
 is represented by 
 $[l,V]$ (cf.\ \eqref{MK3.describe}), 
 where $V\subset l_{\R}^{\perp}/l_{\R}$
 is a positive definite two-dimensional subspace. 
Take a primitive vector $e\in l$.
Take an orthonormal basis $v_1, v_2$ of $V$
 and take their representatives in $l_{\R}^{\perp}$,
 which we also denote by $v_1, v_2$.
Consider a marked K3 surface $(X, \alpha_X)$
 with period $v_1+\sqrt{-1} v_2$ such that $\alpha_X^{-1}(e)$
 is in the closure of the K\"ahler cone of $X$ (cf., e.g., \cite[Chapter 8, Remark 2.13]{Huy}). 
Then from Fact \ref{bp.free}, there is an elliptic fibration structure 
$X\twoheadrightarrow B\simeq \mathbb{P}^{1}$ with the fiber class $e$. 
Hence we can associate $B$ with the McLean metric
 as in \eqref{eq:McLean}.
It is easy to see that the metric space $B$ does not
 depend on the choice of the orthonormal basis $v_1,v_2$ of $V$ and
 their representatives in $l_{\R}^{\perp}$. 
In fact, for different choices of $v_1,v_2$,
 the corresponding tropical K3 surfaces are related
 by the tropical hyperK\"ahler rotation (see \S\ref{trop.K3.1}).
Therefore, to each point in ${\mathcal{M}}_{\rm K3}(a)$,
 we can assign a compact metric space with diameter $1$ 
 homeomorphic to $S^{2}$. 
Because of the presence of complex conjugation as we discussed in \S\ref{trop.K3.1},
 the spheres do not admit canonical orientation. 

The stratum ${\mathcal{M}}_{\rm K3}(a)$ and its closure will be studied
 in detail in terms of Weierstrass models in \S\ref{along.boundary.sec}.


\subsubsection*{Case 3}

Let $V\subset \Lambda_{\rm K3}\otimes \Q$ be an isotropic $3$-dimensional subspace
 such that $V_{\Z}^{\perp}/V_{\Z}\simeq \Gamma_{16}$ 
 as in the previous section.
Then the stratum ${\mathcal{M}}_{\rm K3}(b_1)$ is identified with
 the set of all inner products on $V\otimes \R$
 up to $GL(V_\Z)$-action and rescaling.
If we choose a base point in ${\mathcal{M}}_{\rm K3}(b_1)$, 
 or equivalently, if we choose an inner product on $V\otimes \R$, then 
 we have an isomorphism 
\[{\mathcal{M}}_{\rm K3}(b_1)\simeq 
 GL(V_{\Z})\backslash 
 GL(V\otimes \mathbb{R})/(\mathbb{R}^{\times}\cdot O(V \otimes \mathbb{R})).\] 
From the construction, $V_{\R}$ appears as a limit of degenerating $3$-dimensional subspaces 
 generated by the periods and the K\"ahler class on K3 surfaces. 
${\mathcal{M}}_{\rm K3}(b_1)$ is also isomorphic to $MT_3$, the moduli of compact 
 (un-oriented) flat tori $T=(V\otimes \R)/V_{\Z}$ 
 of dimension $3$ up to rescaling as we saw in \S\ref{Abel.sec}.
The compact tori $T$ has a natural involution
 $\iota$ induced by the $(-1)$-multiplication on $\mathbb{R}^3$.
The flat metric on $T$ descends to a metric on the quotient $T/\iota$, 
 and by rescaling, we get a compact metric space with diameter $1$. 
 
 
 \subsubsection*{Case 4}

Corresponding to ${\mathcal{M}}_{\rm K3}(c_1)$, 
 consider a flag $(V_1\subset V_3 \subset \Lambda_{\rm K3}\otimes \mathbb{Q})$
 of isotropic subspaces such that $V_{3,\Z}^{\perp}/V_{3,\Z}\simeq \Gamma_{16}$.
Then ${\mathcal{M}}_{\rm K3}(c_1)$ is identified with
 the set of all inner products on $(V_3/V_1)\otimes \R$
 up to $GL(V_{3,\Z}/V_{1,\Z})$-action and rescaling.
If we choose a base point in ${\mathcal{M}}_{\rm K3}(c_1)$, 
 or equivalently, if we choose an inner product on $(V_3/V_1)\otimes \R$, then 
 we have an isomorphism 
\[{\mathcal{M}}_{\rm K3}(c_1)\simeq 
GL(V_{3,\Z}/V_{1,\Z})\backslash
 GL((V_{3}/V_{1})\otimes \R))/(\mathbb{R}^{\times}\cdot O((V_{3}/V_{1})\otimes \R)).\] 
${\mathcal{M}}_{\rm K3}(c_1)$ is also isomorphic to $MT_2$, 
 the moduli of compact 
 (un-oriented) flat tori $((V_{3}/V_{1})\otimes \R)/(V_{3,\Z}/V_{1,\Z})$  
 of dimension $2$ up to rescaling. 
As in the case of ${\mathcal{M}}_{\rm K3}(b_1)$, we take a quotient
 of a $2$-dimensional torus by the $(-1)$-multiplication,
 and get a compact metric space with diameter $1$. 
 
 
\subsubsection*{Case 5}

The strata ${\mathcal{M}}_{\rm K3}(b_2)$ and 
 ${\mathcal{M}}_{\rm K3}(c_2)$ also have natural isomorphism
\begin{align*}
&{\mathcal{M}}_{\rm K3}(b_2)\simeq 
 GL(V_{\Z})\backslash 
 GL(V\otimes \mathbb{R})/(\mathbb{R}^{\times}\cdot O(V \otimes \mathbb{R})), \\
&{\mathcal{M}}_{\rm K3}(c_2)\simeq 
GL(V_{3,\Z}/V_{1,\Z})\backslash
 GL((V_{3}/V_{1})\otimes \R))/(\mathbb{R}^{\times}\cdot O((V_{3}/V_{1})\otimes \R)).
\end{align*} 
Here, $V_{\Z}^{\perp}/V_{\Z}$ and $V_{3,\Z}^{\perp}/V_{3,\Z}$
 are isomorphic to $E_8^{\oplus 2}$ instead of $\Gamma_{16}$ as in the Cases 3 and 4 above.
As locally symmetric spaces, they are isomorphic to ${\mathcal{M}}_{\rm K3}(b_1)$ and 
 ${\mathcal{M}}_{\rm K3}(c_1)$, respectively.
The remaining stratum ${\mathcal{M}}_{\rm K3}(d)$ is a point. 
We assign the segment with length $1$ for every point in these three strata. 

\begin{Def}\label{geo.real}
Summing up all the Cases $1$ to $5$ above, we obtain a map 
\[\Phi \colon \overline{\mathcal{M}_{\rm K3}}^{\rm Sat}
\to {\it CMet}_{1},\]
  which we call the \textit{geometric realization map}.
\end{Def}

Here, ${\it CMet}_{1}$ denotes 
 the set of isometry classes of compact metric spaces with diameter one
 equipped with the Gromov-Hausdorff topology.

\begin{Conj}\label{K3.Main.Conjecture2}
The geometric realization map
\[\Phi \colon \overline{\mathcal{M}_{\rm K3}}^{\rm Sat}
\to {\it CMet}_{1}\]
 defined above is continuous.
\end{Conj}

\begin{Rem}
We expect that the above conjectural continuity of $\Phi$
 on a neighborhood of the boundary component $\mathcal{M}_{\rm K3}(b_1)$
 should be compatible with a result of Foscolo~\cite{Fos}, which studied 
 the collapsing of K3 surfaces to three dimensional tori
 modulo the $\mu_2$-action caused by the $(\pm1)$-multiplication. 

\end{Rem}

\begin{Rem}\label{K3.erg}
Another interesting remark could be that, from \cite[Theorem~7]{Moore}, \cite[\S3]{Ver.erg}, 
for ``almost every" compact complex K3 surface $X$ with its K\"ahler cone $K(X)$, with a fixed marking, 
the natural map sends $O(\Lambda_{\rm K3})\times K(X)$ to 
a \textit{dense}(!) subset in $K\Omega^{o}$. 
Here, ``almost every" means that the corresponding subset in $K\Omega^{o}$ has full measure and 
actually the condition is more explicitly studied in \cite[\S4]{Ver.erg} and \cite{Ver.erg.er}. 
In particular, these results of Verbitsky imply that, once the above Conjecture~\ref{K3.Main.Conjecture2} holds, 
it classifies all the possible Gromov-Hausdorff limits (with fixed diameters) of sequences of Ricci-flat K\"ahler metrics on a \textit{fixed} 
K3 surfaces $X$, for almost every $X$. This remark works completely similarly for the higher dimensional hyperK\"ahler case. 
We thank Cristiano Spotti for bringing our attention to the work of Verbitsky \cite{Ver.erg}, and 
also appreciate some discussion with Yosuke Morita and Yuki Arano. 
\end{Rem}

\begin{Rem}\label{remark.HNU2}
As mentioned in Remark~\ref{remark.HNU}, \cite{HU} proved that
 generically the McLean metric in \eqref{eq:McLean} determines the corresponding Jacobian elliptic
 surface up to complex conjugation.
Therefore, the map
 $\Phi|_{\mathcal{M}_{\rm K3}(a)}\colon \mathcal{M}_{\rm K3}(a)\to \Phi(\mathcal{M}_{\rm K3}(a))$
 is generically one to one.
\end{Rem}

\begin{Rem}\label{MK3.to.F2d}
It is easy to see from the definition that the composite of maps $\Phi$ and 
 $j\colon \overline{\mathcal{F}_{2d}}^{\rm Sat}\to \overline{\mathcal{M}_{\rm K3}}^{\rm Sat}$
 given at the end of \S\ref{STK.MK3} 
 equals $\Phi_{\rm alg}$, the geometric realization map for $\mathcal{F}_{2d}$.
Hence Conjecture~\ref{K3.Main.Conjecture2} 
 implies Conjecture~\ref{K3.Main.Conjecture}.
Similarly, Theorems~\ref{K3a.GH.conti}, \ref{K3.Main.Theorem2}, and \ref{Kummer2} below imply
 the corresponding theorems for $\mathcal{F}_{2d}$, namely,
 Proposition~\ref{F2d.boundary.conti}, Theorem~\ref{K3.Main.Conjecture.18.ok},
 and Theorem~\ref{Kummer}, respectively.
\end{Rem}

\medskip

The remaining part of this chapter is devoted to partial confirmation
 of this conjecture. 
We summarize our results as follows. 

\smallskip
 
The first step is the following continuity on the open locus $\mathcal{M}_{\rm K3}$. The continuity  on the locus of smooth K3 surfaces 
 is a corollary to the implicit function 
 theorem again, but at the locus of 
$\mathcal{M}_{\rm K3}$ which 
 parametrizes non-smooth orbifolds, the continuity of $\Phi$ is not trivial. 
 
\begin{Prop}\label{MK3.conti}
The restriction of  $\Phi$ to $\mathcal{M}_{\rm K3}$ is continuous. 
\end{Prop}
Although Proposition~\ref{MK3.conti} seems to be known to experts 
 (see related \cite{Kob90, And92}), we include 
its proof in the next section.

The closure of the boundary component
 $\mathcal{M}_{\rm K3}(a)$ is 
\[
\overline{\mathcal{M}_{\rm K3}(a)}
=\mathcal{M}_{\rm K3}(a)
 \sqcup \mathcal{M}_{\rm K3}(c_1)
 \sqcup \mathcal{M}_{\rm K3}(c_2)
 \sqcup \mathcal{M}_{\rm K3}(d).
\]
The continuity of $\Phi$ restricted to
 $\overline{\mathcal{M}_{\rm K3}(a)}$ can be proved 
 by using the analysis of McLean metrics in \S\ref{along.boundary.sec}. 
By Theorem~\ref{Mwbar.GH.conti} and Remark~\ref{geo.real.MW.MK3} we obtain 

\begin{Thm}\label{K3a.GH.conti}
The restriction of  $\Phi$ to the closure of the $36$-dimensional stratum
 $\overline{\mathcal{M}_{\rm K3}(a)}$ is continuous. 
\end{Thm}

The next theorem is one of our main theorems as it gives global 
moduli-theoretic picture of the general collapsing of K\"ahler K3 surfaces to 
tropical K3 surfaces. 
\begin{Thm}\label{K3.Main.Theorem2}
The restriction of $\Phi$ to
 $\mathcal{M}_{\rm K3} \sqcup \mathcal{M}_{\rm K3}(a)$
 is continuous.
\end{Thm}
\noindent
We prove Theorem~\ref{K3.Main.Theorem2} in
 \S\ref{proof.of.K3.Main.Theorem2}. 

We also confirm the conjecture along the locus of Kummer K3 surfaces, 
 which we will define later and denote by $\mathcal{M}_{\rm Km}$. 
\begin{Thm}\label{Kummer2}
The restriction of $\Phi$ to $\overline{\mathcal{M}_{\rm Km}}$, the closure of $\mathcal{M}_{\rm Km}$ inside 
$\overline{\mathcal{M}_{\rm K3}}^{\rm Sat}$,  is continuous. 
\end{Thm}
\noindent 
Theorem~\ref{Kummer2} provides a substantial evidence to our extension of framework of 
Conjecture \ref{K3.Main.Conjecture} to Conjecture \ref{K3.Main.Conjecture2}. 

We prove the above three partial confirmations of Conjecture \ref{K3.Main.Conjecture2} in the next section. 
 
\section{Partial confirmation of Conjecture \ref{K3.Main.Conjecture2}} 
Here, we partially prove Conjecture~\ref{K3.Main.Conjecture2}, 
including the proofs of Theorems~\ref{K3a.GH.conti}, \ref{K3.Main.Theorem2}, 
\ref{Kummer2}, and some other fundamental related results. 
 
\subsection{Non-collapsing continuity of K\"ahler K3 surfaces}
 \label{MK3.conti.sec}
 
 In this \S\ref{MK3.conti.sec}, we give a proof of 
 Proposition~\ref{MK3.conti} and also show related fundamental results. 
 
 First, the nontriviality of Proposition~\ref{MK3.conti} comes from 
the presence of the discriminant locus in $\mathcal{M}_{\rm K3}$, 
i.e., where the singular orbifolds correspond.
We will see the structure of discriminant locus at the end of \S\ref{proof.of.K3.Main.Theorem2}.

We now give a preparation of the proof of Proposition~\ref{MK3.conti}. 
\begin{Lem}\label{diam.Kahler}
For a bounded subset $S\subset K\Omega$, 
the diameters of the Ricci-flat K\"ahler (orbi-)metrics $(X_s,\omega_{X_s})$ are uniformly bounded above. 
\end{Lem}
This must be known to experts (cf.\  e.g.\ \cite{And92})
 but we write a sketchy proof for convenience after \cite{Tos.lim}. 
\begin{proof}[Proof of Lemma~\ref{diam.Kahler}]
We take a continuous section of $\pi\colon S\to \Omega(\Lambda_{\rm K3})$ as $\sigma(p)$ for $p\in {\rm Im}(\pi)\subset \Omega(\Lambda_{\rm K3})$. 
Now, let us apply \cite[Theorem 3.1]{Tos.lim} or more precisely its refinement which immediately follows from its proof. 
We take $\omega_0$ on $X$ of \textit{loc.cit.}\ as 
 some continuous family of K\"ahler metrics on a simultaneous resolutions
 of $\omega_{X_{\sigma(p)}}$. 
Then we take $\omega$ of \textit{loc.cit.}\ as $\omega_{X_s}$. 
Since the cohomology classes of $\omega_{X_s}$ are bounded,
 $c_1$ and $C_1$ of \textit{loc.cit.}\ can be taken uniformly. 
Hence $C_2, C_3, C_4$
 of \textit{loc.cit.}\ can be also taken uniformly by their definitions. 
Thus, the proof of uniform upper boundedness
 of the diameters ${\rm diam}(X_s,\omega_{X_s})$ is done. 
\end{proof}


We are now ready for the proof of Proposition~\ref{MK3.conti}. 

\begin{proof}[Proof of Proposition~\ref{MK3.conti}]

Take a point $s_0\in K\Omega$ and let $(X_{s_0},\omega_{X_{s_0}})$
 be the corresponding isomorphism class of 
possibly ADE singular Ricci-flat K3 (orbi-)surface. 
It is enough to prove the continuity of $(X_s,\omega_{X_s})$ where $s$ runs over a neighborhood of 
$s_0$, in the Gromov-Hausdorff sense. 

As in \S\ref{Setting},
 we construct a holomorphic deformation of $X_{s_0}$ over a neighborhood $S'$
 of ${\rm pr}(s_0)$ by \cite{Riem} and denote it by $\mathcal{X}\to S'$. 

Fix a small real number $\epsilon>0$. 
We will prove that for small enough $S$, the Gromov-Hausdorff distance of 
$(X_s,\omega_{X_s})$ and $(X_{s_0},\omega_{X_{s_0}})$ are less than $\epsilon$ for an arbitrary $s\in S$, following our discussion in \S\ref{Sing.fiber.nbhd}, \S\ref{GH.conv}. 
We take an open subset
 $\mathcal{U}'=\bigcup_{p\in S'}U_{p}$
 of $\mathcal{X}=\bigcup_{p\in S'}X_{p}$ 
 which contains all the singularities of $X_p$.
For each $s\in S$, 
we set $U_s$ as the pullback of $U_{\rm pr(s)}$ by a (partial) resolution $X_s\to X_{\rm pr(s)}$. 
If we take small enough $\mathcal{U}'$, then 
 we have 
\begin{equation}\label{ADE.loc.vol}
 {\rm vol}(U_s,\omega_{X_s})< \epsilon'^4
\end{equation}
and
\begin{equation}\label{ADE.loc.diam}
{\rm diam}(\partial \overline{U_s},\omega_{X_s})< \epsilon' 
\end{equation}
 for a constant $\epsilon' (\ll \epsilon)$ which is independent of $s$.

Also recall 
\begin{equation}\label{diam.Kahler2}
{\rm diam}(X_s, \omega_{X_s})<C_1
\end{equation} from Lemma~\ref{diam.Kahler}. 
Thus, combining \eqref{ADE.loc.vol} and \eqref{diam.Kahler2}, the Bishop-Gromov inequality implies that there is a constant $C_2>0$ which depends on $C_1$ but independent of $\epsilon'$ such that 
\begin{equation}\label{ball.int2}
B(x,C_2\epsilon')\cap \partial \overline{U_{s}}\neq \emptyset
\end{equation}
for any $s\in S$ and $x\in U_s$. 

We also take a smaller family of open subset 
$$\mathcal{V}':=\bigcup_{p\in S'}V_p$$ satisfying $\overline{V_p}\subset U_p\subset X_p$. 
Again $\mathcal{V}'$ itself is assumed to be open in $\bigcup_{p}X_p$ and 
we set $\mathcal{V}=\bigcup_{s\in S}V_s$ by defining $V_s$ to be the pullback of $V_{\rm pr(s)}$ by the partial resolutions. 

Then, by \cite[Theorem 8]{KT}, \cite[Theorem 21]{Kob90}, 
the metrics on $(X_s\setminus \overline{V_s}, \omega_{X_s})$ vary smoothly. 
From that fact, combined with \eqref{ADE.loc.diam}, \eqref{diam.Kahler2}, \eqref{ball.int2}, 
the Gromov-Hausdorff distance between 
$(X_s,\omega_s)$ and $(X_{s_0},\omega_{s_0})$ is less than $\epsilon$. 
See \S\ref{GH.conv} for the discussion with the same idea. 
\end{proof}


\subsection{General collapsing to tropical K3 surfaces}\label{proof.of.K3.Main.Theorem2}

We prove Theorem~\ref{K3.Main.Theorem2} in this subsection.

\begin{proof}[proof of Theorem~\ref{K3.Main.Theorem2}]
The proof is very similar to that of Theorem~\ref{K3.Main.Conjecture.18.ok}. 
We saw in Theorem~\ref{K3a.GH.conti} that $\Phi$ restricted
 to the closure of ${\mathcal{M}}_{\rm K3}(a)$ is continuous. 
Here, we only need to prove the continuity along a sequence of 
$\mathcal{M}_{\rm K3}$ converging to a point in $\mathcal{M}_{\rm K3}(a)$. 

Fix an isotropic line $l\subset \Lambda_{\rm K3}\otimes \Q$,
 which is unique up to $O(\Lambda_{\rm K3})$-action.
Take a primitive generator $e\in l\cap \Lambda_{\rm K3}$ and
 take an isotropic vector $f\in \Lambda_{\rm K3}\otimes \Q$ such that $(e,f)=1$.
Let $\Lambda':= \{ x\in \Lambda_{\rm K3} \mid (x,e)=(x,f)=0\}$.

For bounded subsets
 $\mathcal{U}, \mathcal{V}\subset \Lambda'\otimes \R$
 and $t>0$, define a subset
 $\mathfrak{S}_{\mathcal{U},t,\mathcal{V}}\subset R\Omega$ by 
 the set of positive definite three dimensional subspaces
 $\langle x_1,x_2,x_3\rangle_{\R} \subset \Lambda_{\rm K3}\otimes \R$
 such that
\begin{align} \label{Siegel.condition2}\nonumber
& \text{$x_1, x_2, x_3$ are orthonormal basis of
 $\langle x_1, x_2, x_3\rangle_{\R}$}, \\ 
& x_1 = c_1 e + v'_1 \quad  (c_1 \in\R,\ v'_1\in \mathcal{V}), \\ \nonumber
& x_2 = c_2 e + v'_2 \quad  (c_2 \in\R,\ v'_2\in \mathcal{V}), \\ \nonumber
& x_3 \in N e + \epsilon f +\epsilon \, \mathcal{U} \quad
 (N>0,\ 0<\epsilon<t).
\end{align}

Consider an arbitrary sequence in $\mathcal{M}_{\rm K3}$
 converging to $[l,V]\in \mathcal{M}_{\rm K3}(a)$,
 where $V\subset l_{\R}^{\perp}/l_{\R}$ denotes a positive definite plane.
By the definition of the Satake topology, we may assume that
 the sequence can be represented in the set
 $\mathfrak{S}_{\mathcal{U},t,\mathcal{V}}$.

We thus consider a sequence of
 Ricci-flat-K\"ahler-metrized ADE K3 surfaces with markings in the sense of
 \cite{Mor}
 $(X_i,\alpha_{i}\colon IH^{2}(X_{i},\Z)\hookrightarrow \Lambda_{\rm K3}, \omega_i)$ 
 with periods $\Omega_i (i=1,2,\cdots)$
 which satisfy the following: 
\begin{enumerate}
\item Three vectors $\alpha_{i}([\re \Omega_i]), \alpha_{i}([\im  \Omega_i]),
 \alpha_{i}([\omega_i])$ have length $1$ and are orthogonal with each other.
\item 
 $\alpha_{i}([\re \Omega_i])\in N_ie+\epsilon_if+\epsilon_i\mathcal{U}$
 and $\epsilon_i\to 0$ as $i\to \infty$.
\item The planes 
 $V_i:=\langle \alpha_{i}([\im \Omega_i]), \alpha_{i}([\omega_i])\rangle_{\R}$
 converge to $V$ as $i\to \infty$
 in the Grassmannian ${\rm Gr}_2(l_{\R}^{\perp}/l_{\R})$. 
\end{enumerate}
Here, $\mathcal{U}\subset \Lambda'\otimes \R$ is a bounded set. 
Let $v_1,v_2$ be an orthonormal basis of $V$.
First, by multiplying appropriate $d_{i}\in \C^{\times}$ to $\Omega_{i}$, 
we can and do assume 
\begin{itemize}
\item $\alpha_{i}([\im \Omega_{i}])\to v_1$ as $i\to +\infty$,
\item $\alpha_{i}([\omega_{i}])\to v_2$ as $i\to +\infty$
\end{itemize}
hold. 

By a hyperK\"ahler rotation,
 we have Ricci-flat-K\"ahler-metrized ADE K3 surfaces
 $(X_i^{\vee}, \omega_i^{\vee})$ with holomorphic volume
 form  such that
\begin{align}\label{hKrot.def}
&\alpha_i([\Omega_i^{\vee}])
 = \alpha_i([\omega_i])
 - \sqrt{-1} \alpha_i([\im \Omega_i]), \\ \nonumber
&\alpha_i([\omega_i^{\vee}]) = \alpha_i([\re \Omega_i]).
\end{align}

From the same arguments as Claim~\ref{e.nef}, it follows that 
\begin{Claim}\label{e.nef.general}
For any bounded subsets $\mathcal{U}, \mathcal{V} \subset \Lambda'\otimes \R$,
 there exists a positive constant $t$ satisfying the following.
Suppose that a marked Ricci-flat K\"ahler ADE K3 surface $(X,\alpha, \omega)$
 satisfies \eqref{Siegel.condition2}
 for $x_1=\alpha([\re \Omega])$,
 $x_2=\alpha([\im \Omega])$ and $x_3=\alpha([\omega])$.
Define $(X^{\vee}, \omega^{\vee})$  by hyperK\"ahler rotation
 as in \eqref{hKrot.def}.
Then $\alpha^{-1}(e)$ is in the closure of the K\"ahler cone of $X^{\vee}$.
\end{Claim}
\noindent
The only 
point we need to explain is that here we consider
 a positive definite form 
$$
(x,y)_{E,\langle v_1,v_2 \rangle}:=-(x,y-2(y,v_1)v_1-2(y,v_2)v_2) 
$$
on $\Lambda'\otimes \R$
instead of $(x,y)_{E,v}=-(x,y-2(y,v)v)$ in the proof of Claim~\ref{e.nef}.

By this claim, $\alpha_i^{-1}(e)$ is in the closure of the K\"ahler cone of
 $X_i^{\vee}$ for large $i$
 so that Fact~\ref{bp.free} ensures the structure of elliptic fibration on
 $X_i^{\vee}$ with the fiber class $\alpha^{-1}(e)$. 

Therefore, we can apply Theorem~\ref{GTZ.extend}
 in the same way as in the proof of Theorem~\ref{K3.Main.Conjecture.18.ok},
 which implies that
 $(X_{i},\omega_{i})$ converges to $\Phi([l,V])$ in the Gromov-Hausdorff sense. 
We complete the proof of Theorem~\ref{K3.Main.Theorem2}. 
\end{proof}

For its own interests of discriminant locus,
 we explain its structure. 
It is known (cf., e.g., \cite{KT}) 
 that there is a codimension $3$ subset 
$\mathcal{M}_{\rm K3}^{\rm ADE}=\bigcup_{\delta^{2}=-2}{\rm Im}\, Z_{\delta}$ 
of $\mathcal{M}_{\rm K3}$ which exactly parametrizes ADE locus \cite[\S5]{KT}, 
where ${\rm Im}\, Z_{\delta}$ is the image of 
$Z_{\delta}\subset {\rm Gr}_{3}^{+,{\rm or}}(\Lambda_{\rm K3}\otimes \R)$
 in $\mathcal{M}_{\rm K3}$,
 which consists of subspaces perpendicular to $\delta$. 
\begin{Def}
Similarly to Remark \ref{Trop.ADE.F2d}, we define the tropical discriminant  
$\mathcal{M}_{\rm K3}(a)^{\rm ADE}$ of $\mathcal{M}_{\rm K3}(a)$
as
$$
\mathcal{M}_{\rm K3}(a)^{\rm ADE}:=
 \overline{\mathcal{M}_{\rm K3}^{\rm ADE}}\cap \mathcal{M}_{\rm K3}(a), 
$$
\end{Def}
Then we can show 
\begin{Claim}[Locally finiteness]\label{delta.finite2}
for a set $U\subset {\rm Gr}_{3}^{+,{\rm or}}(\Lambda_{\rm K3}\otimes \R)$
 of the form \eqref{Siegel.condition2}
 there exist only finitely many 
$O(\Lambda_{\rm K3})\cap {\rm stab}(\R e)$-equivalence classes of 
$\delta$ such that $Z_{\delta}\cap U\neq \emptyset$. 
\end{Claim}
The proof is similar to that of Claim~\ref{delta.finite}. 
$\mathcal{M}_{\rm K3}(a)^{\rm ADE}$ is a real codimension
 $2$ closed subset of $\mathcal{M}_{\rm K3}(a)$ 
 similarly to Remark \ref{Trop.ADE.F2d}.


\subsection{Collapsing of flat Kummer orbifolds}\label{Kummer.confirm}

Let $T$ be a (not necessarily algebraic) 
 complex torus of dimension two, i.e.\ a complex surface
 of the form $\mathbb{C}^2/\Gamma$
 with $\Gamma\simeq \mathbb{Z}^4$, a cocompact subgroup of $\C^2$. 
The quotient $T/\iota$ of $T$ by the $(-1)$-multiplication $\iota$
 can be thought of as a singular K3 surface and 
 has a $16$ $A_1$-singular points.
By blowing up these points, we have
 a smooth K3 surface $\widetilde{T/\iota}$.
If $T$ equips with a flat metric, then it induces a metric on $T/\iota$.
As a metric space, $T/\iota$ is a quotient of real $4$-dimensional
 torus with a flat metric by the $(-1)$-multiplication.
Also, $T/\iota$ has a natural orientation coming
 from the complex structure of $T$.
Let ${\mathcal{M}}_{\rm Km}$ be a moduli space of such
 oriented metric spaces modulo $\mathbb{R}_{>0}$-rescaling.
Then ${\mathcal{M}}_{\rm Km}$ can be identified with the set
 of isomorphism classes of flat metrics on
 the oriented torus $\mathbb{R}^4/\mathbb{Z}^4$.
Hence 
${\mathcal{M}}_{\rm Km}\simeq
 SL(4,\mathbb{Z})\backslash SL(4,\mathbb{R})
 /SO(4)$.

In order to study how ${\mathcal{M}}_{\rm Km}$ is related with
 ${\mathcal{M}}_{\rm K3}$, we recall the definition of the so-called \textit{Kummer lattice}, 
 which we denote by $\Lambda_{\rm Km}$. 
Let us consider a 16-dimensional $\mathbb{Q}$-vector space
 with a basis labeled as $k_{a}$ for
 $a=(a_1,a_2,a_3,a_4)\in \mathbb{F}_2^{\oplus 4}$.
Define a symmetric bilinear form by $(k_{a},k_{b})=-2\delta_{ab}$.
Let $\Lambda_{\rm Km}$ be the lattice of rank $16$ in this vector space
 generated by the vectors
\begin{align*}
&\text{$k_{a}$ for every $a\in \mathbb{F}_2^{\oplus 4}$ and}\\
&\text{$\frac{1}{2}\sum_{a\in h}k_{a}$ for every affine hyperplane
 $h\subset \mathbb{F}_2^{\oplus 4}$}. 
\end{align*} 
It is known that a primitive embedding
 $\Lambda_{\rm Km}\hookrightarrow \Lambda_{\rm K3}$ is unique up to
 ${\rm Aut}(\Lambda_{\rm K3})$ \cite[Lemma 7]{Nik75}.
If we fix such an embedding, 
 the orthogonal complement $\Lambda_{\rm Km}^{\perp}$ 
 in $\Lambda_{\rm K3}$ is isomorphic to $U(2)^{\oplus 3}$, 
 where $U(2)=\mathbb{Z}e\oplus \mathbb{Z}f$ is an indefinite  
 rank $2$ lattice such that $(e,e)=(f,f)=0$ and $(e,f)=2$. 

Suppose that $T$ is a $2$-dimensional complex torus and $\iota$ is the $(-1)$-multiplication (involution) 
on $T$, so that we have a Kummer surface which is the minimal resolution $\widetilde{T/\iota}$ of $T/\iota$ by its definition. 
The Poincar\'e dual of the $16$ exceptional divisors of the minimal resolution in $\widetilde{T/\iota}$ form
 a sublattice of $H^2(\widetilde{T/\iota},\mathbb{Z})$
 with rank $16$ and its saturation is known to be
 isomorphic to $\Lambda_{\rm Km}$ (\cite[Corollary 5]{Nik75}). 
 From the construction, more strongly, we have a natural embedding into the Neron-Severi lattice $\Lambda_{\rm Km}\hookrightarrow {\rm NS}(X)$ 
 if a K3 surface $X$  is a Kummer surface and it also characterizes the class of Kummer surfaces among K3 surfaces, by the theorem of Nikulin 
 \cite[Theorem 1]{Nik75}. We can take a marking 
 $H^2(\widetilde{T/\iota},\mathbb{Z})\simeq \Lambda_{\rm K3}$ 
 such that the classes of $16$ exceptional divisors correspond 
 to $k_a$ and then the period and the K\"{a}hler class of 
 $T/\iota$ lie in the orthogonal complement $\Lambda_{\rm Km}^{\perp}$ of $\Lambda_{\rm Km}$. 
Let $x_i\in \mathbb{R}/\mathbb{Z}\ (i=1,2,3,4)$ be 
 standard coordinates of $T\simeq (\mathbb{R}/\mathbb{Z})^4$. 
The two forms $dx_i\wedge dx_j$ on $T$ for $1\leq i<j\leq 4$ 
 correspond to isotropic vectors 
 $e_{ij}\in \Lambda_{\rm Km}^{\perp}$. 
We have 
\begin{align*}
(e_{ij},e_{i'j'})=
\begin{cases}
\pm 2\  &\text{if $dx_i\wedge dx_j\wedge dx_{i'}\wedge dx_{j'}
 = \pm dx_1\wedge dx_2\wedge dx_3\wedge dx_4$},\\
0 \  &\text{if otherwise}.
\end{cases}
\end{align*}
Hence these $6$ vectors $e_{ij}$ form a lattice
 $U(2)^{\oplus 3}$.
The direct sum
 $U(2)^{\oplus 3}\oplus \Lambda_{\rm Km}$ is a sublattice
 of $\Lambda_{\rm K3}$ of index $2^6$.
The lattice $\Lambda_{\rm K3}$ is generated by
 $U(2)^{\oplus 3}\oplus \Lambda_{\rm Km}$ and the vectors
(see \cite{Gar})
\begin{align*}
&\text{$\frac{1}{2}e_{ij}+
\frac{1}{2}\sum_{a_i=a_j=0}k_{a}$
\ \ \  for $1\leq i< j\leq 4$. }
\end{align*}

We note an isomorphism of symmetric spaces
 $$SL(4,\mathbb{R})/SO(4)
 \simeq O(3,3)/(O(3)\times O(3)),$$
 which can be also identified with the set of
 positive definite $3$-dimensional subspaces 
 of 
 $U(2)^{\oplus 3}\otimes \mathbb{R}$. 
 It is also possible to assign orientation to each $3$-dimensional 
 subspace continuously with respect to the variation of the subspaces, 
 if we regard the above as $SO_{o}(3,3)/(SO(3)\times SO(3))$. 
 We explained that for 
 any flat singular Kummer orbifold $T/\iota$ parametrized in 
 $\mathcal{M}_{\rm Km}$ and its minimal resolution 
 $\widetilde{T/\iota}$, the pullback of the periods and the K\"ahler class 
 as elements of $IH^{2}(T/\iota,\Z)$ (cf., \cite{Mor}) 
are orthogonal to $\Lambda_{\rm Km}\subset \Lambda_{\rm K3}$ for an appropriate 
marking by \cite[Lemma 7]{Nik75}. 
Therefore, the natural embedding 
 $\Psi\colon {\mathcal{M}}_{\rm Km}\hookrightarrow {\mathcal{M}}_{\rm K3}$
 is induced by $U(2)^{\oplus 3}\otimes \R \hookrightarrow \Lambda_{\rm K3}\otimes \R$  and $O(3,3)\hookrightarrow O(3,19)$. 
Let $\overline{\mathcal{M}_{\rm Km}}$ be the closure of
 $\mathcal{M}_{\rm Km}$ in $\mathcal{M}_{\rm K3}$.
Then we can confirm 
 Conjecture \ref{K3.Main.Conjecture2} 
 along $\overline{\mathcal{M}_{\rm Km}}$ as follows.

\begin{proof}[Proof of Theorem~\ref{Kummer2}] 
Consider the algebraic group
 $\mathbb{G}':=O(U(2)^{\oplus 3}\otimes \mathbb{Q})$
 defined over $\mathbb{Q}$ and let $G':=\mathbb{G}'(\mathbb{R})$.
We saw that $\mathcal{M}_{\rm Km}$ is a locally symmetric space for $G'$.
Define $\overline{\mathcal{M}_{\rm Km}}^{{\rm Sat}}$
 to be the Satake compactification of adjoint type.
By the construction of the Satake compactification,
 there is a natural continuous map
 $\Psi:\overline{\mathcal{M}_{\rm Km}}^{{\rm Sat}}
 \to \overline{\mathcal{M}_{\rm K3}}^{{\rm Sat}}$.
 In fact, it is possible to see that
 this map gives a homeomorphism
 of $\overline{\mathcal{M}_{\rm Km}}^{{\rm Sat}}$
 and $\overline{\mathcal{M}_{\rm Km}}$, 
 but we do not use this fact in the proof.

To see the structure of
 $\overline{\mathcal{M}_{\rm Km}}^{{\rm Sat}}$,
 let us fix a maximal split torus of $G'$. 
Define $(\mathbb{R}^{\times})^3$
 as a subgroup of
 $G'=O(U(2)^{\oplus 3}\otimes \mathbb{R})$ by
taking 
\begin{align*}
&e_{12}\mapsto t_1 e_{12},\quad e_{13}\mapsto t_2 e_{13},\quad 
e_{23}\mapsto t_3 e_{23}\\
&e_{34}\mapsto t_1^{-1} e_{34},\quad e_{24}\mapsto t_2^{-1} e_{24},\quad 
e_{14}\mapsto t_3^{-1} e_{14}
\end{align*}
for an element $(t_1,t_2,t_3)\in (\mathbb{R}^{\times})^3$.
The roots 
\[\beta_1((t_1,t_2,t_3))=t_2t_3^{-1},\quad
\beta_2((t_1,t_2,t_3))=t_1t_2^{-1},\quad
\beta_3((t_1,t_2,t_3))=t_2t_3\]
 form the set of simple roots and 
 the corresponding Dynkin diagram is
\begin{align*}
\begin{xy}
\ar@{-} (0,0) *+!D{\beta_1} *{\circ}="A"; (10,0) *+!D{\beta_2} 
 *{\circ}="B"
\ar@{-} "B"; (20,0)*+!D{\beta_3}  *{\circ}="C"
\end{xy} 
\end{align*}
The highest weight $\mu$ of the adjoint representation of $G'$ 
 is orthogonal to $\beta_2$, but not orthogonal
 to $\beta_1$ and $\beta_3$.
Then there are six $G'$-conjugacy classes of $\mu$-saturated parabolic subgroups
 of $G'$, which are listed as
\begin{align*}
\begin{xy}
\ar@{-} (0,0)  *{\bullet}; (10,0)  *{\circ}="A"
\ar@{-} "A"; (20,0)  *{\bullet}
\ar@{} "A"; (10,-5)  *{(a)}
\ar@{-} (30,0)  *{\bullet}; (40,0)  *{\bullet}="B"
\ar@{-} "B"; (50,0)  *{\circ}
\ar@{} "B"; (40,-5)  *{(b_1)}
\ar@{-} (60,0)  *{\circ}; (70,0)  *{\bullet}="C"
\ar@{-} "C"; (80,0)  *{\bullet}
\ar@{} "C"; (70,-5)  *{(b_2)}
\ar@{-} (90,0)  *{\bullet}; (100,0)  *{\circ}="D"
\ar@{-} "D"; (110,0)  *{\circ}
\ar@{} "D"; (100,-5)  *{(c_1)}
\ar@{-} (0,-15)  *{\circ}; (10,-15)  *{\circ}="E"
\ar@{-} "E"; (20,-15)  *{\bullet}
\ar@{} "E"; (10,-20)  *{(c_2)}
\ar@{-} (30,-15)  *{\circ}; (40,-15)  *{\bullet}="F"
\ar@{-} "F"; (50,-15)  *{\circ}
\ar@{} "F"; (40,-20)  *{(d)}
\end{xy} 
\end{align*}
As in the case of $G=O(3,19)$ above, black nodes
 are the roots for the Levi component
 of the corresponding parabolic subalgebra. 
Since all the minimal rational parabolic subgroups of $\mathbb{G}'$
 are conjugate by $GL(4,\mathbb{Z})$,
 every $\mathbb{G}'(\mathbb{Q})$-conjugacy class of
 rational parabolic subgroups is 
 a single $GL(4,\mathbb{Z})$-conjugacy class.
Hence we have a stratification
\begin{multline*}
\overline{\mathcal{M}_{\rm Km}}^{\rm Sat}
 = {\mathcal{M}}_{\rm Km} \sqcup {\mathcal{M}}_{\rm Km}(a)
  \sqcup {\mathcal{M}}_{\rm Km}(b_1)  \sqcup {\mathcal{M}}_{\rm Km}(b_2)\\
  \sqcup {\mathcal{M}}_{\rm Km}(c_1)  \sqcup {\mathcal{M}}_{\rm Km}(c_2)
  \sqcup {\mathcal{M}}_{\rm Km}(d).
\end{multline*}
Note that this is \textit{not} the same compactification as
 the boundary $\partial\overline{A_{2}}^{\rm Sat,\tau_{\rm ad}}$
 of $\overline{A_{2}}^{\rm Sat,\tau_{\rm ad}}\cong \overline{A_{2}}^{\rm T}$ 
discussed in \S\ref{Abel.sec}, which has only one boundary component for 
each dimension. We claim that 
\begin{Claim}
 $\Psi({\mathcal{M}}_{\rm Km}(*))\subset
 {\mathcal{M}}_{\rm K3}(*)$
 for $*=a, b_1, b_2, c_1, c_2, d$. 
 \end{Claim}
 This claim justifies our notation above. 
Let us see this for $*=b_1$. 
The stratum
 ${\mathcal{M}}_{\rm Km}(b_1)$ is isomorphic to
 $GL(3,\mathbb{Z})\backslash
 GL(3,\mathbb{R})/(\mathbb{R}^{\times}\cdot O(3))$
 and it corresponds to the isotropic subspace
 of $U(2)^{\oplus 3}\otimes \mathbb{Q}$
 spanned by $e_{12}, e_{13}, e_{23}$. 
By using the embedding 
 $U(2)^{\oplus 3}\oplus \Lambda_{\rm Km}
 \hookrightarrow \Lambda_{\rm K3}$, 
 we have a $3$-dimensional isotropic subspace $V$
 of $\Lambda_{\rm K3}\otimes \mathbb{Q}$.
Then the lattice 
 $(V\cap \Lambda_{\rm K3})^{\perp}/(V\cap \Lambda_{\rm K3})$
 is generated by $\Lambda_{\rm Km}$ and
\[
\frac{1}{2}\sum_{a_1=a_2=0}k_a, \quad 
\frac{1}{2}\sum_{a_1=a_3=0}k_a, \quad 
\frac{1}{2}\sum_{a_2=a_3=0}k_a.\]
It is easy to see that this lattice is isomorphic to $\Gamma_{16}$.
Therefore,
 $\Psi({\mathcal{M}}_{\rm Km}(b_1))\subset{\mathcal{M}}_{\rm K3}(b_1)$.
Similar arguments show the claim for other strata.

Moreover, it also follows from comparing both spaces that
$\Psi \colon {\mathcal{M}}_{\rm Km}(*) \to {\mathcal{M}}_{\rm K3}(*)$
 is a homeomorphism for each $*=b_1, b_2, c_1, c_2, d$. 
 On the other hand, from the construction, we also see 
 $\mathcal{M}_{\rm Km}(a)$ is a $4$-dimensional 
 locally symmetric space for the group $SL(2)\times SL(2)$. 

To prove Theorem~\ref{Kummer2}, it suffices to show that
 $\Phi\circ\Psi$ is continuous on
 $\overline{\mathcal{M}_{\rm Km}}^{\rm Sat}$. 
Since metric spaces corresponding to ${\mathcal{M}}_{\rm Km}$
 are quotients of $4$-dimensional flat tori by the $(-1)$-multiplication,
 their limits are also quotients of flat tori
 with possibly smaller dimensions. 
By direct calculations as in \S\ref{Abel.sec}, we can confirm
 that $\Phi\circ\Psi$ is continuous.
Here, we do not give detailed arguments,
 but confirm the theorem only in some particular cases. 
 
For example, take a sequence $x_i$ on $SL(4,\mathbb{R})/SO(4)$
 given by $x_i=(t_{i,1},t_{i,2},t_{i,3})\cdot o$,
 where $o$ is a base point in the symmetric space and
 $(t_{i,1},t_{i,2},t_{i,3})$ is an element in the split torus
 of $\mathbb{G}'$ defined above.
Suppose first $t_{i,1}=t_{i,2}=t_{i,3} \to +\infty$ as $i\to \infty$.
Then $x_i$ can be regarded as a sequence of points
 in ${\mathcal{M}}_{\rm Km}$ and then it converges to
 a point in ${\mathcal{M}}_{\rm Km}(b_1)$.
The corresponding metric spaces are quotients of flat tori
 $(\mathbb{R}/\mathbb{Z})^4$.
As $i\to \infty$, the torus $\mathbb{R}/\mathbb{Z}$ corresponding
 to the coordinate $x_4$ shrink and the Gromov-Hausdorff limit
 is three dimensional. 
 
Suppose next that $t_{i,1}=t_{i,2}=t_{i,3}^{-1} \to +\infty$
 as $i\to \infty$.
Then the sequence in ${\mathcal{M}}_{\rm Km}$ converges to a point
 in ${\mathcal{M}}_{\rm Km}(b_2)$.
Regarding corresponding tori,
 the torus $\mathbb{R}/\mathbb{Z}$ corresponding
 to $x_1$ expands, or equivalently, the other three tori
 shrink (recall that we rescale them to normalize the diameter).
Hence the Gromov-Hausdorff limit of those tori is one dimensional and
 its quotient by the involution is a segment.
\end{proof}

\newpage


\chapter{The moduli of tropical K3 surfaces and elliptic K3 surfaces}\label{along.boundary.sec}

In this chapter, we study the behavior of
 tropical K3 surfaces parametrized in
 the $36$-dimensional boundary component
 $\mathcal{M}_{\rm K3}(a)$ of $\mathcal{M}_{\rm K3}$. 
The boundary component ${\mathcal{M}}_{\rm K3}(a)$
 parametrizes Jacobian elliptic K3 surfaces (\textit{modulo complex conjugation}). 
The tropical K3 surfaces we assign via the geometric realization map 
 $\Phi$ are their bases with McLean metrics. 
Our main technical tool is the Weierstrass models of elliptic K3 surfaces. 

\section{GIT moduli compactification for Weierstrass K3 models}
\label{Weier.mod.sec}

\subsection{Stability analysis}

To show expected collapsing of tropical K3 surfaces (or genuine K3 surfaces) to the unit segment, 
 we do an explicit study of Weierstrass models,
 after a suggestion of K.~Ueda. 
We thank him for his kind suggestion. 
Thus, in this \S\ref{Weier.mod.sec}, we prepare a purely algebro-geometric study of the moduli of the Weierstrass models of elliptic K3 surfaces 
and its compactification. 
We begin by recalling the basic of the 
Weierstrass model after Kas \cite{Kas}, Miranda \cite{Mir}, 
on the following explicit structure: 

\begin{Thm}[\cite{Kas}, \cite{Mir}]\label{ADE.Weier.K3}

Any (possibly) ADE singular K3 surface $X$ associated with an elliptic fibration structure 
$\pi\colon X\twoheadrightarrow \mathbb{P}^{1}$ admitting a holomorphic section 
$\sigma\colon \mathbb{P}^{1}\hookrightarrow X$, 
whose fibers are all irreducible, can be explicitly described as follows: 
\begin{align}\label{Weier.form}
X\simeq X^{W}_{g_{8},g_{12}}:=[y^{2}z=4x^{3}-g_{8}(t)xz^{2}+g_{12}(t)z^{3}]\subset 
\mathbb{P}_{\mathbb{P}^{1}}(\mathcal{E}),
\end{align}
where $g_{8}\in \Gamma(\mathcal{O}_{\mathbb{P}^{1}}(8))$, 
$g_{12}\in \Gamma(\mathcal{O}_{\mathbb{P}^1}(12))$, 
$\mathcal{E}:=\mathcal{O}_{\mathbb{P}^{1}}(4)\oplus \mathcal{O}_{\mathbb{P}^{1}}(6)\oplus \mathcal{O}_{\mathbb{P}^{1}}$ on $\mathbb{P}^{1}$ 
and $\pi\colon \mathbb{P}_{\mathbb{P}^{1}}(\mathcal{E})\to \mathbb{P}^{1}$ 
denotes the natural projection. 
Note that 
\begin{align*}
y^{2}z-(4x^{3}-g_{8}(t)xz^{2}+g_{12}(t)z^{3})&\in 
\Gamma(\mathbb{P}_{\mathbb{P}^{1}}(\mathcal{E}),
\mathcal{O}_{\pi}(3)\otimes \pi^{*}\mathcal{O}_{\mathbb{P}^{1}}(12))\\
&\simeq \Gamma(\mathbb{P}^{1},
{\rm Sym}^{3}(\mathcal{E}^{*})\otimes \mathcal{O}_{\mathbb{P}^{1}}(12)). 
\end{align*}
The isomorphism \eqref{Weier.form} 
is over $\mathbb{P}^{1}$, which identifies $\sigma$ with
 the section $\{[0:1:0]\}=\mathbb{P}_{\mathbb{P}^{1}}
(\mathcal{O}_{\mathbb{P}^{1}}(6))
\subset 
\mathbb{P}_{\mathbb{P}^{1}}(\mathcal{E}).$ 
Furthermore, $g_{8}$ and $g_{12}$ satisfy that 
for any $p\in \mathbb{P}^{1}$, 
\begin{align}\label{stab.cond}
\min\{3v_{p}(g_{8}),2v_{p}(g_{12})\}<12. 
\end{align}
Conversely, the surface defined by the above equation satisfying 
(\ref{stab.cond}), 
is a (possibly ADE singular) 
elliptic K3 surface with a holomorphic section and the fibers are all 
irreducible. 
(The condition (\ref{stab.cond})
 is equivalent to that $X$ has only ADE singularities.) 
\end{Thm}

\begin{proof}
As it is convenient, 
 we introduce an inhomogeneous coordinate $t$ on $\mathbb{P}^{1}$
 so that $\mathbb{A}^{1}_{t}\sqcup \{t=\infty\}=\mathbb{P}^{1}_{t}$.
Then we regard $\Gamma(\mathcal{O}_{\mathbb{P}^{1}}(8))$
 (resp.\ $\Gamma(\mathcal{O}_{\mathbb{P}^{1}}(12))$) 
 as polynomials in $t$ of degree $8$ (resp.\  $12$). 
Then the theorem simply follows from the combination
 of \cite[Proposition 5.1]{Mir} and \cite[Lemma 1]{Kas}. 
We note that to apply \cite[Lemma 1]{Kas},
 which is stated for relatively affine projections,
 we need to get rid of the infinity section $[0: 1: 0]$ of $\pi$. 
It does not affect the conclusion since the section only passes through
 the open locus of $X$ where $\pi$ is smooth.
It is because the intersection number
 of the infinity section and the fibers are
 independent of $s$ due to the flatness of $\pi$.
\end{proof}

\begin{Rem}
Also recall that the above explicit equation of $X^{W}_{g_{8},g_{12}}$ 
allows us to see it as the double cover of the Hirzebruch surface 
$F_{4}=\mathbb{P}_{\mathbb{P}^{1}}(\mathcal{O}_{\mathbb{P}^{1}}(4)\oplus
\mathcal{O}_{\mathbb{P}^{1}})$ branched along 
$(4x^{3}-g_{8}(t)xz^{2}+g_{12}(t)z^{3}=0)\subset F_{4}$, 
which is a degree $4$ covering of the base $\mathbb{P}^{1}$. 
\end{Rem}

The above model can be also seen as the relative canonical model 
of elliptic K3 surface with a section 
over $\mathbb{P}^{1}$, i.e., Jacobian elliptic K3 surface, in the terminology of 
the recent 
log minimal model program. As this should be known to experts, we leave the detail 
of its confirmation to the readers. 

Based on \cite{Kas} and \cite{Mir}, 
we give an explicit description of the GIT compact moduli of 
Weierstrass models. As $(g_8,g_{12})$ and $(\lambda^{2}g_8,\lambda^{3}g_{12})$ 
define isomorphic elliptic surfaces over $\mathbb{P}^{1}$, 
we can take a parameter space of $X_{g_{8},g_{12}}^{W}$ naturally as 
\begin{align}\label{WP.par}
\mathbb{P}(
\underbrace{2,2,2,2,2,2,2,2,2}_9,
\underbrace{3,3,3,3,3,3,3,3,3,3,3,3,3}_{13}),
\end{align}
a $21$-dimensional 
weighted projective space, which we denote simply by 
$\mathbb{P}(2^{(9)},3^{(13)})$ or just by $W\mathbb{P}$. 

We denote the discriminant by $\Delta_{24}:=g_{8}^{3}-27g_{12}^{2}
\in \Gamma(\mathcal{O}_{\mathbb{P}^{1}}(24))$. 
By \cite[Proposition 5.1]{Mir}, 
 the stable locus with respect to the action of $SL(2)$ on $W\mathbb{P}$
 is the union of the above ADE locus 
 and the stable locus in the $\Delta_{24}\equiv 0$ case,
 i.e., non-normal surfaces $X_{g_{8},g_{12}}^{W}$. 
Note that $\Delta_{24} \equiv 0$ means
 there exists a section
 $G_4\in \Gamma(\mathbb{P}^{1},\mathcal{O}_{\mathbb{P}}(4))$
 such that $g_8=3G_4^{2}$ and $g_{12}=G_4^{3}$. 
Therefore, the $SL(2)$-action
 on the locus $\{(g_{8},g_{12})\mid \Delta_{24}\equiv 0\}$ 
can be naturally identified with the $4$-th symmetric power of $\mathbb{P}^{1}$
 with natural $SL(2)$-action on it. The GIT stability analysis for the 
latter situation is well-known and classical since the original work \cite{Mum65}: 

\begin{Prop}[cf.\  {\cite[Chapter 3, \S1, \S2 ($n=1$)]{Mum65}}]\label{4pt.GIT}
A section $G_{4}(t)\in \Gamma(\mathbb{P}^{1},\mathcal{O}_{\mathbb{P}^{1}}(4))$ is 
stable (resp., semistable) with respect to the natural $SL(2)$-action 
if its $4$ zeros $p_{1},\cdots,p_{4}$ in $\mathbb{P}^{1}$, 
 with multiplicity, satisfy that 
\begin{align*}
 \#\{i\mid p_{i}=p\}<2\  (\text{resp.,} \le 2) \text{ for any $p\in \mathbb{P}^{1}$}. 
\end{align*}
\end{Prop}
It is also easy to see that the minimal closed orbit inside the 
closure of the unique semistable, non-stable orbit is 
represented by $G_{4}(t)=t^{2}(t-1)^{2}$ for instance, i.e., 
those whose multiplicity type is $(2,2)$. 
This is \textit{polystable}, i.e., a closed orbit in the whole semistable locus. 

In any case, the Geometric Invariant Theory due to Mumford \cite{Mum65} gives an 
$18$-dimensional projective moduli scheme 
$$\overline{M_{W}}:=W\mathbb{P}\sslash SL(2).$$
This includes the open locus $M_{W}$, which parametrizes 
the ADE singular Weierstrass K3 surfaces discussed
 in Theorem~\ref{ADE.Weier.K3}. 
We explicitly determine the structure of this GIT moduli space, 
which we hope to be of independent interest to some readers. 
The boundary $\partial M_{W}:=\overline{M_{W}}\setminus M_{W}$ is given as follows. 

\begin{Prop}[GIT compactification]\label{boundary.Weier}
The boundary $\partial \overline{M_{W}}:=\overline{M_{W}} \setminus M_{W}$ consists of two 
irreducible components as follows. 

\begin{enumerate}
\item  \label{nn.degen}
 A $1$-dimensional component $M_W^{\rm nn}$, parametrizing 
non-normal polystable surfaces 
$$[y^{2}z=(2x-G_{4}(t)z)^{2}(x+G_{4}(t)z)]\subset 
\mathbb{P}_{\mathbb{P}^{1}}(\mathcal{E}),$$
for $G_{4}(t)\in \Gamma(\mathbb{P}^{1},\mathcal{O}_{\mathbb{P}}(4))$. 
\item  \label{ell.sing.degen}
Another $1$-dimensional component $M_W^{\rm seg}$ of $\partial M_{W}$,
 whose Zariski open locus $M_W^{\rm seg, o}$
 parametrizes polystable but non-stable normal surfaces with
 exactly two simple elliptic singularities of type 
 $\tilde{E_8}$. 
Its closure $M_W^{\rm seg}$ intersects with $M_W^{\rm nn}$ only at a point
 which parametrizes the case of \eqref{nn.degen}
 where $G_{4}(t)$ is represented by $t^{2}(t-1)^{2}$ or simply by $t^2$. 
\end{enumerate} 
\end{Prop}

We will denote $M_{W}^{\rm nn}\setminus M_{W}^{\rm seg}$ by $M_{W}^{\rm nn,o}$.

Note that there is a natural map $c\colon M_{W}\to \mathcal{M}_{\rm K3}(a)$
 which identifies elliptic K3 surfaces with its complex conjugates.
Hence $c$ is surjective and generically two to one.
It extends to a map between compactifications.  
We will later see this in \S\ref{MK3.MW}.

\begin{proof}[proof of Proposition~\ref{boundary.Weier}]
Our first observation is that any closed fiber of 
any $X_{g_{8},g_{12}}^{W}$ is isomorphic to 
either a cuspidal rational curve, a nodal rational curve, or an elliptic curve: 
hence always irreducible. 
On the other hand, as they are reduced Cartier divisors
 inside smooth threefold,
 they only have Gorenstein (hence Cohen-Macaulay) singularities. 
Therefore, they are non-normal if and only if the
 singular locus is $1$-dimensional, which must be horizontal with
 respect to the projection onto $\mathbb{P}^{1}$. 
It is equivalent to $\Delta_{24}\equiv 0$,
 which in turn means there exists
 $G_{4}\in \Gamma(\mathbb{P}^{1},\mathcal{O}_{\mathbb{P}^1}(4))$
 such that $g_{8}=3G_{4}^{2}$ and $g_{12}=G_{4}^{3}$. 

From Theorem~\ref{ADE.Weier.K3},
 what we need to do for proving Proposition~\ref{boundary.Weier}
 is to classify $X_{g_{8},g_{12}}^{W}$ which is either 
\begin{enumerate}
\item a non-normal but polystable surface, or 
\item a strictly polystable (non-stable) but still normal surface. 
\end{enumerate} 

For the non-normal case (i), 
 note that the (poly)stability of $X_{3G_{4}^{2},G_{4}^{3}}^{W}$ 
is equivalent to that of $[G_{4}]\in \mathbb{P}(\Gamma(\mathbb{P}^{1},\mathcal{O}_{\mathbb{P}}(4)))$ by \cite[Chapter I, Theorem 1.19]{Mum65}, 
which is explained in Proposition~\ref{4pt.GIT}. 
The consequence is that such polystable non-normal Weierstrass models
 form $\mathbb{P}^{1}\subset \overline{M_{W}}$,
 which is one of the easiest
 special cases of the Deligne-Mumford-Knudsen compactification
 $\overline{M_{0,4}}\simeq \mathbb{P}^{1}$
 parametrized by the classical cross ratio of the four points. 
Among them, the only strictly polystable (i.e., 
polystable but not stable) point is represented by 
$$g_{8}=3G_{4}^{2}=3t^{4}(t-1)^{4},\quad g_{12}=G_{4}^{3}=t^{6}(t-1)^{6},$$ 
i.e., $G_{4}=t^{2}(t-1)^{2}$. 
In that case, the general $\pi$-fiber ($t\neq 0,1$)
 is a nodal rational curve, while $\pi^{-1}(0)$ and $\pi^{-1}(1)$
 are cuspidal rational curves. 
At the two cusps, we have degenerate cusp of
 $T_{2,3,\infty}$-type,
 hence $X$ is in particular semi-log-canonical. 
 
Other stable non-normal surfaces correspond to $G_{4}(t)$ whose 
zeros are all distinct. 
Over each zero point, the $\pi$-fiber is a cuspidal
 rational curve and the cusp is a pinch point (``Whitney's umbrella'') 
as a surface singularity of $X$. We can check this as follows. 
Suppose $t=0$ is one of the zeros of $G_{4}$ and we would like to 
analyze the cusp singularity of $\pi^{-1}(0)$. 
Note that $t^{-1}G_{4}(t)$ is bounded and away from zero
 on a neighborhood of $t=0$. 
So we can replace the coordinate $t$ analytically locally
 by $\tilde{t}:=G_{4}(t)$. 
Then since the Weierstrass equation becomes
 $$y^{2}z=(2x-\tilde{t}z)^{2}(x+\tilde{t}z),$$
 the point $x=y=0$ is a pinch point as a surface singularity. 
Therefore, combining with 
the usual adjunction for the dualizing sheaf, all non-normal 
polystable $X_{g_{8},g_{12}}^{W}$ is semi-log-canonical with 
trivial canonical bundle (dualizing sheaf). 

\medskip

Let us proceed to the case (ii). 
Suppose $(g_{8},g_{12})$ is strictly polystable 
while $\Delta_{24}\not\equiv 0$.
Then recall a result of Miranda \cite[Proposition 5.1]{Mir}, 
which is proved by using the Hilbert-Mumford numerical criterion \cite{Mum65}
 (cf.\  also (\ref{stab.cond}) of Theorem~\ref{ADE.Weier.K3}). 
It proves that from the polystability assumption, it follows 
$$\min\{3v_{p}(g_{8}),2v_{p}(g_{12})\}\le 12$$ 
for any $p\in \mathbb{P}^{1}$ 
 and the equality is attained for some $p$. 
Since the degree of divisor defined by $g_{8}$ is $8$,
 there are at most two (but at least one) such $p$. 
For each such ``de-stabilising'' $p\in \mathbb{P}^{1}$,
 the fiber $\pi^{-1}(p)$ is again a cuspidal rational curve
 and the cusp gives some non-ADE singularity by \cite[Lemma 1]{Kas}. 

Take a destabilizing $p$ and assume $p=0$ without loss of generality
 (by the $SL(2)$-action if necessary).
Then we can write $g_8=c_1 t^4 h_4$ (resp.\ $g_{12}=c_2 t^6 h_6$),
 where $(c_1, c_2)\in \C^{2}\setminus\{(0,0)\}$ and 
 $h_4$ (resp.\ $h_6$) is monic polynomial with respect to 
 $t$ of degree $4$ (resp.\ $6$). 
 Note that one of $c_{i}$s are allowed to be $0$. 

Moreover, by considering the
 $\mathbb{G}_{m}(\mathbb{C})$-action
 $t\mapsto \epsilon t$
 and taking limit with respect to $\epsilon \to 0$ inside $W\mathbb{P}$ 
 (defined around \eqref{WP.par})
 we get the global isotrivial elliptic K3 surface 
 \begin{align}\label{seg.compo}
y^2z=4x^3-c_1t^4xz^2+c_2t^6z^3,
\end{align}
 which is still semistable by \cite[Proposition 5.1]{Mir}. 
Furthermore, the $SL(2)$-orbit of the corresponding point 
\[(g_{8}=c_1 t^{4},\, g_{12}=c_2 t^{6}) \in 
\Gamma(\mathcal{O}_{\mathbb{P}^1}(8))\times \Gamma(\mathcal{O}_{\mathbb{P}^1}(12))\] 
 is clearly a closed subset, 
hence polystable. 
Note that since this is invariant under the 
involution $t\mapsto t^{-1}$ on $\mathbb{P}^{1}$, 
 \eqref{seg.compo} has two simple elliptic singularities of type 
 $\tilde{E_{8}}$  (\cite{Sai}, \cite{Rei}, \cite{Lau}) 
 on the fibers over $0$ and $\infty$, 
 from the equation. 
 $\Delta_{24}\not\equiv 0$ for 
 \eqref{seg.compo} is equivalent to $c_{1}^{3}-27c_{2}^{2}\neq 0$. 
Consequently, 
they form the polystable locus in $\partial\overline{M_W}$,
 which is the other $1$-dimensional boundary component. 
\end{proof}

Let us recall the following well-known theorem due to Kulikov and Pinkham-Persson (see also related \cite{She}) 
for further analysis of the degenerations occurring at the 
boundary of the moduli scheme $\overline{M_W}$.

\begin{Thm}[\cite{Kul}, \cite{PP}]\label{KPP.exist}
Suppose $f\colon \mathcal{X}\to \Delta=\{t\in \mathbb{C}\mid |t|<1\}$ 
is semistable, i.e., $f$ is a holomorphic proper morphism from a smooth 
complex manifold $\mathcal{X}$, 
such that the fibers are K3 surfaces other than $t=0$,
 and the fiber $f^{-1}(0)$ is normal crossing whose components are 
all K\"ahler. Then we can bimeromorphically replace $f^{-1}(0)$ so that 
$K_{\mathcal{X}}$ is relatively trivial over $\Delta$ and 
$f^{-1}(0)$ is still locally normal crossing. 
\end{Thm}

Then the following classification is a well-established classics. 

\begin{Thm}[\cite{Kul}, \cite{Per}]\label{KPP.classif}

Consider a holomorphic flat proper (surjective) family\footnote{such family is often called a ``semistable (minimal) degeneration" or a ``Kulikov degeneration", depending on literatures with 
slightly different conditions.} 
 $f\colon \mathcal{X}\to 
\Delta$ such that 
\begin{itemize}
\item $\mathcal{X}$ is smooth, 
\item $f^{-1}(t)=:X_{t}$ is smooth for $t\neq 0$ and they are all K3 surfaces, 
\item $f^{-1}(0)=:X_{0}$ is normal crossing, 
\item $K_{\mathcal{X}/\Delta}=0$. 
\end{itemize}
Then it can be classified into the following three types. 
We denote the corresponding monodromy action
 of $\pi_{1}(\Delta\setminus\{0\})$ on $H^{2}(X_{t},\mathbb{Z})$
 by $T$ and let $N:=\log(T-1)$. 

\begin{enumerate}

\item[{\rm (Type I)}] $X_{0}$ is (still) a smooth K3 surface. 
$N=0$ in this case. 

\item[{\rm (Type II)}] $X_{0}$ is a union of elliptic ruled surfaces and 
rational surfaces, which form a chain. At both ends of the chain, we have
rational surfaces while the others are all elliptic ruled. The non-normal 
locus is a disjoint union of finite double locus, i.e., 
locally two smooth branches intersect transversally and 
they are all (isomorphic) elliptic curves, which are 
holomorphic sections of each 
elliptic ruled surface. $N^{2}=0$ but $N\neq 0$ in this case. 

\item[{\rm (Type III)}] $X_{0}$ is a union of rational surfaces. 
The intersection of two components are rational smooth curves 
and inside each such double locus, there are exactly $2$ points 
where $3$ analytically locally smooth components intersect. 
$N^{3}=0$ but $N^{2}\neq 0$ in this case.

\end{enumerate}

\end{Thm}
We refer to the literatures \cite{Kul, Per, FM} for 
 the proof and the details of Theorem~\ref{KPP.classif}. 
 
From now on,
 to each of boundary strata 
 of the GIT compactification $\overline{M_{W}}$,
 we specify the ``degeneration type'' (Kulikov type) in the above sense. 
 First, in \S\ref{LC.degen.subsub}, \S\ref{1-dim.boundary.Weierstrass}, 
 and \S\ref{nn.ps.degen}, we do so by analyzing 
 \textit{certain specific}
 algebraic degeneration families. 
These give expectations on each degeneration type, and then later 
in \S\ref{Kul.type}, 
we prove the expectation by essentially more Hodge-theoretic arguments. 
 
\subsection{Log canonical polystable degeneration}\label{LC.degen.subsub}
We first discuss some degenerations toward the boundary component
 $M_{W}^{\rm seg, o}$ of Proposition~\ref{boundary.Weier}. 
We take an arbitrary point in $M_{W}^{\rm seg, o}$ and represent the 
corresponding Weierstrass model whose 
``destabilizing'' point is $p=0\in \mathbb{P}^{1}$ (we can and do 
assume so after ${\rm Aut}(\mathbb{P}^{1})$-action) by 
 \begin{align}\label{seg.compo2}
 y^2z=4x^3-c_1t^4 xz^2+c_2 t^6z^3,
 \end{align}
 with $c_1^3-27c_2^2\neq 0$, 
 following the previous discussion. 
Consider the following standard one parameter degeneration 
 toward the above normal singular surface: 
\begin{align}\label{degen.seg}
X_{s}=[y^{2}z=4x^{3}-c_1 t^4 xz^{2}+(c_{2} t^6+sk_{12})z^{3}],
\end{align}
where $s\in \C$ and 
 $k_{12}\in \Gamma(\mathcal{O}_{\mathbb{P}^{1}}(12))$ 
is of the form: 
$$k_{12}(t)=\prod_{i=1}^{6}(t-a_{i})(t-a_{i}^{-1}),$$
with distinct numbers $a_{1},\cdots,a_{6}\in \mathbb{C}^*$. 
This assumption is, although not essential at all, 
put so that we still have the involution $\iota=\bigcup_{s}\iota_{s}$ on 
$\bigcup_{s}X_{s}$ induced by $t\mapsto t^{-1}$
 for some convenience of discussion. 
Take the open locus $(z\neq 0)\subset \mathbb{P}_{\mathbb{P}^{1}}(\mathcal{O}_{\mathbb{P}^{1}}(4)\oplus \mathcal{O}_{\mathbb{P}^{1}}(6)\oplus \mathcal{O}_{\mathbb{P}^{1}})$, which is the affine 
$\mathbb{A}^{2}_{u, v}$-fiber bundle over the base $\mathbb{P}^{1}$, 
where we put $$u:=\frac{x}{z},\quad v:=\frac{y}{z}$$ from now on. 
Then the intersection of $X_{s}$ with this open subset can be written as 
$$v^{2}=4u^{3}-c_{1}t^{4}u+(c_{2}t^{6}+sk_{12}).$$ 
We take the weighted blow up of $\bigcup_{s}X_{s}$ 
with weights $(6,1,2,3)$ on 
the variables $(s,t,u,v)$ respectively, at the origin $s=t=u=v=0$. 
Then the fiber surface over $s=0$ 
becomes $$Y\cup Y',$$ 
where $Y$ is the strict transform of the original $X_{0}$ while 
canonically we have $Y'\simeq \mathbb{P}(1,2,3)$ 
whose homogeneous coordinates correspond to $t,u,v$ respectively, 
with an $A_{1}$-singularity at $[0:1:0]$ 
and an $A_{2}$-singularity at $[0:0:1]$. 
If we do the same weighted blow up on the (complex) $4$-dimensional 
total space of the 
ambient space, i.e., $\bigcup_{s}\mathbb{P}_{\mathbb{P}^{1}}(\mathcal{E})$, 
then the exceptional divisor $\mathcal{Y}'$ is $\mathbb{P}(6,1,2,3)$ 
where the coordinates are $s,t,u,v$. 
Then $Y'$ can be seen as the weighted hypersurface 
\begin{align}\label{wt.hs.lc.exc}
[v^{2}=4u^{3}-c_{1}t^{4}u+c_{2}t^{6}+s]\subset \mathbb{P}(6,1,2,3)
\end{align}
inside the exceptional divisor $\mathcal{Y}'\simeq \mathbb{P}(6,1,2,3)$. 
Hence, after an appropriate blow up, $Y'$ has a 
natural \textit{rational elliptic surface 
structure} over the base $\mathbb{P}^{1}_{s,t^{6}}$ with homogeneous coordinates 
$s, t^{6}$. 
More precisely, from the weighted hypersurface description 
\eqref{wt.hs.lc.exc}, we see that the projection to 
$\mathbb{P}_{s,t^{6}}^{1}$ 
has indeterminacy only at $[0:0:1:2]$. 

On the other hand, from the general theory \cite{Bri}, \cite{Tju}, 
we can take the simultaneous minimal resolution of the $A_{1}$, $A_{2}$ singularities whose homogeneous coordinates in \eqref{wt.hs.lc.exc} 
are $[-4:0:1:0]$, $[1:0:0:1]$ respectively so that it does not intersect with $Y$, possibly after finite base change 
of the parameter $s$. We denote the strict transform of $Y'$ by 
$\tilde{Y}'$ which is the minimal resolution. 

Also, since the degeneration \eqref{seg.compo2}, \eqref{degen.seg} 
has the involution $\iota$, we have the same 
$\tilde{E_{8}}$ type simple elliptic singularity on the fiber over $t=\infty$ 
at $X_{0}$, we can do the same procedure at that point. 
Then, we get a blow up of $\bigcup_{s}X_{s}$ (possibly after base change 
with respect to $s$) with central fiber over $s=0$ replaced by 
$\tilde{Y}'\cup Y\cup \tilde{Y}''$, where 
$\tilde{Y}'$ (resp., $\tilde{Y}''$) is the 
minimal resolution of $Y'$ (resp., $Y''$). 
This is a simple normal crossing surface and the double locus 
consists of two isomorphic smooth elliptic curves 
$\tilde{Y}'\cap Y$ and $\tilde{Y}''\cap Y$. 
Also, after further blow up at a point of $\tilde{Y}'$ (resp., $\tilde{Y}''$) 
not in $Y$, we get a natural rational elliptic surface structure. 
From these observations, we clearly see that the above degeneration 
\eqref{degen.seg} 
to $X_{0}$ 
is Type II degeneration in Kulikov's sense 
(Theorem~\ref{KPP.classif}) and furthermore 
it is ``stable'' in the sense of \cite[\S3]{Fri}. 
The double locus $\tilde{Y}'\cap Y\simeq \tilde{Y}''\cap Y$, 
$$[v^2=4u^3-c_1 t^4 u+c_2 t^6]\subset \mathbb{P}(1,2,3),$$
is an elliptic curve describing the weight one part of the limit mixed Hodge structure by \cite[Lemma 3.4 (3)]{Fri}.

\begin{Rem}\label{Ftheory1}
It seems that the above degeneration of K3 surfaces is 
close to what has been studied well in the physical context of F-theory, 
and sometimes called ``half K3 surfaces''. 
\end{Rem}

\subsection{Non-normal stable degeneration}\label{1-dim.boundary.Weierstrass}

Next we discuss some degeneration toward the boundary component $M_{W}^{\rm nn, o}$
 of Proposition~\ref{boundary.Weier}. 

Let $G_4(t)$ be a polynomial of degree $4$ with distinct roots.
We consider a one parameter degenerating family 
\begin{align}\label{nn.stable.degen}
y^{2}z=(2x-G_{4}(t)z)^{2}(x+2G_{4}(t)z)+sz^{3}
\end{align} with respect to $s\to 0$, 
and take double cover base 
change for $s$ branching at $s=0$, that is 
$$y^{2}z=(2x-G_{4}(t)z)^{2}(x+2G_{4}(t)z)+s^{2}z^{3}.$$
The double locus $D$ of $X_{0}$ is $2x-G_{4}(t)z=y=s=0$, hence isomorphic to 
$\mathbb{P}^{1}$. We blow this $D$ up. 
Then $X_{0}=[y^{2}z=(2x-G_{4}(t)z)^{2}(x+2G_{4}(t)z)]$ is replaced 
 by a union of two smooth surfaces intersecting at a smooth 
curve $C$. One component is the strict transform of $X_{0}$ 
and the other is the (unique)  
exceptional divisor. Around a neighborhood of any zero of $G_{4}(t)$, $C$ maps 
to $\mathbb{P}^{1}$ with degree $2$ and ramifying at zeros of $G_{4}(t)$. 
Therefore, $C$ is irreducible, smooth, and must be an 
elliptic curve where the projection to 
$\mathbb{P}^{1}$ gives a hyperelliptic structure. 
Note that $D$ does not intersect with $z=0$.
Hence, $C$ 
can be explicitly described as 
$$[y^{2}=(2x-G_{4}(t)z)^{2}G_{4}(t)]\subset \mathbb{P}_{\mathbb{P}^{1}}
(\mathcal{O}_{\mathbb{P}^{1}}(4)\oplus \mathcal{O}_{\mathbb{P}^{1}}(6)),$$ 
where the bundle $\mathcal{O}_{\mathbb{P}^{1}}(4)\oplus 
\mathcal{O}_{\mathbb{P}^{1}}(6)$ are those generated by 
$2x-G_{4}(t)z$, $y$. 
(Note that the ambient space $\mathbb{P}_{\mathbb{P}^{1}}
(\mathcal{O}_{\mathbb{P}^{1}}(4)\oplus \mathcal{O}_{\mathbb{P}^{1}}(6))$ can be 
naturally contracted to a singular quadric $\mathbb{P}(1,1,2)$, along 
$(-2)$-curve to an $A_{1}$-singularity where the equation of the image of 
$C$ becomes easier.)

Therefore, this family 
\eqref{nn.stable.degen} is 
Type II degeneration in the sense of Kulikov-Pinkham-Persson, 
``stable'' in the sense of \cite[\S3]{Fri}, 
and thus again 
by \cite[3.4 (3)]{Fri},  
the weight one part of the corresponding limit mixed Hodge structure is 
represented by the double locus elliptic curve $C$.

\subsection{Non-normal strictly polystable degeneration}
\label{nn.ps.degen}

We consider the same type as
 previous \S\ref{1-dim.boundary.Weierstrass} and 
 set $G_{4}(t)=(t(t-1))^{2}$. 

As discussed in the proof of Proposition~\ref{boundary.Weier}, 
we have two degenerate cusps of
 $T_{2,3,\infty}$-type as the $1$-dimensional ordinary cusp of the 
 fibers over $0$ and $1$.  
Let us consider the family with respect to $s$ 
\begin{align}\label{nn.str.polyst.degen}
X_{s}:=[y^{2}z=(2x-(t(t-1))^{2}z)^{2}(x+2(t(t-1))^{2}z)+sz^{3}]
\end{align}
 and do the same double base change of parameter $s$ to 
$$y^{2}z=(2x-(t(t-1))^{2}z)^{2}(x+2(t(t-1))^{2}z)+s^{2}z^{3}.$$ 
Then we again do the blow up of the double locus $D$ of 
$X_{0}$ defined as $2x-(t(t-1))^{2}z=y=s=0$. 
Then $X_{0}$ is replaced by a normal crossing union of 
smooth surfaces, one is the strict transform of $X_{0}$ 
and the other is the exceptional 
surface which intersects the strict transform at a nodal reducible 
curve $C=C_{1}\cup C_{2}$. 
Each $C_{i}$ is isomorphic to $\mathbb{P}^{1}$ by the projection to the 
base of the elliptic surfaces. 
Thus, this \eqref{nn.str.polyst.degen} is  Type III degeneration in Kulikov's sense
 (Theorem \ref{KPP.classif}). 
 
 We continue our discussion in \S\ref{Kul.type} from a different approach. 

\subsection{Classifying isotrivial elliptic Weierstrass models}\label{Isotrivial.classification}

We now classify $X_{g_8,g_{12}}^{W}$ which is isotrivial in the sense that generic elliptic fibers are isomorphic to each other. 
It means that the $j$-invariant $j:=\dfrac{g_8^3}{\Delta_{24}}$ is constant on 
the Zariski open locus where it is defined. 
From that characterization, such $X_{g_{8},g_{12}}^{W}$ 
consists of the following four possibilities: 

\begin{enumerate}
\renewcommand{\labelenumi}{(\alph{enumi}).}
\item \label{aa} $\Delta_{24}\equiv 0$ ($j\equiv\infty$). 
\item \label{bb} $g_8, g_{12}, \Delta_{24}$ are non-zero
 and $j$ is a constant. 
\item \label{cc} $g_8\equiv 0$. 
\item \label{dd} $g_{12}\equiv 0$. 
\end{enumerate}

The case (a) above is the same as Case (i) discussed in Proposition~\ref{boundary.Weier} and its proof. 

The case (b) above can be written as $g_8=aG_4^{2}$, $g_{12}=bG_4^3$ with 
constants $a,b$ with $ab(a^3-27b^2)\neq 0$, which follows from 
the constancy of $j=\dfrac{g_8^3}{\Delta_{24}}$.

On the other hand, note that for any pair of elliptic curves $E_1, E_2$, 
$(E_1\times E_2)/(\pm 1)\to E_1/(\pm 1)$ or its minimal resolution gives an elliptic fibration structure on the Kummer surface associated to $E_1\times E_2$. 
Here, $\pm$ means the natural $\mu_2(\mathbb{C})$-actions on $E_1$, $E_2$ and its product,
 caused by its group structures or 
 their uniformization by complex vector spaces. 
 The following, which we believe to be known to experts, 
give equivalence of the two families. 

\begin{Prop}\label{Isotriv.Kum}
A Weierstrass model $X_{g_8,g_{12}}^W$ is represented by 
$g_8=aG_4^{2}$, $g_{12}=bG_4^3$ with constants $a,b$
 such that $ab(a^3-27b^2)\neq 0$,
 and that $G_4$ has distinct (four) zeros, 
 if and only if
 it is the Weierstrass model of
 $(E_1\times E_2)/(\pm 1)\twoheadrightarrow 
 E_1/(\pm 1)$
 for some elliptic curves $E_1, E_2$. 
\end{Prop}

\begin{proof}
Note that a Weierstrass model $X_{g_8,g_{12}}^W$
 with $g_8=aG_4^{2}$, $g_{12}=bG_4^3$ as in the proposition
 has four $D_4$-singularities by its explicit description. Therefore, by the classical Kodaira's classification \cite{Kod}, 
its monodromy on the second cohomology of general fibers are 
$\begin{pmatrix}
-1 & 0 \\ 
0 & -1 \\ 
\end{pmatrix}. 
$
Hence, we can take the double ramified covering of $\mathbb{P}^{1}$ 
branched exactly at zeros of $G_4$, 
 which we denote by $E_1\to \mathbb{P}^{1}$. It is explicitly 
 written as the closure of $(x^{2}=G_{4}(t))\subset \mathbb{A}_{t,x}^{2}$ 
 in the singular quadric $\mathbb{P}(1,1,2)$. 
 Then, take its minimal resolution of the total space $X_{g_8,g_{12}}^{W}
 \times_{\mathbb{P}^{1}} E_{1}$, 
 and further take its relative minimal model over $E_{1}$,
 which we denote by  $Y_{\rm min}$. 
By the above singularity and monodromy condition, 
 $Y_{\rm min}\to E_1$ is a $E_2$-fiber bundle. 
Furthermore, pulling back the zero section of $X_{g_{8},g_{12}}^{W}$, 
we have a section for $Y_{\rm min} \to E_1$, 
 which implies $Y_{\rm min}\simeq E_1\times E_2$. 
Then taking the quotient of this by $\mu_2$ and
 taking the relative canonical model (i.e., the Weierstrass model),
 we obtain the proof of one direction; above $X_{g_8,g_{12}}^W$ is 
 $(E_1\times E_2)/(\pm 1)\to E_1/(\pm 1)$. 

On the other hand, consider the Weierstrass model of $(E_1\times E_2)/(\pm 1)
\twoheadrightarrow E_1/(\pm 1)$
 for arbitrary elliptic curves $E_1, E_2$. 
Take an isomorphism 
$E_1/(\pm 1)\simeq \mathbb{P}^1$ so that the image of the $2$-torsion points $E_1[2]$
 are $s_1,s_2,s_3,s_4 \in \mathbb{C}$. 
Then set $G_4(t):=\prod_{i=1}^4(t-s_i)$. 
Also, describe $E_2$ as $y^2=4x^3-ax+b$ and then 
we consider $X_{g_8,g_{12}}^{W}$ 
 where $g_8=aG_4^{2}$, $g_{12}=bG_4^3$. 
Then the relative minimal model of the double base change of
 $X_{g_8,g_{12}}^{W}\to \mathbb{P}^1$ with respect to $E_1\to \mathbb{P}^1$
 is isomorphic to $E_1\times E_2 \to E_1$, by the previous arguments. 
Tracing back the procedure of taking double covering
 and birational transforms, we get an isomorphism
 between $X_{g_8,g_{12}}^{W}$ and the Weierstrass model of $(E_1\times E_2)/(\pm 1)$. 
\end{proof}

From Proposition~\ref{Isotriv.Kum}, the locus of case (b) is naturally isomorphic to the 
product of moduli of elliptic curves 
$\mathbb{A}^{1}\times \mathbb{A}^{1}$ in $M_{W}$ whose closure in 
$\overline{M_{W}}$ includes $\mathbb{A}^{1}\times \{\infty\}=M_{W}^{\rm seg}$ 
and $\{\infty\} \times \mathbb{A}^{1}=M_{W}^{\rm nn}$. 
The $2$-dimensional locus also naturally maps onto 
$\overline{\mathcal{M}_{\rm Km}}\cap \mathcal{M}_{\rm K3}(a)$ (of \S\ref{Geom.Meaning}) 
as we will see more details in next \S\ref{rel.SBB}. 

Case (c) is the locus $[y^2z=4x^3+g_{12}z^3]$ in $W\mathbb{P}$ 
 and its GIT stability for the $SL(2)$-action is 
 equivalent to that of ${\rm Sym}^{12}\mathbb{P}^1$. 
Hence it gives a $9$-dimensional locus in $M_W$, 
which has a structure of locally symmetric space for $U(1,9)$. 
This locus contains a discrete subset parametrizing the 
polytopes with flat metric (union of equilateral triangles) 
discussed in \cite{Thurs}, \cite{Laza}. 

Case (d) is the locus $[y^2z=4x^3-g_{8}xz^2]$ in $W\mathbb{P}$ 
 and its GIT stability for $SL(2)$ is 
 equivalent to that of ${\rm Sym}^{8}\mathbb{P}^1$.  
Hence it gives a $5$-dimensional locus in $M_W$,
 which parametrizes cuboids and
 again has a structure of locally symmetric space for $U(1,5)$. 

\section{From Hodge-theoretic viewpoint}\label{rel.SBB}
\subsection{GIT compactification versus the Satake-Baily-Borel compactification} 
Here, we relate the previously discussed picture of the moduli variety 
constructed by the Geometric Invariant Theory to locally 
Hermitian symmetric space picture through periods. 
As it is well-known, the minimal resolution of the 
Weierstrass models of elliptic K3 surfaces, i.e., 
the Jacobian elliptic K3 surfaces, when associated with a marking, 
are $U$-polarized ample K3 surfaces in the sense of Dolgachev \cite[\S1]{Dol}. 
Here, $U$ denotes the indefinite unimodular lattice of rank $2$.
Therefore, its moduli scheme $M_{W}$ is, through the period map, 
regarded as a subset of 
$18$-dimensional Hermitian locally symmetric space by \cite[\S3]{Dol}. 
More explicitly, fixing a primitive embedding $U\hookrightarrow \Lambda_{\rm K3}$, 
which is 
unique up to $O(\Lambda_{\rm K3})$, we may write 
\begin{align}\label{Wei.p.map}
p_{W}\colon 
M_{W}\hookrightarrow O^{+}(U^{\perp})\backslash \mathcal{D}_{W}.
\end{align}
Here, $\mathcal{D}_{W}$ is a connected component of 
$$\{\C w\in \mathbb{P}(U^{\perp}\otimes \C)
\mid (w,w)=0,\ (w,\bar{w})>0\},$$ 
$U^{\perp}$ denotes the orthogonal complement of $U$ in $\Lambda_{\rm K3}$
and $O^{+}(U^{\perp})$ denotes 
the index $2$ subgroup of $O(U^{\perp})$ 
which preserves the connected component $\mathcal{D}_{W}$. 
This $p_{W}$ is a Zariski open immersion by the local Torelli theorem (cf., also 
\cite[\S2]{Dol}). 

We show that actually $p_W$ is an isomorphism
 in Theorem~\ref{GIT.SBB}, which further shows that such a comparison can be extended to their 
compactifications level. That is, 
we identify our GIT compactification $\overline{M_{W}}$ with the 
Satake-Baily-Borel compactification of the right hand side of 
\eqref{Wei.p.map}, which we denote by 
$\overline{M_{W}}^{\rm SBB}$. 

\begin{Thm}\label{GIT.SBB}
There exists an isomorphism
 $$\overline{p_{W}}\colon \overline{M_{W}} \simeq \overline{M_{W}}^{\rm SBB}$$ 
 of projective varieties extending $p_W$. In particular, 
 $p_W$ itself is an isomorphism: 
$$
M_{W}\simeq O^+(U^{\perp})\backslash \mathcal{D}_{W}. 
$$
Also, via the above isomorphism $\overline{p_{W}}$, 
$M_{W}^{\rm nn,o}\cup M_{W}^{\rm seg,o}$ maps to $1$-dimensional 
cusps and $M_{W}^{\rm seg}\cap M_{W}^{\rm nn}$ maps to 
the $0$-dimensional cusp. 
\end{Thm}

\begin{proof}

First, we extend the natural identity morphism around $M_W^{\rm nn}$. 
We take any point in $M_{W}^{\rm nn, o}$
 and write as $X_{3G_{4}^{2},G_{4}^{3}}$ where 
$G_{4}(t)=t(t-1)(t-2)(t-c)$ with 
$c\neq 0,1,2$. 

Now we regard $\overline{M_{W}}$ as the GIT quotient 
$(\mathbb{C}^{22}\setminus \{0\})\sslash GL(2)$. 
We write $\pi_{W}$ for the quotient map from a subset of $\C^{22}\setminus \{0\}$ 
 to $\overline{M_{W}}$. 
Consider 
$[X_{3G_{4}^{2},G_{4}^{3}}]\in \overline{M_{W}}$ 
and its GIT stable lift to $\C^{22}$, which we denote simply by 
$p=(3G_{4}^{2},G_{4}^{3})$. 
We denote the finite stabilizer of $p$ for the $GL(2)$-action 
by ${\rm stab}(p)$ and 
would like to 
construct an ${\rm stab}(p)$-invariant affine smooth subvariety $V$ 
including $p\in \C^{22}\setminus \{0\}$, which is an \'etale slice in the sense of 
\cite{Luna} (cf., also \cite{Dre} etc.) It is possible by first 
taking a $GL(2)$-invariant affine open neighborhood of $p$ thanks to the 
semistability of the point, 
and then apply the Luna slice theorem (\cite{Luna}), or more precisely the 
version of \cite[Theorem 5.4]{Dre}. 
It is easy to show $\pi_{W}^{-1}(M_{W}^{\rm nn})$ is smooth. 
Therefore, from the conditions 
(especially (v)) of {\it loc.cit.},
 the obtained slice $V$ is smooth and $\pi_{W}^{-1}(M_{W}^{\rm nn})$
 transversally intersects with $V$. 
The intersection is a smooth curve.
 (Note that for generic $p$, the stabilizer ${\rm stab}(p)$ is trivial
 so that the construction of $V$ can be easily made explicit.)

Therefore, 
the extension theorem for the Satake-Baily-Borel compactification 
(\cite[\S4, Theorems 2, 4]{Kie}, cf., also \cite[Theorem A, 3.7]{Bor}, 
\cite[Theorem 1, Corollary]{KO})
applies and we have an extended holomorphic morphism 
$(V\sslash{\rm stab}(p))\to \overline{M_W}^{\rm SBB}$. 
By doing the same thing for all $p$, we get a holomorphic morphism
 $M_W\sqcup M_W^{\rm nn,o}
 \to \overline{M_W}^{\rm SBB}$. 
 
\medskip

Now we proceed to do a similar extension of the map around 
the other $1$-dimensional boundary component $M_W^{\rm seg}$. 
Note that any closed point $q$ of $M_W^{\rm seg, o}$ is represented by 
\[g_8=at^4,\quad g_{12}=bt^6\]
 for some $a, b$ with $ab(a^3-27b^2)\neq 0$. 
Then take a small enough
 $GL(2)$-invariant open affine neighborhood of
 $q:=(at^4,bt^6)\in \C^{22}$ as $U\subset \mathbb{C}^{22}$, which is possible 
by the semistability of $(at^4,bt^6)$. 
Then the strictly semistable (i.e., semistable but not stable) locus inside $U$ is the union of 
a $13$-dimensional subvariety 
\begin{align*}
\{(&g_8=((t-\alpha)(u_{3}t^{3}+u_{2}t^{2}+u_{1}t+u_{0})+A)(t-\alpha)^4, \\ 
&g_{12}=((t-\alpha)(v_{5}t^{5}+v_{4}t^{4}+v_{3}t^{3}+v_{2}t^{2}+v_{1}t+v_{0})+B)(t-\alpha)^6)\in U\},
\end{align*}
where $\alpha, u_{i}, v_{j}, A, B$ are all in $\C$, 
and its image (another $13$-dimensional subvariety) 
translated by the involution $t\mapsto \frac{1}{t}$. 
It is easy to see that the two subvarieties are both 
smooth around the point $q$, 
transversally intersecting, 
so that the intersection is again a $4$-dimensional submanifold of $U$. 

Therefore, by again 
 the same extension theorem for the Satake-Baily-Borel compactification 
 (\cite[\S4, Theorem 2, 4]{Kie}, cf., also \cite[Theorem A, 3.7]{Bor}, 
\cite[Thereom 1, Corollary]{KO}), 
 we get a holomorphic morphism 
 from an analytic neighborhood $U'$ of $q$. 
By the $GL(2)$-invariance of the morphism,
 it extends to a holomorphic map from $GL(2)\cdot U'$. 
Taking all points $q$, we get a holomorphic map from the semistable locus
 $(\C^{22})^{\rm ss}$ of $\C^{22}$. 
This holomorphic map is algebraic
 because that for any affine integral variety ${\rm Spec}(R)=V$,
 $f\in {\rm Frac}(R)$ which extends holomorphically on
 whole $V$ is in $R$. 
Set the GIT quotient morphism
 as $\varpi \colon (W\mathbb{P})^{\rm ss}\to \overline{M_W}$. 
Therefore, we get an algebraic regular $SL(2)$-invariant morphism
 from $\varpi^{-1}(M_W \cup M_W^{\rm seg,o})$ 
to $\overline{M_W}^{\rm SBB}$.  
This descend to $M_W \cup M_W^{\rm seg,o}\to \overline{M_W}^{\rm SBB}$
 by the universality of the GIT quotient 
\cite{Mum65}. 

Therefore, combining the above, we get a regular (algebraic) morphism 
$\varphi \colon \overline{M_W}\setminus (M_W^{\rm nn}\cap M_W^{\rm seg}) \to 
\overline{M_W}^{\rm SBB}$. 
Furthermore, by our discussions in \S\ref{LC.degen.subsub} and \S\ref{1-dim.boundary.Weierstrass}, 
it follows that the image is the complement of the $0$-dimensional cusp. 

We consider the strict transform $\varphi^{-1}_*H$
 of an ample divisor $H$ on $\overline{M_W}^{\rm SBB}$
 which represents multiple of the Hodge line bundle 
and does not pass through the $0$-dimensional cusp. 
It is straightforward, e.g. from the fact that $M_W$ has Picard rank $1$
 (as the complement of the quotient of open subset of
 the Picard rank $1$ variety with higher codimensional Zariski closed subset),
 that $\varphi^{-1}_*H$ is again ample. 
Moreover, both $\overline{M_W}$ and $\overline{M_{W}}^{\rm SBB}$
 are normal.
Hence we conclude that $\varphi$ gives an
 isomorphism by standard arguments 
 (cf.\  \cite{MM}, \cite[3.1.2]{Kol}). 

The last two statements are straightforward by seeing the coincidence of 
the complements, i.e., 
$\overline{p_{W}}(\overline{M_{W}}\setminus M_{W})=\overline{M_{W}}^{\rm SBB}
\setminus (O^{+}(U^{\perp})\backslash \mathcal{D}_{W})$. 
\end{proof}

\begin{Rem}
It seems that the above theorem also follows from the fact that the polystable varieties 
parametrized in $\overline{M_W}$ are all \textit{semi-log-canonical}, hence satisfies \textit{DuBois} property by \cite{Ish, KK}. 
See e.g., \cite[Remark 1.3.13]{GGLR}. 
Also, for interaction between semistability and log canonicity, 
see \cite{Od, Od2, Od0} in a general context. In particular, 
\cite[Corollary~1.1(ii)]{Od0} combined with \cite[Theorem~1.2]{Od} 
implies that the GIT polystability of our Weierstrass models 
parametrized in whole $W\mathbb{P}$ are characterized by the 
K-semistability \cite{Tia, Don} as well. 
\end{Rem}

\begin{Rem}
Before uploading the manuscript on arXiv, 
the first author learned that Kenneth Ascher and Dori Bejleri 
had been also working on related problems on the moduli of elliptic K3 surfaces. (Note added: indeed \cite{AscB} appeared later.) 
We thank K.~Ascher for teaching about it. 
\end{Rem}

Now we give two applications of the above Theorem~\ref{GIT.SBB}. 


 \subsection{Determination of the degeneration type}\label{Kul.type}
 
In \S\ref{LC.degen.subsub}, \S\ref{1-dim.boundary.Weierstrass}, 
 and \S\ref{nn.ps.degen} we studied specific degenerating families toward 
 each boundary component of $\overline{M_{W}}$, which give 
 an expectation on the degeneration types corresponding to each 
 boundary component. Here, we give a Hodge-theoretic proof 
 and the rigorous statements which generalize our algebro-geometric 
 calculation above. Below, by Theorem~\ref{GIT.SBB}, 
 we identify $\overline{M_{W}}$ with $\overline{M_{W}}^{\rm SBB}$. 
 
 \begin{Prop}
 Consider a meromorphic 
 smooth projective family
 $f^{*}\colon \mathcal{X}^{*}\to \Delta^{*}
 =\Delta\setminus \{0\}$ 
 with the relatively ample polarization $\mathcal{L}^{*}$,  
 of polarized Weierstrass models of K3 surfaces. Denote the corresponding 
 morphism to the coarse moduli by 
 $\varphi^{*}\colon \Delta^{*}\to M_{W}$. By the meromorphicity assumption, 
 we mean that 
 $\varphi^{*}$ extends to $\varphi\colon \Delta\to \overline{M_{W}}$. 

 Then the family $f$ is of degeneration type III (resp., II) 
 in the sense of Kulikov-Pinkham-Persson (reviewed in 
Theorem~\ref{KPP.classif}) if and only 
 if $\varphi(0)$ lies in the $0$-dimensional cusp (resp., $1$-dimensional 
cusp).  
 \end{Prop}
 
 \begin{proof}
Recall from \S\ref{Mod.pol.K3} 
that $\mathcal{F}_{2d}$ is uniformized by the 
orthogonal symmetric domain corresponding to $\Lambda_{2d}=
(de_{0}+f_{0})^{\perp}$, the inclusion of two subspaces of 
$\Lambda_{\rm K3}$, 
$\langle e_{0},f_{0} \rangle^{\perp}\hookrightarrow (de_{0}+f_{0})^{\perp}$ 
induces $M_{W}\hookrightarrow \mathcal{F}_{2d}$. This corresponds to 
specify the (family of) degree $2d$ ample line bundles on Weierstrass models. 

It is well-known (at least goes back to \cite{Fri, FriS})
 that for a meromorphic family $f^{*}\colon \mathcal{X}^{*}\to \Delta^{*}$ 
 with the relatively ample polarization $\mathcal{L}^{*}$ of degree $2d$
 which corresponds to $\varphi^{*}\colon \Delta^{*}\to \mathcal{F}_{2d}^{o}$, 
 the family $f^{*}$ is of the degeneration type III (resp., II) in the sense of
 Kulikov-Pinkham-Persson (see Theorem~\ref{KPP.classif})
 if and only if 
$\varphi(0)$ is in the $0$-dimensional cusp
 (resp., $1$-dimensional cusp)
 of the Satake-Baily-Borel compactification
 $\overline{\mathcal{F}_{2d}}^{\rm SBB}$ of $\mathcal{F}_{2d}$. 
Also, it is of the degeneration type I if and only if 
$\varphi(0)\in \mathcal{F}_{2d}$. Since, 
$\overline{M_{W}}\to \overline{\mathcal{F}_{2d}}^{\rm SBB}$ 
sends a $1$-dimensional cusp (resp., a $0$-dimensional cusp) of $\overline{M_{W}}$ 
to a $1$-dimensional cusp (resp., a $0$-dimensional cusp) of 
$\overline{\mathcal{F}_{2d}}^{\rm SBB}$, 
our assertion follows.
We note that if we lift $\varphi(0)$ to a point $[V\subset U^{\perp}\otimes \R]$ in 
${\rm Gr}_{2}^{+,{\rm or}}(U^{\perp}\otimes \R)\subset 
{\rm Gr}_{2}^{+,{\rm or}}(\Lambda_{2d}\otimes \R)$, 
the degeneration type is simply 
determined by the dimension of the linear subspace of 
$V\otimes \R$ where the bilinear form degenerates. 
\end{proof}

\subsection{Relations between compactifications 
  $\overline{\mathcal{M}_{\rm K3}}^{\rm Sat}$ and $\overline{M_{W}}$}\label{MK3.MW}

In \S\ref{Kahler.section} we considered $\mathcal{M}_{\rm K3}$,
 the moduli space of K\"ahler K3 surfaces up to hyperK\"ahler rotation,
 and its Satake compactification $\overline{\mathcal{M}_{\rm K3}}^{\rm Sat}$
 for the adjoint representation.
Recall that one of its boundary component $\mathcal{M}_{\rm K3}(a)$ is a 
$36$-dimensional locally symmetric space, a $(\Z/2\Z)$-quotient of 
$M_W$, by the arguments of Case 2 of \S\ref{Geom.Meaning}. 
The fact corresponds to that $\mathcal{M}_{\rm K3}(a)$ parametrizes \textit{unoriented} metrized spheres, 
forgetting the complex structure on the projective line. 
(In particular, $\mathcal{M}_{\rm K3}(a)$ does not admit a complex structure.) This space $\mathcal{M}_{\rm K3}(a)$ also 
appears in the context of F-theory in string theory 
(cf., e.g., \cite[(5) in \S 1]{ClingherMorgan}). 

Hence, at the compactification level, it follows that 
the closure $\overline{\mathcal{M}_{\rm K3}(a)}$ of $\mathcal{M}_{\rm K3}(a)$ in $\overline{\mathcal{M}_{\rm K3}}^{\rm Sat}$ 
is a $(\Z/2\Z)$-quotient of the Satake-Baily-Borel compactification 
$\overline{M_W}^{\rm SBB}$ of $M_{W}$, from the construction 
of the Satake compactifications (\cite{Sat1, Sat2}, cf.\ also \cite{BJ}). 
Combined with Theorem~\ref{GIT.SBB}, 
it also follows that $\overline{\mathcal{M}_{\rm K3}(a)}$ 
is a $(\Z/2\Z)$-quotient of the GIT compactification 
$\overline{M_W}$ of $M_{W}$ 
studied in our previous section \S\ref{Weier.mod.sec}. 
We denote the corresponding morphism by $c\colon \overline{M_W}\to \overline{\mathcal{M}_{\rm K3}(a)}$. 

Note that $\overline{\mathcal{M}_{\rm K3}(a)}$ is stratified as 
\[\overline{\mathcal{M}_{\rm K3}(a)}=
\mathcal{M}_{\rm K3}(a)\sqcup \mathcal{M}_{\rm K3}(c_1)
 \sqcup\mathcal{M}_{\rm K3}(c_2)\sqcup\mathcal{M}_{\rm K3}(d).\]
By the map $c$,
\begin{enumerate}
\item the stratum $M_{W}^{\rm nn, o}$ maps to $\mathcal{M}_{\rm K3}(c_{1})$, hence corresponding to 
the even unimodular lattice $\Gamma_{16}$ of rank $16$ 
(see also \S \ref{STK.MK3}), 
\item the stratum $M_{W}^{\rm seg, o}$ maps to $\mathcal{M}_{\rm K3}(c_{2})$, hence corresponding to 
the even unimodular lattice $E_{8}^{\oplus 2}$ of rank $16$, and  
\item the stratum $M_{W}^{\rm nn}\cap M_{W}^{\rm seg}$ maps to $\mathcal{M}_{\rm K3}(d)$.
\end{enumerate}
This can be seen for example by studying the locus of Kummer surfaces as in \S\ref{Kummer.confirm}.

\section{McLean metric and its asymptotic behavior}\label{ML.limit}

\subsection{Setting and the statements}

In this section we proceed to the differential geometric side, i.e., 
 the study of the McLean metrics on the base of  elliptic K3 surfaces
 by again using the Weierstrass model 
 and also the study of their Gromov-Hausdorff limits.

We use the setting of Theorem~\ref{ADE.Weier.K3} for Weierstrass K3 surfaces. 
That is, let $g_{8}$ and $g_{12}$ be polynomials in $t$ of degrees 8 and 12, respectively, 
 satisfying \eqref{stab.cond}. 
Then $X_{g_{8},g_{12}}^W:= [y^2z=4x^3-g_8(t)xz^2+g_{12}(t)z^3]$ is
 the Weierstrass model for Jacobian K3 surface. 
After \S\ref{LC.degen.subsub}, we put 
\begin{align*}
u:=\frac{x}{z}\ \text{ and }\ v:=\frac{y}{z},
\end{align*}
so that a holomorphic two form on $X_{g_{8},g_{12}}^W$ is given by
\begin{align*}
\Omega = \frac{du}{v}\wedge dt
 = \frac{du}{\sqrt{4u^3-g_{8}(t)u+g_{12}(t)}}\wedge dt.
\end{align*}
For $a, b \in \mathbb{C}$,  
 let $4u^3-au+b=4(u-\alpha)(u-\beta)(u-\gamma)$.
We assume $a^3-27b^2\neq 0$ so that $\alpha, \beta, \gamma\in \C$
 are distinct numbers.  

We denote the elliptic curve 
$[y^2z=4x^3-axz^2+bz^3]\subset \mathbb{P}^{2}$ by $E_{a,b}$ and 
define the function 
\begin{align*}
\mu(a,b)= \frac{1}{2}
 \,\Bigl|\int_{E_{a,b}}\frac{du}{v}\wedge \overline{\Bigl(\frac{du}{v}\Bigr)}\,\Bigr|. 
\end{align*}
If we define 
\begin{align*}
&\sigma_1(a,b)=\int_{\alpha}^{\infty} \frac{du}{\sqrt{(u-\alpha)(u-\beta)(u-\gamma)}}, \\
&\sigma_2(a,b)=\int_{\gamma}^{\infty} \frac{du}{\sqrt{(u-\alpha)(u-\beta)(u-\gamma)}}, 
\end{align*}
with suitable contours, 
$\mu(a,b)=|\!\im (\sigma_1(a,b)\overline{\sigma_2(a,b)})|$ holds. 
The McLean metric $\omega_{s,{\rm ML}}$ with respect to $\Omega$ is given by
 $$\mu(g_8(t),g_{12}(t)) dt\otimes d\bar{t},$$ as shown in 
 e.g.\ \cite{GW}. 
 It is well-known that this has bounded diameter, 
 as we write an elementary proof of it later for convenience as 
 Corollary \ref{McLean.finite.distance}. Therefore we get a distance structure on the base $\mathbb{P}^{1}$ and write  $B_{g_{8},g_{12}}$, 
 for this metric space. We define the following partial geometric realization, which will be related to $\Phi$ of \S\ref{Geom.Meaning} later. 

\begin{Def}\label{geo.real.ML}
We define the map 
\[\Phi_{\rm ML}:
 \overline{M_W}\to {\it CMet}_{1}\]
as follows (recall that ${\it CMet}_{1}$ denotes the set of isometry classes of compact metric spaces with diameter one equipped with the Gromov-Hausdorff topology); 
\begin{itemize}
\item 
For a point $p$ in $M_W$ represented by $(g_8,g_{12})$ we take the corresponding
 metric space $B_{g_{8},g_{12}}$. 
 \item 
A point in $M_W^{\rm nn}\setminus M_W^{\rm seg}$
 is represented by 
 $(g_{8},g_{12})=(3G_4(t)^2,G_4(t)^3)$
 for $G_4(t)=t(t-1)(t-2)(t-c)$ with $c\neq 0,1,2$. 
The metric 
 $|G_4(t)|^{-1}dt\otimes d\bar{t}$ 
 makes $\mathbb{P}^1$ the compact metric space,
 which we denote by $B_{G_4}$.  
 \item 
For a point $M_W^{\rm seg}$
 (including the intersection point $M_W^{\rm seg}\cap M_W^{\rm nn}$),
 we take a segment. 
 \end{itemize}
Then, the map $\Phi_{\rm ML}$ is defined by assigning
 the spaces above with the metrics
 rescaled so that the diameters become $1$. 
\end{Def}

The main theorem of this section is as follows. 

\begin{Thm}\label{Mwbar.GH.conti}
The map 
\[\Phi_{\rm ML}:
 \overline{M_W}\to {\it CMet}_{1}
\]
 defined above is continuous.
\end{Thm}

\begin{Rem}\label{geo.real.MW.MK3}
It is easy to see that the two geometric realization maps $\Phi$ defined
 in Definition~\ref{geo.real}, and $\Phi_{\rm ML}$ above are
 related in the sense that $\Phi|_{\overline{\mathcal{M}_{\rm K3}(a)}} \circ c=\Phi_{\rm ML}$,
 where $c$ is the quotient map by $\Z/2\Z$ given in \S\ref{MK3.MW}. 
Hence Theorem~\ref{Mwbar.GH.conti} implies Theorem~\ref{K3a.GH.conti}.
\end{Rem}

We split the proof of Theorem~\ref{Mwbar.GH.conti} to 
that of Propositions~\ref{prop:MW.conti}, \ref{prop:nn.conti} and 
\ref{prop:seg.conti}. The completion of the proof will be at 
the end of this whole section. 

Before going to their precise proofs of each proposition, 
we explain below in \S\ref{seg.infinite} that 
the last assignment of the segment in Definition~\ref{geo.real.ML} 
by $\Phi_{\rm ML}$ is partially motivated by the 
following analysis. 

\subsection{Parametrizing open surfaces of infinite diameters at $M_W^{\rm seg}$}\label{seg.infinite} 

For any $X_{g_{8},g_{12}}^{W}$ parametrized in the $1$-dimensional 
boundary (\ref{ell.sing.degen}), i.e., for $g_{8}, g_{12}$ with 
some $p\in \mathbb{P}^{1}$ such that 
\begin{align}\label{destabilizing}
\min\{3v_{p}(g_{8}),2v_{p}(g_{12})\}=12,
\end{align}
and we call $p$ satisfying the above equality \eqref{destabilizing}
 {\it destabilizing}. 
There are at most two destabilizing points. 
We can still think of the McLean metric on the base of such
 an elliptic K3 surface as the generic $\pi$-fibers
 are still elliptic curves and $X_{g_{8},g_{12}}^{W}$ is 
Gorenstein with trivial canonical bundle. 
Take its nonzero holomorphic section as $\sigma_{X}$. Then recall that 
the McLean metric on $\mathbb{P}^{1}\setminus {\rm disc}(\pi)$ 
is defined as $$g(v,w):=-\int_{\pi^{-1}(q)} 
\iota(\tilde{v})\re \sigma_{X}\wedge 
\iota(\tilde{w})\im \sigma_{X},$$
where $v, w\in T_{q}\mathbb{P}^{1}$ with $q\in U\setminus \{p\}$ and 
$\tilde{v}$ (resp., $\tilde{w}$) is a lift of $v$ (resp., $w$). 
The definition of the above metric does not depend on the choice of such lifts. 
By Corollary~\ref{McLean.finite.distance} or 
 \cite[Proposition~2.1]{GTZ2} for instance, 
at least away from 
 one or two destabilizing points, the diameters are bounded. 

Now, we take a minimal resolution of $X=X_{g_{8},g_{12}}^{W}$ 
and denote it by $\varphi\colon \tilde{X}\to X$. Then we take the relative 
minimal model of $\tilde{X}$ over $\mathbb{P}^{1}$ and denote it by 
$\pi_{{\rm min}}\colon X_{{\rm min}}\to \mathbb{P}^{1}$. The 
composite of vertical $(-1)$-curves contraction obtained as a result of 
the relative minimal model program, will be denoted by $\psi\colon \tilde{X}
\to X_{{\rm min}}$. Since $K_{X_{{\rm min}}}$ is relatively 
$\pi_{{\rm min}}$-trivial, we can trivialize it on $\pi_{\rm min}^{-1}(U)$ for a small 
open neighborhood $U$ of 
$p$ for 
each destabilizing $p$. We take a generating holomorphic section of 
$K_{X_{{\rm min}}}|_{\pi_{\rm min}^{-1}(U)}$ as $\sigma_{{\rm min}}$. 

Now we want to compare our McLean metric $g$ on 
$U(\setminus \{p\})$ with $g_{{\rm min}}$ on $U$ defined as 
$$g_{{\rm min}}(v,w):=-\int_{\pi_{{\rm min}}^{-1}(q)}
\iota(\tilde{v})\re \sigma_{{\rm min}}
\wedge \iota(\tilde{w})\im \sigma_{{\rm min}},$$
where $v, w\in T_{q}\mathbb{P}^{1}$ with $q\in U\setminus \{p\}$, 
$\tilde{v}$ (resp., $\tilde{w}$) is a lift of $v$ (resp., $w$). 
By \cite[Proposition~2.1]{GTZ2}, we know that $g_{{\rm min}}$ is a bounded usual 
K\"ahler metric through $p$, on whole $U$. 

Since $X$ has non-canonical (i.e., non-ADE) Gorenstein singularity, 
for our minimal resolution $\varphi\colon \tilde{X}\to X$, we have that 
$K_{\tilde{X}}-\varphi^{*}K_{X}$ is non-zero anti-effective, more strongly, 
all the coefficients along exceptional divisors are negative 
due to the negativity lemma (cf., \cite[Lemma 3.39]{KM}) and the $\varphi$-nefness. 
This implies that if we write 
$$\varphi^{*}(\sigma_{X}|_{\pi^{-1}(U)})=f\cdot \psi^{*}\sigma_{{\rm min}},$$ 
then $f\colon U\to \mathbb{C}$ is holomorphic and 
vanishes only at the destabilizing points $p$. 
Therefore, we have 
\begin{align*}
g(v,w)&=-\int_{(\pi\circ\varphi)^{-1}(q)}
\iota(\tilde{v})\re\varphi^{*}\sigma_{X}\wedge 
\iota(\tilde{w})\im\varphi^{*}\sigma_{X}\\
&=-\int_{(\pi\circ\varphi)^{-1}(q)}
\dfrac
{\iota(\tilde{v})\re\psi^{*}\sigma_{{\rm min}}\wedge 
\iota(\tilde{w})\im\psi^{*}\sigma_{{\rm min}}}
{|f|^{2}}\\ 
&=\dfrac{g_{{\rm min}}(v,w)}{|f|^{2}}. 
\end{align*}
Recall that as ${\rm deg}(g_{8})=8$ and (\ref{destabilizing}), 
we only have at most two destabilizing points. 
Therefore, this $g$ looks ``infinitely long surface" with either one or two 
punctured ends. Later in \S\ref{sec:conv.to.MWseg}, 
we rigorously show the corresponding convergence toward the segment 
of length $1$.

\subsection{General estimate of McLean metric}

To see the asymptotic behavior of the McLean metrics
 we need the following estimate of $\mu$, 
 in which $\log|a^{3}-27b^{2}|$ is the key term. 
\begin{Lem}\label{ML.estimate}
There exist constants $c>0$ and $C>0$ such that 
\begin{align}\label{ell.integral.estimate}
 -C\leq 
 (|a|^{\frac{1}{2}}+|b|^{\frac{1}{3}}) \mu(a,b)
  + c \log \frac{|a^3-27b^2|}{|a|^3+27|b|^2} \leq C
\end{align}
for any $a,b \in \mathbb{C}$ with $a^3-27b^2\neq 0$.
\end{Lem}

\begin{proof} 
The proof is essentially done by estimating 
elliptic integrals. 
Let $4u^3-au+b=4(u-\alpha)(u-\beta)(u-\gamma)$.
We may assume $0\neq |\alpha|\geq |\beta|, |\gamma|$
 and $\re \beta \geq \re \gamma$.
If we replace $a$ and $b$ by $\lambda^2a$ and $\lambda^3 b$
 for a complex number $\lambda$,
 then $\alpha,\beta,\gamma$ are replaced by
 $\lambda\alpha, \lambda\beta, \lambda\gamma$, respectively.
Then $\mu(\lambda^2 a, \lambda^3 b) = |\lambda|^{-1}\mu(a,b)$.
Hence the equation \eqref{ell.integral.estimate} for $a,b$
 and that for $\lambda^2a,\lambda^3b$ are equivalent.
We may thus assume $\alpha=1$.
Since $\alpha+\beta+\gamma=0$, we have $\beta+\gamma=-1$ and
 $\re \beta \geq -\frac{1}{2} \geq \re \gamma$.
Putting $\xi:=\gamma+\frac{1}{2}$,
 we have $\re \xi \leq 0$ and $|\xi-\frac{1}{2}|\leq 1$.

Let $D:=\{z\in \mathbb{C} : |z-\frac{1}{2}|\leq 1
 \text{ and }\re z\leq 0\}$.
By the above argument, it is enough to prove \eqref{ell.integral.estimate}
 for $\xi \in D$ and 
\begin{align}\label{a.b.condition}
 \nonumber
&\alpha=1,\quad \beta=-\frac{1}{2}-\xi,\quad \gamma=-\frac{1}{2}+\xi \\
&a=-4(\alpha\beta+\beta\gamma+\gamma\alpha)
 =4\xi^2+3, \\ \nonumber
&b=-4\alpha\beta\gamma=4\xi^2-1.
\end{align}
We note that both $|a|^{\frac{1}{2}}+|b|^{\frac{1}{3}}$ 
and $|a|^{3}+27|b|^{2}$, which appear in Lemma~\ref{ML.estimate}, 
are bounded and away from zero 
 when $\xi\in D$ and is homogeneous with respect to the 
 variable change of $\alpha, \beta, \gamma$ by 
 $\lambda \alpha, \lambda \beta, \lambda \gamma$. 
Moreover,
\begin{align*}
a^3-27b^2
=16((\alpha-\beta)(\beta-\gamma)(\gamma-\alpha))^2
=4\xi^2(9-4\xi^2)^2.
\end{align*}
Hence $\log |a^3-27b^2| - 2 \log |\xi|$
 is bounded.
As a consequence, it suffices to prove that 
\begin{Claim}\label{ML.estimate2}
there exist $c>0$ and $C>0$ such that
\begin{align}\label{ell.integral.estimate2}
 -C\leq 
 (|a|^{\frac{1}{2}}+|b|^{\frac{1}{3}})\mu(a,b) + c \log |\xi|
 \leq C
\end{align}
for $\xi\in D$ and $a,b$ as in \eqref{a.b.condition}. 
\end{Claim}

We choose the contour for elliptic integral for $\sigma_1(a,b)$ to be
 the half-line $u=1+s\ (s\in \mathbb{R}_{\geq 0})$.
We calculate
\begin{align*}
\sigma_1(a,b)=\int_{0}^{\infty}
 \frac{ds}{\sqrt{s(s+\xi+\frac{3}{2})(s-\xi+\frac{3}{2})}}=:I_1(\xi).
\end{align*}
Note that $\re \bigl(s(s+\xi+\frac{3}{2})(s-\xi+\frac{3}{2})\bigr) > 0$
 for $s>0$.
Then we choose a branch of the square root
 by $\re \sqrt{s(s+\xi+\frac{3}{2})(s-\xi+\frac{3}{2})} > 0$. 
The integral is absolutely convergent and 
 $I_1(\xi)$ is holomorphic on a neighborhood of $D$.
Hence
\begin{align}\label{I1.estimate}
|I_1(\xi)-I_1(0)|\leq |\xi| C_1,
\end{align}
 where $C_1$ is independent of $\xi\in D$.

Similarly, we choose the contour for 
 $\sigma_2(a,b)$ to be the half-line
 $u=\gamma+s\frac{\xi}{|\xi|}
 = \xi-\frac{1}{2}+s\frac{\xi}{|\xi|} \ (s\in \mathbb{R}_{\geq 0})$.
Then
\begin{align*}
\sigma_2(a,b)
&=\int_{0}^{\infty}
 \frac{\frac{\xi}{|\xi|}ds}
 {\sqrt{(s\frac{\xi}{|\xi|}+\xi-\frac{3}{2})
 (s\frac{\xi}{|\xi|}+2\xi)
 s\frac{\xi}{|\xi|}}}\\
&=\int_{0}^{\infty}
 \frac{ds}{\sqrt{s(s+|2\xi|)
 (s\frac{\xi}{|\xi|}+\xi-\frac{3}{2})}}=:I_2(\xi).
\end{align*}
Note that 
 $\re \Bigl(s\frac{\xi}{|\xi|}+\xi-\frac{3}{2}\Bigr) < 0$.
Then we choose a branch of the square root by 
 $\im \sqrt{s\frac{\xi}{|\xi|}+\xi-\frac{3}{2}} > 0$
 and $\sqrt{s(s+|2\xi|)} \geq 0$. 
It is easy to see that the integral from $s=1$ to $+\infty$:
\begin{align*}
\int_{1}^{\infty}
 \frac{ds}{\sqrt{s(s+|2\xi|)
 (s\frac{\xi}{|\xi|}+\xi-\frac{3}{2})}}
\end{align*}
is uniformly bounded for $\xi\in D$.
To estimate the integral from $s=0$ to $1$, we calculate 
\begin{align*}
\frac{1}{\sqrt{s\frac{\xi}{|\xi|}+\xi-\frac{3}{2}}}
-\frac{1}{\sqrt{-\frac{3}{2}}}
=\frac{-\bigl(s\frac{\xi}{|\xi|}+\xi\bigr)}
 { \sqrt{-\frac{3}{2}\bigl(s\frac{\xi}{|\xi|}+\xi-\frac{3}{2}\bigr)}
 \Bigl(
 \sqrt{-\frac{3}{2}} + \sqrt{s\frac{\xi}{|\xi|}+\xi-\frac{3}{2}}
 \Bigr)}.
\end{align*}
Since
 $\sqrt{-\frac{3}{2}\bigl(s\frac{\xi}{|\xi|}+\xi-\frac{3}{2}\bigr)}
 \Bigl(
 \sqrt{-\frac{3}{2}} + \sqrt{s\frac{\xi}{|\xi|}+\xi-\frac{3}{2}}
 \Bigr)$
 is away from zero for $\xi\in D$ and $s\geq 0$, we get 
\begin{align*}
\Biggl|\frac{1}{\sqrt{s\frac{\xi}{|\xi|}+\xi-\frac{3}{2}}}
-\frac{1}{\sqrt{-\frac{3}{2}}}\Biggr|
\leq C_2 (s+|\xi|)
\end{align*}
for a constant $C_2$ which does not depend on $\xi\in D$.
Then 
\begin{align*}
\Biggl|\int_{0}^{1}
 \Biggl(
 \frac{1}{\sqrt{s\frac{\xi}{|\xi|}+\xi-\frac{3}{2}}}
 -\frac{1}{\sqrt{-\frac{3}{2}}}
 \Biggr)
 \frac{ds}{\sqrt{s(s+|2\xi|)}}\Biggr|
\leq \int_{0}^{1}
 \frac{C_2(s+|\xi|) ds}{\sqrt{s(s+|2\xi|)}}
\end{align*}
 and the right hand side is bounded by a constant.
On the other hand,
\begin{align*}
\int_{0}^{1}
 \frac{ds}{\sqrt{-\frac{3}{2}s(s+|2\xi|)}} 
&=\frac{1}{\sqrt{-\frac{3}{2}}}
 \left[\log \bigl(s+|\xi|+\sqrt{s(s+|2\xi|)}\bigr)\right]_{s=0}^1 \\
&=\frac{1}{\sqrt{-\frac{3}{2}}}
 \bigl(\log \bigl(1+|\xi|+\sqrt{1+|2\xi|}\bigr)
 - \log |\xi| \bigr).
\end{align*}
Since $\log \bigl(1+|\xi|+\sqrt{1+|2\xi|})$ is bounded,
\begin{align*}
\frac{\log|\xi|}{\sqrt{-\frac{3}{2}}}
+ \int_{0}^{1}
 \frac{ds}{\sqrt{-\frac{3}{2}s(s+|2\xi|)}} 
\end{align*}
is also bounded.
Therefore, putting $c_1:=\frac{1}{\sqrt{-\frac{3}{2}}}$, we have
\begin{align}\label{I2.estimate}
\left|I_2(\xi)+c_1\log|\xi| \right| \leq C_3
\end{align}
for some constant $C_3$.

In addition, there is a constant $C_4$ such that
\begin{align}\label{ab.estimate}
\bigl| |a|^{\frac{1}{2}}+|b|^{\frac{1}{3}} - (\sqrt{3}+1) \bigr|
\leq C_4|\xi|.
\end{align}

Combining \eqref{I1.estimate} and \eqref{I2.estimate},
 we get
\begin{align*}
\bigl|I_1(\xi)\overline{I_2(\xi)}+\bar{c}_1 I_1(0)\log|\xi|\bigr|\leq C.
\end{align*}
Then by \eqref{ab.estimate},
 $c_1\in \sqrt{-1}\mathbb{R}_{<0}$ and
 $I_1(0)\in \mathbb{R}_{>0}$, 
\begin{align*}
-C\leq
(|a|^{\frac{1}{2}}+|b|^{\frac{1}{3}}) \left|\im (I_1(\xi)\overline{I_2(\xi)})\right|
 + c \log|\xi|\leq C
\end{align*}
for $c:=(\sqrt{3}+1)\im (\bar{c}_1I_1(0))>0$,
 which shows \eqref{ell.integral.estimate2}. 
 Therefore, we obtain the proof of Claim~\ref{ML.estimate2} and 
 thus that of Claim~\ref{ML.estimate} as well. 
\end{proof}

From this, in particular, we see that the McLean metric has finite diameter. 
This is well-known (cf., e.g., \cite{GTZ2}) but we include it for convenience. 
\begin{Cor}\label{McLean.finite.distance}
The base $\mathbb{P}^{1}$ of any elliptic K3 surface 
with the above McLean metric has finite diameter. 
\end{Cor}
\begin{proof}
Obviously, the nontrivial part is about the asymptotic behavior of $\mu$ around the 
discriminants which follows from Lemma~\ref{ML.estimate}. 
Indeed, from Lemma~\ref{ML.estimate} and the stability condition
 \eqref{stab.cond}, for any discriminant point
 $t_0\in \mathbb{C}\subset \mathbb{P}^{1}$, 
$$|\mu(g_{8}(t),g_{12}(t))|=O((t-t_0)^{-2+\epsilon})$$ for small enough $\epsilon>0$  
(more precisely, $O((t-t_0)^{-\frac{5}{3}}\cdot \log|t-t_0|^{-1})$) 
around neighborhoods of $t_0$. Therefore the assertion holds. 
The same argument works for $t_0=\infty$ after 
a transformation by ${\rm Aut}(\mathbb{P}^{1})$. 
\end{proof}

We improve above Corollary \ref{McLean.finite.distance}
 in the proof of Proposition~\ref{prop:MW.conti}
 to a version which is uniform with respect to
 bounded variation of elliptic K3 surfaces. 
See also closely related results in \cite{Yos10}, \cite[Proposition 2.1]{GTZ2}, \cite[Theorem A]{EMM17}, and 
\cite[Theorem 3.4]{TZ} which discuss higher dimensional generalization. 

We also need the following lower bound of $\mu$:
\begin{Lem}\label{ML.estimate3}
There exists a constant $c'>0$ such that 
\begin{align*}
 c'\leq 
 (|a|^{\frac{1}{2}}+|b|^{\frac{1}{3}}) \mu(a,b) 
\end{align*}
for any $a,b \in \mathbb{C}$ with $a^3-27b^2\neq 0$.
\end{Lem}
\begin{proof}
We follow the notation in the proof of Lemma~\ref{ML.estimate}.
We may and do assume \eqref{a.b.condition}.
Then it is easy to see that as a function of $\xi$,
 $\mu(a,b)$ is continuous on the closure $\overline{D}$
 except for $\xi=0$.
By Claim~\ref{ML.estimate2}, $\mu(a,b)$ becomes large when $\xi$ is near $0$.
Therefore, $\mu(a,b)$ is bounded from below by a positive constant.
The lemma follows from this.
\end{proof}

\subsection{Preparing elementary estimates}

When applying Lemma~\ref{ML.estimate},
 the following general estimates Lemma~\ref{integral.estimate0}
 and Lemma~\ref{integral.estimate} are useful 
 for estimating the length of paths with respect to 
 the McLean metrics. 
To discuss them, we define a function 
\begin{align*}
\log^+ x := \begin{cases} \log x & \text{ if $x\geq 1$} \\
 0 &  \text{ if $0<x<1$.}\end{cases}
\end{align*}
Note that $\log^+ (xy)\leq \log^+ x + \log^+ y$. 
Then define a function $F(t)$ on $\C$ by 
\begin{align}\label{F.def}
F(t):= \prod_{i=1}^m |t-u_i|^{-2w_i} \times
 \log^+ \Bigl( C_0 \prod_{i=1}^n |t-v_i|^{-1}\Bigr)
\end{align}
for $m,n\in \Z_{\geq 0}$, $C_0>0$, $w_i>0$, and $u_i , v_i\in \C$.

The following lemma will be used later when estimating
 the length of a segment connecting $2$ points in $\mathbb{P}^{1}$
 with respect to the McLean metric. 

\begin{Lem}[Elementary integral estimate 1]\label{integral.estimate0}
Let $m, n$ be nonnegative integers.
Let $w_i>0 \ (1\leq i\leq m)$, $w:=\sum_{i=1}^m w_i$,
 and assume $w<1$.
Define $F(t)$ for $C_0>0$, $u_i, v_i\in \C$ as in \eqref{F.def}.
Then there exist constants $C_1, C_2>0$ such that 
\begin{align*}
\int_{0}^{a}
F(t)^{\frac{1}{2}}
 dt 
\leq C_1 a^{1-w} (\log^+ C_0)^{\frac{1}{2}}
+ C_2 a^{\frac{1-w}{2}}
\end{align*}
 for $0 < a \leq 1$.
Here, the integral is taken over the segment between $0$ and $a$.
Moreover, $C_1$ and $C_2$ depend only on $w$ and $n$
 but not on any other parameters.
\end{Lem}

\begin{proof}
We have
\begin{align*}
\Bigl(\log^+ \Bigl( C_0 \prod_{i=1}^n |t-v_i|^{-1}\Bigr)\Bigr)^{\frac{1}{2}} 
&\leq 
\Bigl(\log^+ C_0
 +\log^+ \Bigl( \prod_{i=1}^n |t-v_i|^{-1}\Bigr)\Bigr)^{\frac{1}{2}} \\
&\leq 
(\log^+ C_0)^{\frac{1}{2}}
 + \Bigl(\log^+ \Bigl( \prod_{i=1}^n |t-v_i|^{-1}\Bigr)\Bigr)^{\frac{1}{2}}.
\end{align*}
Here, the second inequality follows from the simple inequality
 $(x+y)^{\frac{1}{2}}\leq x^{\frac{1}{2}}+y^{\frac{1}{2}}$
 for $x,y\geq 0$.
Put $w':=\frac{1-w}{2n}$.
It is easy to see that there exists a constant $C_3>0$
 such that
 $\log^+ x \leq C_3 x^{2w'}$ for any $x>0$.
Then we get
\begin{align*}
\log^+ \Bigl( \prod_{i=1}^n |t-v_i|^{-1}\Bigr)
\leq C_3 \prod_{i=1}^n |t-v_i|^{-2w'}.
\end{align*}

Combining above, we estimate the integral as:
\begin{align} \label{integral.estimate2}
&\int_{0}^{a}
\prod_{i=1}^m |t-u_i|^{-w_i} \times
 \Bigl(\log^+ \Bigl( C_0
  \prod_{i=1}^n |t-v_i|^{-1}\Bigr)\Bigr)^{\frac{1}{2}} dt \\ \nonumber
&\leq
 (\log^+ C_0)^{\frac{1}{2}} \int_{0}^{a} \prod_{i=1}^m |t-u_i|^{-w_i} dt
 + C_3^{\frac{1}{2}} \int_{0}^{a} \prod_{i=1}^m |t-u_i|^{-w_i} \prod_{i=1}^n |t-v_i|^{-w'} dt.
\end{align}
By the weighted AM-GM inequality, we get
\begin{align*}
\prod_{i=1}^m |t-u_i|^{-w_i}
\leq \sum_{i=1}^m \frac{w_i}{w} |t-u_i|^{-w}.
\end{align*}
It is easy to see that 
\begin{align*}
\int_{0}^{a} |t-u_i|^{-w}dt \leq C_4 a^{1-w}
\end{align*}
for a constant $C_4 > 0$ which depends only on $w$.
Therefore, 
\begin{align*}
\int_{0}^{a} \prod_{i=1}^m |t-u_i|^{-w_i} dt
\leq C_4 a^{1-w}.
\end{align*}
For the second term of the right hand side of \eqref{integral.estimate2},
 we again use the weighted AM-GM inequality to get
\begin{align*}
\prod_{i=1}^m |t-u_i|^{-w_i} \prod_{i=1}^n |t-v_i|^{-w'}
\leq \sum_{i=1}^m \frac{2w_i}{1+w} |t-u_i|^{-\frac{1+w}{2}}
 +\sum_{i=1}^n \frac{2w'}{1+w}|t-v_i|^{-\frac{1+w}{2}}.
\end{align*}
Hence by the same argument as above, we have
\begin{align*}
\int_{0}^{a} \prod_{i=1}^m |t-u_i|^{-w_i}\prod_{i=1}^n |t-v_i|^{-w'} dt
\leq C_5 a^{\frac{1-w}{2}}
\end{align*}
for some constant $C_5>0$.
The lemma follows from these inequalities.
\end{proof}

Here is the second lemma which we prepare for later use 
when we estimate the length of circles in $\mathbb{P}^{1}$ 
with respect to the McLean metric.  
\begin{Lem}[Elementary integral estimate 2]\label{integral.estimate}
Let $m, n$ be nonnegative integers.
Put $w:=\sum_{i=1}^m {w_i}$ and assume $w<1$.
Define $F(t)$ for $C_0>0$, $u_i, v_i\in \C$ as in \eqref{F.def}.
Then there exist constants $C_1, C_2>0$
 such that 
\begin{align*}
\int_{0}^{2\pi} F(re^{\sqrt{-1}\theta})^{\frac{1}{2}} rd\theta
\leq C_1 r^{1-w}
\Bigl( \log^+ \Bigl(C_0
  \prod_{i=1}^n \min\{r^{-1}, |v_i|^{-1}\}\Bigr)\Bigr)^{\frac{1}{2}}
+ C_2 r^{1-w}
\end{align*}
 for any $0<r\leq 1$.
Moreover, $C_1$ and $C_2$ depend only on $w$ and $n$
 but not on any other parameters.
\end{Lem}

\begin{proof}
We first prove the lemma for the case $r=1$,
 namely, we prove
\begin{align}\label{r=1.estimate}
&\int_{0}^{2\pi}
\prod_{i=1}^m |e^{\sqrt{-1}\theta}-u_i|^{-w_i} \times
 \Bigl(\log^+ \Bigl( C_0 \prod_{i=1}^n
  |e^{\sqrt{-1}\theta}-v_i|^{-1}\Bigr)\Bigr)^{\frac{1}{2}} d\theta\\ \nonumber
&\leq C_1 
\Bigl( \log^+ \Bigl(C_0
  \prod_{i=1}^n \min\{1, |v_i|^{-1}\}\Bigr)\Bigr)^{\frac{1}{2}}+ C_2. 
\end{align}
Suppose that $|v_i|\leq 1$ for $1\leq i\leq k$
 and $|v_i|>1$ for $k<i\leq n$.
Then we have 
\begin{align*}
&\log^+ \Bigl( C_0 \prod_{i=1}^n
  |e^{\sqrt{-1}\theta}-v_i|^{-1}\Bigr) \\
&=
\log^+ \Bigl( C_0 \prod_{i=k+1}^n |v_i|^{-1}
 \prod_{i=1}^k
  |e^{\sqrt{-1}\theta}-v_i|^{-1}
 \prod_{i=k+1}^n
  |v_i^{-1}e^{\sqrt{-1}\theta}-1|^{-1}
\Bigr) \\
&\leq 
\log^+ \Bigl( C_0 \prod_{i=k+1}^n |v_i|^{-1} \Bigr)
+ \log^+ \Bigl( \prod_{i=1}^k
  |e^{\sqrt{-1}\theta}-v_i|^{-1}
 \prod_{i=k+1}^n
  |v_i^{-1}e^{\sqrt{-1}\theta}-1|^{-1}\Bigr).
\end{align*}
By using the inequality
 $(x+y)^{\frac{1}{2}}\leq x^{\frac{1}{2}}+y^{\frac{1}{2}}$
 $(x,y\geq 0)$, we have
\begin{align*}
&\Bigl(\log^+ \Bigl( C_0 \prod_{i=1}^n
  |e^{\sqrt{-1}\theta}-v_i|^{-1}\Bigr)\Bigr)^{\frac{1}{2}} \\
&\leq 
 \Bigl(\log^+ \Bigl( C_0 \! \prod_{i=k+1}^n \! |v_i|^{-1} \Bigr)\Bigr)^{\frac{1}{2}} \!
+ \Bigl(\log^+ \Bigl( \prod_{i=1}^k
  |e^{\sqrt{-1}\theta}-v_i|^{-1} \!
 \prod_{i=k+1}^n
  |v_i^{-1}e^{\sqrt{-1}\theta}-1|^{-1}\Bigr)\Bigr)^{\frac{1}{2}}.
\end{align*}
Then similarly to \eqref{integral.estimate2}, we obtain
\begin{align*}
&\int_{0}^{2\pi}
\prod_{i=1}^m |e^{\sqrt{-1}\theta}-u_i|^{-w_i} \times
 \Bigl(\log^+ \Bigl( C_0 \prod_{i=1}^n
  |e^{\sqrt{-1}\theta}-v_i|^{-1}\Bigr)\Bigr)^{\frac{1}{2}} d\theta\\
&\leq
 \Bigl(\log^+ \Bigl( C_0 \prod_{i=k+1}^n |v_i|^{-1} \Bigr)\Bigr)^{\frac{1}{2}}
 \int_{0}^{2\pi}
\prod_{i=1}^m |e^{\sqrt{-1}\theta}-u_i|^{-w_i} d\theta\\
& + C_3^\frac{1}{2}
\int_{0}^{2\pi}
 \prod_{i=1}^m |e^{\sqrt{-1}\theta}-u_i|^{-w_i}
   \prod_{i=1}^k |e^{\sqrt{-1}\theta}-v_i|^{-w'}
  \prod_{i=k+1}^n  |v_i^{-1}e^{\sqrt{-1}\theta}-1|^{-w'}
 d\theta,
\end{align*}
where $w'=\frac{1-w}{2n}$ and $C_3$ is a constant which depends only on $w'$.
By using weighted AM-GM inequality as in the proof of Lemma~\ref{integral.estimate0},
 the proof of \eqref{r=1.estimate} is reduced to showing that
 the integrals
\begin{align*}
&\int_{0}^{2\pi} |e^{\sqrt{-1}\theta}-u_i|^{-w} d\theta, \quad 
\int_{0}^{2\pi} |e^{\sqrt{-1}\theta}-u_i|^{-\frac{1+w}{2}} d\theta, \\ 
&\int_{0}^{2\pi} |e^{\sqrt{-1}\theta}-v_i|^{-\frac{1+w}{2}} d\theta\quad (1\leq i\leq k),\\
&\int_{0}^{2\pi} |v_i^{-1}e^{\sqrt{-1}\theta}-1|^{-\frac{1+w}{2}} d\theta
 \quad (k+1\leq i\leq n)
\end{align*}
are bounded from above by constants.
To estimate the integral 
$\int_{0}^{2\pi}
 |e^{\sqrt{-1}\theta}-u_1|^{-w}d\theta$,
 we may assume $u_1\in \mathbb{R}_{\leq 0}$.
Then observe that
 $|e^{\sqrt{-1}\theta}-u_1|\geq \frac{1}{\pi}|\theta-\pi|$
 and get 
\[\int_{0}^{2\pi}
 |e^{\sqrt{-1}\theta}-u_1|^{-w}d\theta
 \leq \pi^{w} \int_{0}^{2\pi}|\theta-\pi|^{-w}d\theta. 
\]
It is easy to see (or by Lemma~\ref{integral.estimate0})
 that this is bounded by a constant.
The other integrals can be estimated in a similar way.
We thus proved \eqref{r=1.estimate}.

To finish the proof of Lemma~\ref{integral.estimate}, we fix $0< r \leq 1$.
For a given $F(t)$ as in \eqref{F.def}, 
we put 
\begin{align*}
\overline{F}(t):=F(rt)
&= \prod_{i=1}^m |rt-u_i|^{-2w_i}  \times
 \log^+ \Bigl( C_0 \prod_{i=1}^n |rt-v_i|^{-1}\Bigr) \\
&= r^{-2w} \prod_{i=1}^m |t-r^{-1}u_i|^{-2w_i}  \times
  \log^+ \Bigl( C_0 r^{-n} \prod_{i=1}^n |t-r^{-1}v_i|^{-1}\Bigr) 
\end{align*}
Then by applying the proved $r=1$ case to $\bar{F}(t)$, we have
\begin{align*}
\int_{0}^{2\pi}
 \overline{F}(e^{\sqrt{-1}\theta})^{\frac{1}{2}} d\theta
\leq
 C_1  r^{-w}
 \Bigl( \log^+ \Bigl(C_0 r^{-n}
  \prod_{i=1}^n \min\{1, r|v_i|^{-1}\}\Bigr)\Bigr)^{\frac{1}{2}}
+ C_2 r^{-w},
\end{align*}
which implies the desired estimate. 
We complete the proof of Lemma~\ref{integral.estimate}. 
\end{proof}

\subsection{Estimate of McLean metric near discriminant points}

Using Lemma~\ref{ML.estimate}, we estimate
 the McLean metric near discriminant points uniformly
 for a family of elliptic K3 surfaces.
Let $S'\subset W\mathbb{P}$ be a relatively compact open subset. 
Write the affine cone of the weighted projective space $W\mathbb{P}$ by 
$CW\mathbb{P}$ and its vertex by $v$. 
 Note that one can easily take a relatively compact open subset 
 $S$ in $CW\mathbb{P}\setminus \{v\}$ 
so that the image of $S$ is $S'$. 
On $S$, we have a family of elliptic K3 surfaces 
 $\pi_s: X_s \to B_s(\simeq \mathbb{P}^1)\ (s\in S)$
 and each $B_s$ is equipped with the McLean metric.
Write $(g_{8,s},g_{12,s})$ for the polynomials corresponding to $s\in S$
 and write $\Delta_{24,s}:=g_{8,s}^3-27g_{12,s}^2$.
The McLean metric on $B_s\simeq \mathbb{P}^1$ is given as
 $\mu(g_{8,s}(t),g_{12,s}(t))dt\otimes d\bar{t}$.
We rescale the metric on $B_s$ so that the diameter becomes $1$
 and denote by $B_s^{\rm n}$.

\begin{Prop}\label{ML.estimate.disc}
There exists a system of open neighborhoods 
$\{\mathcal{U}_{\epsilon}\}_{\epsilon>0}$ of 
 the set $\bigsqcup_{s}{\rm disc}(\pi_{s})$ in $S\times \mathbb{P}^1$
 which satisfies: 
 \begin{itemize}
\item $\mathcal{U}_{\epsilon'}\subset \mathcal{U}_{\epsilon}$ for $\epsilon'<\epsilon$, 
\item $\bigcap_{\epsilon>0} \mathcal{U}_{\epsilon} = \bigsqcup_{s}{\rm disc}(\pi_{s})$, 
\item and the length of
 $\partial \overline{\mathcal{U}_{\epsilon}}\,\cap\, (\{s\}\times \mathbb{P}^1)$
 in $B_s^{\rm n}$ converges to zero when $\epsilon\to 0$, 
 uniformly for $s\in S$. 
 \end{itemize}
\end{Prop}

See related Hodge-theoretic estimates in 
\cite[Proposition 2.1, (Lemma 3.1) and its proof]{GTZ1} 
(cf.\  also \cite{Yos10}, \cite{TZ}, \cite{EMM17}).

\begin{proof}
We can suppose $S$ is small enough and
 the discriminant locus does not meet $\infty\in \mathbb{P}^{1}$ 
 by using the ${\rm Aut}(\mathbb{P}^1)$-action if necessary.
Then the discriminant $\Delta_{24,s}(t)$ decomposes as 
\[\Delta_{24,s}(t)=\chi(s)\prod_{i=1}^{24}(t-\chi_{i}(s)),\]
 where $\chi(s)$ is a bounded function on $S$ and away from zero. 
We assume for simplicity that $g_{8,s}, g_{12,s}\not\equiv 0$ for every $s\in S$.
The case $g_{8,s}\equiv 0$ or $g_{12,s}\equiv 0$ can be treated in a similar way.
Moreover, we may and do assume that
 $g_{8,s}(\infty)\neq 0$ and $g_{12,s}(\infty)\neq 0$ 
 as elements of $\Gamma(\mathcal{O}_{\mathbb{P}^{1}}(8)|_{\{\infty\}})$ and 
 $\Gamma(\mathcal{O}_{\mathbb{P}^{1}}(12)|_{\{\infty\}})$, i.e., 
 $\deg(g_{8,s})=8$, $\deg(g_{12,s})=12$, 
 respectively. 
Then we decompose $g_{8,s}$ and $g_{12,s}$ as 
\[g_{8,s}=\alpha(s) \prod_{i=1}^{8}(t-\alpha_{i}(s)), \qquad
g_{12,s}=\beta(s) \prod_{j=1}^{12}(t-\beta_{j}(s)),
\]
 where $\alpha(s)$ and $\beta(s)$ are bounded and away from zero.

We put
\begin{align*}
&U_{\epsilon,i}(s):=\{t\in \mathbb{C} : |t-\chi_{i}(s)|<\epsilon\},\\
&U_{\epsilon}(s):=\bigcup_{i=1}^{24}U_{\epsilon,i}(s), \qquad
 \mathcal{U}_{\epsilon}:=\bigsqcup_{s\in S}U_{\epsilon}(s).
\end{align*}
It remains to bound the length of the boundary $U_{\epsilon}(s)$
 with respect to the metric of $B_s^{\rm n}$.
Clearly, the diameters of $B_s$ are bounded from below by
 a positive constant when $(g_8,g_{12})$ moves inside a bounded region
 in $CW\mathbb{P}$. 
Therefore it is enough to show that the length of
 $\partial\overline{U_{\epsilon}(s)}$
 with respect to $\mu(g_{8,s}(t),g_{12,s}(t))dt\otimes d\bar{t}$
 converges to zero as $\epsilon\to 0$.

To estimate $\mu(g_{8,s},g_{12,s})$ on the boundary
 $\partial\overline{U_{\epsilon}(s)}$
 we introduce the following values.
For subsets $I \subset \{1,2,\dots,8\}$ and $J\subset \{1,2,\dots,12\}$,
 we define
\begin{multline*}
\delta_{I,J}(s):=\max
 \Bigl(
 \{|\alpha_i(s)-\alpha_{i'}(s)|: i,i'\in I\} \cup
 \{|\beta_j(s)-\beta_{j'}(s)|: j,j'\in J\} \\
 \cup \{|\alpha_i(s)-\beta_j(s)| : i\in I \text{ and } j\in J\}\Bigr). 
\end{multline*}
Then we moreover define
\begin{align*}
\delta(s):= \min\{\delta_{I,J}(s) : \# I = 4 \text{ and } \# J = 6\}.
\end{align*}
Since the stability condition \eqref{stab.cond} is satisfied on the closure
 $\overline{S}$ of $S$,
 we have $\delta(s)>0$ on $\overline{S}$ and
 hence $\delta(s)>\delta_0$ for a positive constant $\delta_0$.
This implies that for each $t\in \C$ and $s\in S$,
 at least one (depending on $t$) of 
\begin{align}\label{stab.cond.roots} 
&\# \Bigl\{1\leq i\leq 8 : |t-\alpha_i(s)|<\frac{\delta_0}{2}\Bigr\} \leq 3 
 \text{  or } \\ \label{stab.cond.roots2}
&\# \Bigl\{1\leq j\leq 12 : |t-\beta_j(s)|<\frac{\delta_0}{2}\Bigr\} \leq 5 
\end{align}
holds.
Since the boundary $\partial \overline{U_{\epsilon}(s)}$ 
is decomposed as
\[\partial \overline{U_{\epsilon}(s)}
 = \bigcup_{i=1}^{24}
 (\partial \overline{U_{\epsilon}(s)} \cap \partial \overline{U_{\epsilon,i}(s))},\]
it is enough to bound the length
 $\partial \overline{U_{\epsilon}(s)} \cap \partial \overline{U_{\epsilon,i}(s)}$
from above for each $i$. 
Let us fix $1\leq i\leq 24$ and let
 $t_0 \in \partial \overline{U_{\epsilon}(s)} \cap \partial \overline{U_{\epsilon,i}(s)}$.
Suppose that the latter \eqref{stab.cond.roots2} holds for $t=t_0$
 and suppose $|t_0-\beta_j(s)|\geq \frac{\delta_0}{2}$ for $5< j \leq 12$.
(The argument for the \eqref{stab.cond.roots} case is similar.)
By Lemma~\ref{ML.estimate},
\begin{align}\label{mu.estimate.bdry}
\mu(g_{8,s}(t),g_{12,s}(t))
&\leq 
\frac{C}{|g_{8,s}(t)|^{\frac{1}{2}}+|g_{12,s}(t)|^{\frac{1}{3}}} \\ \nonumber
&  + \frac{c}{|g_{8,s}(t)|^{\frac{1}{2}}+|g_{12,s}(t)|^{\frac{1}{3}}}
 \log \frac{|g_{8,s}(t)|^3+27|g_{12,s}(t)|^2}{|\Delta_{24,s}(t)|}.
\end{align}
We will estimate the right hand side for
  $t\in \partial U_{\epsilon}(s) \cap \partial U_{\epsilon,i}(s)$.

If $t\in \partial U_{\epsilon}(s) \cap \partial U_{\epsilon,i}(s)$, 
 then $|t-t_0|\leq 2\epsilon$ and we get
\begin{align*}
|g_{12,s}(t)| \geq |\beta(s)| \Bigl(\frac{\delta_0}{2}-2\epsilon\Bigr)^7
  \prod_{j=1}^5|t-\beta_j(s)|. 
\end{align*}
Assuming $\epsilon \ll \delta_0$, we have
\begin{align*}
|g_{12,s}(t)|^{-\frac{1}{3}}
 \leq C_3 \prod_{j=1}^5|t-\beta_j(s)|^{-\frac{1}{3}}
\quad (t\in \partial \overline{U_{\epsilon}(s)} \cap \partial \overline{U_{\epsilon,i}(s)})
\end{align*}
 for some constant $C_3>0$.

The function $|g_{8,s}(t)|^3+27|g_{12,s}(t)|^2$ is bounded (from above) 
 on $t\in \partial \overline{U_{\epsilon}(s)} \cap \partial 
 \overline{U_{\epsilon,i}(s)}$.
Moreover, $t\in \partial U_{\epsilon}(s)$ implies
 $|t-\chi_j(s)|\geq \epsilon$ for $1\leq j\leq 24$.
Hence we get 
\begin{align*}
 \log \frac{|g_{8,s}(t)|^3+27|g_{12,s}(t)|^2}{|\Delta_{24,s}(t)|}
 \leq 24 \log \epsilon^{-1} +C_4
\quad (t\in \partial \overline{U_{\epsilon}(s)} \cap \partial \overline{U_{\epsilon,i}(s)}).
\end{align*}
Combining with \eqref{mu.estimate.bdry}, 
\begin{align*}
\mu(g_{8,s}(t),g_{12,s}(t))
&\leq 
 C_3 \prod_{j=1}^5|t-\beta_j(s)|^{-\frac{1}{3}}
 (C+24c \log \epsilon^{-1} +cC_4).
\end{align*}
The length of
 $\partial \overline{U_{\epsilon}(s)} \cap \partial \overline{U_{\epsilon,i}(s)}$ is
 therefore estimated as 
\begin{align*}
&{\rm length}(\partial \overline{U_{\epsilon}(s)} \cap \partial 
\overline{U_{\epsilon,i}(s)}) \\
&= \int_{\partial \overline{U_{\epsilon}(s)} \cap \partial \overline{U_{\epsilon,i}(s)}}
  \mu(g_{8,s}(t),g_{12,s}(t))^{\frac{1}{2}} |dt| \\
&\leq C_3^{\frac{1}{2}} (C+cC_4+24c \log \epsilon^{-1})^{\frac{1}{2}}
\int_{\partial \overline{U_{\epsilon,i}(s)}} \prod_{j=1}^5|t-\beta_j(s)|^{-\frac{1}{6}} |dt|\\
&\leq C_3^{\frac{1}{2}} (C+cC_4+24c \log \epsilon^{-1})^{\frac{1}{2}}
 C_5\epsilon^{1-\frac{5}{6}}.
\end{align*}
Here, we use Lemma~\ref{integral.estimate} (with $n=0$)
 for the last inequality.
Then it is easy to see the last function goes to zero as $\epsilon\to 0$,
 which complete the proof of Proposition~\ref{ML.estimate.disc}. 
\end{proof}

The first step for Theorem~\ref{Mwbar.GH.conti} 
is: 

\begin{Prop}\label{prop:MW.conti}
The map $\Phi_{\rm ML}$ is continuous on $M_{W}$.
\end{Prop}

\begin{proof}
We follow the notation in the proof of Proposition~\ref{ML.estimate.disc}.
We estimate the diameter of the disk $U_{\epsilon,i}(s)$
 uniformly with respect to $s\in S$.
As in the proof of Proposition~\ref{ML.estimate.disc},
 assume \eqref{stab.cond.roots2} holds for
 a fixed $t_0\in U_{\epsilon,i}(s)$ and then we get 
\begin{align*}
|g_{12,s}(t)|^{-\frac{1}{3}}
 \leq C_3 \prod_{j=1}^5|t-\beta_j(s)|^{-\frac{1}{3}}
\quad (t\in U_{\epsilon,i}(s)).
\end{align*}
Since $|g_{8,s}(t)|^3+27|g_{12,s}(t)|^2$ is bounded on $U_{\epsilon,i}(s)$,
 Lemma~\ref{ML.estimate} gives an estimate
\begin{align*}
\mu(g_{8,s}(t),g_{12,s}(t))
 \leq C_3 \prod_{j=1}^5|t-\beta_j(s)|^{-\frac{1}{3}}
 (C+ c \log (C_4 |\Delta_{24,s}(t)|^{-1})),
\end{align*}
where $C_4>0$ is a constant.

The diameter of $U_{\epsilon,i}(s)$ can be estimated by an integral
 of $\mu(g_{8,s}(t),g_{12,s}(t))^{\frac{1}{2}}$.
By using Lemma~\ref{integral.estimate0}, it turns out that
 the integral is bounded from above by 
 $C_1 \epsilon^{\frac{1}{6}} (\log^+ |\chi(s)|^{-1})^{\frac{1}{2}}
  + C_2 \epsilon^{\frac{1}{12}}$
 and this converges to zero uniformly on $S$ as $\epsilon \to 0$. 
 This refines Corollary \ref{McLean.finite.distance} in an $s$-uniform manner. 

It is easy to see that
 the metric $\mu(g_{8,s}(t),g_{12,s}(t))dt\otimes d\bar{t}$
 is continuous in $s$ outside $U_{\epsilon}(s)$.
Then we employ an argument similar to \cite[\S6]{GW} (see also \S\ref{GH.conv})
 and conclude that the metric space $B_{s}$ is continuous
 with respect to the Gromov-Hausdorff topology on $S$. 
\end{proof}

\subsection{Convergence to $M_{W}^{\rm nn}$}\label{sec:conv.to.MWnn}

We consider the map $\Phi_{\rm ML}$
 near the boundary component $M_{W}^{\rm nn}$, 
 for the proof of Theorem~\ref{Mwbar.GH.conti}. 
\begin{Prop}\label{prop:nn.conti}
The map $\Phi_{\rm ML}$ is continuous
  on a neighborhood of $M_{W}^{\rm nn}\setminus M_W^{\rm seg}$.
\end{Prop}

\begin{proof}
Let us fix a point in $M_W^{\rm nn}\setminus M_W^{\rm seg}$
 which is represented by
 $(g_{8},g_{12})=(3G_4(t)^2,G_4(t)^3)$
 for $G_4(t)=t(t-1)(t-2)(t-d)$ with $d\neq 0,1,2$.
The metric
$|G_4(t)|^{-1}dt\otimes d\bar{t}$ 
 makes $\mathbb{P}^1$ the compact metric space,
 which we denote by $B_{G_4}$. Also for convenience, 
 we put its rescale  
 $\dfrac{1}{(\sqrt{3}+1)|G_4(t)|}dt\otimes d\bar{t}$
 on $\mathbb{P}^{1}$ as 
 $B_{G_4}^{\rm n}$. 
In the following, we assume $d\neq \infty$.
The case where $d=\infty$,
 namely, the case $G_4(t)=t(t-1)(t-2)$,
 can be treated in a similar way, 
 or alternatively, 
 that point lies in the same orbit with
 $d=\frac{2}{3}$ or $\frac{4}{3}$, 
 under the $PGL(2)$-action on $\mathbb{P}^1$.

Let us consider 
\[(g_8,g_{12})=(3G_4(t)^2+h_{8}(t),G_4(t)^3+h_{12}(t))\]
 with $h_8\in \Gamma(\mathcal{O}_{\mathbb{P}^{1}}(8))$,
 and $h_{12}\in \Gamma(\mathcal{O}_{\mathbb{P}^{1}}(12))$.
Let $h_8(t)=\sum_{i=0}^8 a_i t^i$ and $h_{12}(t)=\sum_{i=0}^{12} b_i t^i$.
Suppose that
\[\Delta_{24}(t)=(3G_4(t)^2+h_{8}(t))^3-27(G_4(t)^3+h_{12}(t))^2\not\equiv 0\]
 and 
 let $t=\chi_i\ (1\leq i\leq 24)$ be the roots (with multiplicity) of
 $\Delta_{24}(t)=0$.
Suppose $|\chi_i|\leq 1$ for $1\leq i\leq k$ and $|\chi_i|> 1$ for
 $k+1\leq i\leq 24$.  
 so we can write 
\begin{align*}
\Delta_{24}(t)= \chi
 \prod_{i=1}^k (t-\chi_i) \prod_{i=k+1}^{24} (\chi^{-1}_i t-1),
\end{align*}
where $\chi\in\mathbb{C}^{\times}$. 

We claim that if all $|a_i|$ and $|b_i|$ are sufficiently small, then 
 $|\chi|$ becomes arbitrarily small. Indeed, since the maximum of the absolute value of 
all the coefficients of $\prod_{i=1}^k (t-\chi_i) \prod_{i=k+1}^{24} (\chi^{-1}_i t-1)$ is bounded below by a positive constant and since $|a_{i}|, |b_{i}|\to 0$ 
implies that every coefficients of $\Delta_{24}$ goes to zero, we have $|\chi|\to 0$. 

Let $0<\epsilon_1 \ll 1$ be a small number.
Thanks to the above claim, 
 we may assume $\log |\chi|^{-1} > \epsilon_1^{-1}$. 
Normalize the metric of $B_{g_8,g_{12}}$
 for $(g_8,g_{12})=(3G_4(t)^2+h_{8}(t),G_4(t)^3+h_{12}(t))$
 by multiplying $c^{-1}(\log |\chi|^{-1})^{-1}$,
 where $c$ is the constant in \eqref{ell.integral.estimate}.
We denote the normalized metric space
 by $B_{g_8,g_{12}}^{\rm n}$.
We will then prove that if $|a_i|$ and $|b_i|$ are sufficiently small,
 then the Gromov-Hausdorff distance of
 $B_{G_4}^{\rm n}$ and $B_{g_{8},g_{12}}^{\rm n}$ becomes arbitrarily small.

The normalized metric on $B_{g_{8},g_{12}}^{\rm n}$ is given by
 $\mu^{\rm n}(g_8(t),g_{12}(t)) dt\otimes d\bar{t}$,
 where 
 $\mu^{\rm n}(g_8(t),g_{12}(t))
 := c^{-1}(\log |\chi|^{-1})^{-1} \mu(g_8(t),g_{12}(t))$.
By \eqref{ell.integral.estimate}, 
\begin{align*}
\left| (|g_8(t)|^{\frac{1}{2}}+|g_{12}(t)|^{\frac{1}{3}})
 \mu(g_8(t),g_{12}(t))
 + c  \log \frac{|\Delta_{24}(t)|}{|g_8(t)|^3+27|g_{12}(t)|^2} \right|
 \leq C.
\end{align*}
Putting 
\begin{align*}
&\tilde{\mu}(g_8,g_{12}):=-(\log |\chi|^{-1})^{-1}
 \log \frac{|\Delta_{24}(t)|}{|g_8(t)|^3+27|g_{12}(t)|^2} \\
&= 1 - \sum_{i=1}^k
 \frac{\log |t-\chi_i|}{\log |\chi|^{-1}}
 - \sum_{i=k+1}^{24}
 \frac{\log |\chi^{-1}_it-1|}{\log |\chi|^{-1}}
 + \frac{\log (|g_8(t)|^3+27|g_{12}(t)|^2)}{\log |\chi|^{-1}},
\end{align*}
we get 
\begin{align}\label{mun.estimate} 
\left| (|g_8(t)|^{\frac{1}{2}}+|g_{12}(t)|^{\frac{1}{3}})
 \mu^{\rm n}(g_8,g_{12})
 - \tilde{\mu}(g_8,g_{12}) \right|
\leq c^{-1} C(\log |\chi|^{-1})^{-1}
< c^{-1} C \epsilon_1.
\end{align}

Our idea for proving the convergence of $B_{g_{8},g_{12}}^{\rm n}$ to 
$B_{G_{4}}^{\rm n}$ (which implies Proposition~\ref{prop:nn.conti}) 
is to see that $\tilde{\mu}$ converges to  $1$ in some sense 
while $|g_{8}|^{\frac{1}{2}}+|g_{12}|^{\frac{1}{3}}$ converges to 
$(\sqrt{3}+1)|G_{4}|$ and then combine with above 
\eqref{mun.estimate}. 

Fix $\epsilon_2>0$.
We utilize an argument similar to \cite[\S6]{GW}.
For $r>0$, define the following disks in $\mathbb{P}^1$:
\begin{align*}
&D_i:=\{t : |t-\chi_i|<r\} \ (1\leq i\leq k),\\
&D_i:=\{t : |\chi^{-1}_i t-1|<r\} \ (k+1\leq i\leq 24),\\
&D'_j:=\{t : |t-j|<r \} \ (j=0,1,2,c), \\
&D'_{\infty}:=\Bigl\{t : |t|> \frac{1}{r} \Bigr\}.
\end{align*}
We remark that these disks may not be disjoint
 as we choose $r$ independently of $h_8$ and $h_{12}$
 (hence independently of $\chi_i$).
We claim that 
 if $r$, $|a_i|$ and  $|b_i|$ are sufficiently small,
 then the diameters of all these sets
 can be smaller than $\epsilon_2$ with respect to
 both of the metrics of $B_{G_4}^{\rm n}$ and $B_{g_8,g_{12}}^{\rm n}$.
This claim can be proved by estimating the length of paths 
 with the use of Lemma~\ref{integral.estimate0} 
 similarly to Proposition~\ref{prop:MW.conti}. 

We fix small $r$ as above.
Let $B_{r}^o:=\mathbb{P}^1\setminus
 \bigl(\bigcup_{i=1}^{24} D_i
 \cup \bigcup_{j=0}^{2} D'_i \cup D'_{\infty}\bigr)$.
On $B_{r}^o$,
 as $|a_i|,|b_i| \to 0$ (and hence $\chi\to 0$), 
 $\tilde{\mu}(g_8,g_{12})$ converges to $1$ uniformly
 and the ratio of the metrics of $B_{G_4}$ and $B_{g_8,g_{12}}^{\rm n}$
 approaches to $1$ by \eqref{mun.estimate}.
Then by using an argument in \cite[\S6]{GW} (see also \S\ref{GH.conv}),
 we conclude that
 the Gromov-Hausdorff limit of $B_{g_8,g_{12}}^{\rm n}$
 is $B_{G_4}^{\rm n}$ as $|a_i|, |b_i| \to 0$.
\end{proof}

\subsection{Convergence to $M_{W}^{\rm seg}$}\label{sec:conv.to.MWseg}

We consider the map $\Phi_{\rm ML}$
 near the boundary component $M_{W}^{\rm seg}$, 
 which is the last step for the proof of Theorem~\ref{Mwbar.GH.conti}. 
\begin{Prop}\label{prop:seg.conti}
The map $\Phi_{\rm ML}$ is continuous
 on a neighborhood of $M_{W}^{\rm seg}$.
\end{Prop}

\begin{proof}
We want to show that if a point in $M_{W}$ is sufficiently close
 to the boundary component $M_{W}^{\rm seg}$, then 
 the corresponding metric space $B_{g_8,g_{12}}$ (if normalized)
 is close to the segment in the sense of Gromov-Hausdorff distance.

Let us take a point $s\in M_{W}$ which is represented by $(g_8,g_{12})$.
We suppose $(g_8,g_{12})$ is sufficiently close to $(at^4, bt^6)$
 and want to prove that $B_{g_8,g_{12}}$ is close to the segment.
When $s \in M_{W}^{\rm seg}\cap M_{W}^{\rm nn}$, it is represented by 
$(a,b)=(3,1)$. 
We assume $ab\neq 0$ in the following. 
The case $a=0$ or $b=0$ can be treated in a similar way 
(see Remarks \ref{ab0.1}, \ref{ab0.2}). 

Let $t=\alpha_i\ (1\leq i\leq 8)$ be the roots (with multiplicity)
 of $g_8(t)=0$.
Since $g_8$ is close to $at^4$, we may assume
 $\alpha_i\ll 1$ for $1\leq i\leq 4$ and
 $\alpha_i\gg 1$ for $5\leq i\leq 8$.
Then we may write 
\begin{align}\label{g8.decomp}
g_8(t)= \alpha
 \prod_{i=1}^4 (t-\alpha_i) \prod_{i=5}^{8} (\alpha^{-1}_i t-1)
\quad (\alpha\in \C).
\end{align}
Similarly, let $t=\beta_i\ (1\leq i\leq 12)$ be the roots
 of $g_{12}(t)=0$ and assume
 $\beta_i\ll 1$ for $1\leq i\leq 6$ and
 $\beta_i\gg 1$ for $7\leq i\leq 12$.
We write 
\begin{align}\label{g12.decomp}
g_{12}(t)= \beta
 \prod_{i=1}^{6} (t-\beta_i)
 \prod_{i=7}^{12} (\beta^{-1}_i t-1)
\quad (\beta\in \C).
\end{align}
By the $SL(2)$-action,
 we may and do assume that $\alpha_1=0$ and $\alpha_8=\infty$.
When $(g_8,g_{12})\to (at^4,bt^6)$, 
\begin{align*}
&\alpha_i,\beta_j\to 0
 \text{ for $1\leq i\leq 4$ and $1\leq j\leq 6$}, \\
&\alpha_i,\beta_j\to \infty
 \text{ for $5\leq i\leq 8$ and $7\leq j\leq 12$,}\\ 
& \alpha\to a,\quad \beta\to b.
\end{align*}
Also, the discriminant $\Delta_{24}=g_8^3-27g_{12}^2$ can be written as
\begin{align*}
\Delta_{24}(t)= \chi
 \prod_{i=1}^k (t-\chi_i) \prod_{i=k+1}^{24} (\chi^{-1}_i t-1)
\quad (\chi\in \C).
\end{align*}
as in \S\ref{sec:conv.to.MWnn}, 
i.e., $|\chi_{i}|\le 1$ for $1\le i\le k$ and $|\chi_{i}|>1$ for $i\ge 
k+1$. (Note that except for the case $s\in M_{W}^{\rm seg}\cap M_{W}^{\rm nn}$, 
when $[X_{g_{8},g_{12}}^{W}]$ is close enough to fixed $s\in M_{W}^{\rm seg}$, 
we have $k=12$, but it is not necessarily true for the most 
difficult case $s\in M_{W}^{\rm seg}\cap M_{W}^{\rm nn}$.
See Remark~\ref{rem:a3b1}.)
We assume moreover that
\begin{align}\label{order.chi}
|\chi_1|\leq |\chi_2|\leq \cdots\leq |\chi_{24}|.
\end{align}

Recall that $B_{g_8,g_{12}}$ is
 defined as $(\mathbb{P}^1, \mu(g_8(t),g_{12}(t))dt\otimes d\bar{t})$.
For $r>0$, the length of the circle $\{|t|=r\}$ 
with respect to the McLean metric 
is given by
\begin{align*}
\rho(r):=
\int_{0}^{2\pi}
\mu(g_8(re^{\sqrt{-1}\theta}),g_{12}(re^{\sqrt{-1}\theta}))^{\frac{1}{2}}
 r d\theta.
\end{align*}

Let us put 
$$d(g_{8},g_{12}):={\rm dist}(0,\infty; B_{g_8,g_{12}}),$$
i.e., the distance of 
$t=0$ and $t=\infty$ in $B_{g_{8},g_{12}}$. 

The following lemma provides a key estimate 
for collapsing of tropical K3 surfaces to the unit segment.

\begin{Lem}\label{claim1:conv.MWseg}
The ratio
\[ \frac{\max \{\rho(r)\mid 0<r<\infty \}}{d(g_{8},g_{12})}\]
 converges to zero as $(g_8,g_{12})\to (at^4,bt^6)$. 
\end{Lem}
Note that the maximum of the numerator is attained because 
 $\rho(r)\to 0$ as $r\to 0$ or $r\to \infty$ 
 (both by Corollary~\ref{McLean.finite.distance}). 

Let us first prove Proposition~\ref{prop:seg.conti} assuming this lemma. 
Write $B_{g_8,g_{12}}^{\rm n}$ for the metric space
 obtained by multiplying the metric of $B_{g_8,g_{12}}$ by $d(g_{8},g_{12})^{-1}$. 
Define the map $\phi:[0,\infty] \to [0,1]$ by 
\[\phi(r):=
\text{the distance between $\{t=0\}$ and $\{|t|=r\}$ in $B^{\rm n}_{g_{8},g_{12}}$}, 
\]
which is a homeomorphism.
Then define
 $\psi_1: B_{g_8,g_{12}}^{\rm n}\to [0,1]$ by 
 $\psi_1(t):=\phi(|t|)$ and
 define $\psi_2: [0,1]\to B_{g_8,g_{12}}^{\rm n}$ by 
 $\psi_2(r):=\phi^{-1}(r)$. 
Lemma~\ref{claim1:conv.MWseg} implies that 
 the distortions of two maps $\psi_1$, $\psi_2$ converge to zero
 as $(g_8,g_{12})\to (at^4, bt^6)$.
Hence Proposition~\ref{prop:seg.conti} follows.

\medskip

We next prove Lemma~\ref{claim1:conv.MWseg} which lasts until the end of this 
section. 
In what follows, we show that
$$\dfrac{\max \{\rho(r)\mid 0<r\leq 1 \}}{d(g_{8},g_{12})} \to 0$$
 as $(g_8,g_{12})\to (at^4,bt^6)$.
The same argument shows 
 $$\dfrac{\max \{\rho(r)\mid 1\le r < \infty \}}{d(g_{8},g_{12})} \to 0$$ 
 again as $(g_8,g_{12})\to (at^4,bt^6)$, e.g. after the involution 
 $t\mapsto \frac{1}{t}$ on $\mathbb{P}^{1}$.

Let us set 
\[
\epsilon:=\max\{|\alpha_i|, |\beta_j|
 : 2\leq i\leq 4,\ 1\leq j\leq 6\}.
\]
Then $\epsilon\to 0$ as $(g_8,g_{12})\to (at^4, bt^6)$.
We may assume $\epsilon$ is sufficiently small.
Lemma~\ref{ML.estimate} gives 
\begin{align}\label{mu.estimate.circle}
&\mu(g_{8}(t),g_{12}(t))\\ \nonumber
&\leq 
 \frac{C}{|g_{8}(t)|^{\frac{1}{2}}+|g_{12}(t)|^{\frac{1}{3}}} 
 + \frac{c}{|g_{8}(t)|^{\frac{1}{2}}+|g_{12}(t)|^{\frac{1}{3}}}
 \log \frac{|g_{8}(t)|^3+27|g_{12}(t)|^2}{|\Delta_{24}(t)|} \\ \nonumber 
&\leq 
 \frac{C}{|g_{8}(t)|^{\frac{1}{2}}+|g_{12}(t)|^{\frac{1}{3}}}
   + \frac{c}{|g_{8}(t)|^{\frac{1}{2}}+|g_{12}(t)|^{\frac{1}{3}}}
 \log \frac{27(|g_{8}(t)|^{\frac{1}{2}}+|g_{12}(t)|^{\frac{1}{3}})^{6}}
 {|\Delta_{24}(t)|} \\ \nonumber
&\leq 
 \frac{C'}{|g_{8}(t)|^{\frac{1}{2}}+|g_{12}(t)|^{\frac{1}{3}}}
   + \frac{c}{|g_{8}(t)|^{\frac{1}{2}}+|g_{12}(t)|^{\frac{1}{3}}}
 \log \frac{(|g_{8}(t)|^{\frac{1}{2}}+|g_{12}(t)|^{\frac{1}{3}})^{6}}
 {|\Delta_{24}(t)|},
\end{align}
where $C':=C+c\log 27$.

In order to estimate the last term of \eqref{mu.estimate.circle},
 we see that the inequality 
\begin{align}\label{log.ineq}
\frac{1}{x}\log^+\Bigl(\frac{x^6}{\delta}\Bigr)\leq 
\frac{6}{y}+\frac{1}{y}\log^+\Bigl(\frac{y^6}{\delta}\Bigr)
\end{align}
holds for any $x\geq y>0$ and $\delta>0$.
Indeed, the left hand side is decreasing as a function of $x$
 for $x\geq e\delta^{\frac{1}{6}}$.
Then \eqref{log.ineq} can be easily seen by considering the three cases
 $y\geq e\delta^{\frac{1}{6}}$,\ 
 $x\geq e\delta^{\frac{1}{6}} > y$,
 and $e\delta^{\frac{1}{6}}>x$,
 separately. 

We now estimate $\rho(r)$ by using \eqref{mu.estimate.circle}.
By the weighted AM-GM inequality, 
\begin{align}\label{wAG}
|g_{8}(t)|^{\frac{1}{2}}+|g_{12}(t)|^{\frac{1}{3}}
&\geq 
\frac{2}{5}|g_{8}(t)|^{\frac{1}{2}}
 +\frac{3}{5}|g_{12}(t)|^{\frac{1}{3}} 
\geq |g_{8}(t)|^{\frac{1}{5}}|g_{12}(t)|^{\frac{1}{5}}.
\end{align}
When $|t|\leq 1$,
 in view of \eqref{g8.decomp} and \eqref{g12.decomp}
 (recall $\alpha_i\ (i>4)$ and $\beta_i\ (i>6)$
 are sufficiently large and $\alpha_1=0$), we have 
\begin{align}\label{g8g12}
|g_{8}(t)|^{\frac{1}{5}}|g_{12}(t)|^{\frac{1}{5}}
 \geq
 C_3 |t|^{\frac{1}{5}}
 \prod_{i=2}^4|t-\alpha_i|^{\frac{1}{5}}
 \prod_{i=1}^6|t-\beta_i|^{\frac{1}{5}}
\end{align}
for some constant $C_3>0$.
By the definition of $\epsilon$, we have
\begin{align} \label{g8g12upper}
|t|^{\frac{1}{5}}
 \prod_{i=2}^4|t-\alpha_i|^{\frac{1}{5}}
 \prod_{i=1}^6|t-\beta_i|^{\frac{1}{5}}
\leq (|t|+\epsilon)^2.
\end{align}
We use \eqref{log.ineq} to estimate the last term
 of \eqref{mu.estimate.circle}
 by setting $x=|g_8|^{\frac{1}{2}}+|g_{12}|^{\frac{1}{3}}$
 and $y$ to be the right hand side of \eqref{g8g12}.
Then combining with \eqref{g8g12upper}, 
 we obtain 
\begin{align}\label{mu.estimate.circle2}
\mu(g_{8}(t),g_{12}(t)) 
&\leq 
 C_3^{-1} |t|^{-\frac{1}{5}}
 \prod_{i=2}^4|t-\alpha_i|^{-\frac{1}{5}}
 \prod_{i=1}^6|t-\beta_i|^{-\frac{1}{5}}  \\ \nonumber
& \ \ 
 \times \Bigl( C' + 6c +
  c\log^+ \bigl(C_3^6 (|t|+\epsilon)^{12}
 |\Delta_{24}^{-1}(t)|\bigr)\Bigr) \\ \nonumber
&\leq 
 C_3^{-1} |t|^{-\frac{1}{5}}
 \prod_{i=2}^4|t-\alpha_i|^{-\frac{1}{5}}
 \prod_{i=1}^6|t-\beta_i|^{-\frac{1}{5}}  \\ \nonumber
& \ \ 
 \times \Bigl( C'' +
  c\log^+ \bigl(C_4 \max\{|t|^{12},\epsilon^{12}\}
 |\Delta_{24}^{-1}(t)|\bigr)\Bigr)
\end{align}
for $0<|t|\leq 1$, where we put
 $C'':=C'+6c$ and $C_4:=2^{12}C_3^6$.

Let us define 
\begin{align}\label{nu.def}
\nu(r):= r^{12}
 |\chi|^{-1}
 \prod_{i=1}^{k} \min \{r^{-1}, |\chi_i|^{-1}\}.
\end{align}
for $0<r\leq 1$. 
This function plays a key role in our further analysis from now on. 
Indeed, firstly we prove 
\begin{Claim}\label{rho.estimate}
There exist constants $C_5, C_6>0$ such that 
\begin{align*}
&\rho(r) \leq C_5 + C_6 \bigl(\log^+ \nu(r)\bigr)^{\frac{1}{2}}
 \ \ \text{ for $\frac{\epsilon}{2} \leq r\leq 1$ and} \\
&\rho(r) \leq C_5 + C_6 \bigl(\log^+ \nu(\epsilon)\bigr)^{\frac{1}{2}}
 \ \ \text{ for $0 < r\leq \frac{\epsilon}{2}$}.
\end{align*}
\end{Claim}
\begin{proof}[proof of Claim~\ref{rho.estimate}]
By \eqref{mu.estimate.circle2}, 
\begin{align*}
\rho(r)
&= \int_{0}^{2\pi}
 \mu(g_8(re^{\sqrt{-1}\theta}),g_{12}(re^{\sqrt{-1}\theta}))^{\frac{1}{2}}
 r d\theta \\
&\leq r^{-\frac{1}{10}} \int_{0}^{2\pi}
C_3^{-\frac{1}{2}} 
 \prod_{i=2}^4|re^{\sqrt{-1}\theta}-\alpha_i|^{-\frac{1}{10}}
 \prod_{i=1}^6|re^{\sqrt{-1}\theta}-\beta_i|^{-\frac{1}{10}} \\ \nonumber
&\qquad \times \Bigl( C'' +
 c\log \bigl(C_4 \max\{r^{12},\epsilon^{12}\}
 |\Delta_{24}^{-1}(re^{\sqrt{-1}\theta})|\bigr)\Bigr)^{\frac{1}{2}}
 r d\theta.
\end{align*}
Since 
\begin{align*}
&\Bigl( C'' +
 c\log \bigl(C_4 \max\{r^{12},\epsilon^{12}\}
 |\Delta_{24}^{-1}(re^{\sqrt{-1}\theta})|\bigr)\Bigr)^{\frac{1}{2}} \\
&\leq C''^{\frac{1}{2}}
 + \Bigr( c \log^+ \bigl(C_4 \max\{r^{12},\epsilon^{12}\}
 |\Delta_{24}^{-1}(re^{\sqrt{-1}\theta})|\bigr)\Bigr)^{\frac{1}{2}},
\end{align*}
we can apply Lemma~\ref{integral.estimate} and conclude that
\begin{align*}
 r^{\frac{1}{10}} \rho(r)&\leq
 C_7 r^{\frac{1}{10}}
 \Bigl(\log^+\Bigl(\max\{r^{12},\epsilon^{12}\}|\chi|^{-1}
 \!\! \prod_{i=k+1}^{24} \! |\chi_{i}|
 \prod_{i=1}^{24}\min\{r^{-1},|\chi_i|^{-1}\}\Bigr)\Bigr)^{\frac{1}{2}}
 \!\! +  \! C_8r^{\frac{1}{10}}\\ 
  &= C_7 r^{\frac{1}{10}}
 \Bigl(\log^+\Bigl(\max\{r^{12},\epsilon^{12}\}|\chi|^{-1}
  \prod_{i=1}^{k}\min\{r^{-1},|\chi_i|^{-1}\}\Bigr)\Bigr)^{\frac{1}{2}}
 + C_8r^{\frac{1}{10}}
\end{align*}
for some constants $C_7, C_8 >0$.
The first inequality of Claim~\ref{rho.estimate} follows from this. 

For the second inequality, recall that (at least) one of 
 $|\alpha_2|, |\alpha_3|, |\alpha_4|, |\beta_1|, \dots, |\beta_6|$
 equals $\epsilon$.
Suppose $|\alpha_4|=\epsilon$ for example.
Then, note that $$|re^{\sqrt{-1}\theta}-\alpha_4|\geq \frac{\epsilon}{2}$$
 if $r\leq \frac{\epsilon}{2}$. 
Therefore,
\begin{align*}
\rho(r)
&\leq \Bigl(\frac{\epsilon r}{2}\Bigr)^{-\frac{1}{10}} \int_{0}^{2\pi}
C_3^{-\frac{1}{2}} 
 \prod_{i=2}^3|re^{\sqrt{-1}\theta}-\alpha_i|^{-\frac{1}{10}}
 \prod_{i=1}^6|re^{\sqrt{-1}\theta}-\beta_i|^{-\frac{1}{10}} \\ \nonumber
&\qquad \times  \Bigl( C'' + 
 c\log \bigl(C_4 \max\{r^{12},\epsilon^{12}\}
 |\Delta_{24}^{-1}(re^{\sqrt{-1}\theta})|\bigr)\Bigr)^{\frac{1}{2}}
 r d\theta.
\end{align*}
Thus, by applying Lemma~\ref{integral.estimate} again, we obtain
\begin{align*}
(\epsilon r)^{\frac{1}{10}}\rho(r)\leq 
 C_9 r^{\frac{1}{5}}
 \Bigl(\log^+\Bigl(\epsilon^{12} |\chi|^{-1}
 \prod_{i=1}^{k}\min\{r^{-1},|\chi_i|^{-1}\}\Bigr)\Bigr)^{\frac{1}{2}}
 + C_{10}r^{\frac{1}{5}}
\end{align*}
for some constants $C_9,C_{10}>0$.
The latter half of Claim~\ref{rho.estimate} follows from this.
\end{proof}

\begin{Rem}\label{rem:a3b1}
Note that if $s\in M^{\rm seg, o}$, or equivalently $a^3\neq 27b^2$, 
 then $k=12$ and $|\chi|^{-1}$ is bounded, which implies that 
 $\max_r \nu(r)$ is bounded above by a constant. 
Hence Claim~\ref{rho.estimate} is equivalent to
 that $\max_r \rho(r)$ is bounded above by a constant. 
The argument here was designed in order to include
 the most difficult case $a^3=27b^2$.
\end{Rem}

\begin{Rem}\label{ab0.1}
We have been assuming $ab\neq 0$ for simplicity
 but in the case $ab=0$ ($a=b=0$ is not possible though),
 i.e., when a sequence of $X_{g_{8},g_{12}}^{W}$ approaches to
 the isotrivial locus of (c), (d) in
 \S\ref{Isotrivial.classification},
 we can prove similarly that $\max_r \rho(r)$ is bounded above by a constant.
(In this case, $\max_r \nu(r)$ is bounded; see Remark~\ref{rem:a3b1}).

Suppose $a\neq 0$ and $b=0$. 
Then the above argument works once
 we replace the definition of $\epsilon$
 by $\epsilon:= \max\{|\alpha_2|,|\alpha_3|,|\alpha_4|\}$, 
 and replace \eqref{wAG}, \eqref{g8g12} by 
\begin{align*}
|g_{8}(t)|^{\frac{1}{2}}+|g_{12}(t)|^{\frac{1}{3}}
 \geq |g_{8}(t)|^{\frac{1}{2}}
 \geq C_3 |t|^{\frac{1}{2}}\prod_{i=2}^{4}|t-\alpha_i|^{\frac{1}{2}}.
\end{align*}
Then Claim~\ref{rho.estimate} follows and
 $\max_r \rho(r)$ is bounded above.

If $a=0$ and $b\neq 0$,
 we suppose $\beta_1=0$ instead of $\alpha_1=0$
 and set $\epsilon:= \max_{2\leq i\leq 6}\{|\beta_i|\}$.
The rest of the arguments are similar.

We will continue discussion for the case $ab=0$ in Remark \ref{ab0.2}. 
\end{Rem}

\bigskip

It is easy to see from the definition of $\nu(r)$ that
\begin{align}\label{nu.easy}
\log \nu(r) \leq \log \nu(r') + 12\Bigl|\log \frac{r}{r'}\Bigr|  
\end{align}
for any $r,r'>0$.
From Claim~\ref{rho.estimate}, by replacing $C_5$ if necessary,
 we get
\begin{align}\label{rho.estimate2}
\max_{0< r\leq 1} \rho(r)
 \leq
 C_5 + C_6 \bigl(\log^+
 \bigl(\max_{2\epsilon\leq r\leq 1} \nu(r)\bigr)
 \bigr)^{\frac{1}{2}}.
\end{align}

\bigskip

We now turn to the estimate of $d(g_{8},g_{12})$ defined in Lemma~\ref{claim1:conv.MWseg}, 
 the distance between $0$ and $\infty$.
We have
\begin{align*}
d(g_{8},g_{12})
&\geq
 \int_{0}^{\infty}
 \min_{0\leq \theta\leq 2\pi}
  \mu(g_{8}(re^{\sqrt{-1}\theta}),
 g_{12}(re^{\sqrt{-1}\theta}))^{\frac{1}{2}} dr  \\
&\geq 
 \int_{2\epsilon}^{1}
 \min_{0\leq \theta\leq 2\pi}
  \mu(g_{8}(re^{\sqrt{-1}\theta}),
 g_{12}(re^{\sqrt{-1}\theta}))^{\frac{1}{2}} dr.
\end{align*}
By Lemma~\ref{ML.estimate}, we have 
\begin{align*}
\mu(g_{8}(t),g_{12}(t))
&\geq 
-\frac{C}{|g_{8}(t)|^{\frac{1}{2}}+|g_{12}(t)|^{\frac{1}{3}}} \\ 
&  + \frac{c}{|g_{8}(t)|^{\frac{1}{2}}+|g_{12}(t)|^{\frac{1}{3}}}
 \log \frac{|g_{8}(t)|^3+27|g_{12}(t)|^2}{|\Delta_{24}(t)|}.
\end{align*}
On $2\epsilon\leq |t|\leq 1$, the 
 functions $|t|^{-4}|g_{8}(t)|$ and $|t|^{-6}|g_{12}(t)|$
 are bounded from above and below by positive constants.
Moreover,
\[|t|^{12}|\Delta_{24}(t)|^{-1} \geq 2^{-24}\nu(|t|)\]
 since $|t-\chi_{i}|\le 2\max\{|t|,|\chi_{i}|\}$ for $1\leq i<k$ 
 and $|\chi_{i}^{-1}t-1|\leq 2$ for $i>k$. 
Therefore, 
\begin{align*}
\mu(g_{8}(t),g_{12}(t))
\geq |t|^{-2} (C_{11}\log \nu(|t|) -C_{12}),
\end{align*}
where $C_{11}$ and $C_{12}$ are positive constants, 
for any $t$ with $2\epsilon\le |t|\le 1$. 
This implies 
\begin{align}\label{mu.estimate.circle3}
\mu(g_{8}(t),g_{12}(t))^{\frac{1}{2}}
\geq |t|^{-1} 
 \bigl(C_{11}^{\frac{1}{2}}\bigl(\log^+ \nu(|t|)\bigr)^{\frac{1}{2}}
 -C_{12}^{\frac{1}{2}}\bigr)
\end{align}
for $2\epsilon\leq |t|\leq 1$.
Suppose that 
\begin{align}\label{r0}
\max\{\nu(r)\mid 2\epsilon\leq r\leq 1\}=\nu(r_0)
\end{align}
with $2\epsilon \leq r_{0}\leq 1$. 
Then it follows from the explicit form of $\nu(r)$ in \eqref{nu.def} and our 
assumption on the order of $|\chi_{i}|$ in \eqref{order.chi} that
\begin{align*}
&r_0\in [2\epsilon, 1] \cap [|\chi_{12}|, |\chi_{13}|] \ \ 
 \text{ if $[2\epsilon, 1] \cap [|\chi_{12}|, |\chi_{13}|]\neq \emptyset$},\\
&r_0 = 2\epsilon \ \ \text{ if $|\chi_{13}|<2\epsilon$}, \\
&r_0 = 1 \ \ \text{ if $1<|\chi_{12}|$}.
\end{align*}
Moreover, by \eqref{nu.easy},
\begin{align}\label{nu.ineq}
 \log \nu(r) \geq \log \nu(r_0) - 12 \Bigl|\log \frac{r}{r_0}\Bigr|
\end{align}
for any $0<r\leq 1$.
The right hand side function $\log \nu(r_0) - 12 \left|\log \frac{r}{r_0}\right|$ here
 is increasing on $r<r_0$ and decreasing on $r_0<r$.
When $\log\nu(r_{0})\ge 0$, we define 
$r'_1, r'_2$ to be points such that $r'_1\leq r_{0}\leq r'_2$ and 
\begin{align}\label{r1.r2.prime.case1}
\log \nu(r_0) - 12 \Bigl|\log \frac{r'_1}{r_0}\Bigr|
=\log \nu(r_0) - 12 \Bigl|\log \frac{r'_2}{r_0}\Bigr|=0.
\end{align}
When $\log \nu(r_0) < 0$, we put $r'_1=r'_2=r_0$.
We then put 
\[ r_1:=\max\{r'_1,2\epsilon\}\text{ and }r_2:=\min\{r'_2,1\}.\]

We see that \eqref{nu.ineq} implies 
\begin{align*}
 (\log^+ \nu(r))^{\frac{1}{2}}
 \geq (\log^+ \nu(r_0))^{\frac{1}{2}}
  - \sqrt{12} \Bigl|\log \frac{r}{r_0}\Bigr|^{\frac{1}{2}}.
\end{align*}
Hence 
\begin{align*}
\int_{2\epsilon}^{1} \frac{(\log^+\nu(r))^{\frac{1}{2}}}{r} dr
\geq 
\int_{r_1}^{r_2} \biggl(\frac{(\log^+\nu(r_0))^{\frac{1}{2}}}{r}
 -\frac{\sqrt{12}}{r}\Bigl|\log \frac{r}{r_0}\Bigr|^{\frac{1}{2}}\biggr) dr.
\end{align*}
We calculate
\begin{align*}
\int_{r_1}^{r_0}
\frac{\sqrt{12}}{r}\Bigl|\log \frac{r}{r_0}\Bigr|^{\frac{1}{2}} dr
&=\int_{r_1}^{r_0}
\frac{\sqrt{12}}{r}\Bigl(\log \frac{r_0}{r}\Bigr)^{\frac{1}{2}} dr \\
&=\sqrt{12}\Bigl[-\frac{2}{3}
 \Bigl(\log \frac{r_0}{r}\Bigr)^{\frac{3}{2}}\Bigr]_{r=r_1}^{r_0} \\
&=\frac{2\sqrt{12}}{3} \Bigl(\log \frac{r_0}{r_1}\Bigr)^{\frac{3}{2}}.
\end{align*}
By the definition of $r_1$, 
\begin{align*}
\sqrt{12}\Bigl(\log \frac{r_0}{r_1}\Bigr)^{\frac{1}{2}}
\leq (\log^+ \nu(r_0))^{\frac{1}{2}}.
\end{align*}
Therefore, 
\begin{align*}
\int_{r_1}^{r_0} \biggl(\frac{(\log^+\nu(r_0))^{\frac{1}{2}}}{r}
 -\frac{\sqrt{12}}{r}\Bigl|\log \frac{r}{r_0}\Bigr|^{\frac{1}{2}}\biggr) dr
\geq \frac{\sqrt{12}}{3}(\log^+ \nu(r_0))^{\frac{1}{2}}\log\frac{r_0}{r_1}.
\end{align*}
The integral from $r_0$ to $r_2$ can be similarly estimated.
Consequently, by \eqref{mu.estimate.circle3}, we proved 
\begin{Claim}\label{d.estimate}
the distance of $0$ and $\infty$ on $B_{g_{8},g_{12}}$ 
has the following lower bound: 
$$d(g_{8},g_{12})\geq
 \frac{\sqrt{12}}{3}(\log^+ \nu(r_0))^{\frac{1}{2}}
 \Bigl(\log\frac{r_0}{r_1}+\log\frac{r_2}{r_0}\Bigr)
= \frac{\sqrt{12}}{3}(\log^+ \nu(r_0))^{\frac{1}{2}} \log\frac{r_2}{r_1},$$
where $r_{0},r_{1},r_{2}$ are defined in \eqref{r0}, \eqref{r1.r2.prime.case1}. 
\end{Claim}

When a sequence of $(g_8, g_{12})$ converges to $(at^4,bt^6)$, recall that $\epsilon \to 0$. 
It is enough to consider one of the following situations:
\begin{enumerate}
\item $\nu(r_0) \to \infty$,
\item $\nu(r_0)$ is bounded.
\end{enumerate}
Note that the case (i) occurs only when $a^3=27b^2$ (see Remark~\ref{rem:a3b1}).

For the case (i), it follows that $\frac{r_2}{r_1}\to \infty$ 
because of the definition \eqref{r1.r2.prime.case1}. 
Then by \eqref{rho.estimate2} and Claim~\ref{d.estimate}, 
 we obtain 
\[d(g_{8},g_{12})^{-1}\max\{\rho(r)\mid 0<r\leq 1\}\to 0.\]

For the case (ii), the estimate
 \eqref{rho.estimate2} implies that
 $\max\{\rho(r)\mid 0<r\leq 1\}$ is bounded.
On the other hand, by the lower bound of McLean metric in 
Lemma~\ref{ML.estimate3}, we have 
\begin{align*}
\mu(g_{8}(t),g_{12}(t))\geq
 \frac{c'}{|g_8(t)|^{\frac{1}{2}}+|g_{12}(t)|^{\frac{1}{3}}}.
\end{align*}
For a fixed $\epsilon_0\ll 1$, 
 the function $|g_8(t)|^{\frac{1}{2}}+|g_{12}(t)|^{\frac{1}{3}}$
 uniformly converges to
 $(|a|^{\frac{1}{2}}+|b|^{\frac{1}{3}})|t|^2$
 on $\{t\in\C\mid\epsilon_0\leq |t|\leq 1\}$
 as $(g_8,g_{12})\to (at^4,bt^6)$.
Therefore, if $(g_8,g_{12})$ is close enough to $(at^4,bt^6)$, 
 we have 
\begin{align*}
d(g_{8},g_{12})&\geq \int_{\epsilon_0}^{1}
 \min_{0\leq \theta\leq 2\pi}
  \mu(g_{8}(re^{\sqrt{-1}\theta}),
 g_{12}(re^{\sqrt{-1}\theta}))^{\frac{1}{2}} dr \\
&\geq \int_{\epsilon_0}^{1}
 \frac{c}{r}  dr = c \log \epsilon_0^{-1}  
\end{align*}
for some constant $c>0$ which depends only on $a$ and $b$.
Since $\epsilon_0$ can be arbitrarily small,
 $d(g_{8},g_{12})\to +\infty$ as $(g_8,g_{12})\to (at^4,bt^6)$.
Hence we obtain the convergence
 $d(g_{8},g_{12})^{-1}\max\{\rho(r)\mid 0<r\leq 1\}\to 0$ also
 in the case (ii). 
 
\begin{Rem}\label{ab0.2} 
Again we make a comment for the case $ab=0$
 as in Remark~\ref{ab0.1}. 
If this is the case, we have (ii), namely, $\nu(r_0)$ is bounded.
Then the above proof of $d(g_{8},g_{12})\to +\infty$
 is still valid and we obtain the desired estimate Lemma~\ref{claim1:conv.MWseg}.
 in the case $ab=0$ as well.
\end{Rem}

We therefore complete the proof of Lemma~\ref{claim1:conv.MWseg}
 and then that of Proposition~\ref{prop:seg.conti}.
\end{proof}

Now, we are ready to prove Theorem~\ref{Mwbar.GH.conti} on the Gromov-Hausdorff continuity 
of the geometric realization maps. 

\begin{proof}[proof of Theorem~\ref{Mwbar.GH.conti}]
The continuity of $\Phi_{\rm ML}$ on $M_W$ is Proposition~\ref{prop:MW.conti}.
The continuity on a neighborhood of $M_W^{\rm nn}\setminus M_W^{\rm seg}$
 (resp.\ of $M_W^{\rm seg}$)
 is Proposition~\ref{prop:nn.conti} (resp.\ Proposition~\ref{prop:seg.conti}).
Therefore, the theorem follows from these three propositions,
 whose proofs are already given.
\end{proof}


\newpage

\chapter{Higher dimensional hyperK\"ahler varieties case}\label{high.dim.HK.sec}

\section{Introduction}

In this chapter, we discuss generalizations of our discussions on K3 surfaces, 
to higher dimensional irreducible symplectic manifolds. Note that our main points of the arguments for K3 case in 
\S\ref{F2d.sec} to \S\ref{Kahler.section} 
were the hyperK\"ahler metrics, Torelli type theorems on the periods, and some structure theorems for K\"ahler cones. 
Fortunately, recently developed understanding of irreducible symplectic manifolds 
provide their higher dimensional generalizations to some extent. Thus, we can naturally believe that after solving further technical problems, 
one can generalize all our conjectures, results and our discussions for K3 surfaces to  higher dimensional irreducible symplectic manifolds. 

Before setting up our notation etc., 
we review our protagonists --- irreducible symplectic manifolds and their singular versions. 

\begin{Def}\label{IHS.def}
\begin{enumerate}
\item 
A (complex) $n$-dimensional compact K\"ahler manifold $X$ is called an \textit{irreducible holomorphic symplectic manifold} if it satisfies the following conditions: 
\begin{itemize}
\item $X$ is simply connected, 
\item $H^0(X,\Omega_X^{2})$ is generated by a holomorphic symplectic form i.e., a symplectic form which is holomorphic. 
\end{itemize}
As it easily follows from the second condition, 
$n$ is always even and 
also $K_X\sim 0$. 
The second condition also excludes the products of irreducible symplectic manifolds. 

It also implies more. 
Recall that by Yau's theorem \cite{Yau}, to each K\"ahler class $c$ of $X$, there is a unique Ricci-flat K\"ahler metric whose K\"ahler class is 
$c$. Then the existence of the holomorphic symplectic form, unique up to multiple, 
implies that its holonomy group automatically becomes the compact symplectic group, i.e., $Sp(\frac{n}{2})$ 
so that it is irreducible. 
It follows from the preservation of the 
holomorphic symplectic form, the Berger classification of the holonomy group combined with the de Rham decomposition. 
Therefore, this $X$ has 
many hyperK\"ahler metrics in the strict sense 
(see e.g., \cite[Chapter~7]{Joy}). For this reason,
 we sometimes call an irreducible symplectic manifold as a compact
 (irreducible) hyperK\"ahler manifold in this paper. 
\medskip
\item (cf., \cite{Beauville, Nam1}) 
In this paper, we call a normal projective variety $X$ (resp., a normal compact Hausdorff complex analytic space) 
an \textit{irreducible holomorphic symplectic variety} (resp., an \textit{irreducible holomorphic symplectic analytic space})\footnote{We do \textit{not} assume stronger condition that 
reflexive forms of other even degrees are also generated by the holomorphic symplectic form, 
as in e.g., \cite[Definition 1.3]{GGK}. See also \cite[\S14]{GGK}.}
if it satisfies the following conditions: 
\begin{itemize}
\item $K_X\sim 0$, 
\item there exists a holomorphic symplectic form on the smooth part $X_{\rm sm}$ of $X$
 that extends to any resolution of singularities of $X$, 
\item an above holomorphic symplectic form is unique up to constant. 
\end{itemize}
The extendability condition above immediately implies that $X$ only has canonical singularities. 
\end{enumerate}
\end{Def}

For more details of the basic of the theory, we refer to e.g., \cite{Huy.HK} for algebro-geometric background and to e.g., \cite{Joy} for more differential geometric background. 

\section{Conjectural picture I --- projective case}\label{HK.alg}

Now we make our setup to discuss the moduli problem for polarized projective situation. 
Fix a lattice $\Lambda$ of signature $(3,r-3)$ for $r>3$. 
Fix any \textit{connected} moduli space $\mathcal{T}_{M}$ of complex $n$-dimensional 
irreducible holomorphic 
symplectic manifolds $X$ with (fixed) marking on the second cohomologies 
$H^{2}(X,\Z)$. 
Write $M$ for the corresponding deformation class. 
Here, the marking means an isometry $\alpha\colon H^{2}(X,\Z)\simeq \Lambda$ for a fixed lattice $\Lambda$ of rank $r:=b_{2}(X)$, 
where the left hand side is associated with the  
Beauville-Bogomolov-Fujiki form $(-,-)_{\rm BBF}$. 
The lattice $\Lambda$ 
generalizes the K3 lattice $\Lambda_{\rm K3}=
U^{\oplus 3}\oplus E_8^{\oplus 2}$ for when $X$ is a K3 surface 
$(n=2)$. Due to the connectivity assumption of $\mathcal{T}_{M}$, the corresponding Fujiki constant (\cite[Theorem 4.7]{Fjk.cst}, also cf., e.g., 
\cite[\S3.1]{GHS}) 
of the symplectic manifolds in the class is a constant because of its deformation invariance. 
By a theorem of Huybrechts \cite[\S8]{Huy.HK} whose proof uses hyperK\"ahler rotations, 
it is known that $\mathcal{T}_{M}$ surjectively maps to the period domain 
$$
\Omega(\Lambda):=\{[w]\in \mathbb{P}(\Lambda\otimes \mathbb{C})\mid
 (w,w)=0,\ (w,\bar{w})>0\}. 
$$
By combining with \cite{Mark.Torelli, Ver}, it further follows that 
the map is actually a Hausdorff reduction. This extends the phenomenon in the case of K3 surfaces. 

As in the case of K3 surfaces, we now consider polarizations, i.e., ample line bundles. 
For that, we fix a primitive vector $\lambda\in \Lambda$ with 
$(\lambda,\lambda)>0$. 
We now restrict to the subspace of $\mathcal{T}_{M}$ 
where $X$ is polarizable by a line bundle in the class $\lambda$. That is, 
we restrict to the period subdomain 
$\Omega(\lambda^{\perp}):=\Omega(\Lambda)\cap \mathbb{P}(\lambda^{\perp}).$ 
Note that it again has two connected components. 
We denote one of the two connected components by $\mathcal{D}_{M}$. 

From \cite{Mat}, \cite{KolMat}, 
the class of polarized symplectic manifold $(X,L:=\alpha^{-1}(\lambda))$ for all $(X,\alpha)$ in $\mathcal{T}_{M}$ 
is bounded in the usual sense that they are parametrized by a finite type scheme, so that 
the natural moduli algebraic stack of the polarized varieties in the (deformation) class $M$ itself 
is also of finite type. We denote the Deligne-Mumford moduli stack by 
$\mathcal{M}$. By Viehweg's theorem 
\cite{Vie}, it has a quasi-projective coarse moduli variety which we also denote 
by $M$.

Let us set $\lambda\in \Lambda$ to be the class of polarization 
of our class of marked symplectic manifolds in $M$. Then we define $\lambda^{\perp}(\subset \Lambda)$ to be the 
orthogonal lattice of $\lambda\in \Lambda$. 
Note that from the Bogomolov-Tian-Todorov unobstructedness \cite{Bog1, Bog2, Tia.Bog, Tod89}, 
it follows that $\mathcal{M}$ is smooth as a Deligne-Mumford stack and hence 
$M$ has only quotient singularities, i.e., 
orbifolds. Thanks to the recent global Torelli theorem by Verbitsky 
(cf., \cite{Ver, Mark.Torelli}), 
$M$ is a Zariski open subset of a Hermitian locally symmetric space of orthogonal 
type $\Gamma\backslash \mathcal{D}_{M}$ (cf., e.g., \cite[Theorem 3.7]{GHS}). 
Here, $\mathcal{D}_{M}\subset \mathbb{P}(\lambda^{\perp}\otimes{\C})$ is the natural 
generalization of $\mathcal{D}(\Lambda_{2d})$ in the K3 surface case, 
i.e., the 
connected component of the 
period domain which is an orthogonal type IV Hermitian 
symmetric domain. The arithmetic subgroup $\Gamma$ of $O(\lambda^{\perp})$ 
is generated by polarization-stabilizing monodromies on 
$H^{2}$ of the family over $M$ (cf.\  \cite{Ver,Mark.Torelli,GHS} for the 
detailed definition), 
which was one of the main points of the breakthrough \cite{Ver} (cf. also \cite{Ver.er}).\footnote{$\Gamma$ is denoted by ``${\rm Mon}^{2}(X,H)$'' in \cite[\S3]{GHS}, \cite[\S7]{Mark.Torelli}.}
Note that the above arguments also reprove the boundedness of $M$ in particular. 

\medskip

We conjecture that our pictures and discussions for polarized K3 surfaces in \S\ref{F2d.sec} 
 extend to polarized irreducible 
holomorphic symplectic varieties case, after solving technical difficulties. To state the conjecture, 
we review that also in our situation,  we set the notation for the Satake compactification of $\Gamma\backslash \mathcal{D}_M$ for the adjoint representation, 
extending \eqref{K3.Sat.strata} in \S\ref{K3.Sat.sec}: 
\begin{align}\label{HK.mod.Sat.strata}
\overline{\Gamma\backslash \mathcal{D}_M}^{\rm Sat, \tau_{\rm ad}}=(\Gamma\backslash \mathcal{D}_M)
\sqcup \bigsqcup_{l} (\Gamma\backslash \mathcal{D}_M)(l)
\sqcup \bigsqcup_{p} (\Gamma\backslash \mathcal{D}_M)(p),
\end{align}
where $l$ (resp., $p$) run over all $\Gamma$-equivalent classes of 
one dimensional (resp., two dimensional) vector subspaces of $(\lambda^{\perp})\otimes \Q$, 
which are isotropic with respect to the Beauville-Bogomolov-Fujiki pairing. 

The first rough statement of our conjectures is as follows. Below, 
we regard $M$ as an open subspace of $\Gamma\backslash \mathcal{D}_{M}$. 
\begin{Conj}[Basic non-explicit version]\label{HK.conj}
There 
is a continuous map $\Psi_{\rm alg}$, which we again 
call the geometric realization map, 
from the Satake compactification $\overline{\Gamma\backslash \mathcal{D}_M}^{\rm Sat,\tau_{\rm ad}}$ 
with respect to the adjoint representation of $O(\lambda^{\perp}\otimes \R)$ 
to the space of all compact metric spaces with diameter $1$, associated with the Gromov-Hausdorff topology. 

More precisely,
 the $(b_{2}(X)-4)$-dimensional boundary strata $(\Gamma\backslash \mathcal{D}_M)(l)$
 (cf., \eqref{HK.mod.Sat.strata} above) parametrize via $\Psi_{\rm alg}$ the $\frac{n}{2}$-dimensional complex projective space 
$\mathbb{P}^{\frac{n}{2}}$ with special K\"ahler metrics in the sense of \cite{Str90, DW96, Hit99, Freed99} etc.\ and the 
 $0$-dimensional cusps $(\Gamma\backslash \mathcal{D}_M)(p)$ 
 parametrize metric spaces which
 are homeomorphic to the closed ball of real dimension $\frac{n}{2}$. 
\end{Conj}

Note that this is a weaker form than our conjecture \ref{K3.Main.Conjecture} for K3 case
 in the sense that $\Psi_{\rm alg}$ is \textit{not} geometrically explicitly described yet. 
Such a more precise (stronger) version
 with explicit $\Psi_{\rm alg}$ will be partially given in 
Conjecture~\ref{Good.HK.conj} later (see also Theorem~\ref{Moduli.Shimura} and Theorem~\ref{HK.GH.continuity}). 
The expectation in above Conjecture~\ref{HK.conj} 
that the $(b_{2}(X)-4)$-dimensional boundary strata $(\Gamma\backslash \mathcal{D}_M)(l)$ 
of $\overline{\Gamma\backslash \mathcal{D}_M}^{\rm Sat,\tau_{\rm ad}}$ 
should parametrize $\mathbb{P}^{\frac{n}{2}}$ is partially motivated by the 
question if the base manifolds of Lagrangian fibrations 
on irreducible holomorphic symplectic manifolds are always complex projective 
space (cf., e.g., \cite{Mat.Pn}, \cite[21.4]{Huy}). The question is proved affirmatively, under the 
smoothness assumption of the base, by \cite{Hwa, GL}. 

If we regard the moduli space $\mathcal{F}_{2d}$ 
of polarized K3 surfaces as the moduli of 
their $e$-th symmetric products for fixed $e\in \Z_{>0}$ (which are known 
to be holomorphic symplectic varieties with canonical singularities), 
our partial 
proof of Conjecture~\ref{HK.conj} for K3 surfaces 
naturally extend to that higher dimensional case, or more precisely, of its singular 
extension (as the symmetric products are singular).\footnote{In such non-smoothable setup, the corresponding global Torelli theorem is only known along equisingular direction (\cite{BL}, \cite{BL2}) 
as far as the authors know. }
Indeed, we only need to follow our arguments in \S \ref{F2d.sec}, 
while replacing each K3 surfaces and each morphism appearing in the 
discussions for the 
K3 surface case (Theorems \ref{K3.Main.Conjecture.18.ok}, 
\ref{K3.Main.Theorem2}) by their 
$e$-th symmetric products respectively. In this case, 
the metrized closed ball parametrized at the $0$-dimensional cusps 
is $$\{(x_{1},\cdots,x_{e})\in \mathbb{R}^{e}\mid 
0\le x_{1}\le x_{2}\le \cdots \le x_{e}\le 1\}$$ 
with restriction of the standard flat metric on $\mathbb{R}^{e}$. 
If we consider the moduli of different type of 
finite quotient of self-product of 
K3 surfaces, we obtain different metrized closed balls 
(and also the positive dimensional boundary components could 
possibly parametrize metrized weighted projective varieties of dimension $n$ 
rather than $\mathbb{P}^{n}$.) 
At the moment of this writing, 
we do not have examples of $M$ where the metric spaces which 
are associated to the $0$-dimensional cusps are not expected to be 
polytopes with the restriction of the flat metrics. 

As we mentioned, we believe that, benefiting from the recent global Torelli theorem (cf., \cite{Ver, Mark.Torelli}) and more recent developments 
which generalize the results for K3 surfaces, 
basically the general idea and the outline of our arguments on the K3 surfaces case 
in the previous chapters \S\ref{F2d.sec}, \S\ref{GTZ.extend.proof}, 
naturally extend and would show a similar statement to Theorem~\ref{K3.Main.Conjecture.18.ok} 
for higher dimensional 
irreducible holomorphic symplectic 
varieties at the end. 
Nevertheless, for such higher dimensional extension, 
a number of technical problems will arise in general 
which we discuss below. Among them, there are algebro-geometric problems and 
analytic problems. The analytic side would be solved if one can modify and refine the proof of 
\cite{GW, GTZ1, GTZ2, TZ} and 
Theorem~\ref{GTZ.extend} 
to a version uniform with respect to bounded variations of reference 
\textit{higher dimensional} 
hyperK\"ahler varieties. We have not worked much out on this. For the 
algebro-geometric problems side, we show our partial progress below. In particular, 
when $X_{s}$ is deformation equivalent to the class of ${\rm K3}^{[\frac{n}{2}]}$ or the class of generalized Kummer varieties, 
most of the algebro-geometric problems are essentially solved. 


\section{Partial compactification and non-collapsing limits}\label{geom.prob.HK}

We now start to go into algeblo-geometric side of problems toward Conjecture \ref{HK.conj}. 
First we establish a generalization of the partial compactification. 
Recall that $\mathcal{F}_{2d}^{o}$, the moduli of 
smooth polarized K3 surfaces of degree $2d$, is a Zariski open subset 
of $\mathcal{F}_{2d}$, the moduli of polarized 
possibly ADE singular K3 surfaces of degree $2d$. 
In our higher dimensional situation, symplectic varieties in the sense of \cite{Beauville, Nam1} 
(see Definition~\ref{IHS.def}) play the same role as ADE singular degenerations in the 
K3 surfaces case, and by using them, we establish an analog 
of a partial compactification $\mathcal{F}_{2d}^{o}\subset \mathcal{F}_{2d}$ in Theorem \ref{Moduli.Shimura} 
later. 

We recall our setup from \S \ref{HK.alg} as follows, for our singular extension.  
Fix again a connected quasi-projective 
moduli orbifold $M$ of smooth polarized irreducible holomorphic 
symplectic manifolds. 
There is a positive integer $m_0$ such that any polarized variety $(Y,N)\in M$ satisfies that 
$N^{\otimes {m_0}}$ is very ample. This is a standard fact follows from the quasi-compactness of $M$ but also follows from \cite{Mat}. 
On the other hand, the global Torelli theorem \cite{Ver, Mark.Torelli} 
(cf., also \cite[Theorem 3.7]{GHS}) asserts that 
the period map, which associates to each marked polarized 
smooth irreducible holomorphic symplectic variety $X$ 
its weight $2$ polarized $\mathbb{Z}$-Hodge structure, gives a Zariski 
open immersion $p\colon M\hookrightarrow 
\Gamma\backslash \mathcal{D}_{M}$ where $\Gamma\backslash\mathcal{D}_{M}$ is an 
orthogonal 
Hermitian locally symmetric space (cf., e.g., 
\cite[Theorem 3.7]{GHS}) with respect to a certain discrete subgroup 
$\Gamma$ of $O(\Lambda)$. 
Now, we would like to regard $M\hookrightarrow \Gamma\backslash \mathcal{D}_{M}$ 
as a partial compactification and will discuss a geometric meaning to it. 

For that we prepare 
a series of quasi-projective partial compactifications of $M$ as follows. 
Here, a \textit{normal $\mathbb{Q}$-Gorenstein degeneration along $M$ with exponent $m$}
 $(m\in \mathbb{Z}_{>0})$, means a 
polarized normal $\mathbb{Q}$-Gorenstein variety $(X,L)$ such that there is a $\mathbb{Q}$-Gorenstein\footnote{We put this condition for future extension of the results to those with only \textit{numerically trivial} canonical divisors, 
but in this paper's case where the canonical line bundles are trivial, 
the ($\mathbb{Q}$-)Gorensteinness condition of the total space of deformations 
automatically holds. Indeed, it follows from the adjunction plus the openness of Gorensteinness of the fibers, or by more general fact that 
all the fibers have only canonical singularities 
(\cite{Ka}, also cf.\ \cite[Theorem 9.1.13]{Ish.bk}). }
 flat family over a smooth pointed curve  $(\mathcal{X},\mathcal{L})\to C\ni 0$ 
whose general fiber $(X_t,L_t)$ $(t\neq 0)$ satisfies that there is some $(X_t,N_t)\in M$
 with $N_t^{\otimes m}=L_t$ while $(\mathcal{X}_0,\mathcal{L}_0)=(X,L).$ 

Let us now set $M^{(m)}$ as the set of normal 
$\mathbb{Q}$-Gorenstein degenerations $(X,L)$ along $M$ 
with exponent $m$ such that 
 $X$ has only canonical singularities, $K_X=0$, and $L$ is \textit{very} ample. 
 
 (If we change the definition of $M^{(m)}$ 
 by only assuming ampleness of $L$, we can still prove the essentially equivalent theorem as
 Theorem~\ref{Moduli.Shimura}, 
 thanks to the effective basepoint-free theorem \cite[1.1]{Kol.bp}
 and the very ample lemma \cite[\S4]{Fjn2}, \cite[\S7]{Fjn1} (cf., also \cite[1.2]{Kol.bp}. 
 We omit the details but we would like to thank very much Chen Jiang for pointing out these references for us.)
 
 Furthermore, we consider the corresponding moduli functor $\mathcal{M}^{(m)}$ whose value $\mathcal{M}^{(m)}(Y)$ for an algebraic scheme $Y$ over $\C$ 
consists of the natural 
equivalence classes (cf., \cite[\S1.1]{Vie}) of flat polarized proper 
families $(\mathcal{X},\mathcal{L})$ over $Y$ 
of canonical singular normal $\mathbb{Q}$-Gorenstein degenerations along $M$ 
with the exponent $m$. The equivalence relation for them is, as usual (cf., e.g., 
\cite[1.1]{Vie}), defined as 
twisting the polarization $\mathcal{L}$ 
by taking the tensor product with the pullback of a line bundle on the base $Y$. 
See \cite{Vie} for the details.  
Then we define $M^{(m)}$, not only as a set,
 but as the corresponding coarse moduli scheme. 
Indeed, as discussed in Viehweg's \cite[Chapter 8 (esp., \S 8.7)]{Vie}, 
such coarse moduli scheme exists as a quasi-projective scheme for each $m$. 
Also note that \cite[\S 4 (the proof of Theorems 0.3, 0.8)]{KLSV} proves that such $X$ is automatically a symplectic variety in the sense of \cite{Beauville, Nam1} (see Definition~\ref{IHS.def}),
 depending on e.g., 
a rigidity result of $\mathbb{Q}$-factorial terminal symplectic singularity 
\cite{Nam3}. 

As a quotient stack, $\mathcal{M}^{(m)}$ is again a Deligne-Mumford stack 
since for any $[(X,L)]\in \mathcal{M}^{(m)}(\C)$, 
we have that ${\rm Aut}(X,L):=\{g\in {\rm Aut}(X)\mid g^{*}L\simeq L\}$ 
is a finite group by \cite[Proposition 4.6]{Ambro} and the fact that 
for any polarized variety $(X,L)$, ${\rm Aut}(X,L)$ is a linear 
algebraic group (also see \cite[1.5]{Od2} for another more differential 
geometric approach). Further, 
from the construction (cf., \cite[esp. Chapter 8]{Vie}), 
there is a natural map 
$\iota\colon M\to M^{(m)}$ whose image is a Zariski open subset. 
Moreover, $\iota$ is an open immersion since the twisting of line bundle 
does not lose information by the torsion freeness of ${\rm Pic}(X)\subset 
H^{2}(X,\Z)$. 
The complement $M^{(m)}\setminus M$ 
will be called the \textit{Heegner locus} in this chapter. Indeed, if ${\rm dim}(X)=2$ (i.e., $X$ being K3 surfaces), 
it is nothing but the union of the Heegner divisors. Now we can state our main theorem in this \S \ref{geom.prob.HK}.

\begin{Thm}[Partial compactification]\label{Moduli.Shimura}
For sufficiently divisible $m$, the period map 
$p\colon M\hookrightarrow \Gamma\backslash \mathcal{D}_{M}$ 
extends to an isomorphism 
\begin{align*}
M^{(m)}\simeq \Gamma\backslash \mathcal{D}_{M}.
\end{align*}
\end{Thm}

\begin{Cor}\label{HK.bded}
The partial compactification $$M\subset M^{(m)}$$ 
\noindent 
for sufficiently divisible $m$ does not depend on $m$. 
\end{Cor}

The parametrized log-terminal varieties with numerical trivial classes are K-stable \cite{Od2} in 
the sense of Donaldson \cite{Don} (also cf.\ \cite[\S 4]{Od12}) and 
$X$ admits a singular K\"ahler-Einstein metric in the class $c_{1}(L)$ 
by \cite{EGZ.sing}, hence 
this partial compactification also serves as a further evidence for the K-moduli conjecture
 (cf.\ e.g.\ \cite[\S 3]{Od12}) as explained in 
\cite{Tos.WP, Zha.WP}. Also note that 
the local Torelli theorem for singular symplectic varieties 
\cite[Theorem 8]{Nam1} conditionally 
implies that $M^{(m)}$ has only quotient singularities 
but we recover the consequence from Theorem \ref{Moduli.Shimura} in 
our generality. We now discuss the proof. 

\medskip
\begin{proof}[Proof of Theorem~\ref{Moduli.Shimura}]

We first recall the following result which plays a crucial role in 
our proof. We thank Professor Yoshinori Namikawa for teaching it to us. 

\begin{Fac}[{\cite[Theorem 2.2]{Nam2}($+$\cite[Corollary 2]{Nam3})}]\label{simul.resol}
For a generically smooth holomorphic proper family $\mathcal{X}\twoheadrightarrow T\ni t$, 
after shrinking $T$ to an open neighborhood of $t$ if necessary and 
taking a finite base change $\tilde{T}\to T$, we have a family $\tilde{\mathcal{X}}$ over $\tilde{T}$ which is a simultaneous symplectic resolution of 
$\mathcal{X}\times_{T}\tilde{T}\to \tilde{T}$. 
\end{Fac}

In fact, the above references \cite{Nam2}, \cite{Nam3} 
prove more. Suppose that the fiber $\mathcal{X}_{t}$ admits 
a $\mathbb{Q}$-factorial terminalization, say $\mathcal{X}'_{t}$ 
(which is true by \cite{BCHM}). Then \cite{Nam2}, \cite{Nam3} say it is automatically a 
symplectic resolution, hence $\mathcal{X}'_{t}$ is smooth in particular. 
Furthermore, \textit{loc.cit.}\ shows a natural map from ${\rm Def}(\mathcal{X}'_{t})$ to 
${\rm Def}(\mathcal{X}_{t})$ exists which preserves the fibers, 
and is finite dominant. 
A somewhat related result of Huybrechts \cite[\S 8]{Huy.HK} 
is the surjectivity of period map from 
any connected moduli of marked compact hyperK\"ahler manifolds. 
The proof therein essentially uses the hyperK\"ahler rotation 
technique combined with results on the structure of the K\"ahler cone.

Also a fundamental fact we use is that 
\begin{Claim}\label{Def.sp.inclusion}
The natural forgetful morphism 
${\rm Def}(X,L)\to {\rm Def}(X)$, 
from the deformation space of polarized symplectic variety $(X,L)$ to 
that of bare $X$, is an immersion. 
\end{Claim}
\begin{proof}[Proof of Claim~\ref{Def.sp.inclusion}]
This holds since the Picard variety is discrete (i.e., irregularity vanishes) 
for any symplectic variety which admits a symplectic resolution. Indeed, 
suppose $f\colon W\to Y$ is a symplectic resolution. Then,  
$R^{1}f_{*}\mathcal{O}_{W}=0$ holds since $Y$ has only rational singularities. 
This vanishing together with another vanishing 
$H^{1}(\mathcal{O}_{W})=0$ implies $H^{1}(\mathcal{O}_{Y})\simeq H^{1}(f_{*}\mathcal{O}_{W})=0$ 
by the Leray spectral sequence 
for $\mathcal{O}_{W}$ with respect to $f$. 
\end{proof}

Since $M^{(m)}$ is the coarse moduli of smooth Deligne-Mumford quotient stack, 
or by \cite[\S 3.5]{Vie} for instance, we have a flat family of 
polarized symplectic varieties 
on a finite Galois cover $\widetilde{M}^{(m)}$ of $M^{(m)}$. 
We denote the covering map by $r_{M^{(m)}}\colon \widetilde{M}^{(m)}\to M^{(m)}.$ 
Take any point $\tilde{y}\in \widetilde{M}^{(m)}$. 
By applying Fact~\ref{simul.resol} (\cite{Nam2}) to 
an analytic neighborhood $\widetilde{U}_{\tilde{y}}$ of $\tilde{y}\in \widetilde{M}^{(m)}$, 
we obtain a finite proper covering $\tilde{r}_{y}\colon \widetilde{\widetilde{U}}_{\tilde{y}}\to \widetilde{U}_{\tilde{y}}$ 
and a family over $\widetilde{\widetilde{U}}_{\tilde{y}}$ which simultaneously resolves the pulled back 
family over $\widetilde{U}_{\tilde{y}}$. 
By the finiteness of the Galois group of $r_{M^{(m)}}$, it follows 
$r_{M^{(m)}}$ is an open morphism and a closed morphism as well 
(with respect to both topologies, 
either complex analytic or Zariskian). We set the image of 
$\widetilde{U}_{\tilde{y}}$ in $M^{(m)}$ as $U_{y}$ which is a complex analytic 
open subset. By taking the Galois closure 
of the finite proper covering $\tilde{\tilde{r}}_{y}\colon \widetilde{\widetilde{U}}_{\tilde{y}}\to U_{y}$ 
if necessary, we can assume $\tilde{\tilde{r}}_{y}$ is also Galois. 
We denote the Galois group of $\tilde{\tilde{r}}_{y}$ by $\Gamma_y$. 

Then we see that 
\begin{Claim}[Local period map]\label{descend}
Fix a marking on the hyperK\"ahler manifolds parametrized by $\widetilde{\widetilde{U}}_{\tilde{y}}$, 
which is deformation invariant and compatible with the original marking on $\mathcal{T}_M$. It is not necessarily unique but 
anyhow we get the associated period map 
$\tilde{\tilde{p}}_{\tilde{y}}\colon \widetilde{\widetilde{U}}_{\tilde{y}}\to \mathcal{D}_{M}$. 

We claim that this descends 
to $U_{y}\to \Gamma\backslash \mathcal{D}_{M}$, which we denote by $p_{y}$
 and call it a local period map. 
\end{Claim}
\begin{proof}
Consider the graph $G(\tilde{\tilde{p}}_{\tilde{y}})$ of $\tilde{\tilde{p}}_{\tilde{y}}$ in 
$\widetilde{\widetilde{U}}_{\tilde{y}}\times (\Gamma\backslash \mathcal{D}_M)$ and let 
$\Gamma_y$ act on the product with the Galois action on the first component 
and the 
trivial action on the second component. Note that there is an analytic closed  proper subset $Z'_{\tilde{y}}\subset \widetilde{\widetilde{U}}_{\tilde{y}}$, 
such that the map $\widetilde{\widetilde{U}}_{\tilde{y}}\to U_{y}$ 
is \'etale away from $Z'_{\tilde{y}}$ 
and the family (before the simultaneous resolution) over the complement $\widetilde{\widetilde{U}}_{\tilde{y}}\setminus Z'_{\tilde{y}}$ 
is smooth, i.e., the image of $\widetilde{\widetilde{U}}_{\tilde{y}}\setminus 
Z'_{\tilde{y}}$ is contained 
in $M$. Denote the image of $Z'_{\tilde{y}}$ in $M^{(m)}$ by $Z_{y}$ 
which is an analytic closed subset of $U_{y}$. 
From the above arguments, we have $(U_{y}\setminus Z_{y})\subset M$. 
Note that $U_{y}\setminus Z_{y}$ is still path-connected 
since $Z_{y}\subset U_{y}$ is an analytic proper closed subset, 
hence has real codimension at least $2$. 
Thus, from the definition of the arithmetic subgroup 
$\Gamma$ using the monodromy, 
$p_{\tilde{y}}^{o}:=\tilde{\tilde{p}}_{\tilde{y}}|_{(\widetilde{\widetilde{U}}_{\tilde{y}}\setminus Z'_{\tilde{y}})}$ 
descends to $p\colon (U_{y}\setminus Z_{y})\to \Gamma\backslash 
\mathcal{D}_{M}$ (cf.\ \cite{Ver} also see \cite{Ver.er, Mark.Torelli, GHS}). 
Therefore, we see that $\Gamma_y$-action preserves $G(\tilde{\tilde{p}}_{\tilde{y}})$ 
which is nothing but the closure of the graph of $p_{\tilde{y}}^{o}$ with respect to the 
complex analytic topology. 
Therefore, $\tilde{\tilde{p}}_{\tilde{y}}$ descends to a holomorphic 
morphism from $U_y$, i.e., 
$p\colon (U_{y}\setminus Z_{y})\to \Gamma\backslash 
\mathcal{D}_{M}$ extends to whole $U_{y}$. 
We finish the proof of Claim~\ref{descend}. 
\end{proof}

This descended map $p_y$ is compatible with 
the period map $p$ in the sense $p_{y}|_{U_{y}\cap M}=p|_{U_{y}\cap M}$. 
Also, because of that and the 
principle of analytic continuation, for any $\tilde{y}, \tilde{y}'$ in 
$\widetilde{M}^{(m)}$ and their images $y,y' \in M^{(m)}$, 
$p_{y}|_{U_{y}\cap U_{y'}}=p_{y'}|_{U_{y}\cap U_{y'}}$ always holds. 
Therefore, by considering an open covering $\{\tilde{U}_{\tilde{y}}\}
_{\tilde{y}}$ of 
$\widetilde{M}^{(m)}$ and 
gluing the partial extension of $p$ from $U_{y}$s by the 
principle of  analytic continuation, 
$p\colon M\to \Gamma\backslash \mathcal{D}_{M}$ extends to a holomorphic map 
\begin{align*}
p^{(m)}\colon M^{(m)}\to \Gamma\backslash \mathcal{D}_{M}.
\end{align*} This map 
$p^{(m)}$ is automatically  
algebraic by \cite[Theorem 3.10]{Bor}, a generalization of the classical 
big Picard theorem.

Take a sequence of positive integers $\{m_{i}\}_{i=1,2,\cdots}$ 
such that $m_{i}|m_{i+1}$ and for any positive integer $m$, 
there exists $i$ with $m|m_{i}$. 
Now we consider 
the sequence of partial compactifications 
$M\subset M_{i}:=M^{(m_{i})}$. We denote $p^{(m_{i})}$ by $p_{i}$. 
We will prove the following three claims: 

\begin{Claim}\label{exhaustion}
As partial compactifications, $M_{i+1}$ includes 
$M_{i}$. That is, there is an open immersion 
$M_{i}\hookrightarrow M_{i+1}$ which extends the identity map of $M$. 
\end{Claim}

\begin{Claim}\label{pm.inj}
$p^{(m)}$ is an open immersion for any $m\in \Z_{>0}$. 
\end{Claim}

\begin{Claim}\label{pm.surj}
For any point $x\in \Gamma\backslash \mathcal{D}_{M}$, 
it is in the image of $p^{(m)}$ (resp., $p_{i}$) for sufficiently divisible $m$ 
(resp., sufficiently large $i$).  
\end{Claim}

We first confirm that the above three claims imply 
Theorem~\ref{Moduli.Shimura}. 
From Claim~\ref{exhaustion}, Claim~\ref{pm.inj} 
and \cite[Theorem 3.10]{Bor}, 
we can think of the sequence $\{M_{i}\}_{i}$ as 
open subschemes of $\Gamma\backslash \mathcal{D}_{M}$ as 
$M_{1}\subset M_{2}\subset M_{3}\subset \cdots 
\subset (\Gamma\backslash \mathcal{D}_{M})$. 
Claim~\ref{pm.surj} implies that the union $M_i$ coincides with the whole 
$\Gamma\backslash \mathcal{D}_{M}$ and hence by the Noetherian property of $\Gamma\backslash \mathcal{D}_{M}$, 
we conclude that $M_i$ coincides with whole $\Gamma\backslash \mathcal{D}_{M}$ 
for large enough $i$. 

Now we prove the three above claims one by one. 

\begin{proof}[Proof of Claim~\ref{exhaustion}]
Note that at the level of functors, $\mathcal{M}^{(m_{i})}
\to \mathcal{M}^{(m_{i+1})}$ is naturally defined as follows. 
For the equivalence class of a flat family $[(\mathcal{X},\mathcal{L})\to T]
\in \mathcal{M}^{(m_{i})}(T)$, we consider 
$[(\mathcal{X},\mathcal{L}^{\otimes \bigl(\frac{m_{i+1}}{m_{i}}\bigr)})\to T]
\in \mathcal{M}^{(m_{i+1})}(T)$. 
This naturally induces a morphism 
$M_{i}\to M_{i+1}$ extending the identity of 
$M$, by the definition of coarse moduli scheme. 
Since the Picard 
varieties of the fibers $\mathcal{X}_{t}$ $(t\in T)$ are all free abelian groups, 
the morphism $M_{i}(\mathbb{C})\to M_{i+1}(\mathbb{C})$ is injective. 

We would like to show that $M_{i}\to M_{i+1}$ is analytically 
(or \'etale locally) biholomorphic, so that it is indeed a Zariski open immersion. 
We can prove it as follows. 
For any $[(X,L)]\in M^{(m_{i})}$, it is easy to see that 
taking power of the polarization does not 
change the deformation space, i.e., 
${\rm Def}(X,L)$ and ${\rm Def}(X,L^{\otimes \bigl(\frac{m_{i+1}}{m_{i}}\bigr)})$ 
are isomorphic again. 
Indeed, if we use \cite[Theorem 2.2]{Nam2} again, then 
these deformation spaces are both identified with 
the image of a smooth codimension $1$ subspace inside 
smooth ${\rm Def}(\tilde{X})$ tangent to 
$c_{1}(\tilde{L})^{\perp}$, where 
$\tilde{X}\to X$ is a symplectic resolution and the pullback of $L$ is 
denoted by $\tilde{L}$ (cf., e.g., \cite[\S3.3]{Sernesi}). 
Also note that obviously we have 
${\rm Aut}\Bigl(X,L^{\otimes \bigl(\frac{m_{i+1}}{m_{i}}\bigr)}\Bigr)=
{\rm Aut}(X,L)$ because of the torsion-freeness of the Picard group. 
Also, recall that these are finite groups by, 
e.g., \cite[Proposition 4.6]{Ambro}. 
On the other hands, the same Luna slice type arguments as 
\cite[\S 3]{Od15} (cf., also \cite[\S 3.2, \S 3.3]{OSS}) show that 
$\mathcal{M}^{(m_{i})}$ (resp., $\mathcal{M}^{(m_{i+1})}$) 
is analytically locally or 
\'etale locally isomorphic to $[{\rm Def}(X,L)/{\rm Aut}(X,L)]$ (resp., 
$[{\rm Def}(X,L^{\otimes \bigl(\frac{m_{i+1}}{m_{i}}\bigr)})/
 {\rm Aut}\Bigl(X,L^{\otimes \bigl(\frac{m_{i+1}}{m_{i}}\bigr)}\Bigr)]$). 
Hence, $\mathcal{M}^{(m_{i})}\to \mathcal{M}^{(m_{i+1})}$ 
is an open immersion as stack and also 
induces a local biholomorphism between a neighborhood of 
$[(X,L)]\in M^{(m_{i})}$ and one of 
$[(X,L^{\otimes \bigl(\frac{m_{i+1}}{m_{i}}\bigr)})]\in M^{(m_{i+1})}$. 
We finish the proof of Claim~\ref{exhaustion}. 
\end{proof}

\begin{proof}[Proof of Claim~\ref{pm.inj}]
Take any $[(X,L)]\in M^{(m)}$. By Fact~\ref{simul.resol}, in particular we obtain a symplectic resolution $\tilde{X}\to X$. 
From Claim~\ref{Def.sp.inclusion}, the local Torelli theorem 
for the \textit{smooth} $\tilde{X}$ and Fact~\ref{simul.resol} again, 
it follows that 
all the fibers of $p^{(m)}$ are discrete. Since it is also an algebraic morphism, 
$p^{(m)}$ is a quasi-finite morphism. 
On the other hand, since 
$\Gamma\backslash \mathcal{D}_{M}$ is an orbifold, hence normal in particular. 
Combining with the fact and 
$p^{(m)}$ is birational by the global Torelli theorem (cf., \cite{Ver,Mark.Torelli}), 
it follows that $p^{(m)}$ is a Zariski open immersion. 
\end{proof}

For Claim~\ref{pm.surj}, we give two proofs. Our first proof below 
uses the semistable Minimal Model Program, and the other uses 
the Weil-Petersson geometry. 

\begin{proof}[First proof of Claim~\ref{pm.surj}]
Take an arbitrary point $x\in \Gamma\backslash \mathcal{D}_{M}$. 
Recall that (e.g., from our discussion in \S\ref{HK.alg}) 
although $M$ is only \textit{coarse} moduli scheme, it was realized as a finite quotient of a variety $\widetilde{M}$ on which the family exists 
(cf.\ \S\ref{HK.alg} or \cite[\S 3.5]{Vie} which depends on \cite{Sesh}). 
By the density of $M\subset \Gamma\backslash \mathcal{D}_{M}$, we can take a 
smooth sub-curve $C\subset \Gamma\backslash \mathcal{D}_{M}$ whose generic point lies in $M$ and passes through $x$. 
After passing to a finite cover $r_{C}\colon\widetilde{C}\to C$ from another smooth curve $\widetilde{C}$, we can assume that 
it lifts to both $\mathcal{D}_{M}$ and also $\widetilde{M}$ because the isotropy group of $x\in \Gamma\backslash \mathcal{D}_{M}$ is 
finite. Summarizing up, we have a compatible pair of holomorphic 
morphisms $\widetilde{C}\to \mathcal{D}_{M}$ and $(\widetilde{C}\setminus r_{C}^{-1}(x))\to \widetilde{M}$ where the latter passes through $x\in M$. 
Take a point $\tilde{x}\in r_{C}^{-1}(x)$ and consider its analytic neighborhood $\tilde{x}\in \Delta\subset \widetilde{C}$, which is biholomorphic to 
a unit disk. From the morphism $\Delta \setminus \{\tilde{x}\}\to 
\widetilde{M}$, we obtain a corresponding polarized family of 
compact irreducible holomorphic symplectic manifolds 
$\pi\colon 
(\mathcal{X},\mathcal{L})\to \Delta\setminus \{\tilde{x}\}$, which can be compactified 
as a projective morphism to whole $\Delta$ because of the 
algebraicity of the morphism $(\widetilde{C}\setminus r_{C}^{-1}(x))\to \widetilde{M}$. 

Now we would like to apply \cite[Theorem 0.8]{KLSV} to the family $\pi$.\footnote{Notes added: After we have written up this part and finishing the whole draft (end of August 2018), the first author had a chance to 
attend Christian Lehn's talk. There the first author learned that 
Benjamin Bakker and he seem to have also somewhat similar arguments to prove different statements, 
which studies the locus along equisingular deformation in the moduli. Notes further added: 
This indeed appeared as \cite{BL} later (cf., also their \cite{BL} written before).}
Indeed, because of the existence of $\widetilde{C}\to \mathcal{D}_{M}$, it follows that the 
associated variation of Hodge structure $R^2\pi_*\mathbb{Z}$ has trivial monodromy. 
Therefore, \cite[Theorem 0.8]{KLSV} applies to $\pi$ and hence, 
possibly after a finite base change of $\Delta$ (which we believe to be 
unnecessary), 
we get another filling of $\pi$ as a smooth holomorphic proper map 
$\pi^{\rm sm}\colon \bar{\mathcal{X}}^{\rm sm}\to \Delta$ 
whose central fiber 
$(\pi^{\rm sm})^{-1}(\tilde{x})$ is a smooth irreducible holomorphic 
symplectic manifold (not necessarily algebraic) 
which is in particular, non-uniruled. 
We can and do take its blow up $\bar{\mathcal{X}}^{\rm sm}$ which is 
projective over $\Delta$ by \cite{Moi}, and further apply the 
semistable reduction \cite{KKMS} 
to obtain $\bar{\pi}\colon \overline{\mathcal{Y}}\to \Delta$ 
such that $\overline{\mathcal{Y}}$ is projective over $\Delta$ with 
a reduced simple normal crossing central fiber. 
Then we apply the log minimal model program to 
$(\overline{\mathcal{Y}},\bar{\pi}^{*}
\tilde{x})\to \Delta$ by \cite{Fjn.ss}. It is possible since the family is nothing but a restriction of 
a projective family to an analytic open subset of the base $\widetilde{C}$. 
We denote the obtained log minimal model (recently 
called dlt model) by $\mathcal{X}_{\rm min}\to 
\Delta$, which is obviously still isomorphic to $\mathcal{X}$ over 
$\Delta^*:=\Delta\setminus\{\tilde{x}\}$. 
Replacing $\Delta$ by an open neighborhood of $\tilde{x}$ 
and replacing $\mathcal{L}$ by its power $\mathcal{L}^{\otimes d}$ 
for some $d\in \Z_{>0}$ if necessary, 
we can and do take an effective algebraic divisor $\mathcal{D}^*$ corresponding to $\mathcal{L}$ on $\mathcal{X}$. Its closure in $\mathcal{X}_{\rm min}$ 
will be denoted by $\mathcal{D}_{\rm min}$. 
By our construction which uses $\mathcal{X}_{\rm sm}$, applying 
\cite[Theorem 1.1]{Tak.uniruled} for instance, it follows that $\mathcal{X}_{\rm min}$ is 
isomorphic to $\bar{\mathcal{X}}^{\rm sm}$ in codimension $1$.

We take small enough rational number $0<\epsilon\ll 1$ and take relative log canonical model of 
$(\mathcal{X}_{\rm min},\epsilon \mathcal{D}_{\rm min})$ which is possible by \cite{BCHM} and denote it by 
$(\bar{\mathcal{X}},\epsilon \mathcal{D})\to \Delta$. From the construction, $K_{\bar{\mathcal{X}}}$ is relatively trivial over $\Delta$ 
hence $\mathcal{D}$ is ample and $\mathbb{Q}$-Cartier divisor. Suppose its Cartier index is $l>0$. Then, we set $\bar{\mathcal{L}}:=\mathcal{O}
_{\bar{\mathcal{X}}}(dl\bar{\mathcal{D}})$ for $d\in \Z_{>0}$. 
If we take $d\gg 0$, we obtain the desired filling over $\tilde{x}\in\Delta$ 
since $\bar{\mathcal{L}}$ will be relatively very ample over $\Delta$. 
We can do the same for whole preimage in $\widetilde{C}$ of $x\in C$, 
so that we get a filling over whole $\widetilde{C}$ which gives an element of 
$\mathcal{M}^{(dl)}(\widetilde{C})$ which extends the original family locally 
given by $\pi$ around $\tilde{x}\in \widetilde{C}$. 
Obviously this implies that $\tilde{x}$ is in the image of $p^{(dl)}$ 
for such $d\gg 0$. 
Therefore, we complete the proof of Claim~\ref{pm.surj}. 
\end{proof}

We give another somewhat different 
proof of Claim~\ref{pm.surj} involving some 
discussion on the Weil-Petersson metric, 
which is a natural discussion after \cite{Wang, Tos.WP,Tak,Zha.WP}. 

\begin{proof}[Second proof of Claim~\ref{pm.surj}]
We use the same notation $\tilde{x}\in \Delta\subset \widetilde{C}$ 
and the corresponding punctured smooth family 
$(\mathcal{X},\mathcal{L})\to \Delta\setminus \{\tilde{x}\}$ 
as the first proof. 
As the Weil-Petersson (orbi-)metric on $M$ coincides with the restriction of the 
Bergman (orbi-)metric on $\Gamma\backslash \mathcal{D}_{M}$ by 
\cite[Theorem 4.8]{Schu}, the induced Weil-Petersson metric on 
$\Delta\setminus \{\tilde{x}\}$ has finite distance. 
Therefore, applying \cite[Theorem 1.2]{Tos.WP}, \cite[Theorem 1.3(ii)]{Zha.WP}, 
there is a filling $(\bar{\mathcal{X}},\bar{\mathcal{L}})\to \Delta$ as 
a holomorphic proper map 
with relatively ample $\bar{\mathcal{L}}$ such that the 
central fiber $\mathcal{X}_{\tilde{x}}$ has only 
canonical singularities with trivial canonical class. 
The rest of our second proof (of this Claim~\ref{pm.surj}) 
is same as the first proof. 
\end{proof}
Hence, we complete the proof of Theorem~\ref{Moduli.Shimura}. 
\end{proof}

\begin{Rem}[Parametrizing Quasi-polarized resolutions]
The above partial compactification $M\subset M^{(m)}(\simeq \Gamma\backslash \mathcal{D}_{M})$ 
for sufficiently divisible $m$ can be also regarded as a coarse moduli of quasi-polarized smooth irreducible symplectic manifolds as in the case of K3 surfaces. 
Here, \textit{quasi-polarization} means a nef and big line bundle, which is automatically semiample by the basepoint freeness theorem 
(cf., e.g., \cite[\S3]{KMM}). 

Indeed, such a  reinterpretation 
follows from the arguments in the proof of Theorem~\ref{Moduli.Shimura}, especially the Fact~\ref{simul.resol}. 
\end{Rem}

From now on, we fix $m$ which satisfies Theorem~\ref{Moduli.Shimura}. 
Now we establish the continuity of Ricci-flat metrics on the obtained partial compactification. 
Recall that for each $(X,L)\in M^{(m)}$, there is a unique weak Ricci-flat K\"ahler current $\omega_{X}$ with $[\omega_{X}]\in c_1(L)$ 
by \cite[Corollary E]{EGZ.sing}, which we use now. 

\begin{Thm}[Gromov-Hausdorff continuity for non-collapsing]\label{HK.GH.continuity}
Each Ricci-flat K\"ahler current $\omega_{X}$ with $[\omega_{X}]=c_{1}(L)$ for $(X,L)\in M^{(m)}$ (\cite[Corollary E]{EGZ.sing}) 
gives a finite diameter distance on $X$. 
Therefore we obtain a map 
\[\Phi\colon M^{(m)}\to {\it CMet}_{1} ,\] 
which is continuous. 
\end{Thm}

\noindent 
Note that the above theorem extends Proposition~\ref{easy.confirmation} and Claim~\ref{GH.conti.ADE} for the case of (polarized) K3 surfaces to 
higher dimensional (polarized) hyperK\"ahler manifolds. 
Indeed, the main points of the proof remain the same. 

\begin{proof}[proof of Theorem~\ref{HK.GH.continuity}]
First we confirm that the diameters of the Ricci-flat metrics have locally uniform bounded diameter across the Heegner locus in the moduli $M^{(m)}$. 
From now, we denote it by $H(M^{(m)})$. 
This  follows from \cite[Proposition 3.1]{Tos.lim} (cf., also \cite[Lemma 1.3]{DPS}) as in Lemma~\ref{diam.Kahler}, 
once we set the reference metrics ``$\omega_0$" in \cite[\textit{loc.cit.}]{Tos.lim} as the Fubini-Study metrics induced by $|L^{\otimes m}|$ for 
bounded (continuously metrized) family of $(X,L)\in M^{(m)}$. 

Then, take a point $x_\infty$ in the Heegner locus of $M^{(m)}$ 
and a sequence $\{x_i\}$ of $M^{(m)}\setminus H(M^{(m)})$ 
which converges to $x_\infty$. 
Suppose $x_i=[(X_i,L_i)]$ and $x_\infty=[(X_\infty,L_\infty)]$, and write the corresponding weak Ricci-flat K\"ahler current 
on $X_i$ as $\omega_i$ and $X_\infty$ as $\omega_{\infty}$. 

We prove that $(X_i,\omega_i)$ converges to $(X_\infty,\omega_\infty)$ 
in the Gromov-Hausdorff sense, by contradiction. 
Suppose there is a subsequence of $(X_i,\omega_i)$ which does not converge to $(X_\infty,\omega_\infty)$, 
then after replacing the 
original sequence $(X_i,L_i)$ by that subsequence if necessary, 
the Gromov-Hausdorff distance 
\begin{equation}\label{GH.far}
d_{\rm GH}((X_i,\omega_i),(X_\infty,\omega_\infty))>\epsilon
\end{equation} holds for a 
uniform positive real number $\epsilon$. 
From the uniform upper bound of the diameters of $\omega_i$, we can apply the  Gromov's precompactness theorem 
to $(X_i,\omega_i)$ to see that there is a subsequence which converges in the Gromov-Hausdorff sense. We write that 
compact metric space as $W$. Then now we can apply Donaldson-Sun theorem \cite{DS} to see that $W$ must be 
$(X_\infty, \omega_\infty)$. More precisely, it follows from 
\cite[Theorem 1.1, the second statement]{DS} combined with (especially Proposition 4.5 and Remark of) \cite[\S4.4]{DS}. 
This contradicts to the assumption \eqref{GH.far}. 
Thus we complete the proof of Theorem \ref{HK.GH.continuity}. 
\end{proof}

\section{Tropical hyperK\"ahler manifolds and their moduli}\label{Alg.prob}

The aim of this section is to extend (\S\ref{K3.Sat.sec} and) \S\ref{trop.K3.1} for higher dimensional hyperK\"ahler varieties. 
By the following several steps, extending those in the case of K3 surfaces (\S\ref{trop.K3.1}), 
we assign a metrized complex projective space $\Psi_{\rm alg}(x)$ to each point $x=[e,v]$ in $(\Gamma\backslash \mathcal{D}_M)(l)$ 
for $l=\Q e$.

Below, we explain our steps of the construction of $\Psi_{\rm alg}$, extending that of \S\ref{trop.K3.1}. 
Our main achievement here is the complete definition in the case of K3$^{[\frac{n}{2}]}$ type or generalized Kummer varieties type in 
Definition~\ref{part.geom.realization.HK}. 
We assume some familiarity to our discussion for K3 surfaces case in \S\ref{F2d.sec}, rather than repeating 
all the details since many arguments are simply extending those in \S\ref{F2d.sec}. 
We denote the Satake's partial compactification (\cite{Sat1}) of $\mathcal{D}_{M}$ (``the rational closure") with respect to 
the adjoint representation simply by $\mathcal{D}^{*}_M$ in the following. Recall from \cite{Sat2} that there is a stratification 
$$\mathcal{D}^{*}_M=\mathcal{D}_M\sqcup \bigsqcup_l \mathcal{D}_{M}(l)\sqcup \bigsqcup_p \mathcal{D}_M(p),$$
where $l$ (resp., $p$) runs over all  
 isotropic lines (resp., planes) in $\lambda^{\perp}\otimes \Q$ and 
$$\Gamma\backslash \mathcal{D}^{*}_M=\overline{\Gamma\backslash \mathcal{D}_M}^{\rm Sat,\tau_{\rm ad}}.$$ 

\medskip
\begin{Stp}
[Assigning holomorphic symplectic manifolds to boundary]\label{HKR.HK}
\noindent
Let $e\in \Lambda$ be an isotropic vector 
 with $e\perp \lambda$ and set $l=\Q e$. 
Then the corresponding boundary stratum, which we denote by $(\Gamma\backslash \mathcal{D}_M)(l)$, is 
abstractly a quotient of the real $(r-4)$-dimensional ball $\mathcal{D}_M(l)$ by a discrete group. More precisely, 
it is the quotient of 
$$
\mathcal{D}_{M}(l)=\{v\in \langle e,\lambda \rangle_{\R} ^{\perp}/e_{\R} \mid v^2>0\}/\R^{\times}\simeq O(1,r-4)/(O(1)\times O(r-4)) 
$$
by an arithmetic subgroup $\Gamma\cap {\rm stab}(l)$ of $O(\lambda^{\perp}\otimes \R)$ 
as a generalization of \eqref{F2d(l)}, \eqref{F2d(l)2}. 
The surjectivity of period map 
$$
\mathcal{T}_{M}\twoheadrightarrow \Omega(\Lambda)
$$
was proved by 
Huybrechts \cite[\S8]{Huy.HK} and allows us to 
generalize verbatim the method of taking complex K3 surface $X$ in \S\ref{trop.K3.1} to our higher dimensional setting. 
That is, to each point $\tilde{x}=[e,v]\in \mathcal{D}_M(l)$ (\cite{Sat1}), we 
assign a marked irreducible holomorphic symplectic manifold $(X,\alpha_X)$ in our specified  
deformation class $\mathcal{T}_{M}$. 
Note that this Step~\ref{HKR.HK} works for moduli of \textit{any} class of smooth (polarized) 
irreducible holomorphic symplectic manifolds. 

In the Hodge decomposition of $H^2(X,\C)$, 
$H^{2}(X,\R)\cap (H^{0,2}\oplus \overline{H^{0,2}})$ is orthogonal to $H^{1,1}$ 
with respect to the Beauville-Bogomolov-Fujiki form (cf., e.g., \cite[\S5]{Huy.HK}). This fact and the 
Lefschetz (1,1) theorem imply that $e\in NS(X)$. Hence we have a holomorphic line bundle $N$ 
with $c_1(N)=e$. 

We wish to give a holomorphic Lagrangian fibration structure on these $X$ 
defined by $N$, 
 which we discuss below. 
The idea is that our desired $\Psi_{\rm alg}(x)$ for Conjecture~\ref{HK.conj} would be the base variety of the obtained Lagrangian fibration on $X$. 

\end{Stp}
\medskip
\begin{Stp}[Nefness of $e$ at the boundary component]\label{HK.nef1}

This Step~\ref{HK.nef1} is to extend the former half of \S\ref{trop.K3.1} to our higher dimensional situation, 
under the assumption that 
\begin{Ass}\label{b2}
$r=b_{2}(X)\ge 5$ holds for (any) irreducible symplectic manifold $X$ parametrized in $\mathcal{T}_{M}$. 
Here, $b_2$ denotes the second Betti number. 
\end{Ass}
\noindent
Note that this is only a priori nontrivial in the sense that all known (so far) holomorphic symplectic manifolds $X$ has either $b_2(X)=23, 7, 24$ or $8$. 

We take a point $\tilde{x}\in \mathcal{D}_M(l)$ and denote the image of $\tilde{x}$ in $\overline{\Gamma\backslash \mathcal{D}_M}^{\rm Sat,\tau_{\rm ad}}$ 
by 
$x\in (\Gamma\backslash \mathcal{D}_M)(l)$. We make the following claim, which is our purpose of this Step~\ref{HK.nef1}. 
\begin{Claim}\label{e.nef1.HK}
For the equivalence class of the 
marked irreducible symplectic manifold $(X,\alpha_X)$ associated to $\tilde{x}\in \mathcal{D}_M(l)$ in the previous Step~\ref{HKR.HK}, 
if we change the marking $\alpha_X$, $\alpha_X^{-1}(e)$ becomes nef. 
\end{Claim}
\begin{proof}[proof of Claim~\ref{e.nef1.HK}]
This follows from 
extension of the arguments in \cite[Chapter III, Remark~2.13]{Huy} for K3 case, given the recent development by \cite{AV1}. 
We now give its details. Suppose $\alpha_X^{-1}(e)$ is not nef, i.e., 
not in the closure of the K\"ahler cone of $X$. 
Then from the description of the K\"ahler cone in \cite[Theorem~1.19]{AV1}, there is a rational curve $R$ which they call ``MBM'' type, 
such that $$(\alpha_X^{-1}(e),R)<0.$$ Now we take a class $D_{R}\in H^{2}(X,\Z)$ such that 
$(D_{R},a)_{\rm BBF}=(R,a)$ for any $a\in H^{2}(X,\Z)$ 
where the left hand side means the Beauville-Bogomolov-Fujiki form and the right hand side means the 
usual cup product. From \cite[Corollary~1.4]{AV4} and our Assumption~\ref{b2} that $b_{2}(X)\ge 5$, 
we can assume $$-k=(R,D_{R})=(D_{R})_{\rm BBF}^{2}<0$$ and the value $-k=(D_{R})_{\rm BBF}^{2}$ is bounded below from 
a (negative) constant which only depend on the deformation equivalence class of $X$. 
This is first proved in the case of K3$^{[\frac{n}{2}]}$-type by Mongardi 
\cite[Corollary~2.7, also cf.\  (2.12), (2.14), (2.15)]{Mon} and Bayer-Hassett-Tscinkel \cite[Proposition 2]{BHT}. 
Then the uniform lower boundedness is later generalized as \cite[Corollary~1.4]{AV4}. 
Note that there is a slight difference of the terminology --- Mongardi \cite{Mon} defines the \textit{wall divisors} 
which are, roughly speaking, corresponding to the responsible facet of K\"ahler cone, while 
Amerik-Verbitsky introduces slight variant called \textit{Monodromy Birational Minimal (MBM)} curves 
(cf., \cite[Proposition~1.5]{Mon} for the difference). 

Now we consider the reflection on $H^{2}(X,\R)$ with respect to $D_{R}$. That is, we define 
\begin{align} \label{sR.HK}
s_{R}\colon H^{2}(X,\R)&\to \;\;H^{2}(X,\R)\\  \nonumber
 x\;\;\;\;\:&\mapsto x+\frac{2}{k}(x,R)D_{R}.
\end{align}
Note that above $s_{R}$ 
preserves the lattice $H^{2}(X,\Z)$ with the Beauville-Bogomolov-Fujiki form by the definition. 
On the other hand, for a fixed ample class $H\in H^{2}(X,\Z)$, we have 
\begin{align*}
0\le (s_{R}(\alpha_X^{-1}(e)),H)_{\rm BBF}&=(\alpha_X^{-1}(e),H)_{\rm BBF}+\frac{2}{k}(\alpha_X^{-1}(e),R)(R,H)\\ 
&<(\alpha_X^{-1}(e),H)_{\rm BBF}.
\end{align*}
Note that $k$ takes only finite possibilities, hence bounded,  
by the result explained above (\cite{Mon, BHT, AV4}). 
Hence, we cannot continue this for infinitely many times. 
Therefore, as in K3 case \cite[Chapter~VIII, Remark~2.13]{Huy} after changing the marking $\alpha_{X}$ by 
finite reflections of the form $s_{R}$ as \eqref{sR.HK}, $\alpha_X^{-1}(e)$ eventually becomes nef. 
\end{proof}

\end{Stp}
\medskip

\begin{Stp}[Nefness of $e$ on hyperK\"ahler rotations at a Siegel set]\label{HK.nef2}
Next we generalize Claim~\ref{e.nef} in the case of K3 surfaces to our setting of 
higher dimensional irreducible holomorphic symplectic manifolds. 
We arrange the setting as verbatim extension of \S\ref{trop.K3.1}. 
If $y$ lies in a certain Siegel subset in $\mathcal{D}_M\subset \mathcal{D}^{*}_M$, 
take a marked hyperK\"ahler manifold $X(y)$ and apply the same 
hyperK\"ahler rotation as in K3 case (Construction~\ref{HK.rotation.}) and denote it by $(X^{\vee}(y),\alpha_{X^{\vee}(y)})$. 
We wish to prove nefness of the line bundle $\mathcal{O}_{X^{\vee}(y)}(\alpha_{X^{\vee}(y)}^{-1}(e))$ on the 
holomorphic symplectic manifold 
$X^{\vee}(y)$, if we take appropriate Siegel set where $y$ lies, for each $\tilde{x}$. 

Under our Assumption~\ref{b2}, our proof of Claim~\ref{e.nef} 
works after simply replacing the $(-2)$-curves which played a role in the proof by MBM curve classes \cite{AV1, AV4}. 
Here, we use the same fact as in the previous Step~\ref{HK.nef1}, i.e., 
those whose self-intersection numbers 
(with respect to the Beauville-Bogomolov-Fujiki form) of MBM curves are bounded (\cite[Corollary~1.4]{AV4}). 
Combining with their previous work \cite[Theorem~1.19]{AV1}, we obtain the desired assertion. 

\end{Stp}
\medskip

\begin{Stp}[Basepoint freeness of $e$] 
To show the basepoint-freeness of $\mathcal{O}_{X}(\alpha_{X}^{-1}(e))$ 
and $\mathcal{O}_{X^{\vee}(y)}(\alpha_{X^{\vee}(y)}^{-1}(e))$ appeared in the previous Steps \ref{HK.nef1}, \ref{HK.nef2} 
from their nef-ness, seems to be in general 
unsettled difficult problem.  This is (a somewhat stronger version of) special case of the generalized abundance conjecture. 

Nevertheless, the ${\rm K3}^{[\frac{n}{2}]}$-type case is conditionally (for very general ones)
 confirmed by Markman
\cite[Theorem~1.3]{Mark13}. In the case of generalized Kummer varieties type, i.e., those 
deformation equivalent to generalized Kummer varieties, 
partial affirmative result is obtained in \cite[Proposition~3.38]{Yos13}.\footnote{The ``type'' refers to deformation equivalence class, 
at least in our paper. }
Then the assumptions in their results became for those K3$^{[\frac{n}{2}]}$-type and generalized Kummer varieties type, 
by \cite[Corollary~1.1]{Mat.semiample}. The proof in \textit{op.cit.}\ depends on these partial confirmations \cite{Mark13, Yos13} 
and extended version (\cite[Theorem 5.5]{Nak.kw}, \cite[\S4]{Fjn}) of Kawamata's theorem \cite{Kawa}. 
Anyway, we get a Lagrangian fibration structure from Matsushita's theorem \cite{Mat.fib} (projective case), \cite{Huy.HK.book} (general 
K\"ahler case). 
\end{Stp}
\medskip

Below, for two more steps in this section, we temporarily discuss under the assumption that 
$X$ admits the desired holomorphic Lagrangian fibration defined by $N$ 
 and denote it by $X\twoheadrightarrow B$. 
 By the above discussion, assuming that $M$ is of K3${}^{[\frac{n}{2}]}$ type or generalized Kummer varieties type is enough for that. 

\medskip
\begin{Stp}[The base of the Lagrangian fibration]\label{Lag.fib.HK}
The base spaces $B$ of the obtained Lagrangian fibrations $X\to B$ 
are proved to be isomorphic to $\mathbb{P}^{\frac{n}{2}}$ at least in the case deformation equivalent to 
${\rm K3}^{[\frac{n}{2}]}$-type case and the case deformation equivalent to the generalized Kummer varieties, 
by  Matsushita \cite[Theorem 1.4]{Mat.base}. 
For general holomorphic Lagrangian fibrations, the same conclusion 
is expected as Matsushita's conjecture \cite{Mat.fib} and is indeed 
known to be true if the base is smooth \cite{Hwa, GL}. Also there is a very recent progress on this conjecture, 
after the appearance of our first manuscript in \cite{BogKur}. 

Summarizing up, we can give a generalization of the geometric realization map $\Psi_{\rm alg}$ in the polarized projective K3 surface case in \S\ref{trop.K3.1} and 
\S\ref{Alg.K3.statements.sec} as follows. 

\begin{Def}[Partial geometric realization map]\label{part.geom.realization.HK}
In the case $X$ is deformation equivalent to K3$^{[\frac{n}{2}]}$-type or generalized Kummer varieties\footnote{We need this condition only for the explicit definition of $\Psi_{\rm alg}$ as discussed in above several steps. We expect this 
condition to be removable in future. },
to each point $\tilde{x}=[e,v]\in (\Gamma\backslash \mathcal{D}_M)(l)$, we consider the base manifold $B(\simeq \mathbb{P}^{\frac{n}{2}})$ 
of the Lagrangian fibration $X\twoheadrightarrow B$ obtained in Steps~\ref{Lag.fib.HK}. We denote the McLean metric by $g_{\rm ML}$ and 
associate it rescale $\frac{g_{\rm ML}}{{\rm diam}(g_{\rm ML})^2}$. Then we set 
\[\tilde{\Psi}_{\rm alg}(\tilde{x}):=
 \biggl(B,\frac{g_{\rm ML}}{{\rm diam}(g_{\rm ML})^2}\biggr).\] 
Then this map 
\[\tilde{\Psi}_{\rm alg}\colon \mathcal{D}_M 
\sqcup \bigsqcup_{l}  \mathcal{D}_M(l)\to  {\it CMet}_{1}\]
where $l$ runs over all isotropic lines in $\lambda^{\perp}\otimes \Q$, obviously descend to 
$$\Psi_{\rm alg}\colon (\Gamma\backslash \mathcal{D}_M)
\sqcup \bigsqcup_{l}  (\Gamma\backslash \mathcal{D}_M(l))\to {\it CMet}_{1},$$
where $l$ runs over all $\Gamma$-equivalence classes of isotropic lines in $\lambda^{\perp}\otimes \Q$. 
This is a part of our desired geometric realization map $\Psi_{\rm alg}$. 
\end{Def}

We remark that, Theorem~\ref{MS.Sat} combined with the above construction of $\Psi_{\rm alg}$ for $\bigsqcup_{l}(\Gamma\backslash \mathcal{D}_M(l))$, 
provides an analog of \cite{ACP} for the moduli of irreducible symplectic varieties in this case and their tropical version. 

By above Definition~\ref{part.geom.realization.HK}, we can put more precision on Conjecture~\ref{HK.conj} as a generalization of Conjecture~\ref{K3.Main.Conjecture} in the algebraic K3 surface case. 

\medskip

\begin{Conj}[Partial refinement of Conjecture~\ref{HK.conj}]\label{Good.HK.conj}
In the situation of Definition~\ref{part.geom.realization.HK}, the defined map 
\[\Psi_{\rm alg}\colon (\Gamma\backslash \mathcal{D}_M)
\sqcup \bigsqcup_{l} (\Gamma\backslash \mathcal{D}_M)(l)\to {\it CMet}_{1}\]
 is continuous.
\end{Conj}

\end{Stp}

\begin{Rem}[Analytic problems to prove Conjectures~\ref{HK.conj} and \ref{Good.HK.conj}]\label{Anal.prob} 

Let us consider what we have to do for proving the conjecture above. 
Recall that in the analytic arguments for K3 surfaces in \S\ref{GTZ.extend.proof} of our paper, i.e., 
the proof of Theorem~\ref{GTZ.extend} we frequently use 
${\rm dim}(X_{s})=2$ or ${\rm dim}(B_{s})=1$. Therefore, to 
generalize our arguments in a direct manner, one needs at least extensions of 
\begin{itemize}
\item \S\ref{H.upper}, 
\item \S\ref{unif.L.infinity} where our Moser iteration 
arguments use ${\rm dim}(X_{s})=2$, 
\item Lemma~\ref{2.9.new}, 
\S\ref{(3.9)new.ref.metric} for new reference metrics, 
\item \S\ref{Tos.Thm4.1}, 
\item \S\ref{Sing.fiber.nbhd} where the discussion uses the fact that 
${\rm disc}(\pi_{s}) (\subset B_{s})$ 
is $0$-dimensional. 
\end{itemize}
All these become technical obstacles to solve Conjecture~\ref{HK.conj}.\footnote{Note that, as far as the locus of smooth manifolds 
(the complement of the closure of the Heegner locus) is concerned, 
we have much less problems to solve. }

Once we would be able to extend the above listed parts (the previous paragraph) 
and Conjecture~\ref{Good.HK.conj}, for the case of the deformation class $M$, 
\textit{such resolution combined with} our Theorem~\ref{Moduli.Shimura} and 
the discussion in \S\ref{F2d.sec} (especially \S\ref{K3.along.disk}) give a proof of the 
conjecture of Kontsevich-Soibelman \cite[Conjecture 1]{KS} for hyperK\"ahler manifolds of the class $M$. 
As the arguments are completely the same as in \S\ref{K3.along.disk}, we omit the details. 

\end{Rem}

\medskip
\begin{Stp}[Type II collapsing]
For each isotropic plane $p\subset \lambda^{\perp}\otimes \Q$ with respect to the Beauville-Bogomolov-Fujiki form, 
and the corresponding $0$-dimensional boundary components $(\Gamma\backslash \mathcal{D}_M)(p)$, 
we wish to associate a metrized polytope $\Psi_{\rm alg}((\Gamma\backslash \mathcal{D}_M)(p))$ as a generalization of the unit interval in K3 case 
to refine Conjecture~\ref{HK.conj} in a more explicit manner. 
This is \textit{not} yet done in general. In fact, we believe that even in K3 case we lack enough conceptual understanding. 
See \cite{HSVZ} for the recent related development in K3 case, 
and we strongly wish to discuss and develop our understanding further on this problem in near future. 

\end{Stp}

\medskip

Also, it has been expected by experts 
that the dual intersection complex of maximally degenerating 
irreducible symplectic manifolds is homeomorphic to the complex projective space of half real dimension. 
Indeed, there are some recent progresses related to it by 
\cite{Nag1, NX, Nag2, KLSV, BM17}.

\section{Conjectural picture II --- general hyperK\"ahler case}\label{HK.Kahler}

We also expect that similar phenomenon as Conjecture \ref{K3.Main.Conjecture2} also holds 
for general (\textit{not} necessarily projective) compact hyperK\"ahler manifolds. We first review and make our setup 
on the moduli of such K\"ahler metrized hyperK\"ahler manifolds. Recall that 
there was some discussion toward this direction in \cite[\S3, \S4]{HuyMM}. 

We keep fixing the same notations as \S\ref{HK.alg}; 
a lattice $\Lambda$ of signature $(3,r-3)$ for $r>3$ and any deformation equivalence class $\mathcal{T}_M$ 
of $\Lambda$-marked irreducible smooth holomorphic 
symplectic manifolds $X_{\rm ref}$. We only allow 
markings $\alpha_X$ of $H^2(X,\Z)$ for deformation $X$ of $X_{\rm ref}$ 
so that $(X,\alpha_X)\in \mathcal{T}_M$ (cf., e.g., \cite[\S3.2]{GHS}, \cite[\S3.1]{Mark.Torelli}). 
We say such a marking is $M$-\textit{admissible}. 

We set $$\mathcal{P}_{X_{\rm ref}}:=\tilde{\Gamma}\backslash {\rm Gr}_{3}^{+,{\rm or},o}(\Lambda\otimes \R),$$ 
where ${\rm Gr}_{3}^{+,{\rm or},o}(\Lambda\otimes \R)$ denotes a connected component of the 
Grassmannian of real $3$-dimensional oriented positive definite subspaces 
of $\Lambda\otimes_{\Z}\R$ and $\tilde{\Gamma}\subset O(\Lambda)$ is the subgroup generated by 
monodromies along deformations of $X_{\rm ref}$ (cf., e.g., \cite[\S3.2]{GHS}).\footnote{which is denoted by ``${\rm Mon}^{2}(X_{\rm ref})$'' in {\textit{op.cit}}. }
If $X_{\rm ref}$ is a K3 surface, 
it is generated by $(-2)$-reflections (cf., \cite[Chapter 14, Proposition~5.5]{Huy} and generalization \cite[Theorem~1.2]{Mark06}) and 
$\mathcal{P}_{X_{\rm ref}}$ is identified with $\mathcal{M}_{\rm K3}$ through the refined period map. 
Define the set $\mathcal{K}^{\rm sm}(X_{\rm ref})$
 as 
\begin{align*}
&\{(X,\omega)\mid X {\text{ is a deformation of }}X_{\rm ref},\ 
 \omega \text{ is a K\"ahler class on }X \text{ of diameter }1\}
\end{align*}
and define $\mathcal{K}(X_{\rm ref})$ as
\begin{align*}
\{(X,\omega)\mid  &\ X \text{ is a symplectic K\"ahler analytic space (cf., Definition~\ref{IHS.def})} \\
&\text{which is a degeneration of smooth deformations of } X_{\rm ref}, \\ 
& \omega \text{ is a Ricci-flat K\"ahler metric on }X
 \text{ (cf., \cite{EGZ.sing}) of diameter }1\}. 
\end{align*}

We set $$\mathcal{M}^{\rm sm}(X_{\rm ref}):=\mathcal{K}^{\rm sm}(X_{\rm ref})/\sim,$$
where the equivalence relation $\sim$ denotes the $SO(3)$-action generated by the hyperK\"ahler rotations. 
Then, by considering $M$-admissible markings $\alpha$ on the smooth deformations $X$ of $X_{\rm ref}$, we consider the natural refined period map 
\begin{align*}
P^{\rm sm}\colon \mathcal{M}^{\rm sm}(X_{\rm ref})&\to \;\;\;\;\;\;\;\;\;\;\;\;\mathcal{P}_{X_{\rm ref}}\\ 
(X,\omega) &\mapsto \langle\alpha_{\R}([{\rm Re}(\sigma_X)]),\alpha_{\R}([{\rm Im}(\sigma_X)]),\alpha_{\R}([\omega])\rangle,
\end{align*}
where $\sigma_X$ is a holomorphic $2$-form normalized as $[{\rm Re}(\sigma_X)]^2=[{\rm Im}(\sigma_X)]^2=([\omega]^2=)1$. 
From the definitions of $M$-admissible markings and $\tilde{\Gamma}$, this is well-defined and generalizes the case of K3 surfaces studied in 
\cite{Tod, Looi, KT, Kro89a, Kro89b}. 
Further, the above map is injective by the Hodge theoretic Torelli theorem \cite{Ver, Ver.er}, \cite[\S3.2]{Mark.Torelli} 
(cf., also \cite[Theorem~3.5]{GHS}). 

\medskip

Our conjecture for general (not necessarily projective) K\"ahler  irreducible holomorphic symplectic manifolds is as follows. 

\begin{Conj}
In the above situation, the following holds. 
\begin{enumerate}
\item \label{HK.rot.singular}
We can extend the $SO(3)$-action (generated by the hyperK\"ahler rotations) on $\mathcal{K}^{\rm sm}(X_{\rm ref})$ to that 
on $\mathcal{K}(X_{\rm ref})$ which gives the 
equivalence relation $\sim$ on it and the quotient space 
$\mathcal{M}(X_{\rm ref}):=\mathcal{K}(X_{\rm ref})/\sim$. 
There is a natural topology on it, with respect to which the hyperK\"ahler metrics are Gromov-Hausdorff continuous. 
\item \label{HK.mod}
There is a natural homeomorphism $$P\colon 
\mathcal{M}(X_{\rm ref})\to \mathcal{P}_{X_{\rm ref}}$$ 
which extends $P^{\rm sm}$. 
Through this homeomorphism $P$, 
we identify 
$\mathcal{M}(X_{\rm ref})$ and $\mathcal{P}_{X_{\rm ref}}$. 
\item \label{HK.general.conj}
There exists a continuous map 
\[\Psi\colon \overline{\mathcal{P}_{X_{\rm ref}}}^{\rm Sat,\tau_{\rm ad}}
 \to {\it CMet}_{1},\] 
where, the left hand side is the Satake compactification with respect to the adjoint representation of $O(\Lambda\otimes \R)$.
The map $\Psi$ on $\mathcal{P}_{X_{\rm ref}}$ is 
 defined by associating the hyperK\"ahler metric corresponding via
 the homeomorphism $P$ in \eqref{HK.mod}. 
We call $\Psi$ the geometric realization map again. 
\end{enumerate}
\end{Conj}

The former half of the above \eqref{HK.rot.singular} holds if hyperK\"ahler rotations can be extended across the symplectic singularities. 
\eqref{HK.mod} is a higher dimensional generalization of the work of Kobayashi-Todorov \cite{KT}. 

Our Theorems \ref{Moduli.Shimura}, \ref{HK.GH.continuity} on the polarized (algebraic) case as well as 
\cite[Proposition~5.1, Theorem~5.9]{Huy.HK} give some partial understanding toward above \eqref{HK.mod}. 
We also note that the arguments of Remark \ref{K3.erg} after \cite{Ver.erg} also work verbatim in this higher dimensional setting. 

The last statement \eqref{HK.general.conj} is the main part, a weak (non-explicit) generalization of Conjecture~\ref{K3.Main.Conjecture2} for K3 surfaces case. 

\chapter{Towards general $K$-trivial varieties case}\label{high.dim.gen.sec}

We also wish to extend our picture for general collapsing of Ricci-flat K\"ahler metrics on $K$-trivial varieties, say $X$. 
Here, $K$-trivial\footnote{we employ this terminology in this work, since other popular name ``Calabi-Yau variety'' 
often assume stronger assumption which we do not impose here.}
simply means ($X$ is compact and) $K_X\equiv 0$. 
Unfortunately, for such general setting, we know much fewer things. 
Nevertheless, we would like to give a partial formulation in Problem~\ref{CY.conj}. 
Our setup is: 

\begin{Stpp}
We first fix $\mathcal{M}$ which is a connected finite type moduli (Deligne-Mumford) 
stack, with its coarse moduli space $\mathcal{M}\to M$, 
of polarized $K$-trivial varieties $(X,L)$
with only log-terminal singularities on $X$. We crucially assume \textit{boundedness}, 
i.e., there are no other $\mathbb{Q}$-Gorenstein degeneration of a family in the deformation class, 
which is still log terminal polarized $K$-trivial varieties. This boundedness problem is solved for 
compact hyperK\"ahler manifolds case as Theorem \ref{Moduli.Shimura} but not yet 
for other $K$-trivial varieties case as far as the authors are aware (cf., e.g., \cite[Remark 1.2]{Zha.WP}). 
We solved this for hyperK\"ahler case in Theorem~\ref{HK.bded}. As we inferred in the proof of Theorem~\ref{HK.bded}, 
this boundedness is easily reduced to the boundedness of {\it exponent} of the limit $\mathbb{Q}$-line bundle $L$, 
i.e., the minimal positive integer $m$ such that $L^{\otimes m}$ becomes genuine line bundle, by the 
effective base point freeness type theorem (\cite[1.1]{Kol.bp}, \cite[\S4]{Fjn2}, \cite[\S7]{Fjn1}, \cite[1.2]{Kol.bp}).\footnote{We thank Chen Jiang for pointing out this to us.} We also believe this issue can be approached at least partially 
through the analysis done in \cite[\S3]{SS}. 

\begin{Rem}
Here is a note added in January, 2021. 
A while after the appearance of this paper, 
this boundedness is proved unconditionally 
by Caucher Birkar \cite[Corollary 1.6]{Bir20}, 
as kindly informed by Chenyang Xu which we appreciate. 
In the case of irreducible holomorphic symplectic varieties, 
it gives another proof to Corollary \ref{HK.bded} but recall that our 
approach in  \S\ref{geom.prob.HK} 
via period map gives more precise information including 
the orbi-smoothness of the partial compactification 
i.e., Theorem \ref{Moduli.Shimura}. 
\end{Rem}
\end{Stpp}

\begin{Stpp}
Recall from the celebrated theorem of Yau \cite{Yau}, Eyssidieux-Guedj-Zeriahi \cite{EGZ.sing}, that each such 
$K$-trivial variety $X$ admits a unique Ricci-flat K\"ahler metric in the first Chern class $c_1(L)$ of the polarization $L$. 
On the other hand, Viehweg \cite{Vie} showed 
that $M$ is always quasi-projective. Note that the Bogomolov-Tian-Todorov theorem \cite{Bog1, Bog2, Tia.Bog, Tod89} 
essentially states that the Zariski open substack $\mathcal{M}^{\rm sm}$ 
of $\mathcal{M}$ parametrizing 
smooth $K$-trivial varieties is a \textit{smooth} Deligne-Mumford stack 
(given algebraization of a semi-universal deformation). 
\end{Stpp}

\begin{Stpp} We make the second assumption  (a version of unobstructedness) 
that $\mathcal{M}$ itself can be written as 
a quotient stack $[T/G]$ where 
$T$ is a \textit{smooth} quasi-projective scheme while $G$ is a finite group. 
Equivalently, $T$ admits a $G$-equivariant semi-universal deformations of 
complex log-terminal $K$-trivial varieties in our fixed class. 
Note that the possible singularities on $X$ could violate the Bogomolov-Tian-Todorov type unobstructedness and  
make this assumption nontrivial (cf., e.g., \cite{Gross.obst, Nam.obst}).  
\end{Stpp}

\begin{Stpp}
 We take a $G$-equivariant projective compactification $T\subset \bar{T}$ such that 
$(\bar{T},\partial \bar{T}:=\bar{T}\setminus T)$ is a divisorial log terminal pair (the notion invented by Shokurov 
and is very close to being simple normal crossing cf., e.g., \cite[Chapter~3]{KM}). 
Our final assumption in this chapter is that 
$(\bar{T},\partial \bar{T})$ has non-negative 
logarithmic Iitaka dimension.\footnote{
In the case when $M$ is locally Hermitian-symmetric, 
even log general typeness had been verified as a 
classical result of Mumford \cite{Mum77}. }  
\end{Stpp}

In the above setting, take the Gromov-Hausdorff compactification of $M$ 
and denote by $\overline{M}^{\rm GH}$. 
On the other hand, consider a $G$-equivariant log minimal 
(dlt\footnote{divisorially log terminal}) model 
of $(\bar{T},\partial \bar{T})$ and apply its Morgan-Shalen-Boucksom-Jonsson  
compactification of $\mathcal{M}\subset [\bar{T}/G]$ 
\cite[Appendix A.10, A.11]{TGC.II} which we denote as 
$M\subset \overline{M}^{\rm min.MSBJ}$. 
We expect this is independent of the choice of the dlt model of $\bar{T}$ 
(\cite[A.10]{TGC.II}). 

Then we would like to ask the following: 

\begin{Prob}[General $K$-trivial case]\label{CY.conj}
Is it true that the Gromov-Hausdorff compactification 
$\overline{M}^{\rm GH}$ is dominated by $\overline{M}^{\rm min.MSBJ}$? 
More precisely speaking, is there a continuous map $\Phi_{\rm alg}$ from 
$\overline{M}^{\rm min.MSBJ}$ to $\overline{M}^{\rm GH}$ preserving the open dense subspace $M$?
\end{Prob}

Although the authors have not confirmed this speculation beyond the 
cases of abelian varieties and some hyperK\"ahler varieties
 as our Theorems \ref{Ag.TGC.Satake.MS}, \ref{K3.Main.Conjecture.18.ok}, and \S\ref{HK.alg}, \S\ref{Alg.prob}, 
 we expect some relations with the mirror symmetry. 
Indeed, the attached boundary of the above-mentioned Morgan-Shalen-Boucksom-Jonsson  
compactification 
``around maximally unipotent points" are expected to be (and partially proved to be) a union of the 
$\mathbb{R}_{>0}$-quotients of the K\"ahler cones of 
``mirror duals" (cf., e.g., \cite{HosTak}). 
It would be interesting to see how this is related to our results on above Problem~\ref{CY.conj} for K3 surfaces in 
which we used the hyperK\"ahler rotations. 

\smallskip

The above Problem~\ref{CY.conj} also seems to inherit some spirit of 
 the Griffiths-Morrison conjecture\footnote{Very recently, 
 a paper \cite{GGLR} appeared.} 
on the existence of analog of 
Satake-Baily-Borel compactification of period domain quotients 
(cf.\  \cite[9.2, 9.5]{Griffiths}, \cite{Morrison}, \cite{GGLR}), 
although ours is rather 
an analog of a \textit{different} (non-algebraic) minimal Satake compactification and uses 
the idea of the log minimal model program. 

\subsection*{Note added in the revision} 
There are (partially ongoing) progresses on the understanding of 
collapses along type II  degeneration of K3 surfaces (\cite{Od20}, 
\cite{Osh}, and related works \cite{ABE}, \cite{HSZ}). 


\newpage

\addcontentsline{toc}{chapter}{Summary of notations}
\subsection*{Summary of Notations}

\begin{itemize}

\item ${\it CMet}_{1}$: the set of isometry classes of compact metric spaces with diameter one equipped with the Gromov-Hausdorff topology (\S \ref{Intro.K3})

\item $A_{g}$: the moduli space of $g$-dimensional 
principally polarized abelian varieties (\S \ref{Abel.sec}) 

\item $\overline{A_{g}}^{\rm T}$: the tropical geometric 
compactification of $A_{g}$ (\S \ref{Abel.sec}, also \cite{TGC.II}).

\item $\Lambda_{\rm K3}$: the K3 lattice (\S \ref{Mod.pol.K3})

\item $\Lambda_{2d}$: the orthogonal complement of 
a primitive vector with the squared norm $2d$ in $\Lambda_{\rm K3}$ (\S \ref{Mod.pol.K3})

\item $\mathcal{F}_{2d}$: 
the moduli space of polarized K3 surfaces 
of degree $2d$ (\S \ref{Mod.pol.K3})

\item $\overline{\mathcal{F}_{2d}}^{\rm Sat}$: 
the Satake compactification of $\mathcal{F}_{2d}$ for 
adjoint representation (\S \ref{K3.Sat.sec})

\item $\mathcal{F}_{2d}(l), \mathcal{F}_{2d}(p)$: 
boundary components of $\overline{\mathcal{F}_{2d}}^{\rm Sat}$ 
(\S \ref{K3.Sat.sec})

\item $\Phi_{\rm alg}$: the geometric realization map
 for $\mathcal{F}_{2d}$ (\S \ref{Alg.K3.statements.sec})

\item $\mathcal{M}_{\rm K3}$: the moduli space of 
K\"ahler K3 surfaces up to hyperK\"ahler rotation (\S \ref{STK.MK3})

\item $\overline{\mathcal{M}_{\rm K3}}^{\rm Sat}$: the Satake compactification of $\mathcal{M}_{\rm K3}$ for adjoint representation (\S \ref{STK.MK3})

\item $\mathcal{M}_{\rm K3}(-)$: boundary components of 
$\overline{\mathcal{M}_{\rm K3}}^{\rm Sat}$ (\S \ref{STK.MK3})

\item $\Gamma_{16}$: an even unimodular lattice of rank $16$, 
which is also an overlattice of the root lattice $D_{16}$. 
Also sometimes denoted by $D_{16}^{+}$ (\S \ref{STK.MK3}) 

\item $\Phi$: the geometric realization map
 for $\mathcal{M}_{\rm K3}$ (\S \ref{Geom.Meaning})

\item $M_{W}$: the moduli space of Weierstrass elliptic K3 surfaces 
(\S \ref{Weier.mod.sec})

\item $\overline{M_{W}}$: the GIT compactification of $M_W$ (\S \ref{Weier.mod.sec})

\item $M_{W}^{\rm nn}, M_{W}^{\rm seg}, M_{W}^{\rm nn, o}, M_{W}^{\rm seg,o}$: boundary components of $\overline{M_{W}}$
(\S \ref{Weier.mod.sec})

\item $\Phi_{\rm ML}$: the geometric realization map
 for $M_W$ (\S \ref{ML.limit})

\end{itemize}


\bigskip
\bigskip

\footnotesize 
\noindent
{\bf Contact of Yuji Odaka}:\\ 
email address: \footnotesize 
{\tt yodaka@math.kyoto-u.ac.jp} \\
Department of Mathematics, Kyoto university, 
Oiwake-cho, Kitashirakawa, Sakyo-ku, Kyoto city, Kyoto 606-8285. JAPAN \\

\medskip

\noindent
{\bf Contact of Yoshiki Oshima}:\\ 
email address: \footnotesize 
 {\tt oshima@ist.osaka-u.ac.jp} \\
Department of Pure and Applied Mathematics, Graduate School of Information Science and Technology, 
Osaka University, 1-5 Yamadaoka, Suita, Osaka 565-0871. JAPAN\\

\end{document}